\documentclass[11pt,english]{smfart} 
\usepackage{amsmath, amsfonts, amsthm, amssymb, amscd, mathrsfs}
\usepackage[mathcal]{euscript}
\usepackage{xcolor}
\usepackage {a4wide}

\usepackage[left]{lineno} 

\makeatletter
\renewcommand{\section}{\@startsection {section}{1}{\z@}%
             {-3.5ex \@plus -1ex \@minus -.2ex}%
             {2.3ex \@plus .2ex}%
             {\normalfont \Large \scshape\bfseries}}

\renewcommand{\subsection}{\@startsection{subsection}{2}{\z@}%
             {-3.25ex\@plus -1ex \@minus -.2ex}%
             {1.5ex \@plus .2ex}%
             {\normalfont\large\scshape\bfseries}}

\renewcommand{\subsubsection}{\@startsection{subsubsection}{2}{\z@}%
             {-3.25ex\@plus -1ex \@minus -.2ex}%
             {1.5ex \@plus .2ex}%
             {\normalfont\normalsize\scshape\bfseries}}
\makeatother

\theoremstyle{plain}
\newtheorem{theorem}{Theorem}[section]
\newtheorem{proposition}[theorem]{Proposition}
\newtheorem{propositiondefinition}[theorem]{Proposition and Definition}
\newtheorem{lemma}[theorem]{Lemma}

\newtheorem{definition}[theorem]{Definition}
\newtheorem{remark}[theorem]{Remark}
\numberwithin{equation}{section}
\newcommand \bei {\begin{itemize}}
\newcommand \eei {\end{itemize}}

\newcommand \phib {\phi_0}
\newcommand \psib {\phi_1}

\newcommand \be         {\begin{equation}}
\newcommand \bel {\be\label}

\newcommand \eps \epsilon

\newcommand \nablabd {{\overline \nabla}^\dag}

\newcommand \Gammad { {\Gamma^{\dag}}}
\newcommand \Gammab {\overline{\Gamma}}

\newcommand \coeff \kappa

\newcommand \Lcal {\mathcal L}

\newcommand \Ocal {\mathcal O}

\newcommand \Mb {\overline M}

\newcommand \gb {\overline g}

\newcommand \Rb {\overline R}

\newcommand \nablab {\overline \nabla}
\newcommand \nablad {{\nabla^\dag}}
\newcommand \delb {\overline \del}

\newcommand \Rd {{R^\dag}}
\newcommand \Rbd{\overline{R}^{\dag}}

\newcommand \tr {\text{tr \hskip-.cm}}
\newcommand \gd {{g^\dag}}
\newcommand \Kd {{K^\dag}}
\newcommand \Gd  {{G^\dag}}
\newcommand \Nd {{N^\dag}}
\newcommand \gbd {{\overline g^\dag}}

\newcommand \sigmad {{\sigma^\dag}}
\newcommand \Jd {{J^\dag}}

\newcommand \omegad{{\omega^\dag}}

\newcommand \del \partial

\newcommand \Rh {R_0}
\newcommand \Sh {R_1}

\newcommand \Ric {\text{Ric}}
\newcommand \Rm {\text{Rm}}
\newcommand \Acal   {\mathcal A}

\newcommand \RR         {\mathbb R}
\newcommand \ee         {\end{equation}}
\newcommand \la \langle
\newcommand \ra \rangle

\newcommand \rhoR {\theta}

\newcommand \gv{{g^{\ddag}}}
\newcommand \gvzero{{g^{\ddag}_0}}
\newcommand \gvone{{g^{\ddag}_1}}
\newcommand \Rv{{R^{\ddag}}}
\newcommand \nablav{{\nabla^{\ddag}}}
\newcommand \Gammav{{\Gamma^{\ddag}}}
\newcommand \Gv{{G^{\ddag}}}
\newcommand \Kv{{K^{\ddag}}}

\newcommand \phizerod {\phi_0^\dag}
\newcommand \phioned {\phi_1^\dag}

\newcommand \phizerodd{\phi_0^\ddag}
\newcommand \phionedd{\phi_1^\ddag}

\newcommand \rhozerod {\rho_0^\dag}
\newcommand \rhooned {\rho_1^\dag}

\newcommand \varrhozero {\varrho_0}
\newcommand \varrhoone {\varrho_1}

\newcommand \TJ T
\newcommand \TE {T^\dag}

\begin{document}
 
\title[Mathematical Validity of the Theory of Modified Gravity]
{\vskip1.cm MATHEMATICAL VALIDITY OF THE 
$f(R)$ THEORY 
\\
OF MODIFIED GRAVITY} 
\author[Philippe G. L{\smaller e}FLOCH and Yue MA]{\hskip2.cm Philippe G. L{\normalsize E}FLOCH
and Yue MA 
\newline  
\newline
\centerline{\sl Laboratoire Jacques-Louis Lions}
\newline 
\centerline{\sl Centre National de la Recherche Scientifique}
\newline 
\centerline{\sl Universit\'e Pierre et Marie Curie}
\newline 
\centerline{\sl 4 Place Jussieu, 75252 Paris, France} 
\newline  
}
\date{December 2014}

\maketitle


\vskip2.cm 
\begin{abstract} We establish here a well-posedness theory for the $f(R)$ theory of modified gravity, which is a generalization of Einstein's theory of gravitation. The scalar curvature $R$ of the spacetime, which arises in the integrand of the Einstein-Hilbert functional, is replaced here by an arbitrary function $f=f(R)$.  
The field equations involve {\sl up to fourth-order derivatives} of the unknown spacetime metric, and the main challenge is to understand the structure of these high-order derivative terms. First of all, we propose a  {\sl formulation of the initial value problem in modified gravity} when the initial data are prescribed on a spacelike hypersurface. In addition to the induced metric and second fundamental form of the initial slice and the initial matter content, an initial data set for modified gravity must also provide the {\sl spacetime scalar curvature} and its time-derivative. Next, in order to tackle the initial value problem, we introduce an {\sl augmented conformal formulation}, as we call it, in which the spacetime scalar curvature is regarded as an {\sl independent variable.} In particular, in the so-called wave gauge, we prove that the field equations of modified gravity reduce a {\sl coupled system of nonlinear wave-Klein-Gordon equations} with defocusing potential, whose main unknowns are the conformally-transformed metric and the scalar curvature, as well as the matter fields. Based on this novel formulation, we are able to establish the existence of {\sl maximal globally hyperbolic developments of modified gravity} when, for definiteness, the matter is represented by a scalar field. We analyze the so-called {\sl Jordan coupling} and we work with the {\sl Einstein metric}, which is conformally equivalent to the physical metric ---the conformal factor depending upon the (unknown) scalar curvature. Our analysis of these conformal field equations in the Einstein metric leads us to a rigorous validation of the theory of modified gravity. We derive quantitative estimates in suitable functional spaces, which are uniform in terms of the nonlinearity $f(R)$, and we prove that asymptotically flat spacetimes of modified gravity are {\sl `close' to Einstein spacetimes}, when the defining function $f(R)$ in the action functional of modified gravity is `close' to the Einstein-Hilbert integrand~$R$.  
\end{abstract}

\vfill

{\small  

\hskip6.cm  $
\aligned
\text{Email addresses:} \quad & \text{{\tt contact@philippelefloch.org}}
\\
& \text{{\tt ma@ann.jussieu.fr}}
\endaligned
$
}

\newpage 

\setcounter{secnumdepth}{2} \setcounter{tocdepth}{2}

\tableofcontents


\section{Introduction}

In recent years, new observational data have suggested that alternative theories of gravity, based on extensions of Einstein's field equations of general relativity, may be relevant in order to explain, for instance, the accelerated expansion of the Universe and certain instabilities observed in galaxies (without explicitly introducing notions such as 'dark energy' or 'dark matter'). Among these theories, the so-called {\bf $f(R)$--theory of modified gravity} (associated with a prescribed function $f(R)$ of the scalar curvature $R$) was recognized as a physically viable alternative to Einstein's theory. Despite the important role played by this theory in physics\footnote{The reader may refer to the physical and numerical literature \cite{BLMN,BransDicke,Garcia,MagnanoSokolowski,SalgadoMartinez}.}, the corresponding field equations have not been investigated by mathematicians yet. This is due to the fact that the modified gravity equations are significantly more involved than the Einstein equations: they contain up to {\sl fourth-order derivatives} of the unknown metric, rather than solely second-order derivatives.

Our purpose in this monograph is to initiate a rigorous mathematical study of the modified gravity equations and, as our main objectives, to 
\bei 

\item define a suitable {\bf notion of initial data set in modified gravity,} 

\item describe an {\bf initial value formulation} from an arbitrary spacelike hypersurface, 

\item establish the {\bf existence of a globally hyperbolic maximal developments} associated with any given initial data set,

\item and, importantly, to {\bf provide a rigorous validation} that the modified gravity theory is an `approximation' of Einstein's theory.

\eei

As already mentioned, in addition to the (second-order) Ricci curvature terms arising in Einstein equations, the field equations of the $f(R)$-theory involve fourth-order derivatives of the metric and, in fact, second-order derivatives of its scalar curvature. The corresponding system of partial differential equations (after a suitable choice of gauge) consists of a system of {\sl nonlinear wave equations}, which is significantly more involved than Einstein's system. Yet, a remarkable mathematical structure is uncovered in the present work, which is based on a novel formulation, referred to as the {\bf augmented conformal formulation}: we introduce an extended system in which,
both, the metric and its scalar curvature are regarded as {\sl independent unknowns};
we establish the well-posedness of the initial value problem for this augmented formulation, and finally recover
the solutions of interest for the original system of modified gravity.

Recall first that Einstein theory is based on the {\bf Hilbert-Einstein action} 
\be
\Acal_{\text{HE}}[\phi, g] := \int_M \Big( {R_g  \over 16 \pi} + L[\phi,g] \Big) \, dV_g
\ee
associated with a $(3+1)$--dimensional spacetime $(M,g)$ with Lorentzian signature $(-, +,+,+)$
whose canonical volume form is denoted by $dV= dV_g$. Here, and thereafter, we denote by $\Rm= \Rm_g$, $\Ric= \Ric_g$, and $R=R_g$ the Riemann, Ricci, and scalar curvature of the metric $g$, respectively. 
Observe that the above functional $\Acal_{\text{EH}}[g] $ is determined from the scalar curvature $R_g$ and a Lagrangian $L[\phi,g]$, the latter term describing the matter content represented by fields $\phi$ defined on $M$.

It is well-known that critical metrics for the action $\Acal_{\text{EH}}[g]$ ( at least formally)
satisfy Einstein's equation
\be
\label{Eq1-01}
G_g := \Ric_g - {R_g \over 2} \, g = 8 \pi \, T[\phi,g],
\ee
in which the right-hand side\footnote{Greek indices $\alpha, \beta=0, \ldots, 3$ represent spacetimes indices.}
\bel{Eq:12}
T_{\alpha\beta}[\phi,g] := -2 \, {\delta L \over \delta g^{\alpha\beta}} [\phi,g]  + g_{\alpha\beta}\, L[\phi,g]
\ee
is referred to as the stress-energy tensor of the matter model. In the vacuum,
 for instance, these equations are equivalent to the Ricci-flat condition
\bel{eqRicci}
\Ric_g = 0.
\ee

The `higher-order' gravity theory of interest is defined as follows. A smooth function $f: \RR \to \RR$ being prescribed, the {\bf action of the $f(R)$-modified gravity theory} read\footnote{See Buchdahl \cite{Buck}, as well as the earlier proposal by Brans and Dicke~\cite{BransDicke}.}:  
\bel{eq:action}
\Acal_{\text{NG}}[\phi,g] =: \int_M \Big( {f(R_g) \over 16 \pi} + L[\phi, g]\Big) \, dV_g,
\ee
whose critical points satisfy the {\bf field equations of modified gravity}
\bel{Eq1-14}
\aligned
N_g :&= f'(R_g) \, G_g - \frac{1}{2} \Big( f(R_g) - R_g f'(R_g) \Big) g
+  \big( g \, \Box_g   - \nabla d\big)  \big( f'(R_g) \big) 
\\
&= 8 \pi \, T[\phi,g].
\endaligned
\ee
The modified gravity tensor $N_g$ thus ``replaces" Einstein's tensor $G_g$, while the right-hand side\footnote{further discussed shortly below} is still given by  the same expression \eqref{Eq:12}.
Observe that, by taking the trace of \eqref{Eq1-14}, we deduce the scalar equation
\begin{equation}\label{Eq1-15}
\tr{N_g} = f'(R_g)R_g - 2f'(R_g) + 3\Box_gf'(R_g) = 8\pi\tr(T),
\end{equation}
which can be regarded as an {\sl evolution equation for the spacetime curvature} and will play an important role. 

Concerning the matter content, we point out (cf.~Section~\ref{sec original} below for the derivation) that the modified gravity tensor $N_g$ is {\sl divergence free,} that is,
\be
\nabla^\alpha N_{\alpha\beta} = 0,
\ee
so that the matter field satisfies the {\bf matter evolution equation} 
\bel{Eq1-15bis}
\nabla^\alpha T_{\alpha\beta} =0.
\ee

Furthermore, for the nonlinear theory to be a formal extension of the classical theory, we must assume that $f(R) \simeq R$ in the small curvature limit $R \to 0$. Since we will see later that the (positive)
sign of the coefficient $\coeff:=f''(0) >0$ is critical for nonlinear stability, it is convenient to set
\bel{Eq1-16derive}
f'(R) = 1 + \coeff \big( R  + \kappa \Ocal(R^2) \big),
\ee
which after integration yields
\bel{Eq1-16}
f(R) = R + \coeff \Big( \frac{R^2}{2}  + \kappa \Ocal(R^3) \Big),
\ee
where, by definition, the remainder $\Ocal(z^2)/z^2$ remains bounded when $z \to 0$ (uniformly in $\kappa$, if this parameter is taken to vary). In particular, the function $f$ is increasing and strictly convex in a neighborhood of the origin and, therefore, one-to-one\footnote{Of course, the term $\kappa \Ocal(R^3)$ in \eqref{Eq1-16} could be taken to vanish identically, which corresponds to the quadratic action $\int_M \Big( R_g + {\kappa \over 2} (R_g)^2  + 16 \pi L[\phi, g]\Big) \, dV_g$ often treated in the physical literature.}.

As we will see, in local coordinates, the field equations \eqref{Eq1-14} take the form of a nonlinear system of fourth-order partial differential equations (PDE's), while the Einstein equation \eqref{Eq1-01} leads to only second-order equations. Our  challenge in the present work is investigating the role of these fourth-order terms and generalizing the mathematical methods that were originally developed for the Einstein's equations. Furthermore, one formally would
expect to recover Einstein's theory by letting the coefficient $\coeff$ tend to zero. However, this limit is  {\sl very singular,} since this involves analyzing the convergence of a fourth-order system (with no definite type yet) to a system of second-order (hyperbolic-elliptic) PDE's. 


Before we can proceed further, we need to make an important observation concerning the modeling of the matter content of the spacetime. In the physics literature, the choice of the frame\footnote{From a mathematical standpoint, all frames are equivalent.} in which measurements are made is still somewhat controversial \cite{MagnanoSokolowski}. Yet, this issue is essential in order for properly formulating the coupling between the gravity equations and the matter fields. Two standpoints were proposed by physicists. In the so-called ``Jordan frame'', the original metric $g_{\alpha\beta}$ is considered to be the physically relevant metric, while in the ``Einstein frame'', the {\bf conformally-transformed metric} 
\be
\gd_{\alpha\beta}:= f'(R_g)g_{\alpha\beta}
\ee
 is
considered to be the physically relevant metric.  In the present work, these two approaches will be referred to as the ``Jordan coupling'' and ``Einstein coupling'' for the matter. Hence, the ``Jordan coupling'' refers to the minimal coupling of the matter field to the geometry of the spacetime (represented by the tensor $N_g$) described by the ``Jordan metric'' (i.e. the original metric) $g_{\alpha\beta}$. On the other hand, the ``Einstein coupling'' refers to the minimal coupling of the matter field to the geometry of the spacetime described by the metric $\gd_{\alpha\beta}$.

It is important to observe that different matter couplings lead to different physical theories, which may or may not be equivalent to each other. Of course, a given physical theory can also be expressed in various choices of metrics, that is, for the problem under consideration,  the ``Jordan coupling" could also be expressed with the ``Einstein metric'' $\gd_{\alpha\beta}$, while the ``Einstein coupling'' could also be stated in the ``Jordan metric'' $g_{\alpha\beta}$. 
A coupling which is minimal (in the sense that the action takes the decoupled form \eqref{eq:action}), in general, will no longer be minimal in another choice of metric. This, therefore, suggests that the Einstein metric is not the physical metric in the Jordan coupling theory, while the Jordan metric is not the physical metric in the Einstein coupling. 
This has apparently led to great confusion in the literature. In summary, observe that our notion of Einstein coupling is equivalent to the notion of Einstein frame adopted in \cite{MagnanoSokolowski}. 

In this work, we will treat the Jordan coupling but expressed in the (conformal) Einstein metric $\gd$. This coupling has the minimal form \eqref{eq:action}, but only in the original metric $g$. If one would insist on stating the problem in terms of the Einstein metric, then the coupling would not be minimal. This presentation appears to be optimal from the standpoint of establishing a well-posed theory for the initial value problem. 

Throughout this monograph, the matter model of interest is a {\bf real massless scalar field,} defined by its standard stress-energy tensor, and we consider the following two possible couplings:
\bel{eq-94}
\aligned
\TJ_{\alpha\beta}
& := \nabla_\alpha \phi \nabla_{\beta}\phi - \frac{1}{2} g_{\alpha\beta} g^{\delta\lambda} \nabla_\delta \phi \nabla_{\lambda}\phi,
\\
\TE_{\alpha\beta}
& := f'(R_g)\big(\nabla_\alpha \phi \nabla_{\beta}\phi - \frac{1}{2} g_{\alpha\beta} g^{\delta\lambda} \nabla_\delta \phi \nabla_{\lambda}\phi\big), 
\endaligned
\ee
and, for convenience, the Einstein coupling is stated in the Jordan metric. As should be expected from the above discussion, different choices of coupling lead to systems of PDE's of rather different nature. In fact, we will show that the Einstein coupling leads to an {\sl ill-defined Cauchy problem}. Therefore, in the rest of this section, we restrict attention to the Jordan coupling. 

We are now in a position to state our main result in a preliminary form. We recall that the initial value problem for the Einstein equations is classically formulated as follows. 
(We refer to the textbook by Choquet-Bruhat \cite{CB} for the terminology and historical references.) 
Given a Riemannian $3$--manifold $(\Mb,\gb)$ together with a $2$-covariant tensor field $K$
(plus suitable matter data) satisfying certain constraint equations, one seeks for a (globally hyperbolic)
 development of this so-called initial data set. By definition, such a development consists of a Lorentzian manifold $(M, g)$ satisfying the Einstein equations such that $\Mb$ is embedded in $M$ as a spacelike hypersurface with induced metric $\gb$ and second fundamental form $K$. The maximal (globally hyperbolic) development, by definition, is the unique development of the initial data set in which any such development can be  isometrically embedded.

In short, our formulation of the initial value problem for the theory of modified gravity is as follows. Since the field equations \eqref{Eq1-14} are fourth-order in the metric, additional initial data are required, which are denoted by $\Rh, \Sh$ and are specified on the initial slice $\Mb$: they represent the scalar curvature and the time derivative of the scalar curvature of the (to-be-constructed) spacetime. They must of course also satisfy certain Gauss-Coddazzi-type constraints. 
In addition, since the matter is modeled by a scalar field, say $\phi$, we also prescribe some initial data denoted by $\phib, \psib$, and representing the initial 
values of the scalar field and its time derivative, respectively. 
The prelimary statement above will be made more precise in the course of our analysis and all necessary 
terminology will be introduced. For definiteness, the results are stated with asymptotically flat data but this is unessential. 

\begin{theorem}[The Cauchy developments of modified gravity]
\label{thm main}
Consider the field equations \eqref{Eq1-14}  of modified gravity based on a function $f=f(R)$ satisfying
\eqref{Eq1-16derive} and assume that the matter is described by a scalar field with Jordan coupling \eqref{eq-94}.
Given an asymptotically flat initial data set\footnote{in the sense of Definition \ref{def 7-4-1}, below} 
 $(\Mb, \gb, K, \Rh,\Sh, \phib, \psib)$, there exists a unique  
maximal globally hyperbolic development $(M,g)$ of these data, which satisfies\footnote{in the sense of Definition~\ref{def:formu}, below} the modified gravity equations \eqref{Eq1-14} 
Furthermore, 
if an initial data set $(\Mb, \gb, K, \Rh,\Sh, \phib, \psib)$ for modified gravity is ``close'' (in a sense that will be made precise later on) to an initial data set $(\Mb,\gb',K',\phib', \psib')$ for the classical Einstein theory, 
 then the corresponding development of modified gravity is also close to the corresponding Einstein development. 
This statement is uniform in term of the gravity parameter $\kappa$ and modified gravity developments converge to Einstein developments when $\kappa \to 0$. 
\end{theorem}
 
Our results provide the first mathematically rigorous proof that the theory of modified gravity admit a well-posed Cauchy formulation and, furthermore, can be regarded as an approximation of Einstein's classical theory of gravity, as anticipated by physicists.

A key contribution of the present work is a re-formulation of the field equations of modified gravity
 as a system of second-order hyperbolic equations and, more precisely, a coupled {\sl system of wave-Klein-Gordon equations}\footnote{Wave-Klein-Gordon systems have brought a lot of attention in mathematical analysis: see, for instance,  Bachelot \cite{Bachelot88,Bachelot94}, Delort et al. \cite{Delort01,Delort04}, Katayama \cite{Katayama}, Lannes \cite{Lannes13}, and LeFloch and Ma \cite{LeFlochMa1} and the references therein.}. For further results on the mathematical aspects of the $f(R)$ theory, we refer to LeFloch and Ma \cite{LeFlochMa2}. We advocate here the use of {\sl wave coordinates associated with the Einstein metric} and our formulation in such a gauge leads us to propose the following definition. Importantly, our formulation contains an {\sl augmented variable} denoted by $\rho$, which represents the scalar curvature of the spacetime\footnote{specifically $\rho = \frac{1}{2}\ln f'(R_g)$}.

\begin{definition}
The {\bf augmented conformal formulation} of  the field equations of modified gravity
 (with Jordan coupling and in wave coordinates associated with the  Einstein metric) reads:
\begin{equation}
\label{eq:500} 
\aligned
\gd^{\alpha'\beta'}\del_{\alpha'}\del_{\beta'}\gd_{\alpha\beta}
&=  F_{\alpha\beta}(\gd;\del \gd,\del \gd)  - 12\del_{\alpha}\rho\del_{\beta}\rho
+
V(\rho) 
\gd_{\alpha\beta}  -16\pi\del_{\alpha}\phi\del_{\beta}\phi,
\\
\gd^{\alpha'\beta'}\del_{\alpha'}\del_{\beta'}\phi
& = -2\gd^{\alpha\beta}\del_{\alpha}\phi\del_{\beta}\rho,
\\
\gd^{\alpha'\beta'}\del_{\alpha'}\del_{\beta'}\rho- {\rho \over 3 \kappa}
&= 
W(\rho)
- \frac{4\pi}{3e^{2\rho}}\gd^{\alpha\beta}\del_{\alpha}\phi\del_{\beta}\phi,
\\
\gd^{\alpha\beta}\Gammad_{\alpha\beta}^{\lambda}
&= 0, 
\endaligned
\end{equation}
in which $F_{\alpha\beta}(\gd;\del \gd,\del \gd)$ are quadratic expressions (defined in Section~3 below), 
$\del \gd$ is determined by the Ricci curvature, and the function $V=V(\rho)$ and $W=W(\rho)$ are of quadratic order as $\rho \to 0$.
\end{definition}

Clearly, we recover Einstein equations by letting $\kappa \to 0$ and thus $f(R) \to R$. Namely, we will show that $\rho \to 0$ so that \eqref{eq:500} reduces to the standard formulation in wave coordinates \cite{CB}. In particular, in this limit, we do recover the expression $R = 8\pi \, \nabla_\alpha \phi \nabla^\alpha \phi$ of the scalar curvature in terms of the norm of the scalar field.

An outline of the rest of this monograph is as follows. In Sections~2 and 3, we formulate the initial value problem first in the Jordan metric and then in the Einstein metric. We find that the second formulation is simpler, since the Hessian of the scalar curvature is eliminated by the conformal transformation.
Furthermore, we demonstrate that the Einstein coupling is ill-posed.
 The conformal formulation is analyzed in Section~4, where the wave gauge is introduced and the wave-Klein-Gordon structure of the field equations is exhibited. Section~5 contains one of our main result and proposes an augmented formulation of the conformal system of modified gravity. The local existence theory with bounds that are uniform in $\coeff$
is developed in Sections 6 to 8 and leads us in Section 9 to our main statement concerning the comparison between the modified and the classical theories.


\newpage 

\section{Formulation of the Cauchy problem in the Jordan metric}
\label{sec original}

\subsection{The $3+1$ decomposition of spacetimes}

In this section, we formulate the initial value problem for the modified gravity system, by prescribing suitable initial data on a spacelike hypersurface. Until further notice, our setup is as follows, in which we follow the textbook~\cite{CB} for classical gravity. We are thus interested in time-oriented spacetimes $(M,g)$ endowed with a Lorentzian metric $g$ with signature $(-, +, +, +)$, which are
homeomorphic to $[0,t_{\max}) \times M_t$ and admit a global foliation by spacelike hypersurfaces $M_t \simeq \{t\} \times \Mb$. The foliation is determined by a time function $t: M \to [0, t_{\max})$ and a three-dimensional manifold $\Mb$ and, throughout, we assume that
\bel{eq:globalhyper}
M \text{ is globally hyperbolic and every } M_t \text{ is a Cauchy surface,}
\ee
which, for instance, ensures that a wave equation posed on any such Cauchy surface enjoys the usual local existence and uniqueness property. (See \cite{CB,HE} for the definitions.) 

We introduce local coordinates adapted to the above product structure, that is, $(x^{\alpha}) = (x^0=t, x^i)$, and we choose the basis of vectors $(\del_i)$ as the `natural frame' of each slice $M_t$. This also defines the 'natural frame' $(\del_t, \del_i)$ on $M$. Then, by definition, the 'Cauchy adapted frame' is $e_i = \del_i$ and $e_0 = \del_t - \beta^i\del_i$, where $\beta = \beta^i\del_i$ a time-dependent field, tangent to $M_t$ and called the {\bf shift vector,} and we impose the restriction that $e_0$ is orthogonal to each $M_t$. 

Now we introduce the dual frame $(\theta^{\alpha})$ of the Cauchy adapted frame $(e_{\alpha})$ by setting 
\be
\theta^0 := dt, \qquad \theta^i := dx^i + \beta^i dt.
\ee
In particular, in our notation, the spacetime metric reads 
\be
g = -N^2 \theta^0\theta^0 + g_{ij}\theta^i\theta^j,
\ee
where the function $N>0$ is referred to as the {\bf lapse function} of the foliation. 

Next, the {\bf structure coefficients} of the Cauchy adapted frame are defined by
\be
C_{0j}^i = - C_{j0}^i = \del_j\beta^i
\ee
and the corresponding Levi-Civita connection associated with $g$ is represented by a family $\omega_{\gamma\alpha}^{\beta}$ defined by 
\be
\nabla e_{\alpha} = \omega_{\gamma\alpha}^{\beta}\theta^{\gamma}\otimes e_{\beta}
\ee
and, consequently, 
\be
\nabla \theta^\alpha = -\omega_{\gamma\beta}^{\alpha}\theta^{\gamma}\otimes\theta^{\beta}.
\ee

We denote by $\gb=\gb_t$ the induced Riemannian metric associated with the slices $M_t$ and by $\nablab$ the Levi-Civita connection of $\gb$, whose Christoffel symbols (in the natural frame) are denoted by $\Gammab_{ab}^c$. We can also introduce the second fundamental form $K = K_t$ defined by 
\be
K(X,Y) := g(\nabla_X n, Y) 
\ee
for all vectors $X,Y$ tangent to the slices $M_t$, where $n$ denotes the future-oriented, unit normal to the slices. In the Cauchy adapted frame, it reads
\be
K_{ij} = -\frac{1}{2N}\big(<e_0,g_{ij}> - g_{lj}\del_i\beta^l - g_{il}\del_j\beta^l \big).
\ee
Here, and throughout this monograph, we use the notation $<e_0,g_{ij}>$ for the action of the vector field $e_0$ on the function $g_{ij}$. Next, we define the {\bf time-operator} $\delb_0$ acting on a two-tensor defined on the slice $M_t$ by
\be
\delb_0 T_{ij} = <e_0,T_{ij}> - T_{lj}\del_i\beta^l - T_{il}\del_j\beta^l
\ee
which is again a two-tensor on $M_t$. Therefore, with this notation, we thus have 
\be
K = -\frac{1}{2N}\delb_0 \gb.
\ee

Standard calculations (Cf.~Section VI.3 in \cite{CB}) yield us the following formulas for the connection in terms of the $(3+1)$-decomposition:
\be
\aligned
& \omega_{00}^0 = N^{-1}<e_0,N>, 
\\
& \omega_{00}^i = Ng^{ij}\del_j N, \qquad \qquad \omega_{0i}^0 = \omega_{i0}^0 = N^{-1}\del_i N,
\\
& \omega_{ij}^0 = \frac{1}{2}N^{-2}\big(<e_0,g_{ij}> - g_{hj}\del_i\beta^h - g_{ih}\del_j\beta^k\big) = -N^{-1}K_{ij},
\\
&\omega_{0j}^i = -NK_j^i + \del_j\beta^i, \qquad \qquad  \omega_{j0}^i = -NK^i_j,
\\
& \omega_{jk}^i = \Gammab_{jk}^i.
\endaligned
\ee
Here, $\Gammab_{jk}^i$ denotes the Christoffel symbol of the connection $\nablab$ in the natural coordinates $\{x^i\}$.

It is also a standard matter to derive the Gauss--Codazzi equations for each slice:
\be
\label{GC-1}
\aligned
& R_{ij,kl} = \Rb_{ij,kl} + K_{ik} k_{lj} - K_{il} \, K_{kj},
\\
& R_{0i,jk} = N(\nablab_j K_{ki} - \nablab_k K_{ji}),
\\
& R_{0i,0j} = N(\delb_0K_{ij} + NK_{ik}K^k_j + \nablab_i\del_jN).
\endaligned
\ee
In addition by suitable contractions of these identities, we arrive at
\begin{subequations}\label{GC-2}
\begin{equation}\label{GC-2a}
R_{ij} = \Rb_{ij} - \frac{\delb_0K_{ij}}{N} + K_{ij}K^l_l - 2K_{il}K^{l}_j - \frac{\nablab_i\del_jN}{N},
\end{equation}
\begin{equation}\label{GC-2b}
R_{0j} =N\big( \del_jK_l^l - \nablab_l K^l_j\big),
\end{equation}
and, for the $(0,0)$-component of the Einstein curvature, 
\begin{equation}\label{GC-2c}
 G_{00} = \frac{N^2}{2} \big(\Rb - K_{ij}K^{ij} + (K_l^l)^2\big).
\end{equation}
\end{subequations}
These equations clarify the relations between the geometric objects in the spacetime $M$ and the ones of the slices $M_t$.

The equation \eqref{GC-2a} yields the evolution of the tensor $K$ and together with the definition $\del_0 \gb = -2NK$, we thus find the system for the metric and second fundamental form: 
\begin{equation}\label{GC-evolution}
\aligned
&\delb_0\gb_{ij} = -2NK_{ij}.
\\
&\delb_0 K_{ij} = N\big(\Rb_{ij} - R_{ij}\big) + NK_{ij}K_l^l - 2NK_{il}K^l_j - \nablab_i\del_jN. 
\endaligned
\end{equation} 


\subsection{Evolution and constraint equations}

Our objective now is to combine the equations \eqref{GC-2} and the field equations \eqref{Eq1-14} in order to derive the fundamental equations of modified gravity.

We recall first some elementary identities about the Hessian of a a function expressed in the Cauchy adapted frame. Given any smooth function $f: M \to 0$, we can write
$$
\nabla df = \nabla \big(<e_{\gamma},f>\theta^{\gamma}\big) = <e_{\beta},<e_{\alpha},f>>\theta^{\beta}\otimes \theta^{\alpha}- <e_{\gamma},f>\omega_{\beta\alpha}^{\gamma}\theta^{\beta}\otimes\theta^{\alpha}.
$$
First of all, we compute the components
\bel{eq:formHess}
\aligned
\nabla_i\nabla_j f
&= \del_i\del_j f - < e_\gamma, f > \omega_{ij}^{\gamma} = \del_i\del_jf - \Gammab_{ij}^k\del_k f - \omega_{ij}^0<e_0,f>
\\
& = \nablab_i\nablab_j f + K_{ij}N^{-1}(\del_t - \beta^l\del_l)f
\\
&= \nablab_i\nablab_j f + K_{ij}\mathcal{L}_nf,
\endaligned
\ee
where $\mathcal{L}_n$ is the Lie derivative associated with the normal unit vector of the slice $M_t$. Then, for the other components, we find
$$
\aligned
\nabla_j\nabla_0f = \nabla_0\nabla_jf  &= <e_0,\del_jf> - <e_0,f>\omega_{0j}^0 - \del_i f \omega_{0j}^i
\\
&=(\del_t - \beta^i\del_i)\del_jf - (\del_t-\beta^i\del_i)f\,N^{-1}\del_jN + N\del_i f K^i_j - \del_if\del_j\beta^i
\\
&=(\del_t - \beta^i\del_i)\del_jf - (\del_t-\beta^i\del_i)f\,\del_j\ln{N} + N\del_i f K^i_j - \del_if\del_j\beta^i
\\
&=\del_j(\del_t -\beta^i\del_i)f - (\del_t-\beta^i\del_i)f\,\del_j\ln{N} + N\del_i f K^i_j
\\
&=N\del_j\big(N^{-1}(\del_t - \beta^i\del_i)f\big) + N K^i_j\del_i f 
\endaligned
$$
and 
$$
\aligned
\nabla_0\nabla_0f &= (\del_t-\beta^i\del_i)(\del_t-\beta^i\del_i)f - (\del_t - \beta^i\del_i)f\, \omega_{00}^0 - \del_if \, \omega_{00}^i
\\
&=(\del_t-\beta^i\del_i)(\del_t-\beta^i\del_i)f - (\del_t - \beta^i\del_i)f\, (N^{-1}\del_j N)- \del_if \, (N^{-1}(\del_t - \beta^i\del_i)N)
\\
&=(\del_t-\beta^i\del_i)(\del_t-\beta^i\del_i)f - (\del_t - \beta^i\del_i)f\cdot \del_j \ln{N} -  \del_if \, (\del_t - \beta^i\del_i)\ln{N}.
\endaligned
$$

In particular, the trace of the Hessian of a function is the so-called D'Alembert operator, expressed in the Cauchy adapted frame as
$$
\aligned
\Box_g f &= g^{\alpha\beta}\nabla_{\alpha}\nabla_{\beta}f = -N^2\nabla_0\nabla_0f + g^{ij}\nabla_i\nabla_j f
\\
&=-N^{-2}\nabla_0\nabla_0f + \gb^{ij}\nablab_i\nablab_j f + \gb^{ij}K_{ij}N^{-1}(\del_t - \beta^l\del_l)f
\\
&=-N^{-2}\nabla_0\nabla_0f + \Delta_{\gb} f + \gb^{ij}K_{ij}\mathcal{L}_n f,
\endaligned
$$
where $\Delta_{\gb}f$ is the Laplace operator associated with the metric $\gb$.

To proceed wiht the formulation of the field equation \eqref{Eq1-14}, we need first to rewrite it in a slightly different form, by defining the tensor: 
\be
{E_g}_{\alpha\beta} := {N_g}_{\alpha\beta} - \frac{1}{2}\tr(N_g)g_{\alpha\beta},
\ee
where $\tr(\cdot)$ is the trace with respect to the metric $g$. Then, we have the following relation in terms of the Ricci tensor:  
\be
{E_g}_{\alpha\beta} = f'(R_g)R_{\alpha\beta} + \frac{1}{2}f'(R)W_1(R_g)g_{\alpha\beta} - \bigg(\frac{1}{2}g_{\alpha\beta}\Box_g + \nabla_{\alpha}\nabla_{\beta}\bigg)f'(R_g),
\ee
where we have introduced the {\bf function $W_1$} by 
\bel{def:Wone}
W_1(r) := \frac{f(r) - rf'(r)}{f'(r)}, \qquad r \in \RR.
\ee

In view of \eqref{Eq1-14}, we know that $E_g$ satisfies the field equations 
\begin{equation}
{E_g}_{\alpha\beta} = 8\pi \big(T_{\alpha\beta} - \frac{1}{2}\tr(T)g_{\alpha\beta}\big)
=: 8 \pi H_{\alpha\beta}, 
\end{equation} 
where we have introduce the new {\bf matter tensor} $H_{\alpha\beta}$. 
More precisely, it will be most convenient to introduce, for different components, a different form of the equations, that is, we write the field equations as: 
\begin{equation}\label{eq main reformed}
\aligned
&{E_g}_{ij} = 8\pi H_{ij},
\\
&{E_g}_{0j} = 8\pi H_{0j},
\\
&{N_g}_{00} = 8\pi T_{00},
\endaligned
\end{equation}
or, equivalently, 
\begin{subequations}\label{eq main reformed'}
\begin{equation}\label{eq main reformed'a}
\aligned
&f'(R_g)R_{ij} + \frac{1}{2}f'(R)W_1(R_g)g_{ij} 
 - \Big(\frac{1}{2}g_{ij}\Box_g + \nabla_i\nabla_j\Big)f'(R_g) 
\\
& =  8\pi \big(T_{ij} - \frac{1}{2}\tr(T)g_{ij}\big),
\endaligned
\end{equation}
\begin{equation}\label{eq main reformed'b}
f'(R_g)R_{0j}  - \nabla_0\nabla_jf'(R_g) =  8\pi T_{0j},
\end{equation}
\begin{equation}\label{eq main reformed'c}
f'(R_g) \, {G_g}_{00} - \frac{1}{2} f'(R_g)W_1(R_g) g_{00}
+  \big( g_{00} \, \Box_g   - \nabla_0 \nabla_0 \big)  \big( f'(R_g) \big) = 8 \pi T_{00}.
\end{equation} 
\end{subequations}

For completeness we check the following equivalence. 

\begin{lemma} If a metric $g_{\alpha\beta}$ and a matter tensor $T_{\alpha\beta}$ satisfy the field equations \eqref{Eq1-14}, then they also satisfy \eqref{eq main reformed}. The converse is also true. 
\end{lemma}

\begin{proof} The equations \eqref{eq main reformed} are clearly equivalent to 
$$
\aligned
&\big(N_g - 8\pi T\big)_{ij} = -\frac{1}{2}g_{ij}\tr(8\pi T - N_g),
\\
&\big(N_g - 8\pi T\big)_{ij} = 0,
\\
&{N_g}_{00} - 8\pi T_{00} = 0.
\endaligned
$$
By taking the trace of the tensor $(N_g - 8\pi T)$, we find 
$
\tr(N_g - 8\pi T) = -\frac{3}{2}\tr(8\pi T - N_g)
$
and thus $\tr(8\pi T - N_g)=0$, which proves the result. 
\end{proof}

Hence, in view of \eqref{eq main reformed'a} and by using \eqref{eq:formHess}, we have arrived at the field equations of modified gravity in a preliminary form. First of all, we have 
\begin{subequations}\label{eq field 2}
\begin{equation}\label{eq field 2a}
\aligned
R_{ij}&= \frac{1}{f'(R_g)}\bigg({E_g}_{ij} - \frac{1}{2}f'(R_g)W_1(R_g)g_{ij} + \frac{1}{2}\big(g_{ij}\Box_g + 2\nabla_i\nabla_j \big)f'(R_g)\bigg)
\\
&= \frac{E_{ij}}{f'(R_g)} - \frac{1}{2}W_1(R_g)\gb_{ij}
 + \frac{\big(\gb_{ij}\Delta_{\gb} + 2\nablab_i\nablab_j\big)f'(R_g) }{2f'(R_g)}
\\
 &- \frac{\gb_{ij}\nabla_0\nabla_0f'(R_g)}{2N^2f'(R_g)} + \big(K_{ij} + \frac{1}{2}\gb_{ij}K\big)\mathcal{L}_n\ln(f'(R_g)),
\endaligned
\end{equation}
where $K := \gb^{ij}K_{ij}$ is the trace of $K$ with respect to $\gb$. We also have
\begin{equation}\label{eq field 2b}
\aligned
R_{0j} &= \frac{1}{f'(R_g)}\bigg({N_g}_{0j} + \nabla_0\nabla_jf'(R_g)\bigg)
\\
       &= \frac{{N_g}_{0j}}{f'(R_g)}
       + \frac{N\del_j\big(N^{-1}(\del_t-\beta^i\del_i)f'(R_g)\big)}{f'(R_g)}
       + NK_j^i\del_i\big(\ln(f'(R_g))\big)
\\
&= \frac{{N_g}_{0j}}{f'(R_g)} + \frac{N\del_j\big(f''(R_g)\mathcal{L}_nR_g\big)}{f'(R_g)} + NK_j^i\del_i\big(\ln(f'(R_g))\big),
\endaligned
\end{equation}
and, finally, 
\begin{equation}\label{eq field 2c}
\aligned
G_{00} &= \frac{1}{f'(R_g)}\big({N_g}_{00} + \frac{1}{2}f'(R_g)W_1(R_g)g_{00} - (g_{00}\Box_g - \nabla_0\nabla_0)f'(R_g)\big)
\\
       &= \frac{{N_g}_{00}}{f'(R_g)} + \frac{1}{2}g_{00}W_1(R_g) - \frac{g_{00}}{f'(R_g)}\big(\Delta_{\gb}
        + \gb^{ij}K_{ij}\mathcal{L}_n\big)f'(R_g).
\endaligned
\end{equation}
\end{subequations}

Next, by combining \eqref{eq field 2a} with \eqref{GC-2a}, \eqref{eq field 2b} with \eqref{GC-2b} and \eqref{eq field 2c} with \eqref{GC-2c}, the evolution equations and constraint equations for the system of
modified gravity are formulated as follows: 
\begin{equation}
\aligned
\delb_0 K_{ij}&= N\Rb_{ij}- NR_{ij} + NK_{ij}K_l^l - 2NK_{il}K^l_j - \nablab_i\del_jN
\\
&= N\Rb_{ij} + NK_{ij}K_l^l - 2NK_{il}K^l_j - \nablab_i\del_jN
\\
& \quad -\frac{NE_{ij}}{f'(R_g)} + \frac{N}{2}W_1(R_g)\gb_{ij}
 - \frac{N\big(\gb_{ij}\Delta_{\gb} + 2\nablab_i\nablab_j\big)f'(R_g) }{2f'(R_g)}
\\
 &+ \frac{N\gb_{ij}\nabla_0\nabla_0f'(R_g)}{2N^2f'(R_g)} - N\big(K_{ij} + \frac{1}{2}\gb_{ij}K\big)\mathcal{L}_n\ln(f'(R_g)),
\\
\\
\delb_0\gb_{ij} &= -2NK_{ij},
\endaligned
\end{equation} 
\begin{equation}
\aligned
\Rb - K_{ij}K^{ij}+(K_j^j)^2
 &=  \frac{2{N_g}_{00}}{N^2f'(R_g)} + \frac{2\Delta_{\gb}f'(R_g)}{f'(R_g)}
  \\
&  + 2\gb^{ij}K_{ij}\mathcal{L}_n\big(\ln f'(R_g)\big) - W_1(R_g),
\endaligned
\end{equation}
and 
\begin{equation}
\del_j K_i^i - \nablab_i K_j^i = \frac{{N_g}_{0j}}{Nf'(R_g)}
+ \frac{\del_j\big(f''(R_g)\mathcal{L}_nR_g\big)}{f'(R_g)} + K^i_j\del_i \big(\ln(f'(R_g))\big).
\end{equation}

It remains to consider the coupling with the matter field, described by the stress-energy tensor $T_{\alpha\beta}$. Recall that the equations read
${N_g}_{\alpha\beta} = 8\pi T_{\alpha\beta}$,
where $T_{\alpha\beta} = \TJ_{\alpha\beta}$ for the Jordan coupling
and $T_{\alpha\beta} = \TE_{\alpha\beta}$ for the Einstein coupling. We also define the mass density $\sigma$ and the momentum vector $J$ (measured by an observer moving orthogonally to the slices)
by the relations
\be
\sigma := N^{-2}T_{00}, \qquad J_j := - N^{-1}T_{0j}.
\ee
We can thus conclude this section and introduce a definition suitable for modified gravity.

\begin{propositiondefinition}
\label{2:defconstr}
The equations for modified gravity in the Cauchy adapted frame $\{e_0,e_1,e_2,e_3\}$
decompose as follows: 

\noindent 1. {\bf Evolution equations:}
\begin{equation}\label{eq evolution}
\aligned
\delb_0 K_{ij}&= N\Rb_{ij} + NK_{ij}K_l^l - 2NK_{il}K^l_j - \nablab\del_jN
\\
                  &-\frac{8\pi N\big (T_{ij} - \frac{1}{2}g_{ij}\tr(T)\big)}{f'(R_g)} + \frac{N}{2}W_1(R_g)\gb_{ij}
\\
&
 - \frac{N\big(\gb_{ij}\Delta_{\gb} + 2\nablab_i\nablab_j\big)f'(R_g) }{2f'(R_g)}
+ \frac{N\gb_{ij}\nabla_0\nabla_0f'(R_g)}{2N^2f'(R_g)} 
\\
 &
- N\big(K_{ij} + \frac{1}{2}\gb_{ij}K\big)\mathcal{L}_n\ln(f'(R_g)),
\\
\delb_0\gb_{ij} &= -2NK_{ij}.
\endaligned
\end{equation}

\noindent 2. {\bf Hamiltonian constraint:}
\begin{equation}\label{eq constraint H}
\aligned
\Rb - K_{ij}K^{ij}+(K_j^j)^2
&= \frac{16\pi\sigma}{f'(R_g)} + \frac{2\Delta_{\gb}f'(R_g)}{f'(R_g)}
\\
&
+ 2\gb^{ij}K_{ij}\mathcal{L}_n\big(\ln f'(R_g)\big) - W_1(R_g),
\endaligned
\end{equation}

\noindent 3. {\bf Momentum constraint:}
\begin{equation}\label{eq constraint M}
\aligned
\del_j K_i^i - \nablab_i K_j^i = - \frac{8\pi J_i}{f'(R_g)}
+ \frac{\del_j\big(f''(R_g)\mathcal{L}_nR_g\big)}{f'(R_g)} + K^i_j\del_i \big(\ln(f'(R_g))\big).
\endaligned
\end{equation}
\end{propositiondefinition}

Observe that, in the classical gravity theory, the factor $f'(R_g)$ is constant and equal to unit, so that the terms containing $f'(R_g)$ in the right-hand sides of the constraint equations \eqref{eq constraint H} and \eqref{eq constraint M} vanish identically; consequently, we can recover here the standard equations \eqref{GC-3} given below.

These new constraint equations are very involved compared with the classical ones: they contain fourth-order derivatives of the metric $\gb$ and, more precisely, second-order derivatives of the scalar curvature $R_g$. In particular, we can not recognize directly the elliptic nature of the classical constraint equations.

\begin{remark} 
Recall here the constraint equations for the {\sl classical} theory of general relativity, when the Einstein equations $G_{\alpha\beta} = 8 \pi T_{\alpha\beta}$ are imposed: the last two equations in \eqref{GC-2} yield
\be
\label{GC-3}
\aligned
&  \Rb + K_{ij} \, K^{ij} - (K_i^i)^2   = 16 \pi \sigma,
\\
& \nablab^i K_{ij} - \nablab_j K_l^l  = 8 \pi \, J_j,
\endaligned
\qquad \qquad \text{(when $f' \equiv 0$).}
\ee
\end{remark} 

\subsection{The divergence identity}

As in the classical gravity theory, we expect that the matter should be divergence-free $\nabla^{\alpha}T_{\alpha\beta} = 0$, which is now proven.

\begin{lemma}[The divergence identity in modified gravity]
\label{prop E_bianchi}
The contracted Bianchi identities
\bel{eq E-bianchi}
\nabla^\alpha R_{\alpha\beta} = \frac{1}{2} \nabla_\beta R
\ee
imply the divergence-free property for the modified gravity tensor 
\begin{equation}
\nabla^\alpha {N_g}_{\alpha\beta} = 0.
\end{equation}
\end{lemma}

\begin{proof}
The following calculation holds in an arbitrary (possibly only local) natural frame. We compute the three relevant terms:
\begin{align*}
\nabla^\alpha (\nabla_\alpha\nabla_\beta f'(R)-g_{\alpha\beta}\Box_g f'(R))
&=(\nabla^\alpha \nabla_\alpha\nabla_{\beta} - \nabla_{\beta}\nabla^{\lambda}\nabla_{\lambda})f'(R)
\\
&=(\nabla_\alpha\nabla_\beta\nabla^\alpha -\nabla_\beta\nabla_\alpha\nabla^\alpha )(f'(R))
\\
&=[\nabla_\alpha,\nabla_\beta](\nabla^\alpha (f'(R)))
=R_{\alpha\beta}\nabla^\alpha (f'(R)),
\end{align*}
then
\begin{align*}
\nabla^\alpha (f'(R) \,
R_{\alpha\beta})&=R_{\alpha\beta}\nabla^\alpha (f'(R))+f'(R)\,\nabla^\alpha R_{\alpha\beta}\\
&=R_{\alpha\beta}\nabla^\alpha (f'(R))+\frac{1}{2}f'(R)\,\nabla_{\beta}R,
\end{align*}
and, finally,
\begin{align*}
\nabla^\alpha \Big( \frac{1}{2} f(R) \, g_{\alpha\beta} \Big)
&=\frac{1}{2}\nabla_{\beta}(f(R))
=\frac{1}{2}f'(R)\,\nabla_{\beta}R
\end{align*}
Combining these three identities together yields us the desired identity.
\end{proof}

As a first application of Lemma~\ref{prop E_bianchi}, we now determine which coupling (formulated in the Jordan metric as far as this section is concerned) is mathematically sound. On one hand, consider first the Jordan coupling, corresponding to 
\bel{eq:Jcou}
 \TJ_{\alpha\beta}
= \del_{\alpha}\phi\del_{\beta}\phi -\frac{1}{2}g_{\alpha\beta}g^{\alpha'\beta'}\del_{\alpha'}\phi\del_{\beta'}\phi. 
\ee
By the field equation ${N_g}_{\alpha\beta} = 8\pi \TJ_{\alpha\beta}$ and \eqref{eq E-bianchi}, we find
$\nabla^{\alpha\beta}\TJ_{\alpha\beta} = 0$
and, after some calculations, 
\be
\del_{\beta}\phi \Box_g\phi = 0.
\ee
Consequently, if the scalar field $\phi$ satisfies the wave equation 
\be
\Box_g \phi = 0,
\ee
then the tensor $\TJ_{\alpha\beta}$ is divergence-free, as required.
We conclude that we need to solve a single scalar equation for the evolution of the matter.

On the other hand, if we assume the Einstein coupling
\be
T_{\alpha\beta}
= \TE_{\alpha\beta}
= f'(R_g)\big(\del_{\alpha}\phi\del_{\beta}\phi -\frac{1}{2}g_{\alpha\beta}g^{\alpha'\beta'}\del_{\alpha'}\phi\del_{\beta'}\phi\big),
\ee
then the field equation ${N_g}_{\alpha\beta} = \TE_{\alpha\beta}$ together with \eqref{eq E-bianchi} lead us to
$\nabla^{\alpha}\TE_{\alpha\beta} = 0$, 
which now reads 
$$
f''(R_g)T_{\alpha\beta}\nabla^{\alpha}R_g + f'(R_g)\del_{\beta}\phi\Box_g\phi = 0.
$$
This (vectorial) equation can be written as
$$
\del_{\beta}\phi\,\Box_g\phi = \frac{f''(R_g)}{f'(R_g)}\big(\del_{\alpha}\phi\del_{\beta}\phi - \frac{1}{2}g_{\alpha\beta}g^{\alpha'\beta'}\del_{\alpha'}\phi\del_{\beta'}\phi\big)\nabla^{\alpha}R_g
$$
or equivalently
\be
\bigg(\frac{f''(R_g)}{f'(R_g)}\del_{\alpha}\phi \nabla^{\alpha}R - \Box_g\phi\bigg)
\,
\nabla_\beta \phi
= \frac{f''(R_g)}{2f'(R_g)} \,\bigg(  g^{\alpha'\beta'}\del_{\alpha'}\phi\del_{\beta'}\phi \bigg)
\nabla_\beta R.
\ee
Now, for general initial data, this is an {\sl over-determined}\footnote{Unless we impose the unnatural restriction that $\nabla \phi$ and $\nabla R$ are co-linear}
 partial differential system (since the unknown of this vectorial system is a single scalar field): this strongly suggests that the Einstein coupling is not 
mathematically (nor physically) meaningful. 

Consequently, from now on, {\sl we focus our attention on the Jordan coupling.}


\subsection{The initial value problem for modified gravity}

Before we can formulate the Cauchy problem for the system \eqref{Eq1-14}, we need to specify the stress-energy tensor. In agreement with our discussion in the previous section, we assume a scalar field and the Jordan coupling \eqref{eq:Jcou}  
and the matter fields then read 
\be
\aligned
\sigma &= N^{-2}T_{00} = |\mathcal{L}_n\phi|^2 + \frac{1}{2}\big|\nabla\phi\big|_g^2 \
= \frac{1}{2}\big(|\mathcal{L}_n\phi|^2 + \gb^{ij}\del_i\phi\del_j\phi\big),
\\
J_i &= -\mathcal{L}_n\phi\,\del_i\phi.
\endaligned
\ee

\begin{definition}\label{def initial data}
An {\bf initial data set for the modified gravity theory} 
$$
(\Mb, \gb, K, \Rh, \Sh, \phib, \psib)
$$
consists of the following data:
\begin{itemize}

\item a $3$-dimensional manifold $\Mb$ endowed with a Riemannian metric $\gb$ and a symmetric $(0,2)$-tensor  field $K$,

\item two scalar fields denoted by $\Rh$ and $\Sh$ defined on $\Mb$ and representing the (to-be-constructed) spacetime curvature and its time derivative,

\item two scalar fieds $\phib$ and $\psib$ defined on $\Mb$.

\end{itemize}

\noindent These data are required to satisfy the {\bf Hamiltonian constraint of modified gravity}
\begin{equation}
\label{constraint Halmitonian-ini}
\aligned
\Rb - K_{ij}K^{ij}+(K_j^j)^2
 &= \frac{8\big(\psib^2 + \gb^{ij}\del_i\phib\del_j\phib\big)}{f'(\Rh)} + \frac{2\Delta_{\gb}f'(\Rh)}{f'(\Rh)}
\\
&  +2\gb^{ij}K_{ij}\mathcal{L}_n\frac{f''(\Rh)\Sh}{f'(\Rh)}- W_1(\Rh),
\endaligned
\end{equation}
and
the {\bf momentum constraint of modified gravity}
\begin{equation}
\label{constraint momentum-ini}
\aligned
\del_j K_i^i - \nablab_i K_j^i &= \frac{8\pi \psib\,\del_i\phib}{f'(\Rh)}
+ \frac{\del_j\big(f''(\Rh)\Sh\big)}{f'(\Rh)} + K^i_j\del_i \big(\ln(f'(\Rh))\big).
\endaligned
\end{equation}
\end{definition}

\begin{definition}
\label{def:formu}
Given an initial data set $(\Mb, \gb, K, \Rh, \Sh, \phib, \psib)$ as in Definition~\ref{def initial data},
the {\bf initial value problem for the modified gravity theory} consists of finding
a Lorentzian manifold $(M, g)$ and a matter field $\phi$ defined on $M$ such that the following properties hold:
\begin{itemize}

\item[1.] The field equations of modified gravity  \eqref{Eq1-14} are satisfied. 

\item[2.] There exists an embedding $i: \Mb \to M$ with pull back metric $\gb = i^* g$ and
second fundamental form $K$.

\item[3.] The field $\Rh$ coincides with the restriction of the spacetime scalar curvature $R$ on $\Mb$, while
$\Sh$ coincides with the Lie derivative $\Lcal_n R$ restricted to $\Mb$, where $n$ denotes the normal to $\Mb$.

\item[4.] The scalar fields $\phib,\, \psib$ coincides with the restriction of $\phi, \Lcal_n \phi$ on $\Mb$, respectively.
\end{itemize}
Such a solution to \eqref{Eq1-14} is referred to as a {\bf modified gravity development of the initial data set} $(\Mb, \gb, K, \Rh, \Sh, \phib, \psib)$.
\end{definition}

Similarly as in classical gravity, we can define~\cite{CB} the notion of {\bf maximal globally hyperbolic} development for the modified gravity theory. 

Observe that the initial value problem for modified gravity reduces to the classical formulation in the special case
of vacuum and vanishing geometric data $\phib = \psib =\Rh = \Sh \equiv 0$. 
For the modified gravity theory, we have just shown that, similarly as in classical gravity, these prescribed fields can not be fully arbitrary, and certain constraints (given above) must be assumed. 
 

\subsection{Preservation of the constraints}

We need to address the following issue: if the evolution equations are satisfied by symmetric two-tensors $(\gb, K)$, the stress-energy tensor $T$ is divergence-free, and furthermore the constraint equations are satisfied on some initial slice $M_0$, then are all of field equations satisfied? In other words, we want to establish the preservation of the constraint equations along a flow of solutions. 

This is the first instance where we establish a ``preservation property'' and, later in this text, other similar situations will occur. The common character of these results is as follows.
A differential system being given, the equations therein can be classified into two categories: one is easer to handle (the evolution equations in this example) while the other is more difficult (the constraint equations here).
Our strategy is to replace the most difficult equations by some equations which can be deduced form the original system (in this example, the trace-free equation of $T$) but are also easier to handle. This leads us to a new system to be studied first, and an essential task is to check the equivalence between the original system and the new system. 

Before we give a precise statement, we make the following observation. The evolution equations \eqref{eq evolution} are equivalent to \eqref{eq main reformed'a} and the constraint equations \eqref{eq constraint H} and \eqref{eq constraint M} are equivalent to \eqref{eq main reformed'b} and \eqref{eq main reformed'c}. So we suppose that \eqref{eq main reformed'a} together with the divergence condition $\nabla^{\alpha}T_{\alpha\beta} = 0$ are satisfied in the spacetime $M = \bigcup_{t\in [0,t_{\max})} M_t$, and the constraint equations \eqref{eq main reformed'b} and \eqref{eq main reformed'c} are satisfied on the initial slice. Then we will prove that the  equations \eqref{eq main reformed'b} and \eqref{eq main reformed'c} are satisfied in the whole spacetime.  More precisely, we have the following result. 

\begin{proposition}\label{prop 2-5-1}
With the notation above, let us suppose that the equations
\begin{equation}\label{eq 2-5-1}
{E_g}_{ij} - 8\pi H_{ij}  = 0  \qquad \text{ in the spacetime } \bigcup_{t \in [0, t_{\max})} M_t
\end{equation}
and
\begin{equation}\label{eq 2-5-2}
\nabla^{\alpha}T_{\alpha\beta} = 0  \qquad \text{ in the spacetime } \bigcup_{t \in [0, t_{\max})} M_t
\end{equation} 
 and, moreover, 
\begin{equation}\label{eq 2-5-3}
{E_g}_{0j} = 8\pi H_{0j},\quad {N_g}_{00} = 8\pi T_{00}
\qquad \text{ in the initial slice } M_0 = \{t=0\}.
\end{equation} 
Then, it follows that 
\begin{equation}\label{eq 2-5-4}
{E_g}_{0j} = 8\pi H_{0j},\quad {N_g}_{00} = 8\pi T_{00}\quad \text{ in the spacetime } \bigcup_{t \in [0, t_{\max})} M_t.
\end{equation}
\end{proposition}

\begin{proof} The calculations are made in the Cauchy adapted frame and, for convenience, we introduce the notation
$$
\Sigma_{\alpha\beta} := {N_g}_{\alpha\beta} - 8\pi T_{\alpha\beta}. 
$$
We will prove that $\Sigma_{\alpha\beta} = 0$ which is equivalent to the desired result.

By the condition $\nabla_{\alpha}T_{\alpha\beta} = 0$ and the identity \eqref{eq E-bianchi}, we have 
\begin{equation}\label{eq pr 2-5-1}
\nabla^{\alpha}\Sigma_{\alpha\beta} = 0.
\end{equation}
By the definition of ${E_g}_{\alpha\beta}$ and $H_{\alpha\beta}$, the following identity holds:
\begin{equation}\label{eq pr 2-5-2}
\Sigma_{\alpha\beta} = {E_g}_{\alpha\beta} - H_{\alpha\beta} - \frac{1}{2}g_{\alpha\beta}\tr(E_g - H)
\end{equation}
and, in particular, 
\begin{equation}\label{eq pr 2-5-3}
\Sigma_{00} = {E_g}_{00} - H_{00} - \frac{1}{2}g_{00}\tr(E_g - H). 
\end{equation}

Now, due to the fact that $g^{0i} = g_{0i} = 0$ and \eqref{eq 2-5-1}, we have 
\begin{equation}\label{eq pr 2-5-4}
\tr(E_g - H) = g^{00}({E_g}_{00} - H_{00}).
\end{equation}
Combining \eqref{eq pr 2-5-3} and \eqref{eq pr 2-5-4} yields 
\begin{equation}\label{eq pr 2-5-5}
\tr(E_g - H) = 2g^{00}\Sigma_{00}
\end{equation}
and, by substituting this into \eqref{eq pr 2-5-2},
\begin{equation}\label{eq pr 2-5-6}
\Sigma_{\alpha\beta} = {E_g}_{\alpha\beta} - H_{\alpha\beta} - g_{\alpha\beta}g^{00}\Sigma_{00}.
\end{equation}
Here, we can compute more precisely the spatial components in view of \eqref{eq 2-5-1}
\begin{equation}\label{eq pr 2-5-7}
\Sigma_{ij} = -g_{ij} g^{00} \Sigma_{00}.
\end{equation}

Given this material, we are now in a position to calculate
$
\Sigma^\alpha_\beta = g^{\alpha\alpha'}\Sigma_{\alpha'\beta}.
$
When $\alpha = \beta = 0$, we find 
\begin{equation}\label{eq pr 2-5-8}
\Sigma_0^0 = \Sigma_{0\alpha}g^{\alpha 0} = g^{00}\Sigma_{00}.
\end{equation}
For $\beta = 0$, $1\leq \alpha\leq 3$, we set $a = \alpha$ and, by recalling that $g^{a0} = 0$, we obtain  
$$
\Sigma_0^a = g^{a\alpha'}\Sigma_{\alpha'0} = g^{aa'}\Sigma_{a'0}
$$
For $1\leq b\leq 3$, we have 
$$
\Sigma_b^0 = g_{b\beta}g^{0\alpha}\Sigma_{\alpha}^{\beta} = g^{00}g_{bc}\Sigma_0^c.
$$
For $1\leq b\leq 3$ and $1\leq a\leq 3$, by applying \eqref{eq pr 2-5-7} and \eqref{eq pr 2-5-8}, we obtain 
$$
\Sigma_b^a = g^{a\alpha}\Sigma_{b\alpha} = g^{aa'}\Sigma_{ba'} = -g^{aa'}g_{ba'}g^{00}\Sigma_{00} = -\delta_b^a\Sigma_0^0.
$$
Hence, we conclude with
\begin{equation}\label{eq pr 2-5-9}
\aligned
\Sigma_b^0 = g^{00}g_{bc}\Sigma_0^c,\quad \Sigma_b^a = -\delta_b^a\Sigma_0^0.
\endaligned
\end{equation}

Now recall that the identity \eqref{eq pr 2-5-1} can be written as
$$
\nabla_{\alpha}\Sigma_{\beta}^\alpha = 0, 
$$
which leads to
$$
<e_{\alpha},\Sigma_{\beta}^{\alpha}>
- \omega_{\alpha\beta}^{\delta}\Sigma_\delta^\alpha + \omega_{\alpha\delta}^{\alpha}\Sigma_\beta^\delta = 0.
$$
When $\beta = 0$, we have 
\begin{equation}\label{eq pr 2-5-10}
\textcolor{black}{<e_0, \Sigma_0^0>} + \del_a\Sigma^a_0
- \omega_{\alpha\beta}^{\delta}\Sigma_\delta^\alpha + \omega_{\alpha\delta}^{\alpha}\Sigma_\beta^\delta = 0.
\end{equation}

For $1\leq b\leq 3$, we can take the equation \eqref{eq pr 2-5-9} and write 
$$
\textcolor{black}{<e_0,g^{00}g_{bc}\Sigma_0^c>} + \del_a\big(-\delta_b^a\Sigma_0^0\big)
- \omega_{\alpha b}^\gamma\Sigma_{\gamma}^\alpha + \omega_{\alpha\gamma}^{\alpha}\Sigma_{b}^{\gamma}=0, 
$$
which leads us to
\begin{equation}\label{eq pr 2-5-11}
\textcolor{black}{<e_0,\Sigma_0^c>} - g_{00}\gb^{bc}\del_b\Sigma_0^0
- g_{00}\gb^{bc}\big( \omega_{\alpha b}^\gamma\Sigma_{\gamma}^\alpha - \omega_{\alpha\gamma}^{\alpha}\Sigma_{b}^{\gamma}\big)
+ g_{00}\gb^{bc}<e_0, g^{00}g_{bc'}>\Sigma_0^{c'}
=0.
\end{equation}

We now consider the equations \eqref{eq pr 2-5-10} and \eqref{eq pr 2-5-11} together, and we observe that, in view of \eqref{eq pr 2-5-9}, the lower-order terms are {\sl linear combinations} $\Sigma_0^{\alpha}$ with $0\leq \alpha\leq 3$. Hence, these equations form a first-order differential system with linear source-terms.
This system can also be written in a standard symmetric hyperbolic form.
Namely, by recalling the notation $\gb=\gb_t$ for the induced Riemannian metric on the slices $M_t$,
we introduce 
$$
V := (\Sigma_0^0, \Sigma_0^a)^T,\quad \rho^a := g_{00}(\gb^{a1},\gb^{a2},\gb^{a3})^T 
$$
and 
$$
\sigma_1 = (1,0,0), \quad \sigma_2 = (0,1,0),\quad \sigma_3 = (0,0,1).
$$
The principal part of the system defined by \eqref{eq pr 2-5-10} and \eqref{eq pr 2-5-11} can be put in the form
\begin{equation}\label{eq 2-5-system}
\textcolor{black}{<e_0, V>} + \sum_a A^a\del_a V = F, 
\end{equation}
where
$$
A^a =
\left(
\begin{array}{cc}
0 & \sigma_a
\\
-g_{00}\rho^a &0
\end{array}
\right )
$$
and $F$ is a linear form on $V$. By multiplying this equation by the matrix 
$$
A_0 :=
\left(
\begin{array}{cc}
1 &0
\\
0 &-g^{00}g_{ab}
\end{array}
\right )
= -g_{00}
\left(
\begin{array}{cc}
-g_{00} &0
\\
0 &g_{ab}
\end{array}
\right ), 
$$
we conclude that \eqref{eq 2-5-system} becomes
\be
A_0 \textcolor{black}{<e_0, V>} + \sum_{a}A_0 A^a\del_a V = A_0 F.
\ee
Note that $A_0A^a$ are symmetric:
$$
A_0A^a =
\left(
\begin{array}{cc}
0 &\sigma^a
\\
(\sigma^a)^T &0
\end{array}
\right).
$$
and the system \eqref{eq 2-5-system} is thus symmetrizable. Clearly, \eqref{eq 2-5-3} implies that $V = 0$ on the initial slice $\{t=0\}$. Thanks to our global hyperbolicity assumption \eqref{eq:globalhyper} and by a standard uniqueness argument, we therefore conclude that $V = 0$ in the whole spacetime.
\end{proof}


\newpage 

\section{Formulation of the Cauchy problem in the Einstein metric}\label{sec conformal}

\subsection{Conformal transformation}

In view of the derivation made in Section~2, it is clear that the evolution and the constraint equations of modified gravity are, both, very involved and do not have a standard (hyperbolic, elliptic) type within the general class of PDE's. The main difficulty comes from the fourth-order term 
$$
\nabla_{\alpha}\nabla_{\beta}f'(R_g). 
$$
As we will now show it, the conformal transformation 
\bel{eq:transf1}
\gd_{\alpha\beta} := e^{2\rho} \, g_{\alpha\beta},\qquad \quad \gd^{\alpha\beta} = e^{-2\rho}g^{\alpha\beta}
\ee
(which depends upon second-order derivatives of the metric solution) will overcome some of the difficulties: 
where the {\bf conformal factor} is defined by
\be
\rho := \frac{1}{2}\ln f'(R_g)
\ee
or, equivalently, $f'(R_g) = e^{2\rho}$. 
We now proceed by deriving several relevant expressions in the {\bf conformal metric} $\gd$ in order to derive a tractable formulation of the field equations. 

We begin by deriving an expression for the gravity tensor $N_g$ in terms of the Einstein metric.

\begin{lemma}
With the notation above, the following identity holds 
\begin{equation}
\label{main eq trans1}
e^{2\rho}\Rd_{\alpha\beta} - 6e^{2\rho}\del_{\alpha}\rho\del_{\beta}\rho + \frac{\gd_{\alpha\beta}}{2}W_2(\rho)
=
{N_g}_{\alpha\beta} - \frac{1}{2}g_{\alpha\beta}\tr(N_g),
\end{equation}
where the  {\bf function $W_2=W_2(\rho)$} is defined implicitly by 
\begin{equation}
W_2(s) = \frac{f(r) - f'(r)r}{f'(r)},\quad e^{2s} = f'(r), \qquad r \in \RR.
\end{equation}
\end{lemma}

We also recall that the function $W_1$ is defined by \eqref{def:Wone} and it will be also convenient (in the proof below) to introduce the {\bf function $W_3=W_3(\rho)$} by 
\be
W_3(s):= f(r),\quad e^{2s} = f'(r), \qquad r \in \RR.  
\ee

\begin{proof} We need to analyze the tensor 
$$
{N_g}_{\alpha\beta} = f'(R_g)R_{\alpha\beta} - \frac{1}{2}f(R_g)g_{\alpha\beta}
+ \big(g_{\alpha\beta}\Box_g - \nabla_{\alpha}\nabla_{\beta}\big)f'(R_g)
$$
and its trace
$
\tr(N_g) = f'(R_g)R_g - 2f(R_g) + 3\Box_gf'(R_g).
$
Recall first the identities 
$$
\aligned
&\nabla_{\alpha}\nabla_{\beta}e^{2\rho} = 2e^{2\rho}\nabla_{\alpha}\nabla_{\beta}\rho + 4e^{2\rho}\nabla_{\alpha}\rho\nabla_{\beta}\rho,
\\
&\Box_ge^{2\rho} = 2e^{2\rho}\Box_g\rho + 4e^{2\rho}g(\nabla \rho,\nabla \rho), 
\endaligned
$$
which imply 
$$
{N_g}_{\alpha\beta} = e^{2\rho}R_{\alpha\beta} - \frac{1}{2}g_{\alpha\beta}W_3(\rho)
+ 2e^{2\rho}\big(g_{\alpha\beta}\Box_g - \nabla_{\alpha}\nabla_{\beta}\big)\rho + 4e^{2\rho}\big(g_{\alpha\beta}g(\del\rho,\del\rho) - \del_{\alpha}\rho\del_{\beta}\rho\big)
$$
and
$$
\Box_g \rho - \frac{W_3(\rho)}{6e^{2\rho}} + 2g(\del\rho,\del\rho) - \frac{1}{6}W_2(\rho) = \frac{\tr (N_g)}{6e^{2\rho}}.
$$
Moreover, we have the following relation between the Ricci curvature tensors of $g$ and $\gd$:
\begin{equation}
\label{conformal Ricci curvature}
\Rd_{\alpha\beta} = R_{\alpha\beta} - 2\big(\nabla_\alpha\nabla_{\beta}\rho
- \nabla_\alpha\rho\nabla_{\beta}\rho\big) - \big(\Box_g\rho+2g(\nabla \rho,\nabla \rho)\big)g_{\alpha\beta}
\end{equation}
and, therefore, we see that $N_g$ can be expressed as 
$$
{N_g}_{\alpha\beta} = e^{2\rho}\Rd_{\alpha\beta} - 6e^{2\rho}\nabla_{\alpha}\rho\nabla_{\beta}\rho + 6e^{2\rho}g_{\alpha\beta}g(\nabla\rho,\nabla\rho) + 3e^{2\rho}g_{\alpha\beta}\Box_g \rho
- \frac{1}{2}W_3(\rho)g_{\alpha\beta}.
$$
It remains to combine this result with the trace equation above.
\end{proof}

We are now in a position to state the field equations in the conformal metric. 
At this juncture, it is unclear how the scalar field $\rho$ should be recovered in term of the Einstein metric, and this is precisely the issue that we will address next, by deriving an evolution equation on $\rho$.

\begin{propositiondefinition}
The {\bf field equations of modified gravity in the Einstein metric} $\gd_{\alpha\beta} = e^{2\rho}g_{\alpha\beta}$
with $\rho= \frac{1}{2}\ln(f'(R_g))$
read
\begin{equation}\label{main eq trans}
 e^{2\rho}\Rd_{\alpha\beta} - 6e^{2\rho}\del_{\alpha}\rho\del_{\beta}\rho + \frac{\gd_{\alpha\beta}}{2}W_2(\rho)
=
8\pi \big(T_{\alpha\beta} - \frac{1}{2}\gd_{\alpha\beta}\gd^{\alpha'\beta'}T_{\alpha'\beta'}\big).
\end{equation}
\end{propositiondefinition}

\begin{remark} For any sufficiently regular function $w$, one alaso has 
\be
\label{conformal d'Alembertian}
\Box_{\gd}w = e^{-2\rho}\big(\Box_g w + 2g^{\alpha\beta}\del_\alpha\rho\del_{\beta}w \big) = e^{-2\rho}\Box_gw + 2\gd^{\alpha\beta}\del_{\alpha}\rho\del_{\beta}w, 
\ee
so that the trace equation transforms into
\begin{equation}\label{trace eq trans}
\Box_{\gd}\rho = \frac{W_2(\rho)}{6e^{2\rho}} + \frac{W_3(\rho)}{6e^{4\rho}} + \frac{1}{6e^{4\rho}}\tr(N_g).
\end{equation}
\end{remark}


\subsection{Evolution and constraint equations in the conformal metric}

As in the previous section, we can formulate the evolution equations and constraint equations associated with the conformal field equation \eqref{main eq trans1}. To do so, as before, we introduce a foliation of the spacetime $M = [0,+\infty)\times M_t$ and a Cauchy adapted frame $\{e_0,e_1,e_2,e_3\}$ associated with the transformed metric $\gd$. Then, by similar calculations as above, we find
\begin{subequations}\label{GC-2-trans}
\begin{equation}\label{GC-2-transa}
\Rd_{ij} = \Rb_{ij} - \frac{\delb_0\Kd_{ij}}{N} + \Kd_{ij}\Kd^l_l - 2\Kd_{il}\Kd^{l}_j - \frac{\nablabd_i\del_j\Nd}{\Nd},
\end{equation}
\begin{equation}\label{GC-2-transb}
\Rd_{0j} = \Nd\big( \del_j\Kd_l^l - \nablabd_l \Kd^l_j\big),
\end{equation}
\begin{equation}\label{GC-2transc}
\Gd_{00} = \frac{{\Nd}^2}{2} \big(\Rbd - \Kd_{ij}\Kd^{ij} + (\Kd_l^l)^2\big).
\end{equation}
\end{subequations}
Here, $\nablabd$ refers to the covariant derivative on the slice $M_t$ with respect to $\gd$, and we observe that \eqref{GC-2-transa} yields the evolution equations
\begin{equation}\label{GC-evolution-deux}
\aligned
&\delb_0 \Kd_{ij} = \Nd\big(\Rbd_{ij} - \Rd_{ij}\big) + \Nd\Kd_{ij}\Kd_l^l - 2\Nd\Kd_{il}\Kd^l_j - \nablabd_i\del_j\Nd,
\\
&\delb_0\gbd_{ij} = -2\Nd\Kd_{ij}.
\endaligned
\end{equation}

Moreover, the transformed field equations \eqref{main eq trans1} read 
$$
\Rd_{\alpha\beta} = e^{-2\rho}\big({N_g}_{\alpha\beta} - \frac{1}{2}g_{\alpha\beta}\tr(N_g)\big) + 6\del_{\alpha}\rho\del_{\beta}\rho - \frac{1}{2e^{2\rho}}\gd_{\alpha\beta}W_2(\rho)
$$
and, by taking the trace of this equation {\bf with respect to the metric $\gd$}, we have 
$$
\Rd = -e^{-4\rho}\tr(N_g) + 6\gd(\del\rho,\del\rho)g_{\alpha\beta} + e^{-2\rho}\gd_{\alpha\beta}W_2(\rho). 
$$
This leads us to
$$
\Gd_{\alpha\beta} = {N_g}_{\alpha\beta} + 6\nabla_{\alpha}\varrho\nabla_{\beta}\varrho - 3\gd(\del\rho,\del\rho)\gd_{\alpha\beta} + e^{-2\rho}\gd_{\alpha\beta}W_2(\rho),
$$

We have thus derived the evolution equations and constraint equations. 
The evolution equations read 
$$
\aligned
\delb_0 \Kd_{ij} &= \Nd\Rbd_{ij} + \Nd\Kd_{ij}\Kd_l^l - 2\Nd\Kd_{il}\Kd^l_j - \nablabd_i\del_j\Nd
\\
&-\Nd e^{-2\rho}\bigg({N_g}_{ij} - \frac{1}{2e^{2\rho}}\gd_{ij}\tr(N_g) + 6e^{2\rho}\del_i\rho\del_j\rho + \frac{1}{2}\gd_{ij}W_2(R_g)\bigg), 
\\
\delb_0\gbd_{ij} &= -2\Nd\Kd_{ij},
\endaligned
$$
while the Hamilton constraint equation reads 
$$
\Rbd - \Kd_{ij}\Kd^{ij} + (\Kd_l^l)^2 = \frac{2{N_g}_{00}}{e^{2\rho}\Nd^2} + 12|\mathcal{L}_{n^{\dag}}\rho|^2
+ 6\gd(\nablad \rho,\nablad\rho) - e^{-2\rho}W_2(\rho)
$$
and the momentum constraint equations read
$$
\del_j\Kd_l^l - \nablabd\Kd_j^l = \frac{{N_g}_{0j}}{e^{2\rho}\Nd} + 6\mathcal{L}_{n^{\dag}}\rho\,\del_j\rho.
$$
Here, $n^\dag$ denotes the normal unit vector of the slice $M_t$.

Finally, we consider the Jordan coupling with matter field (this choice of coupling being revisited in the next subsection):
\be
{N_g}_{\alpha\beta} = 8\pi T_{\alpha\beta} 
\ee
and, furthermore, we define the matter fields 
\be
\Jd_j := -\frac{T_{0j}}{\Nd}, \qquad \sigmad := \frac{T_{00}}{\Nd^2}.
\ee

\begin{definition}
In the Einstein metric,
the equations of modified gravity in a Cauchy adapted frame $\{e_0,e_1,e_2,e_3\}$
can be decomposed as follows:
\begin{itemize}

\item[1.] {\bf Evolution equations:}
$$
\aligned
\delb_0 \Kd_{ij} &= \Nd\Rbd_{ij} + \Nd\Kd_{ij}\Kd_l^l - 2\Nd\Kd_{il}\Kd^l_j - \nablabd_i\del_j\Nd
\\
&-\Nd e^{-2\rho}\Big(8\pi T_{ij} - 4\pi e^{-2\rho}\gd_{ij}\tr(T) + 6e^{2\rho}\del_i\rho\del_j\rho + \frac{1}{2}\gd_{ij}W_1(R_g)\Big),
\\
\delb_0\gbd_{ij} &= -2\Nd\Kd_{ij}.
\endaligned
$$

\item[2.] {\bf Hamiltonian constraint}:
\begin{equation}\label{eq constraint conformalT H}
\Rbd - \Kd_{ij}\Kd^{ij} + (\Kd_l^l)^2 = \frac{16\sigmad}{e^{2\rho}} + 12|\mathcal{L}_{n^{\dag}}\rho|^2
+ 6\gd(\nablad \rho,\nablad\rho) - e^{-2\rho}W_2(\rho)
\end{equation}

\item[3.] {\bf Momentum constraints}:
\begin{equation}\label{eq constraint conformalT M}
\del_j\Kd_l^l - \nablabd\Kd_j^l = -\frac{\Jd_j}{e^{2\rho}} + 6\mathcal{L}_{n^\dag}\rho\, \del_j\rho.
\end{equation}
\end{itemize}
\end{definition}

Let us again emphasize again that the constraint equations are equivalent to 
${N_g}_{00} = 8\pi T_{00}$, ${N_g}_{0a} = 8\pi T_{0a}$, while the evolution equations are equivalent to
$$
{N_g}_{ab} - \frac{1}{2}g_{ab}\tr(N_g) = 8\pi \big(T_{ab} - \frac{1}{2}g_{ab}\tr(T)\big).
$$


\subsection{The divergence identity}
\label{sec div conformal}

In order to derive an evolution equation for the matter field, we need the divergence of the tensor $N_g$ with respect to the conformal metric $\gd$. 

\begin{lemma}\label{lem divN conformal}
The modified gravity tensor in terms of the conformal metric satisfies the identity 
\begin{equation}\label{eq divN-bianchi convformal}
\nablad^{\alpha}{N_g}_{\alpha\beta}
 = e^{-2\rho}\big(2g^{\gamma\delta}\del_{\gamma}\rho{N_g}_{\delta\beta} - \tr(N_g)\del_{\beta}\rho\big).
\end{equation}
\end{lemma}

\begin{proof} We work in an arbitrary (possibly only locally defined) natural frame. The desired identity follows from 
$$
\Gammad^{\gamma}_{\alpha\beta} = \Gamma^{\gamma}_{\alpha\beta} + \delta_{\alpha}^{\gamma}\del_{\beta}\rho + \delta_{\beta}^{\gamma}\del_{\alpha}\rho - g_{\alpha\beta}\nabla^\gamma\rho.
$$
We have 
$$
\aligned
&\nablad^{\alpha}{N_g}_{\alpha\beta}
 = e^{-2\rho}g^{\alpha\gamma}\nablad_{\gamma}{N_g}_{\alpha\beta}
\\
&= e^{-2\rho}g^{\gamma\alpha}\big(\del_{\gamma}{N_g}_{\alpha\beta} - \Gamma_{\gamma\alpha}^{\delta}{N_g}_{\beta\delta} - \Gamma_{\gamma\beta}^{\delta}{N_g}_{\alpha\delta}\big)
\\
&\quad - e^{-2\rho}g^{\alpha\gamma}\big(\delta_{\gamma}^{\delta}\del_{\alpha}\rho + \delta_{\alpha}^{\delta}\del_{\gamma}\rho -
 g_{\gamma\alpha}\nabla^{\delta}\rho\big){N_g}_{\beta\delta}
\\
& \quad
  -e^{-2\rho}g^{\alpha\gamma}\big(\delta_{\gamma}^{\delta}\del_{\beta}\rho + \delta_{\beta}^{\delta}\del_{\gamma}\rho - g_{\gamma\beta}\nabla^{\delta}\rho\big){N_g}_{\alpha\delta}, 
\endaligned
$$ 
thus 
$$
\aligned
&\nablad^{\alpha}{N_g}_{\alpha\beta}
\\
&= e^{-2\rho}\nabla^{\alpha}{N_g}_{\alpha\beta}
\\
&\quad- e^{-2\rho}\big(\nabla^{\delta}\rho + \nabla^{\delta}\rho - 4\nabla^{\delta}\rho\big){N_g}_{\beta\delta}
- e^{-2\rho}\big(\del_{\beta}\rho\tr(N_g) + \nabla^{\alpha}\rho{N_g}_{\alpha\beta} - \nabla^{\alpha}\rho{N_g}_{\alpha\beta}\big)
\endaligned
$$
Recalling that $\nabla^{\alpha}{N_g}_{\alpha\beta} = 0$ by Lemma~\ref{prop E_bianchi}, we concude that 
$$
\aligned
\nablad^{\alpha}{N_g}_{\alpha\beta}
= e^{-2\rho}\big(2g^{\gamma\delta}\del_{\gamma}\rho{N_g}_{\delta\beta} - \tr(N_g)\del_{\beta}\rho\big).
\endaligned
$$
\end{proof}

With the Jordan coupling, the divergence of the stress-energy tensor is thus expressed as 
\be
\nablad^{\alpha}T_{\alpha\beta} = \big(2g^{\delta\gamma}\del_\delta\rho \, T_{\gamma\beta} - {\tr(T)\del_\beta}\rho  \big)e^{-2\rho}, 
\ee
which (together with an equation of state for the matter field) determine the evolution equation of the matter field.

\begin{remark} 
We conclude this section with a discussion of the Einstein coupling. We rely on \eqref{eq divN-bianchi convformal} and now show that the only meaningful choice of coupling (now viewed in the Einstein metric)
is the Jordan coupling.

Observe frst that the Jordan coupling 
$$
T_{\alpha\beta}  = \TJ_{\alpha\beta} = \del_{\alpha}\phi\del_{\beta}\phi - \frac{1}{2}g_{\alpha\beta}|\nabla \phi|^2_g 
$$
implies
$$
\aligned
\del_{\beta}\rho \,\tr(T) - 2g^{\delta\gamma}\del_\delta\rho \, T_{\gamma\beta}
&=-|\nabla\phi|_g^2 \del_{\beta}\rho - 2g(\del\rho,\del\phi)\del_{\beta}\phi + |\nabla\phi|^2_g\del_{\beta}\rho
\\
&= - 2g(\del\rho,\del\phi)\del_{\beta}\phi, 
\endaligned
$$
which leads us to
\begin{equation}\label{eq 3-3-1}
\aligned
\big(2g^{\delta\gamma}\del_\delta\rho T_{\gamma\beta} - \tr(T)\del_{\beta}\rho\big)e^{-2\rho} = 2\gd(\del\rho,\del\rho)\del_{\beta}\phi.
\endaligned
\end{equation}
From the identity 
$$
\nablad^{\alpha}T_{\alpha\beta} = \del_{\beta}\phi\,\Box_{\gd}\phi
$$
combined with \eqref{eq divN-bianchi convformal}, we have 
$$
\del_{\beta}\phi\,\Box_{\gd}\phi
= 2e^{-2\rho} g^{\alpha'\beta'}\del_{\alpha'}\phi \del_{\beta'}\rho \,\del_{\beta}\phi
$$
and this leads us to the {\bf wave equation for the matter field}
\begin{equation}
\label{eq conformal wave phi}
\Box_{\gd}\phi = 2\gd^{\alpha'\beta'}\del_{\alpha'}\phi \del_{\beta'}\rho.
\end{equation}

On the other hand, let us consider the Einstein coupling:
$$
T_{\alpha\beta} = \TE_{\alpha\beta} = e^{2\rho}\big(\del_{\alpha}\phi\del_{\beta}\phi - \frac{1}{2}g_{\alpha\beta}|\nabla \phi|^2_g\big),
$$
which gives 
$$
\aligned
\nablad^{\alpha}T_{\alpha\beta} &= e^{2\rho}\del_{\beta}\phi\Box_\gd\phi
+ 2e^{2\rho}\gd^{\alpha\alpha'}\del_{\alpha'}\rho \big(\del_{\alpha}\phi\del_{\beta}\phi - \frac{1}{2}g_{\alpha\beta}|\nabla \phi|^2_g\big)
\\
&= e^{2\rho}\del_{\beta}\phi\big(\Box_\gd\phi + 2\gd(\del\rho,\del\phi)\big) - \del_{\beta}\rho|\nabla\phi|^2_g.
\endaligned
$$
In combination with \eqref{eq 3-5-1}, we find 
$$
e^{2\rho}\del_{\beta}\phi\big(\Box_\gd\phi + 2\gd(\del\rho,\del\phi)\big) - \del_{\beta}\rho|\nabla\phi|^2_g
= 2\gd(\del\rho,\del\phi)\del_{\beta}\phi, 
$$
and therefore 
$$
e^{2\rho}\del_{\beta}\phi\big(\Box_\gd \phi + 2\gd(\del\rho,\del\phi) - 2e^{-2\rho}\gd(\nablad \rho,\nablad \phi)\big) =
\del_{\beta}\rho|\nabla\phi|^2_g. 
$$
In agreement with what we noticed with the Jordan metric, the Einstein coupling leads to an {\sl over-determined} partial differential system. This suggests again that the Einstein coupling may not lead to a well-posed initial value problem.
\end{remark} 


\subsection{The conformal version of the initial value problem}

We are now ready to formulate the notion of initial data set and the notion of Cachy development in terms of the conformal metric. In agreement to our discussion in the previous section, we work with the Jordan coupling and a real massless scalar field:
\be
T_{\alpha\beta} = \del_{\alpha}\phi\del_{\beta}\phi - \frac{1}{2}g_{\alpha\beta}|\nabla \phi|_g^2, 
\ee
and we set 
$$
\aligned
\sigmad &:= \frac{1}{2}\big(|\mathcal{L}_{n^\dag}\phi|^2 + \gb^{ij}\del_i\phi\del_j\phi\big),
\\
\Jd_j &:= -\mathcal{L}_{n^\dag}\phi\,\del_j\phi.
\endaligned
$$
 
\begin{definition}\label{def36}
An {\bf initial data set for the modified gravity theory in the Einstein metric} $(\Mb, \gbd, \Kd, \rhozerod, \rhooned, \phizerod, \phioned)$ consists of the following data:

\begin{itemize}

\item a $3$-dimensional manifold $\Mb$ endowed with a Riemannian metric $\gbd$ and a symmetric $(0,2)$-tensor field $\Kd$,

\item two scalar fields denoted by $\rhozerod$ and $\rhooned$ on $\Mb$ and representing the (to-be-constructed) conformal factor and its time derivative,

\item two scalar field $\phizerod$ and $\phioned$ defined on $\Mb$.

\end{itemize}

Furthermore, these data are required to satisfy the {\bf Hamiltonian constraint of modified gravity in the Einstein metric}
\begin{equation}
\label{constraint trans Halmitonian}
\aligned
\Rbd - \Kd_{ij}\Kd^{ij}+(\Kd_j^j)^2
=  \, 
& 8 \, e^{-2\rho} \, \Big((\phioned)^2
+ \gbd^{ij}\del_i\phizerod\del_j\phizerod\Big) + 6 (\rhooned)^2 
\\
& \quad + 6\gbd^{ij}\del_i\rhozerod\del_j\rhozerod
 - e^{-2\rhozerod}W_2(\rhozerod),
\endaligned
\end{equation}
and the {\bf momentum constraint of modified gravity in the Einstein metric}
\begin{equation}
\label{constraint trans momentum}
\aligned
\del_j \Kd_i^i - \nablabd_i \Kd_j^i
 &= \frac{\phioned\del_j\phizerod}{e^{2\rho}} + 6\rhooned\del_j\rhozerod.
\endaligned
\end{equation}
\end{definition}

\begin{definition}
\label{def cauchy conformal}
Given an initial data set $(\Mb, \gbd, \Kd, \rhozerod, \rhooned, \phizerod, \phioned)$ as in Definition~\ref{def36},
the {\bf initial value problem for the modified gravity theory in the Einstein metric}
consists of finding
a Lorentzian manifold $(M, g)$ and a two-tensor field $T_{\alpha\beta}$ on $M$
\begin{itemize}
\item[1.]The conformal metric $\gd$ is defined with the relation $\gd_{\alpha\beta} = e^{2\rho}g_{\alpha\beta}$ with the conformal factor $\rho = \frac{1}{2}\ln(f'(R_g))$ where $R_g$ is the scalar curvature of $g$.

\item[2.] The field equations of modified gravity  \eqref{main eq trans} are satisfied with $\rho = \frac{1}{2}\ln f'(R_g)$.

\item[3.] There exists an embedding $i: \Mb \to M$ with pull back metric $\gbd = i^* \gd$ and
second fundamental form $\Kd$.

\item[4.] The field $\rhozerod$ coincides with the restriction of the conformal factor $\rho$ on $\Mb$, while
$\rhooned$ coincides with the Lie derivative $\mathcal{L}_{n^\dag}\rho$ restricted to $\Mb$, where $n^\dag$ denotes the normal unit vector of $\Mb$.

\item[5.] The scalar fields $\phizerod,\, \phioned$ coincides with the restriction of $\phi,\, \Lcal_{n^\dag} \phi$ on $\Mb$.
\end{itemize}
Such a solution to \eqref{main eq trans} is referred to as a {\bf modified gravity development of the initial data set} $(\Mb, \gbd, \Kd, \rhozerod, \rhooned, \phizerod, \phioned)$.
\end{definition}

The notion of {\bf maximal globally hyperbolic development} is then defined along the same lines as in \cite{CB} for the classical gravity.  We observe that our formulation of the initial value problem for modified gravity reduces to the classical formulation in the special case
of vanishing geometric data $\phizerod=\phioned=\Rh = \Sh \equiv 0$.
On the other hand, without matter fields and for non-vanishing geometric data $\Rh$ and $\Sh$, the
spacetimes under consideration do not satisfy Einstein vacuum equations.
Similarly as in classical gravity, these fields can not be fully arbitrary prescribed but certain constraints (given above) must be assumed.


\subsection{Preservation of the constraints}

Next, we establish the preservation of the constraints, as follows.  

\begin{proposition}
Let $(\gbd, \Kd)$ be symmetric two-tensors defined in $M = \cup_{t\in[0,t_{\max})}M_t$. If the following equations hold in $M$
\begin{equation}\label{eq 3-5-1}
{N_g}_{ij} - \frac{1}{2}\tr(N_g)g_{ij} = 8\pi \big(T_{ij} - \frac{1}{2}\tr(N_g)g_{ij}\big),
\end{equation}
\begin{equation}\label{eq 3-5-1.5}
\nablad^{\alpha}T_{\alpha\beta} =  e^{-2\rho}\big(g^{\gamma\delta}\del_{\gamma}\rho T_{\gamma\beta} - \del_{\beta}\rho\tr(T)\big),
\end{equation}
and
\begin{equation}\label{eq 3-5-2}
{N_g}_{0\beta}= 8\pi T_{0\beta}.
\end{equation}
holds  on the initial slice $M_0$, then \eqref{eq 3-5-2} holds throughout the spacetime $M$.
\end{proposition}

\begin{proof} Recalling the notation $\Sigma_{\alpha\beta} = {N_g}_{\alpha\beta} - 8\pi T_{\alpha\beta}$, we are going to prove that $\sigma_{00}=0$.
We note that \eqref{eq 3-5-1} can be written as
\begin{equation}\label{eq 3-5-3}
{\Sigma}_{ij} - \frac{1}{2}\tr^{\dag} (\Sigma)\gd_{ij} = 0.
\end{equation}
By taking the trace of the tensor $\Sigma_{\alpha\beta} - \frac{1}{2}\tr{\Sigma}g_{\alpha\beta}$ {\bf with respect to $\gd$}, we find 
$$
\tr^{\dag}\Sigma - 2\tr^{\dag}\Sigma = -\Nd^{-2}\Sigma_{00} + \gd^{ij}\big({\Sigma}_{ij} - \frac{1}{2}\tr(\Sigma)g_{ij}\big). 
$$
Combining with \eqref{eq 3-5-3}, we thus have 
\begin{equation}\label{eq 3-5-4}
\tr^\dag \Sigma = -\Nd^{-2}\Sigma_{00}.
\end{equation}
Combining \eqref{eq 3-5-4} together with \eqref{eq 3-5-3}, we then obtain 
\begin{equation}\label{eq 3-5-5}
\Sigma_{ij} = -\frac{\Sigma_{00}}{2\Nd^2}\gd_{ij}.
\end{equation}

Along the same lines as in the proof of Proposition \ref{prop 2-5-1}, we have 
\begin{equation}\label{eq 3-5-6}
\Sigma_b^0 = \gd^{00}\gd_{bc}\Sigma_0^c,\qquad \Sigma_b^a = -\delta_b^a\Sigma_0^0.
\end{equation}
Let us consider the identity \eqref{eq divN-bianchi convformal} combined with \eqref{eq 3-5-1.5}, and note the identity
\begin{equation}\label{eq 3-5-7}
\nablad^{\alpha}\Sigma_{\alpha\beta}
= e^{-2\rho}\big(2\del_{\gamma}\rho\Sigma_{\beta}^{\gamma} - \del_{\beta}\rho \tr(\Sigma)\big). 
\end{equation}
We observe that by \eqref{eq 3-5-6}, the right-hand-side is a linear form of the function $\Sigma_0^{\beta}$
and, by definition,
\begin{equation}\label{eq 3-5-8}
\nablad_{\alpha}\Sigma_\beta^\alpha
= \textcolor{black}{<e_{\alpha},\Sigma_\beta^{\alpha}>} - {\omegad}_{\alpha\beta}^{\delta} \Sigma_\delta^\alpha + {\omegad}_{\alpha\delta}^{\alpha}\Sigma_{\beta}^{\delta}.
\end{equation}
By combining \eqref{eq 3-5-7} and \eqref{eq 3-5-8}, we arrive at the first-order linear differential system
$$
\textcolor{black}{<e_{\alpha},\Sigma_\beta^{\alpha}>} - {\omegad}_{\alpha\beta}^{\delta} \Sigma_\delta^\alpha + {\omegad}_{\alpha\delta}^{\alpha}\Sigma_{\beta}^{\delta} = e^{-2\rho}\big(2\del_{\gamma}\rho\Sigma_{\beta}^{\gamma} - \del_{\beta}\rho \tr(\Sigma)\big), 
$$
whose principal part is
\begin{equation}
\aligned
\textcolor{black}{<e_0,\Sigma_0^0>} + \del_a\Sigma_0^a &= \text {lower order terms},
\\
\textcolor{black}{<e_0,\gd^{00}\gd_{bc}\Sigma_0^c>} - \del_b\Sigma_0^0 &= \text{lower order terms}.
\endaligned
\end{equation}
can be symmetrized by the same procedure as we did for the system \eqref{eq pr 2-5-10} and \eqref{eq pr 2-5-11}. Recall also that by \eqref{eq 3-5-2}, this system has vanishing initial data and, therefore, in view of our global hyperbolicity assumption \eqref{eq:globalhyper}, the desired result is proven.
\end{proof}


\newpage 

\section{The conformal formulation in wave coordinates}

\subsection{The wave gauge}

We now turn our attention to solving the system \eqref{main eq trans} by developping approriate techniques of PDE's. Our first task is to express \eqref{main eq trans} in well-chosen coordinates. In view of the expression of the left-hand-side of \eqref{main eq trans}, we observe that if we remove the terms in $\rho$, the principal part (that is, the second-order terms in $\gd$) is determined by $\Rd_{\alpha\beta}$. In order to investigate its structure, we perform first some basic calculations, which are valid for general Lorentzian manifolds in arbitrary local coordinates.

Let $(M,g)$ be a Lorentzian manifold with metric $g$ of signature $(-,+,+,+)$ and consider any local coordinate system $\{x^0,x^1,x^2,x^3\}$. Let $\Gamma_{\alpha\beta}^\gamma$ be the associated Christoffel symbols, and consider the D'Alembert operator $\Box_g = \nabla^{\alpha}\nabla_{\alpha}$ associated with $g$. The following lemma follows from a straighforward but tedious calculation. 

\begin{lemma}[Ricci curvaturein general coordinates]\label{lem Ricci}
With the notation\footnote{Of course, $\Gamma_{\gamma}$ are coordinate-dependent functions and are not the components of a tensor field.}
\be
\Gamma^\lambda:
=g^{\alpha\beta}\Gamma_{\alpha\beta}^\lambda,
\qquad 
\Gamma_{\lambda} := g_{\lambda\beta}\Gamma^\beta,
\ee
one has 
\be
R_{\alpha\beta}=-\frac{1}{2}g^{\alpha'\beta'}\del_{\alpha'}\del_{\beta'}g_{\alpha\beta}+\frac{1}{2}(\del_\alpha\Gamma_\beta
+\del_\beta\Gamma_\alpha) + \frac{1}{2} F_{\alpha\beta}(g; \del g, \del g),
\ee
where
$F_{\alpha\beta}(g; \del g, \del g)$ are nonlinear functions in the metric coefficients and are quadratic in their first-order derivatives.
The D'Alembert operator and the reduced D'Alembert operator 
$\widetilde\Box_g u := g^{\alpha'\beta'}\del_{\alpha'}\del_{\beta'}$ satisfy the relation 
\be
\Box_g u = g^{\alpha'\beta'}\del_{\alpha'}\del_{\beta'} u + \Gamma^{\delta}\del_\delta u
= \widetilde\Box_g u + \Gamma^{\delta}\del_\delta u, 
\ee
and these two operators thus coincide whenever the coefficients $\Gamma^\lambda$ vanish identically. 
\end{lemma}

\begin{proof} We will use the definitions 
$$
R_{\alpha\beta} = \del_{\lambda}\Gamma_{\alpha\beta}^{\lambda} - \del_{\alpha}\Gamma_{\beta\lambda}^{\lambda} + \Gamma_{\alpha\beta}^{\lambda}\Gamma_{\lambda\delta}^{\delta} - \Gamma_{\alpha\delta}^{\lambda}\Gamma_{\beta\lambda}^{\delta}
$$
and 
$$
\Gamma_{\alpha\beta}^{\lambda} = \frac{1}{2}g^{\lambda\lambda'}\big(\del_{\alpha}g_{\beta\lambda'}
+ \del_{\beta}g_{\alpha\lambda'} - \del_{\lambda'}g_{\alpha\beta}\big).
$$
We compute the expression of Ricci tensor and focus on the first two terms: 
$$
\aligned
&\del_{\lambda}\Gamma_{\alpha\beta}^{\lambda} - \del_\alpha\Gamma_{\beta\lambda}^{\lambda}
\\
&=
\frac{1}{2}\del_{\lambda}\big(g^{\lambda\delta}(\del_\alpha g_{\beta\delta} + \del_{\beta}g_{\alpha\delta}
- \del_\delta g_{\alpha\beta})\big)
- \frac{1}{2}\big(\del_\alpha(g^{\lambda\delta}(\del_{\beta}g_{\lambda\delta} + \del_{\lambda}g_{\beta\delta} - \del_\delta g_{\beta\lambda})\big)
\\
&=-\frac{1}{2}\del_{\lambda}\big(g^{\lambda\delta}\del_\delta g_{\alpha\beta}\big)
+ \frac{1}{2}\del_{\lambda}\big(g^{\lambda\delta}(\del_\alpha g_{\beta\delta} + \del_{\beta}g_{\alpha\delta})\big)
-\frac{1}{2}\del_\alpha\big(g^{\lambda\delta}\del_{\beta}g_{\lambda\delta}\big),
\endaligned
$$
so that 
\begin{equation}\label{eq pr 4-1-1}
\aligned
\del_{\lambda}\Gamma_{\alpha\beta}^{\lambda} - \del_\alpha\Gamma_{\beta\lambda}^{\lambda}
= 
&  -\frac{1}{2}g^{\lambda\delta}\del_{\lambda}\del_\delta g_{\alpha\beta}
   + \frac{1}{2}g^{\lambda\delta}\del_\alpha\del_{\lambda}g_{\delta\beta}
 \\
&  +\frac{1}{2}g^{\lambda\delta}\del_{\beta}\del_{\lambda}g_{\delta\alpha}
   -\frac{1}{2}g^{\lambda\delta}\del_\alpha\del_{\beta}g_{\lambda\delta} + \text{l.o.t.}, 
\endaligned
\end{equation}
where \text{l.o.t.} are quadratic terms involving first- or zero-order derivatives.

We focus on the term $\del_{\alpha}\Gamma_{\beta} + \del_{\beta}\Gamma_{\alpha}$ and obtain 
$$
\aligned
\Gamma^{\gamma} = \Gamma_{\alpha\beta}^{\gamma}g^{\alpha\beta}
&= \frac{1}{2}g^{\alpha\beta}g^{\gamma\delta}\big(\del_\alpha g_{\beta\delta} + \del_{\beta}g_{\alpha\delta} - \del_\delta g_{\alpha\beta}\big)
\\
&=g^{\gamma\delta}g^{\alpha\beta}\del_\alpha g_{\beta\delta} - \frac{1}{2}g^{\alpha\beta}g^{\gamma\delta}\del_\delta g_{\alpha\beta}
\endaligned
$$
and
$$
\Gamma_{\lambda} = g_{\lambda\gamma}\Gamma^{\gamma} = g^{\alpha\beta}\del_\alpha g_{\beta\lambda}
- \frac{1}{2}g^{\alpha\beta}\del_{\lambda}g_{\alpha\beta}.
$$
So, we have 
$$
\del_\alpha\Gamma_{\beta}
= \del_\alpha\big(g^{\delta\lambda}\del_\delta g_{\lambda\beta}\big)
- \frac{1}{2}\del_\alpha\big(g^{\lambda\delta}\del_{\beta}g_{\lambda\delta}\big),
$$
and therefore 
\begin{equation}\label{eq pr 4-1-2}
\del_\alpha\Gamma_{\beta} + \del_{\beta}\Gamma_\alpha
= g^{\gamma\delta}\del_\alpha\del_{\lambda}g_{\delta\beta}
+ g^{\lambda\delta}\del_{\beta}\del_{\lambda}g_{\delta\alpha}
- g^{\lambda\delta}\del_{\alpha\beta}g_{\lambda\delta} + \text{l.o.t.}
\end{equation}
It remains to compare \eqref{eq pr 4-1-1} and \eqref{eq pr 4-1-2}. 
\end{proof}

Observe that the field equation \eqref{Eq1-14} or the conformally transformed field equations \eqref{main eq trans1}, both, contain linear terms in the Ricci curvature. In order to exhibit the hyperbolicity property for the linear part of these systems (at least for the second-order terms in \eqref{Eq1-14}), we now introduce the so-called wave coordinate conditions. Recall that local coordinates are called a {\bf wave coordinate system} if its Christoffel symbols satisfy 
\begin{equation}
\Gamma^{\lambda} = g^{\alpha\beta}\Gamma_{\alpha\beta}^{\gamma} = 0.
\end{equation}
In view of \eqref{main eq trans1}, the principal part (after removing the terms in $\rho$) of 
$$
{N_g}_{\alpha\beta} - \frac{1}{2}g_{\alpha\beta}\tr(N_g) - \frac{1}{2}e^{2\rho}\big(\del_{\alpha}\Gammad + \del_{\beta}\Gammad\big)
$$
is $\frac{1}{2}e^{2\rho}\gd^{\alpha\beta}\del_{\alpha'}\del_{\beta'}\gd_{\alpha\beta}$ ---which is a quasi-linear wave operator. From this observation, we see that the equations \eqref{main eq trans} in wave coordinates 
with 
\be
\Gammad_{\lambda} = 0 
\ee
can be reformulated as:  
\begin{equation}
{N_g}_{\alpha\beta} - \frac{1}{2}g_{\alpha\beta}\tr(N_g) - \frac{1}{2}e^{2\rho}\big(\del_{\alpha}\Gammad_{\beta} + \del_{\beta}\Gammad_{\alpha}\big) = 8\pi\big(T_{\alpha\beta} - \frac{1}{2}\tr(T)g_{\alpha\beta}\big),
\end{equation}
while the trace equation \eqref{trace eq trans} becomes
\begin{equation}
\Box_{\gd}\rho + \Gammad^{\delta}\del_\delta\rho = \frac{W_2(\rho)}{6e^{2\rho}} + \frac{W_3(\rho)}{6e^{4\rho}} + \frac{4\pi \tr(T)}{3e^{4\rho}}. 
\end{equation}
Hence, in view of Lemma~\ref{lem Ricci}, the above system can be written in terms of the metric $\gd$ and its derivatives.
We emphasize that the trace equation \eqref{eq conformal wave c} and the evolution equation of matter field \eqref{eq conformal wave d} below are genuinely coupled to the field equations.

\begin{lemma}[Conformal field equations in wave coordinates] 
\begin{subequations}\label{eq conformal wave}
The field equations \eqref{main eq trans} in wave coordinates take the form 
\begin{equation}\label{eq conformal wave a}
\aligned
\gd^{\alpha'\beta'}\del_{\alpha'}\del_{\beta'}\gd_{\alpha\beta} 
 &=  F_{\alpha\beta}(\gd;\del \gd,\del \gd)  - 12\del_{\alpha}\rho\del_{\beta}\rho
\\
& \quad + \frac{W_2(\rho)}{e^{2\rho}}\gd_{\alpha\beta}
- \frac{8\pi}{e^{2\rho}}\big(2T_{\alpha\beta} - \tr(T)g_{\alpha\beta}\big), 
\endaligned
\end{equation}
supplemented with a `constraint equation' derived from the wave coordinate condition: 
\begin{equation}\label{eq conformal wave b}
\gd^{\alpha\beta}\Gammad_{\alpha\beta}^{\gamma} = 0.
\end{equation}
This system must also be supplemented with the trace equation
\begin{equation}\label{eq conformal wave c}
\gd^{\alpha'\beta'}\del_{\alpha'}\del_{\beta'}\rho = \frac{W_2(\rho)}{6e^{2\rho}} + \frac{W_3(\rho)}{6e^{4\rho}} + \frac{4\pi \tr(T)}{3e^{4\rho}}
\end{equation}
with $\rho = \frac{1}{2}f'(R_g)$, as well as the evolution equation for the matter 
\begin{equation}\label{eq conformal wave d}
\nablad^\alpha T_{\alpha\beta} = \big(\tr(T)\del_{\beta}\rho - 2g^{\delta\gamma}\del_{\gamma}\rho T_{\gamma\beta}\big)e^{-2\rho}.
\end{equation}
\end{subequations}
\end{lemma}


\subsection{A nonlinear wave system for the modified gravity theory}

The aim of this subsection is to study the `essential system' consisting of \eqref{eq conformal wave a} and \eqref{eq conformal wave b}. If we remove the terms in $\rho$, this is  a quasi-linear wave system with constraints (see \eqref{eq conformal wave b}), whose structure is quite involved. The strategy we propose is to replace these constraints by another differential equation which will turn out to be simpler to handle. We will write a new system which may (a priori) not be equivalent to the original system \eqref{main eq trans}. This system is defined as follows.

\begin{definition} The {\bf wave-reduced system in geometric form} 
associated with \eqref{main eq trans} is, by definition, 
\begin{subequations}\label{eq conformal wave-reduced geo}
\begin{equation}\label{eq conformal wave-reduced geo a}
{N_g}_{\alpha\beta} - \frac{1}{2}g_{\alpha\beta}\tr(N_g) - \frac{1}{2}e^{2\rho}\big(\del_{\alpha}\Gammad_{\beta} + \del_{\beta}\Gammad_{\alpha}\big) = 8\pi\big(T_{\alpha\beta} - \frac{1}{2}\tr(T)g_{\alpha\beta}\big),
\end{equation}
\begin{equation}\label{eq conformal wave-reduced geo b}
\nablad^\alpha T_{\alpha\beta} = \big(\tr(T)\del_{\beta}\rho - 2g^{\delta\gamma}\del_{\gamma}\rho T_{\gamma\beta}\big)e^{-2\rho}, 
\end{equation}
\end{subequations}
where ${N_g}_{\alpha\beta}$ is defined by \eqref{main eq trans1} and $e^{2\rho} = f'(R_g)$.
\end{definition}

Thanks to \eqref{main eq trans1} and Lemma \ref{lem Ricci}, the above system reads also 
\begin{equation}\label{eq conformal wave-reduced}
\aligned
&\gd^{\alpha'\beta'}\del_{\alpha'}\del_{\beta'}\gd_{\alpha\beta} =  F_{\alpha\beta}(\gd;\del \gd,\del \gd)  - 12\del_{\alpha}\rho\del_{\beta}\rho
+ \frac{W_2(\rho)}{e^{2\rho}}\gd_{\alpha\beta} 
\\
& \hskip5.cm  - \frac{8\pi}{e^{2\rho}}\big(2T_{\alpha\beta} - \tr(T)g_{\alpha\beta}\big)
\\
&\nablad^\alpha T_{\alpha\beta} = \big(\tr(T)\del_{\beta}\rho - 2g^{\delta\gamma}\del_{\gamma}\rho T_{\gamma\beta}\big)e^{-2\rho}.
\\
&\rho = \frac{1}{2}\ln f'(R_g), 
\endaligned
\end{equation}
which is \eqref{eq conformal wave} without the constraint equations \eqref{eq conformal wave b} but includes the evolution equation \eqref{eq conformal wave d}. 

In the next subsection, we are going to establish the following result.

\begin{proposition}[Preservation of the wave coordinate conditions in modified gravity]
\label{prop presev 1} 
Consider a globally hyperbolic spacetime $M = \cup_{t\in[0,t_{\max})}M_t$ with metric $\gd$, together with a matter field $T$ defined on $M$. Suppose that the wave coordinate conditions
\begin{equation}
\label{eq preserve 1-1}
\Gammad^{\gamma} := \gd^{\alpha\beta}\Gammad_{\alpha\beta}^{\gamma} = 0
\qquad \text{ on the initial slice $M_0$,}
\end{equation}
together with the constraint equations \eqref{eq constraint conformalT H} and \eqref{eq constraint conformalT M}. Then, the wave conditions \eqref{eq preserve 1-1} are satisfied within the whole of $M$. 
\end{proposition}

In other words, if one wants to find a solution of \eqref{main eq trans}, what we need to do is to find first a solution of \eqref{eq conformal wave-reduced} with the constraint equations \eqref{eq preserve 1-1} satisfied on the initial slice, with four additional constraint equations to be required on the initial data set. Recall again that the interest of relying on \eqref{eq conformal wave-reduced geo} rather than on \eqref{main eq trans} is that the former one has a hyperbolic principal part (after removing the terms in $\rho$).


\subsection{Preservation of the wave coordinate conditions}

The key to Proposition~\ref{prop presev 1} is the contracted Bianchi identity and, amazingly, the precise form of the modified gravity tensor $N_g$ is not used at this juncture. We begin with some lemmas and the derivations of key identities.

\begin{lemma}\label{lem preserve-wave-1}
For any (pseudo-riemannian) manifold $(M,g)$, the following identity holds:
\begin{equation}\label{eq 2-4-0}
\aligned
\nabla^{\alpha}\big(\del_{\alpha}\Gamma_\beta + \del_{\beta}\Gamma_{\alpha} - g_{\alpha\beta}g^{\alpha'\beta'}\del_{\alpha'}\Gamma_{\beta'}\big)
&=g^{\alpha'\beta'}\del_{\alpha'}\del_{\beta'}\Gamma_{\beta}
- g^{\alpha\alpha'}\Gamma_{\alpha\alpha'}^{\gamma}\del_{\gamma}\Gamma_{\beta}
- g^{\alpha\alpha'}\Gamma^{\delta}_{\alpha'\beta}\del_\delta\Gamma_{\alpha}
\\
&- \del_{\beta}g^{\alpha'\beta'}\big(\del_{\alpha'}\Gamma_{\beta'}\big), 
\endaligned
\end{equation}
where $\Gamma_{\alpha\beta}^{\gamma}$ denote the Christoffel symbols and 
$\Gamma^{\gamma} := g^{\alpha\beta}\Gamma^{\gamma}_{\alpha\beta}$ and $\Gamma_{\gamma} = g_{\gamma'\gamma}\Gamma^{\gamma'}$.
\end{lemma}

\begin{proof} From the standard identities
$$
\aligned
&\nabla^{\alpha}\del_{\beta}u = g^{\alpha\alpha'}\del_{\alpha'}\del_{\beta}u - \Gamma_{\alpha'\beta}^{\gamma}\del_{\gamma}u,
\\
&\Box_gu = g^{\alpha\beta}\del_{\alpha}\del_{\beta}u - g^{\alpha\beta}\Gamma_{\alpha\beta}^{\gamma}\del_{\gamma}u, 
\endaligned
$$
we find
$$
\aligned
&\nabla^{\alpha}\big(\del_{\alpha}\Gamma_\beta + \del_{\beta}\Gamma_{\alpha} - g_{\alpha\beta}g^{\alpha'\beta'}\del_{\alpha'}\Gamma_{\beta'}\big)
\\
&=
\nabla^\alpha\del_{\alpha}\Gamma_{\beta} + \nabla^{\alpha}\del_{\beta}\Gamma_{\alpha} - \del_{\beta}\big(g^{\alpha'\beta'}\del_{\alpha'}\Gamma_{\beta'}\big)
\\
&= \Box_{g}\Gamma_{\beta} + g^{\alpha\alpha'}\del_{\alpha'}\del_{\beta}\Gamma_{\alpha}
- g^{\alpha\alpha'}\Gamma^{\delta}_{\alpha'\beta}\del_\delta\Gamma_{\alpha}
- \del_{\beta}g^{\alpha'\beta'}\big(\del_{\alpha'}\Gamma_{\beta'}\big)
-  g^{\alpha'\beta'}\del_{\beta}\del_{\alpha'}\Gamma_{\beta'}
\endaligned
$$
thus
$$
\aligned
&\nabla^{\alpha}\big(\del_{\alpha}\Gamma_\beta + \del_{\beta}\Gamma_{\alpha} - g_{\alpha\beta}g^{\alpha'\beta'}\del_{\alpha'}\Gamma_{\beta'}\big)
\\
&=  \Box_{g}\Gamma_{\beta} - g^{\alpha\alpha'}\Gamma^{\delta}_{\alpha'\beta}\del_\delta\Gamma_{\alpha}
- \del_{\beta}g^{\alpha'\beta'}\big(\del_{\alpha'}\Gamma_{\beta'}\big)
\\
&= g^{\alpha'\beta'}\del_{\alpha'}\del_{\beta'}\Gamma_{\beta}
- g^{\alpha\alpha'}\Gamma_{\alpha\alpha'}^{\gamma}\del_{\gamma}\Gammad_{\beta}
- g^{\alpha\alpha'}\Gamma^{\delta}_{\alpha'\beta}\del_\delta\Gamma_{\alpha}
- \del_{\beta}g^{\alpha'\beta'}\big(\del_{\alpha'}\Gamma_{\beta'}\big).
\endaligned
$$
\end{proof}

Our next lemma establishes a relation between the wave condition and the evolution equation of the wave-reduced system \eqref{eq conformal wave-reduced}. Recall that sufficiently regular is assumed throughout so that all terms uder consideration are continuous functions at least. 

\begin{lemma}\label{lem pre 1}
Consider an arbitrary manifold $(M,\gd)$ and a (matter) tensor $T_{\alpha\beta}$. Then, if in some local coordinate system $\{x^0,x^1,x^2,x^3\}$, $\gd$, $T_{\alpha\beta}$ satisfy the system of equations \eqref{eq conformal wave-reduced geo}, then the following equations hold (in the domain described by the coordinates):
\begin{equation}
\aligned
\gd^{\alpha'\beta'}\del_{\alpha'}\del_{\beta'} \Gammad_{\beta} &= F_{\beta}(\rho,\gd;\Gammad_{\gamma},\del\Gammad_{\gamma}), 
\endaligned
\end{equation}
where $F_{\beta}(\rho,\gd;\cdot,\cdot)$ is a combination of linear and bilinear forms, and one recalls $\Gammad^{\gamma}_{\alpha\beta}$ are the Christoffel symbols with 
$\Gammad_{\beta} = g_{\beta\beta'}\Gammad^{\beta'} = g_{\beta\beta'}g^{\alpha\gamma}\Gammad_{\alpha\gamma}^{\beta'}$.
\end{lemma}

\begin{proof}
Taking the trace of \eqref{eq conformal wave-reduced geo a} with respect to the metric $g$,
\begin{equation}\label{eq 4-4-pr2-1}
\tr(N_g - 8\pi T) = -e^{2\rho}g^{\alpha\beta}\del_{\alpha}\Gammad_{\beta}
\end{equation}
and combining with \eqref{eq conformal wave-reduced geo a}, we obtain 
\begin{equation}\label{eq 4-4-pr2-2}
\aligned
{N_g}_{\alpha\beta} - 8\pi T_{\alpha\beta}
&= \frac{1}{2}g_{\alpha\beta}\tr(N_g - 8\pi T)
+ \frac{1}{2}e^{2\rho}\big(\del_{\alpha}\Gammad_\beta + \del_{\beta}\Gammad_{\alpha}\big)
\\
&= \frac{1}{2}e^{2\rho}\big(\del_{\alpha}\Gammad_\beta + \del_{\beta}\Gammad_{\alpha}
- \gd_{\alpha\beta}\gd^{\alpha'\beta'}\del_{\alpha'}\Gammad_{\beta'}\big)
\endaligned
\end{equation}
Taking the trace of \eqref{eq 4-4-pr2-2} {\bf with respect to $\gd$}, we obtain 
\begin{equation}\label{eq 4-4-pr2-3}
\frac{1}{2}\nablad^{\alpha}\big(e^{2\rho}\big(\del_{\alpha}\Gammad_\beta + \del_{\beta}\Gammad_{\alpha}
- \gd_{\alpha\beta}\gd^{\alpha'\beta'}\del_{\alpha'}\Gammad_{\beta'}\big)\big)
= \nablad^{\alpha}\big({N_g}_{\alpha\beta} - 8\pi T_{\alpha\beta}\big), 
\end{equation}
whose left-hand-side is evaluated by using \eqref{eq 2-4-0}: 
$$
\aligned
&\frac{1}{2}\nablad^{\alpha}\big(e^{2\rho}\big(\del_{\alpha}\Gammad_\beta + \del_{\beta}\Gammad_{\alpha}
- \gd_{\alpha\beta}\gd^{\alpha'\beta'}\del_{\alpha'}\Gammad_{\beta'}\big)\big)
\\
&=\frac{1}{2}e^{2\rho}\nablad^{\alpha}\big(\del_{\alpha}\Gammad_\beta + \del_{\beta}\Gammad_{\alpha}
- \gd_{\alpha\beta}\gd^{\alpha'\beta'}\del_{\alpha'}\Gammad_{\beta'}\big)
\\
&\quad+ e^{2\rho} \big(\del_{\alpha}\Gammad_\beta + \del_{\beta}\Gammad_{\alpha}
- \gd_{\alpha\beta}\gd^{\alpha'\beta'}\del_{\alpha'}\Gammad_{\beta'}\big)\nablad^{\alpha}\rho,
\endaligned
$$
thus 
$$
\aligned
&\frac{1}{2}\nablad^{\alpha}\big(e^{2\rho}\big(\del_{\alpha}\Gammad_\beta + \del_{\beta}\Gammad_{\alpha}
- \gd_{\alpha\beta}\gd^{\alpha'\beta'}\del_{\alpha'}\Gammad_{\beta'}\big)\big)
\\
&= \frac{1}{2}e^{2\rho}\Big(\gd^{\alpha'\beta'}\del_{\alpha'}\del_{\beta'}\Gammad_{\beta}
- \gd^{\alpha\alpha'}\Gammad_{\alpha\alpha'}^{\gamma}\del_{\gamma}\Gammad_{\beta}
- \gd^{\alpha\alpha'}\Gammad^{\delta}_{\alpha'\beta}\del_\delta\Gammad_{\alpha}
- \del_{\beta}\gd^{\alpha'\beta'}\big(\del_{\alpha'}\Gammad_{\beta'}\big)\Big)
\\
&\quad + e^{2\rho} \big(\del_{\alpha}\Gammad_\beta + \del_{\beta}\Gammad_{\alpha}
- \gd_{\alpha\beta}\gd^{\alpha'\beta'}\del_{\alpha'}\Gammad_{\beta'}\big)\nablad^{\alpha}\rho
\\
&=:\frac{1}{2}e^{2\rho}\gd^{\alpha'\beta'}\del_{\alpha'}\del_{\beta'}\Gammad_{\beta}
+ \frac{1}{2}e^{2\rho} \tilde{F}_{\beta}(\rho,\gd;\Gammad_{\gamma},\del \Gammad_{\gamma})
\endaligned
$$
where $F_{\beta}(\rho,\gd;\cdot)$ is a combination of linear and bilinear forms of the functions $\Gammad_{\gamma}$ and $\del \Gammad_{\gamma}$.

The right-hand-side of \eqref{eq 4-4-pr2-3} is computed by using the identity \eqref{eq divN-bianchi convformal} and \eqref{eq conformal wave-reduced geo b}, that is, 
$$
\nablad^{\alpha}\big({N_g}_{\alpha\beta} - 8\pi T_{\alpha\beta}\big)
= e^{-2\rho}\big(\tr(N_g - 8\pi T)\del_{\beta}\rho - 2g^{\gamma\delta}\del_{\gamma}\rho ({N_g}_{\delta\beta}-8\pi T_{\delta\beta})\big).
$$
Then, by \eqref{eq 4-4-pr2-1}, we obtain 
$$
\aligned
&\nablad^{\alpha}\big({N_g}_{\alpha\beta} - 8\pi T_{\alpha\beta}\big)
\\
&= e^{-2\rho}\Big(-e^{2\rho}g^{\alpha'\beta'}\del_{\alpha'}\Gammad_{\beta'}\del_{\beta}\rho - e^{2\rho}g^{\gamma\delta}\del_{\gamma}\rho\big(\del_{\alpha}\Gammad_{\beta} + \del_{\beta}\Gammad_{\alpha} - \gd_{\alpha\beta}\gd^{\alpha'\beta'}\del_{\alpha'}\Gammad_{\beta'}\big)\Big)
\\
&=-g^{\gamma\delta}\del_{\gamma}\rho\big(\del_{\gamma}\Gammad_{\beta}+\del_{\beta}\Gammad_\delta\big)
\endaligned
$$
and, by \eqref{eq 4-4-pr2-3},
\begin{equation}
\gd^{\alpha'\beta'}\del_{\alpha'}\del_{\beta'}\Gammad_{\beta}
= F_{\beta}(\rho,\gd;\Gammad_{\gamma}, \del \Gammad_{\gamma}),
\end{equation}
where
$$F_{\beta}(\rho,\gd;\Gammad_{\gamma}, \del \Gammad_{\gamma}) = -\tilde{F}_{\beta}(\rho,\gd;\Gammad_{\gamma}, \del \Gammad_{\gamma})
 - 2e^{-2\rho}g^{\gamma\delta}\del_{\gamma}\rho\big(\del_\delta\Gammad_{\beta} + \del_{\beta}\Gammad_\delta\big)
$$
is a combination of linear and bilinear forms in $\Gammad_{\gamma}$ and $\del \Gammad_{\gamma}$. 
\end{proof}

\begin{lemma}\label{lem pre 2}
Let $(M,\gd)$ be a globally hyperbolic Lorentzian manifold endowed with foliation $M = [0,t_{max})\times M_t$ (and signature $(-,+,+,+)$), together with a tensor field $T_{\alpha\beta}$. Suppose that the equation \eqref{eq conformal wave-reduced geo} holds on the initial slice $M_0$ and, furthermore,  the wave coordinate conditions and the constraint equations hold on the slice $M_0$: 
$$
\Gammad^{\gamma} := \gd^{\alpha\beta}\Gammad_{\alpha\beta}^{\gamma} = 0
$$
and (as stated in \eqref{constraint trans Halmitonian} and \eqref{constraint trans momentum})
$$
\aligned
&\Rbd - \Kd_{ij}\Kd^{ij} + (\Kd_l^l)^2 = \frac{16\sigmad}{e^{2\rho}} + \frac{12|<e_0,\rho>|^2}{\Nd^2}
+ 6\gd(\nabla \rho,\nabla\rho) - e^{-2\rho}W_2(\rho),
\\
&\del_j\Kd_l^l - \bar{\nablad}\Kd_j^l = -\frac{\Jd_j}{e^{2\rho}} + \frac{6<e_0,\rho>\del_j\rho}{\Nd}.
\endaligned
$$
Then, one has 
$$
\del_0\Gammad^{\lambda} = 0 \qquad \text{ in the spacetime } M.
$$
\end{lemma}

\begin{proof} We work in a Cauchy adapted frame $(e_0,e_1,e_2,e_3)$, that is, 
$$
\aligned
&e_0 = \del_0 - \beta^i\del_i, \qquad \beta^i = \frac{\gd_{0i}}{\gd_{ii}},
\qquad e_i = \del_i, 
\endaligned
$$
so that $\gd(e_0,e_i) = 0$. A tensor can be written in, both, the natural frame and the Cauchy adapted frame. We denote by an underlined letter the components in the Cauchy adapted frame. For example, $\underline{T}_{\alpha\beta}$ are the components of $T$ in the Cauchy frame.

Recall that the momentum constraint equations are equivalent to
\begin{equation}\label{eq pr 4-10-1}
\underline{N_g}_{0j} = 8\pi \underline{T}_{0j},
\end{equation}
the Hamiltonian constraint equation is equivalent to
\begin{equation}\label{eq pr 4-10-2}
\underline{N_g}_{00} = 8\pi \underline{T}_{00}.
\end{equation}
Recall also that the Cauchy adapted frame is expressed in the natural frame via
$e_\alpha = \Phi_{\alpha}^{\beta}\del_{\beta}$, 
where
$$
\big(\Phi_{\alpha}^{\beta}\big)_{\alpha\beta} =
\left(
\begin{array}{cccc}
1 &-\beta^1 &-\beta^2 &-\beta^3
\\
0 &1 &0 &0
\\
0 &0 &1 &0
\\
0 &0 &0 &1
\end{array}
\right)
$$
Then, we have 
$$
\underline{N_g}_{\alpha\beta} = {N_g}_{\alpha'\beta'}\Phi_{\alpha}^{\alpha'}\Phi_{\beta}^{\beta'},
\qquad
\underline{T}_{\alpha\beta} = T_{\alpha'\beta'}\Phi_{\alpha}^{\alpha'}\Phi_{\beta}^{\beta'}.
$$

Observe that the wave-reduced field equation \eqref{eq conformal wave-reduced geo a} can be rewritten in the Cauchy adapted frame as
$$
\underline{N_g}_{\alpha\beta} - \frac{1}{2}\tr{N_g}\underline{g}^{\dag}_{\alpha\beta}
- \frac{1}{2}e^{2\rho}\Phi_{\alpha}^{\alpha'}\Phi_{\beta}^{\beta'}\big(\del_{\alpha'}\Gammad_{\beta'} + \del_{\beta'}\Gammad_{\alpha'}\big)
 = 8\pi\big(\underline{T}_{\alpha\beta} - \frac{1}{2}\tr(T)\underline{g}^{\dag}_{\alpha\beta}\big), 
$$
which is
\begin{equation}\label{eq pr 4-10-3}
\aligned
& \underline{N_g}_{\alpha\beta} - \frac{1}{2}\tr{N_g}\underline{g}^{\dag}_{\alpha\beta}
- \frac{1}{2}e^{2\rho}\big(\Phi_{\beta}^{\beta'}<e_{\alpha},\Gammad_{\beta'}> + \Phi_{\alpha}^{\alpha'}<e_{\beta},\Gammad_{\alpha'}>\big)
\\
& = 8\pi\big(\underline{T}_{\alpha\beta} - \frac{1}{2}\tr(T)\underline{g}^{\dag}_{\alpha\beta}\big).
\endaligned
\end{equation}
Next, by combining \eqref{eq pr 4-10-3} with \eqref{eq pr 4-10-1} with $\alpha = 0$, $1\leq \beta = b\leq 3$ in \eqref{eq pr 4-10-3} and by observing that $\underline{\gd}_{0j} = 0$, we obtain 
$$
\Phi_b^{\beta'}<e_0,\Gammad_{\beta'}> + \Phi_0^{\alpha'}<e_b,\Gammad_{\alpha'}> = 0.
$$
We consider this equation on the initial slice $M_0$. 

Recall that $\Gammad^{\lambda} = 0$ so that $\del_b\Gamma_{\beta} = 0$ for any $1\leq b\leq 3$ and $0\leq \beta\leq 3$. Then observe that $\Phi_b^{\beta'} = \delta_b^{\beta'}$ and $<e_b,\Gammad_{\alpha'}> = \del_b\Gammad_{\alpha'} = 0$, so that
\begin{equation}\label{eq pr 4-10-4}
<e_0,\Gammad_b> = 0,
\end{equation}
which leads us to $\del_t \Gammad_b = 0$.

Now, we can combine \eqref{eq pr 4-10-3} with \eqref{eq pr 4-10-1} and \eqref{eq pr 4-10-2} with $\alpha = \beta = 0$ in \eqref{eq pr 4-10-3}:
$$
\Phi_0^{\beta'}<e_0,\Gammad_{\beta'}> + \Phi_0^{\alpha'}<e_0,\Gammad_{\alpha'}> = - \tr({N_g - 8\pi T})\gd_{00}.
$$
We recall \eqref{eq pr 4-10-1} and \eqref{eq pr 4-10-2} and the fact that $\gd_{0j} = 0$, so that 
\begin{equation}\label{eq pr 4-10-5}
\tr(N_g - 8\pi T)
 = \underline{g}^{\alpha\beta}\big(\underline{N_g}_{\alpha\beta} - 8\pi \underline{T}_{\alpha\beta}\big)
 = \underline{g}^{ij}\big(\underline{N_g}_{ij} - 8\pi \underline{T}_{ij}\big).
\end{equation}
We also recall \eqref{eq pr 4-10-3} with $1\leq \alpha=i\leq 3$ and $1\leq \beta = j\leq 3$, and we observe that $<e_i,\Gammad_{\alpha}> = 0$. This shows that 
$$
\underline{{N_g}}_{ij} - 8\pi \underline{T}_{ij} = \frac{1}{2}\tr(N_g - 8\pi T)\underline{g}_{ij}, 
$$
which leads us to
$$
\underline{g}^{ij}\big(\underline{{N_g}}_{ij} - 8\pi \underline{T}_{ij}\big) =  \frac{1}{2}\tr(N_g - 8\pi T)\underline{g}^{ij}\underline{g}_{ij}. 
$$
Therefore, by \eqref{eq pr 4-10-5}, we have 
$
\tr(N_g - 8\pi T) = \frac{3}{2}\tr(N_g - 8\pi T), 
$
and thus 
$
\tr(N_g - 8\pi T) = 0.
$
We substitute this conclusion in \eqref{eq pr 4-10-4} and obtain 
$$
\Phi_0^{\beta'}<e_0,\Gammad_{\beta'}> + \Phi_0^{\alpha'}<e_0,\Gammad_{\alpha'}> =0.
$$
We finally recall \eqref{eq pr 4-10-4} and get 
$$
\Phi_0^0<e_0,\Gammad_0> + \Phi_0^0<e_0,\Gammad_0> =0
$$
and the desired conclusion is reached. 
\end{proof}

\begin{proof}[Proof of Proposition \ref{prop presev 1}] 
In view of Lemma~\ref{lem pre 1} and Lemma~\ref{lem pre 2}, we see that $\Gammad_{\beta}$ satisfies the initial value problem
$$
\gd^{\alpha'\beta'}\del_{\alpha'}\del_{\beta'} \Gammad_{\beta} = F_{\beta}(\rho,\gd;\Gammad_{\gamma},\del\Gammad_{\gamma})
$$
with initial data
$$
\Gammad_{\beta}|_{x^0=0} = 0, \quad \del_0\Gammad_{\beta}|_{x^0 = 0} = 0.
$$
Since $x^0$ is the time-like direction and the symmetry of $\gd$ guarantees the hyperbolicity of $\gd^{\alpha'\beta'}\del_{\alpha'}\del_{\beta'}$. We also observe that $\Gammad=0$ is a solution of this initial value problem.Thanks to the global hyperbolicity of $\gd$, the desired uniqueness result holds within the domain of determinacy of the initial slice, that is, $M$ itself, thanks to our global hyperbolicity assumption.
\end{proof}


\newpage 

\section{The augmented conformal formulation}

\subsection{A novel formulation}

In this section we will re-formulate again our system and establish a local-in-time existence theory for the system \eqref{main eq trans}. Since $\rho$ is function of the scalar curvature, the syste under consideration now is still a {\sl third-order} system and does not enjoy any special structure as far as the principal part is concerned.
(The third-order terms are $\del \rho$, where $\rho$ is a function of the second-order derivatives of $\gd_{\alpha\beta}$.) To bypass this difficulty, we introduce still another transformation. 

In the present section, we propose, in the system \eqref{eq conformal wave-reduced geo}, to replace the constraint $e^{2\rho} = f'(R_g)$ by the trace equation leading to the evolution law for $\rho$. We then formulate an {\sl argumented formulation,} as we call it. 
For clarity in the presentation, we switch from the letter $\rho$ to the letter $\varrho$, in order the emphasize that the relation $e^{2\rho} = f'(R_g)$ is {\sl no longer imposed}.  

Let us now define the tensor ${N_g^{\ddag}}_{\alpha\beta}$ as 
\begin{equation}\label{eq conformal aug geo0}
{N_g^{\ddag}}_{\alpha\beta}  - \frac{1}{2}g_{\alpha\beta}\tr(N_g^\ddag)
= e^{2\varrho}\Rv_{\alpha\beta} - 6e^{2\varrho}\del_{\alpha}\varrho\del_{\beta}\varrho + \frac{1}{2}\gv_{\alpha\beta}W_2(\varrho).
\end{equation}
Here, $\varrho$ plays the role of $\rho = \frac{1}{2}\ln f'(R_g)$, which is no longer imposed, and we work in the {\bf metric of the augmented system}  
\be
g^{\ddag}_{\alpha\beta} = e^{2\varrho}g_{\alpha\beta}.
\ee 
We also use the notation $\nabla^\ddag$, $R^\ddag_{\alpha}$ and $\Rv$ for the connection, the Ricci curvature tensor and the scalar curvature of the metric $\gv$, respectively. Also, $\Gammav_{\alpha\beta}^{\gamma}$ denote the Christoffel symbols of $\gv$ in a natural frame, and we set $\Gammav^{\gamma} := \gv^{\alpha\beta}\Gammav_{\alpha\beta}^{\gamma}$ and $\Gammav_{\gamma} := \gv_{\gamma\gamma'}\Gammav^{\gamma}$.

\begin{definition}
The {\bf conformal augmented formulation of the modified gravity theory} is the following partial differential system:
\begin{subequations}\label{eq conformal aug geo0-deux}
\begin{equation}\label{eq conformal aug geo0 a}
{N_g^{\ddag}}_{\alpha\beta}  - \frac{1}{2}e^{2\varrho}\big(\del_{\alpha}\Gammav_{\beta} + \del_{\beta}\Gammav_{\alpha}\big)
= 8\pi \big(T_{\alpha\beta} - \frac{1}{2}g_{\alpha\beta}\tr(T)\big).
\end{equation} 
\begin{equation}\label{eq conformal aug geo0 c}
\Box_{\gv} \varrho = \frac{W_2(\varrho)}{6e^{2\varrho}} + \frac{f(\rhoR(\varrho))}{6e^{4\varrho}}
+ \frac{4\pi\tr(T)}{3e^{4\varrho}}\gv_{\alpha\beta}.
\end{equation}
\end{subequations}
\end{definition}

In the proposed standpoint, $\varrho$ becomes an {\sl independent unknown} function, which is no longer regarded as a function of the scalar curvature $R_g$. In this form, the system under consideration is {\sl second-order} and our first task is to compute the divergence of ${N_g^{\ddag}}$.

\begin{lemma}\label{lem 5-3-1-0}
When \eqref{eq conformal aug geo0 c} holds, the following identity also holds:
\begin{equation}\label{eq lem 5-3-1-0}
\nablav^{\alpha}{N_g^{\ddag}}_{\alpha\beta}
= e^{-2\varrho}\big(2g^{\alpha\alpha'}\del_{\alpha'}\varrho {N_g^{\ddag}}_{\alpha\beta} - \tr(8\pi T)\del_{\beta}\varrho\big).
\end{equation}
\end{lemma}

\begin{proof} By \eqref{eq conformal aug geo0-deux}, we have 
$$
{N_g^{\ddag}}_{\alpha\beta}
= e^{2\varrho}G^{\ddag}_{\alpha\beta}
- 6e^{2\varrho}\big(\del_{\alpha}\varrho\del_{\beta}\varrho - \frac{1}{2}g_{\alpha\beta}|\nabla\varrho|^2_{\gv}\big)
- \frac{1}{2}\gv_{\alpha\beta}W_2(\varrho), 
$$
where
$\Gv_{\alpha\beta} := \Rv_{\alpha\beta} - \frac{1}{2}\gv_{\alpha\beta}\Rv$
is the Einstein curvature of $\gv$. We start from the identities
$$
\nablav^{\alpha\beta}G_{\alpha\beta}^{(\varrho)} = 0,
$$
and 
$$
\nablav^{\alpha\beta}\big(\del_{\alpha}\varrho\del_{\beta}\varrho - \frac{1}{2}g_{\alpha\beta}|\nabla\varrho|^2_{\gv}\big)
=\del_{\beta}\varrho\,\Box_{\gv}\varrho, 
$$
and we introduce the function 
\be
s(r) = \frac{1}{2}\ln f'(r), \qquad r \in \RR 
\ee
together with its (local) inverse (of near $0$) denotes by $\theta$. Then, we have 
$$
\aligned
\nablav^{\alpha\beta}\big(\gv_{\alpha\beta}W_2(\varrho)\big)
&= \del_{\beta}\big(W_2(\varrho)\big)
= \del_{\beta}\big(e^{-2\varrho}f(\rhoR(\varrho)) - \rhoR(\varrho)\big)
\\
&= -2e^{-2\varrho}\del_{\beta}\varrho f(\rhoR(\varrho)) + e^{-2\varrho}f'(\rhoR(\varrho))\rhoR'(\varrho)\del_{\beta}\varrho - \rhoR'(\varrho)\del_{\beta}\varrho
\\
&=-2e^{-2\varrho}\del_{\beta}\varrho f(\rhoR(\varrho)) + e^{-2\varrho}e^{2\varrho}\rhoR'(\varrho)\del_{\beta}\varrho - \rhoR'(\varrho)\del_{\beta}\varrho
\\
&= -2e^{-2\varrho}\del_{\beta}\varrho f(\rhoR(\varrho)).
\endaligned
$$
This allows us to compute the divergence of ${N_g^{\ddag}}_{\alpha\beta}$:
$$
\aligned
\nablav^{\alpha}{N_g^{\ddag}}_{\alpha\beta}
&= 2e^{2\varrho} G^{\varrho}_{\alpha\beta} \nablav^{\alpha}\varrho
  -12e^{2\varrho} \big(\del_{\alpha}\varrho\del_{\beta}\varrho - \frac{1}{2}g_{\alpha\beta}|\nabla\varrho|^2_{\gv}\big)\nablav^{\alpha}\varrho
\\
 & \quad -6e^{2\varrho}\del_{\beta}\varrho\Box_{\gv}\varrho + e^{2\varrho}\del_{\beta}f(\rhoR(\varrho))
\\
&= 2e^{2\varrho} G^{\varrho}_{\alpha\beta} \nablav^{\alpha}\varrho
  -12e^{2\varrho} \big(\del_{\alpha}\varrho\del_{\beta}\varrho - \frac{1}{2}g_{\alpha\beta}|\nabla\varrho|^2_{\gv}\big)\nablav^{\alpha}\varrho
  - \gv_{\alpha\beta}W_2(\varrho)\nablav^{\alpha}\varrho
\\
 & \quad -6e^{2\varrho}\del_{\beta}\varrho\Box_{\gv}\varrho + e^{-2\varrho}\del_{\beta}f(\rhoR(\varrho))
 + \gv_{\alpha\beta}W_2(\varrho)\nablav^{\alpha}\varrho
\\
&= 2{N_g^{\ddag}}_{\alpha\beta}  \nablav^{\alpha}\varrho
- 6e^{2\varrho}\del_{\beta}\varrho
\bigg(\Box_{\gv} - \frac{f(\rhoR(\varrho))}{6e^{4\varrho}} - \frac{W_2(\varrho)}{6e^{2\varrho}}\bigg).
\endaligned
$$
Then, by \eqref{eq conformal aug geo c}, we find 
$$
\Box_{\gv}\varrho
= \frac{f(\rhoR(\varrho))}{6e^{4\varrho}} + \frac{W_2(\varrho)}{6e^{2\varrho}} + \frac{4\pi \tr(T)}{3e^{4\varrho}} 
$$
and the desired conclusion follows. 
\end{proof}

By Lemma~\ref{lem 5-3-1-0}, we see that the evolution law for $T$ is 
\begin{equation}\label{eq conformal aug geo0 b}
\nablav^\alpha T_{\alpha\beta}
= \big(2g^{\delta\gamma}\del_{\gamma}\varrho T_{\gamma\beta} - \tr(T)\del_{\beta}\varrho\big)e^{-2\varrho}. 
\end{equation}
The following question arises at this juncture:  Will the relation $e^{2\varrho} = f'(R_g)$ (with $\gd_{\alpha\beta} = e^{2\varrho}g_{\alpha\beta}$) hold if we solely solve the equations \eqref{eq conformal aug geo0-deux}? The following subsections precisely provide a (positive) answer to this question.

First, in Section~\ref{section52}, we will re-formulate the initial value problem for the augmented conformal formulation, by building upon our previous formulations of the field equations in the Jordan and Einstein metrics.

Then, in Section~\ref{section53}, in order for the principal part of \eqref{eq conformal aug geo0-deux} to be hyperbolic, we will write our augmented system in wave coordinates. Finally, in Section~\ref{section54}, we will prove that once the wave constraint equations $\Gammav_\lambda = 0$ hold on the initial slice, then the augmentation relation $e^{2\varrho} = f'(R_g)$ is guaranteed by \eqref{eq conformal aug}.
 

\subsection{Initial value formulation for the augmented conformal system}
\label{section52}

In this section, we revisit Definitions \ref{def36} and \ref{def cauchy conformal}, which is is required for the clarity of the following discussion.
 
\begin{definition}\label{def36aug}
An {\bf initial data set for the augmented conformal formulation of modified gravity} $(\Mb, \gb^{\ddag}, \Kv, \varrhozero, \varrhoone, \phizerodd, \phionedd)$ consists of the following data:
\begin{itemize}
\item a $3$-dimensional manifold $\Mb$ endowed with a Riemannian metric $\gb^{\ddag}$ and a symmetric $(0,2)$-tensor field $\Kv$,
\\
\item two scalar fields denoted by $\varrhozero$ and $\varrhoone$ on $\Mb$ and representing the (to-be-constructed) conformal factor and its time derivative,
\\
\item and two scalar field $\phizerodd$ and $\phionedd$ defined on $\Mb$
\end{itemize}
Moreover, these data are required to satisfy the {\bf Hamiltonian constraint of modified gravity in the augmented conformal form}
\begin{equation}
\label{constraint trans-aug Hamiltonian}
\aligned
\Rb^{\ddag} - \Kv_{ij}\Kv^{ij}+(\Kv_j^j)^2 
= 
&8 \, e^{-2\rho} \, \big((\phionedd)^2
+ {\gb^{\ddag}}^{ij}\del_i\phizerodd\del_j\phizerodd\big) 
\\
&+ 6 (\varrhoone)^2 + 6\gbd^{ij}\del_i\varrhozero\del_j\varrhozero
 - e^{-2\varrhozero}W_2(\varrhozero)
\endaligned
\end{equation}
and the {\bf momentum constraint in the augmented conformal form}
\begin{equation}
\label{constraint trans-aug momentum}
\del_j \Kv_i^i - {\nablab^{\ddag}}_i \Kv_j^i = \frac{\phionedd\del_j\phizerodd}{e^{2\varrhozero}} + 6\varrhoone\del_j\varrhozero.
\end{equation}
Here, $\Rb^{\ddag}$ and ${\nablab^{\ddag}}$ is the scalar curvature and the connection of the metric $\gb^{\ddag}$, respectively.
\end{definition}

\begin{definition}\label{Cauchy conformal aug} Given an initial data set $(\Mb, \gb^{\ddag}, \Kv, \varrhozero, \varrhoone, \phizerodd, \phionedd)$ as in Definition~\ref{def36aug},
the {\bf initial value problem in modified gravity in the augmented conformal form} consists of finding
a Lorentzian manifold $(M, g)$ and a two-tensor field $T_{\alpha\beta}$ on $M$ such that: 
\begin{itemize}
\item[1.]The augmented conformal metric $\gv_{\alpha\beta} = e^{2\varrho}g_{\alpha\beta}$ with conformal factor $\varrho$ satisfies the {\bf evolution equation} \eqref{eq conformal aug geo c}.

\item[2.] The augmented conformal field equations  \eqref{eq conformal aug geo} are satisfied.

\item[3.] There exists an embedding $i: \Mb \to M$ with pull back metric $\gb^{\ddag} = i^* \gv$ and
second fundamental form $\Kv$.

\item[4.] The field $\varrhozero$ coincides with the restriction of the conformal factor $\varrho$ on $\Mb$, while
$\varrhoone$ coincides with the Lie derivative $\mathcal{L}_{n^\ddag}\varrho$ restricted to $\Mb$, where $n^\ddag$ denotes the normal unit vector of $\Mb$ (with respect to $\gv$).

\item[5.] The scalar fields $\phizerodd,\, \phionedd$ coincides with the restriction of $\phi,\, \Lcal_{n^\ddag} \phi$ on $\Mb$.
\end{itemize}
Such a solution to \eqref{main eq trans} is referred to as a {\bf modified gravity development of the initial data set} $(\Mb, \gb^{\ddag}, \Kv, \varrhozero, \varrhoone, \phizerodd, \phionedd)$.
\end{definition}

Note that, as in Section \ref{sec conformal}, the geometric form of the constraint equations is the {\bf Hamiltonian constraint} ${N_g}_{00} = 8\pi T_{00}$ and the {\bf momentum constraint}
${N_g}_{0a} = 8\pi T_{0a}$.


\subsection{Augmented conformal formulation in wave coordinates}
\label{section53}

We now reduce the conformal augmented system \eqref{eq conformal aug geo0-deux} in wave coordinates to a system with hyperbolic principal part. Indeed, we obtain the following formulation, where we replace the wave constraints by the evolution law of $T$ given by \eqref{eq lem 5-3-1-0}.

\begin{definition}
The {\bf conformal augmented formulation of modified gravity} is, by definition, the following partial differential system:
\begin{subequations}\label{eq conformal aug geo}
\begin{equation}\label{eq conformal aug geo a}
{N_g^{\ddag}}_{\alpha\beta} - \frac{1}{2}g_{\alpha\beta}\tr(N_g^{\ddag}) - \frac{1}{2}e^{2\varrho}\big(\del_{\alpha}\Gammav_{\beta} + \del_{\beta}\Gammav_{\alpha}\big)
= 8\pi \big(T_{\alpha\beta} - \frac{1}{2}g_{\alpha\beta}\tr(T)\big),
\end{equation}
\begin{equation}\label{eq conformal aug geo b}
\nablav^\alpha T_{\alpha\beta}
= \big(2g^{\delta\gamma}\del_{\gamma}\varrho T_{\gamma\beta} - \tr(T)\del_{\beta}\varrho\big)e^{-2\varrho},
\end{equation}
\begin{equation}\label{eq conformal aug geo c}
\Box_{\gv} \varrho + \Gammav^{\delta}\del_\delta\varrho = \frac{W_2(\varrho)}{6e^{2\varrho}} + \frac{f(\rhoR(\varrho))}{6e^{4\varrho}}
+ \frac{4\pi\tr(T)}{3e^{4\varrho}}\gv_{\alpha\beta}.
\end{equation}
\end{subequations}
\end{definition}

By Lemma~\ref{lem Ricci}, we then have the following expressions in coordinates, in which we emphasize that $\gv$ need not coincide with $\gd$. 

\begin{lemma}
The conformal augmented formulation of modified gravity theory in coordinates reads 
\begin{subequations}\label{eq conformal aug}
\begin{equation}\label{eq conformal aug a}
\aligned
\gv^{\alpha'\beta'}\del_{\alpha'}\del_{\beta'}\gv_{\alpha\beta}
=  F_{\alpha\beta}(\gv;\del \gv,\del \gv)  
& - 12\del_{\alpha}\varrho\del_{\beta}\varrho
+ \frac{W_2(\varrho)}{e^{2\varrho}}\gv_{\alpha\beta} 
\\
& -16\pi\big(T_{\alpha\beta} - \frac{1}{2}g_{\alpha\beta}\tr(T)\big),
\endaligned
\end{equation}
\begin{equation}\label{eq conformal aug b}
\nablav^\alpha T_{\alpha\beta}
= \big(2g^{\delta\gamma}\del_{\gamma}\varrho T_{\gamma\beta} - \tr(T)\del_{\beta}\varrho\big)e^{-2\varrho},
\end{equation}
\begin{equation}\label{eq conformal aug c}
\gv^{\alpha'\beta'}\del_{\alpha'}\del_{\beta'}\varrho = \frac{W_2(\varrho)}{6e^{2\varrho}} + \frac{f(\rhoR(\varrho))}{6e^{4\varrho}}
+ \frac{4\pi\tr(T)}{3e^{4\varrho}},
\end{equation}
\end{subequations}
with 
\be
\gv_{\alpha\beta} = e^{2\varrho}g_{\alpha\beta}. 
\ee
\end{lemma}


\subsection{Preservation of the constraints}
\label{section54}

Our first task is to address the problem of the preservation of the constraints.  

\begin{proposition}\label{prop aug 1}
Let $(M,\gv)$ be a globally hyperbolic Lorentzian manifold endowed with foliation $M = [0,t_{max})$ and with signature $(-,+,+,+)$. Let $T$ be a symmetric two-tensor (representing the matter content) and let $\varrho$ be a scalar field defined in $(-\eps,\eps)\times \mathbb{R}^3$. Furthermore, assume that $(\gv_{\alpha\beta}, \phi, \varrho)$ is a solution to the conformal field equations \eqref{eq conformal aug}. Let $g_{\alpha\beta} := e^{-2\varrho}\gv_{\alpha\beta}$ be the metric conformal to $\gv$. Then, provided the constraint equations \eqref{constraint trans-aug Hamiltonian} and \eqref{constraint trans-aug momentum} together with the constraint equations 
\begin{equation}\label{eq preserve aug 1-1}
\gv^{\alpha\beta}\Gammav_{\alpha\beta}^{\gamma} = 0
\end{equation}
are satisfied on the slice $\{x^0 = 0\}$ (where $R_g$ is the scalar curvature of $g_{\alpha\beta}$ and $\Gammav_{\alpha\beta}^{\gamma}$ are the Christoffel symbols of $\gv$), then \eqref{eq preserve aug 1-1} holds in the whole of $(-\eps,\eps)\times \mathbb{R}^3$. Furthermore, one has 
\begin{equation}\label{eq preserve aug 1-2}
e^{2\varrho} = e^{2\rho} =  \frac{1}{2}\ln(f'(R_g))
\end{equation}
in $(-\eps,\eps)\times \mathbb{R}^3$, so that $(\gv_{\alpha\beta},\phi,\varrho)$ is also a solution to \eqref{main eq trans}.
\end{proposition}

The rest of this section is devoted to the proof of this proposition. Recall that, throughout, we assume that the $\gv_{\alpha\beta}$, $\varrho$ and $T_{\alpha\beta}$ are sufficiently regular, so that all relevant derivatives are continuous at least. For the proof, we need some preliminary material and, first of all, we compute the divergence of ${N_g^{\ddag}}$. 

\begin{lemma}\label{lem 5-3-1}
When \eqref{eq conformal aug geo c} holds, the following identity holds:
\begin{equation}\label{eq lem 5-3-1}
\nablav^{\alpha}{N_g^{\ddag}}_{\alpha\beta}
= e^{-2\varrho}\big(2g^{\alpha\alpha'}\del_{\alpha'}\varrho {N_g^{\ddag}}_{\alpha\beta} - \tr(8\pi T)\del_{\beta}\varrho\big)
+ 6e^{2\varrho}\del_{\beta}\varrho\,\Gammav^{\delta}\del_\delta\varrho.
\end{equation}
\end{lemma}

\begin{proof} First of all, in view of \eqref{eq conformal aug geo0-deux}, we have 
$$
{N_g^{\ddag}}_{\alpha\beta}
= e^{2\varrho}G^{\ddag}_{\alpha\beta}
- 6e^{2\varrho}\big(\del_{\alpha}\varrho\del_{\beta}\varrho - \frac{1}{2}g_{\alpha\beta}|\nabla\varrho|^2_{\gv}\big)
- \frac{1}{2}\gv_{\alpha\beta}W_2(\varrho),
$$
where  
$\Gv_{\alpha\beta} := \Rv_{\alpha\beta} - \frac{1}{2}\gv_{\alpha\beta}\Rv$
is the Einstein curvature of $\gv$. We have the identities
$$
\aligned
\nablav^{\alpha\beta}G_{\alpha\beta}^{(\varrho)} &= 0,
\\
\nablav^{\alpha\beta}\big(\del_{\alpha}\varrho\del_{\beta}\varrho - \frac{1}{2}g_{\alpha\beta}|\nabla\varrho|^2_{\gv}\big)
&=\del_{\beta}\varrho\,\Box_{\gv}\varrho.
\endaligned
$$
As before, by introducing the function 
\be
s(r) = \frac{1}{2}\ln f'(r), \qquad r \in \RR
\ee
with (local) inverse denoted by $\theta$ and defined near $0$ at least, we can write 
$$
\aligned
\nablav^{\alpha\beta}\big(\gv_{\alpha\beta}W_2(\varrho)\big)
&= \del_{\beta}\big(W_2(\varrho)\big)
 = \del_{\beta}\big(e^{-2\varrho}f(\rhoR(\varrho)) - \rhoR(\varrho)\big)
\\
&= -2e^{-2\varrho}\del_{\beta}\varrho f(\rhoR(\varrho)) + e^{-2\varrho}f'(\rhoR(\varrho))\rhoR'(\varrho)\del_{\beta}\varrho - \rhoR'(\varrho)\del_{\beta}\varrho
\\
&=-2e^{-2\varrho}\del_{\beta}\varrho f(\rhoR(\varrho)) + e^{-2\varrho}e^{2\varrho}\rhoR'(\varrho)\del_{\beta}\varrho - \rhoR'(\varrho)\del_{\beta}\varrho
\\
&= -2e^{-2\varrho}\del_{\beta}\varrho f(\rhoR(\varrho)).
\endaligned
$$

We are now in a position to compute the divergence  
$$
\aligned
\nablav^{\alpha}{N_g^{\ddag}}_{\alpha\beta}
&= 2e^{2\varrho} G^{\varrho}_{\alpha\beta} \nablav^{\alpha}\varrho
  -12e^{2\varrho} \big(\del_{\alpha}\varrho\del_{\beta}\varrho - \frac{1}{2}g_{\alpha\beta}|\nabla\varrho|^2_{\gv}\big)\nablav^{\alpha}\varrho
\\
 & \quad -6e^{2\varrho}\del_{\beta}\varrho\Box_{\gv}\varrho + e^{2\varrho}\del_{\beta}f(\rhoR(\varrho))
\\
&= 2e^{2\varrho} G^{\varrho}_{\alpha\beta} \nablav^{\alpha}\varrho
  -12e^{2\varrho} \big(\del_{\alpha}\varrho\del_{\beta}\varrho - \frac{1}{2}g_{\alpha\beta}|\nabla\varrho|^2_{\gv}\big)\nablav^{\alpha}\varrho
  - \gv_{\alpha\beta}W_2(\varrho)\nablav^{\alpha}\varrho
\\
 & \quad -6e^{2\varrho}\del_{\beta}\varrho\Box_{\gv}\varrho + e^{-2\varrho}\del_{\beta}f(\rhoR(\varrho))
 + \gv_{\alpha\beta}W_2(\varrho)\nablav^{\alpha}\varrho
\\
&= 2{N_g^{\ddag}}_{\alpha\beta}  \nablav^{\alpha}\varrho
- 6e^{2\varrho}\del_{\beta}\varrho
\bigg(\Box_{\gv} - \frac{f(\rhoR(\varrho))}{6e^{4\varrho}} - \frac{W_2(\varrho)}{6e^{2\varrho}}\bigg).
\endaligned
$$
In view of \eqref{eq conformal aug geo c}, we find 
$$
\Box_{\gv}\varrho
= - \Gammav^{\delta}\del_\delta\varrho + \frac{f(\rhoR(\varrho))}{6e^{4\varrho}} + \frac{W_2(\varrho)}{6e^{2\varrho}} + \frac{4\pi \tr(T)}{3e^{4\varrho}}
$$
and this yields the desired conclusion. 
\end{proof}

\begin{lemma}\label{lem 5-3-2}
The equation 
\eqref{eq conformal aug a} leads also to the wave equation for $\varrho$ 
\begin{equation}\label{eq lem 5-3-2}
\Box_{\gv}\varrho = \frac{R_g}{6e^{2\varrho}} + \frac{W_2(\varrho)}{3e^{2\varrho}} + \frac{4\pi \tr(T)}{3e^{4\varrho}} - \frac{g^{\alpha\beta}\del_{\alpha}\Gammav_{\beta}}{6e^{2\varrho}}.
\end{equation}
\end{lemma}

\begin{proof} We recall that \eqref{eq conformal aug a} can be written as
\begin{equation}\label{eq pr 5-4-1}
\aligned
&
e^{2\varrho}\Rv_{\alpha\beta} - 6e^{2\varrho}\del_{\alpha}\varrho\del_{\beta}\varrho + \frac{\gv_{\alpha\beta}}{2}W_2(\varrho)
-\frac{e^{2\varrho}}{2}\big(\del_{\alpha}\Gammav_{\beta} + \del_{\beta}\Gammav_{\alpha}\big)
\\
&= 8\pi \big(T_{\alpha\beta} - \frac{1}{2}g_{\alpha\beta}\tr(T)\big), 
\endaligned
\end{equation}
and that \eqref{conformal Ricci curvature} implies 
$$
\Rv_{\alpha\beta}
= R_{\alpha\beta} - 2\big(\nabla_{\alpha}\nabla_{\beta}\varrho - \del_{\alpha}\varrho\del_{\beta}\varrho\big)
-\big(\Box_g \varrho + 2g(\nabla\varrho,\nabla\varrho)\big)g_{\alpha\beta}. 
$$
We substitute this relation into \eqref{eq pr 5-4-1} and obtain 
$$
\aligned
&e^{2\varrho}R_{\alpha\beta} - 2e^{2\varrho}\nabla_{\alpha}\nabla_{\beta}\varrho
- 4e^{2\varrho}\del_{\alpha}\varrho\del_{\beta}\varrho - g_{\alpha\beta}e^{2\varrho}\Box_g\varrho
- 2e^{2\varrho}g_{\alpha\beta}g(\nabla\varrho,\nabla\varrho) + \frac{\gv_{\alpha\beta}}{2}W_2(\varrho)
\\
&= 8\pi \big(T_{\alpha\beta} - \frac{1}{2}\tr(T)g_{\alpha\beta}\big) + \frac{1}{2}e^{2\varrho}\big(\del_{\alpha}\Gammav_\beta + \del_{\beta}\Gammav_{\alpha}\big).
\endaligned
$$
By taking the trace of this equation {\bf with respect to $g$}, we then have 
$$
e^{2\varrho}R_g - 6e^{2\varrho}\Box_g\varrho - 12e^{2\varrho}g(\nabla\varrho,\nabla\varrho) + 2e^{2\varrho}W_2(\varrho)
= -8\pi \tr(T) + e^{2\varrho}g^{\alpha\beta}\del_{\alpha}\Gammav_{\beta}, 
$$
which can also be written as
$$
\Box_g\varrho + 2g(\nabla\varrho,\nabla\varrho)
 = \frac{R_g}{6} + \frac{W_2(\varrho)}{3} + \frac{4\pi tr(T)}{3e^{2\varrho}} - \frac{g^{\alpha\beta}}{6}\del_\alpha\Gammav_{\beta}.
$$
In view of the expression Recall \eqref{conformal d'Alembertian}, the desired result is proven.
\end{proof}
 
In the next lemma, we identify the geometric form of the system \eqref{eq conformal aug}.

\begin{lemma}\label{lem 5-3-3}
Let $(M,\gv)$ be a Lorentzian manifold together wiht a two-tensor $T_{\alpha\beta}$. If, in some local coordinates $\{x^0,x^1,x^2,x^3\}$, $(\gv,T_{\alpha\beta})$ satisfy \eqref{eq conformal aug}, then 
the following equation holds (in the domain of the coordinate system):
\begin{equation}
\aligned
\gv^{\alpha'\beta'}\del_{\alpha'}\del_{\beta'} \Gammav_{\beta} &= F_{\beta}(\varrho,\gv;\Gammav_{\gamma},\del\Gammav_{\gamma}),
\endaligned
\end{equation}
where $F_{\beta}(\varrho,\gv;\cdot,\cdot)$ is a combination of linear and bilinear forms and  $\Gammav^{\gamma}_{\alpha\beta}$ denote the associated Christoffel symbol with
$$
\Gammav_{\beta} = \gv_{\beta\beta'}\Gammav^{\beta'} = \gv_{\beta\beta'}\gv^{\alpha\gamma}\Gammav_{\alpha\gamma}^{\beta'}.
$$
\end{lemma}

\begin{proof} By taking the trace of \eqref{eq conformal aug geo a} we have 
\begin{equation}\label{eq pr 5-5-0}
e^{2\varrho}g^{\alpha'\beta'}\del_{\alpha'}\Gammav_{\beta'} = - \tr(N_g^{\ddag} - 8\pi T).
\end{equation}
Combining this results with \eqref{eq conformal aug geo a}, we obtain 
\begin{equation}\label{eq pr 5-5-1}
\frac{e^{2\varrho}}{2}
\big(\del_{\alpha}\Gammav_{\beta} + \del_{\beta}\Gammav_{\alpha} - g_{\alpha\beta}g^{\alpha'\beta'}\del_{\alpha'}\Gammav_{\beta'}\big)
 = {N_g^{\ddag}}_{\alpha\beta} - 8\pi T_{\alpha\beta}.
\end{equation}
By computing the divergence of this equation (for the metric $\gv$) and evaluating the left-hand-side from \eqref{eq 2-4-0}, we get 
\begin{equation}\label{eq pr 5-5-2}
\aligned
&\frac{1}{2}\nablav^{\alpha}\Big(e^{2\varrho}\big(\del_{\alpha}\Gammav_{\beta} + \del_{\beta}\Gammav_{\alpha} - g_{\alpha\beta}g^{\alpha'\beta'}\del_{\alpha'}\Gammav_{\beta'}\big)\Big)
\\
&= \frac{1}{2}e^{2\varrho}\Big(
\gv^{\alpha'\beta'}\del_{\alpha'}\del_{\beta'}\Gammav_{\beta}
- \gv^{\alpha\alpha'}\Gammav_{\alpha\alpha'}^{\gamma}\del_{\gamma}\Gammav_{\beta}
- \gv^{\alpha\alpha'}\Gammav^{\delta}_{\alpha'\beta}\del_\delta\Gammav_{\alpha}
- \del_{\beta}\gv^{\alpha'\beta'}\big(\del_{\alpha'}\Gammav_{\beta'}\big)
\Big)
\\
&+e^{2\varrho}
\Big(e^{2\varrho}\big(\del_{\alpha}\Gammav_{\beta} + \del_{\beta}\Gammav_{\alpha} - g_{\alpha\beta}g^{\alpha'\beta'}\del_{\alpha'}\Gammav_{\beta'}\big)\Big)\nablav^{\alpha}\varrho
\\
&=:  \frac{1}{2}e^{2\varrho}g^{\alpha'\beta'}\del_{\alpha'}\del_{\beta'}\Gammav_{\beta}
+ \tilde{F}_{\beta}(\varrho,\gv;\Gammav_{\gamma},\del \Gammav_{\gamma}). 
\endaligned
\end{equation}
Here, $\tilde{F}$ is a combination of  linear and bilinear forms on $\Gammav_{\gamma}$ and $\del \Gammav_{\gamma}$ depending on $\varrho,\gv$ and their derivatives.

On the other hand, the right-hand-side is computed from \eqref{eq conformal aug geo b} and \eqref{eq lem 5-3-1}, as follows:
$$
\aligned
&\nablav^{\alpha}\big({N_g^{\ddag}}_{\alpha\beta} - 8\pi T_{\alpha\beta}\big)
\\
&= e^{-2\varrho}\big(2g^{\alpha\alpha'}\del_{\alpha'}\varrho {N_g^{\ddag}}_{\alpha\beta} - \tr(8\pi T)\del_{\beta}\varrho\big)
 + 6e^{2\varrho}\del_{\beta}\varrho\,\Gammav^{\delta}\del_\delta\varrho
\\
& \hskip3.cm 
- 8\pi e^{-2\varrho} \big(2g^{\delta\gamma}\del_{\gamma}\varrho T_{\gamma\beta} - \tr(T)\del_{\beta}\rho\big)
\\
&= e^{-2\varrho}\Big(2g^{\gamma\alpha}\del_{\gamma}\varrho\big({N_g^{\ddag}}_{\alpha\beta}-8\pi T_{\alpha\beta}\big) - \tr(N_g^{\ddag} - 8\pi T)\del_{\beta}\varrho\Big) + 6\del_{\beta}\rho\Gammav{\delta}\del_\delta\rho.
\endaligned
$$
Then, in view of \eqref{eq pr 5-5-0} and \eqref{eq pr 5-5-1}, we obtain 
\begin{equation}\label{eq pr 5-5-3}
\aligned
\nablav^{\alpha}\big({N_g^{\ddag}}_{\alpha\beta} - 8\pi T_{\alpha\beta}\big)
= g^{\gamma\alpha}\del_{\gamma}\varrho\big(\del_{\alpha}\Gammav_{\beta} + \del_{\beta}\Gammav_{\alpha}\big) + 6\del_{\beta}\rho\Gammav^{\delta}\del_\delta\rho, 
\endaligned
\end{equation}
which is a linear form in the functions $\Gammav_{\alpha}$ and $\del_{\gamma}\Gammav_{\alpha}$.
Finally, we arrive at the desired conclusion by combining \eqref{eq pr 5-5-2} and \eqref{eq pr 5-5-3} with \eqref{eq pr 5-5-1} together. 
\end{proof}

\begin{lemma}\label{lem 5-3-4}
Provided the equations \eqref{eq conformal aug} hold on the initial slice $\{t=0\}$ and the condition 
$$
\Gammav_{\lambda} = 0
$$
and the constraint equations \eqref{constraint trans-aug Hamiltonian} and \eqref{constraint trans-aug momentum} also hold on the initial slice $\{t=0\}$, then it follows that 
$$
\del_t\Gammav_{\lambda} = 0 
$$
on this initial slice.
\end{lemma}

\begin{proof}
We observe that the constraint equations \eqref{constraint trans-aug Hamiltonian} and \eqref{constraint trans-aug momentum} are equivalent to
${N_g^{\ddag}}_{0j} = 8\pi T_{00}$
and
${N_g^{\ddag}}_{00} = 8\pi T_{00}$, respectively. 
Consequently, the proof of the lemma follows from the same calculation which was performes in the proof of Lemma~\ref{lem pre 2}. Therefore, we omit the details.
\end{proof}

\begin{proof}[Proof of Proposition \ref{prop aug 1}]
By Lemmas~\ref{lem 5-3-3} and \ref{lem 5-3-4} and by the global hyperbolicity of the metric $\gv$, we see that in the spacetime $M$,
\begin{equation}\label{eq pr 5-3-4-1}
\Gammav_{\beta} = 0.
\end{equation}

Recalling Lemma~\ref{lem 5-3-2} and combining \eqref{eq lem 5-3-2} with \eqref{eq pr 5-3-4-1}, we obtain 
\begin{equation}\label{eq pr 5-3-4-2}
\Box_{\gv}\varrho = \frac{R_g}{6e^{2\varrho}} + \frac{W_2(\varrho)}{3e^{2\varrho}} + \frac{4\pi \tr(T)}{3e^{4\varrho}}.
\end{equation}
In a similar way, \eqref{eq conformal aug geo c} becomes
\begin{equation}
\Box_{\gv} \varrho  = \frac{W_2(\varrho)}{6e^{2\varrho}} + \frac{f(\rhoR(\varrho))}{6e^{4\varrho}}
+ \frac{4\pi\tr(T)}{3e^{4\varrho}}.
\end{equation}
By comparing these two equations, we thus get
$$
R_g + W_2(\varrho) - f(\rhoR(\varrho)) = 0,
$$
which (by the definition of $W_2$) leads us to 
$\rhoR(\varrho) = R_g.$
It remains to recall the definition of $\rhoR$, $e^{2\varrho} = R_g$.
\end{proof}


\subsection{The local existence theory}

The standard theory of local-in-time existence for the initial value problem associated with hyperbolic problems  can now be applied to an arbitrary initial set in these sense of Definition \ref{def36}. 
Let us sketch the strategy of proof. For simplicity in this discussion and without genuine restriction, we can consider that the initial data set and, therefore, the solutions are close to data in Minkowski space.

First of all, we need to construct the (local-in-time) solution of the problem \eqref{eq conformal aug sec6} whose the initial data set must be expressed in wave coordinates, say $(\Mb, \gvzero_{\alpha\beta}, \gvone_{\alpha\beta},\varrho_0,\varrho_1, \phi_0,\phi_1)$. This PDE initial data set is determined from the geometric initial data set denoted by $(\Mb, \gb^{\ddag}, \Kv, \varrhozero, \varrhoone, \phizerodd, \phionedd)$. Without restriction, smallness (and regularity) assumptions are here made on the initial data set.

Second, we need to check that this (local-in-time) solution leads to a globally hyperbolic spacetime. Then, according to Proposition \ref{prop aug 1}, we can conclude that this solution preserves the constraints \eqref{eq preserve aug 1-1} and \eqref{eq preserve aug 1-2} and, consequently, is also a solution to the field equations \eqref{main eq trans} with $\gv = \gd$ and $\varrho = \rho = \frac{1}{2}\ln f'(R_g)$. We thus conclude that this solution solves the initial value problem in Definition \ref{def36}.

Third, we need to observe that the solution $(g,\phi)$ constructed from $g_{\alpha\beta}:=e^{-2\rho}\gd_{\alpha\beta}$ is a solution to the original initial value problem, stated in Definition \ref{def:formu} with the corresponding initial data set determined by  $(\Mb, \gb^{\ddag}, \Kv, \varrhozero, \varrhoone, \phizerodd, \phionedd)$ (via the conformal transformation).

We omit the details and refer to Choquet-Bruhat \cite{CB} for the existence and uniqueness statements in classical gravity, which based on our reformulation and discussion above can be extactly restated for the field equations of modified gravity. For the rest of this work, our objective is to revisit such a theory and, while re-proving this existence result, to establish that {\sl modified gravity developments} remain close to 
{\sl classical Einstein developments}. 


\newpage 

\section{Local existence theory. Formulation and main statement} 

\subsection{Construction of the PDE initial data set}

Our objective is thus to establish an existence theory for the Cauchy problem associated with the modified gravity field equations \eqref{eq conformal aug}, when the initial data are assumed to be asymptotic flat. For the sake of simplicity and without genuine loss of generality as far as our method of proof is concerned, we focus our presentation on quadratic functions $f(r) = r + \frac{\coeff}{2}r^2$. It is straighforward to modified our argument to cover more general functions $f$. 

We need first to introduce several notations, before we can state one of our main results in Theorem \ref{prop 7-10-1} below. Recall that the matter model we are considering is the real massless scalar field with Jordan coupling and that, in agreement with Section \ref{sec div conformal} (see~\eqref{eq conformal wave phi}),
the system \eqref{eq conformal aug} under consideration reads 
\begin{subequations}\label{eq conformal aug sec6}
\begin{equation}\label{eq conformal aug sec6 a}
\aligned
&\gv^{\alpha'\beta'}\del_{\alpha'}\del_{\beta'}\gv_{\alpha\beta}
\\
& =  F_{\alpha\beta}(\gv;\del \gv,\del \gv)  - 12\del_{\alpha}\varrho\del_{\beta}\varrho
+ \frac{W_2(\varrho)}{e^{2\varrho}}\gv_{\alpha\beta}  -16\pi\del_{\alpha}\phi\del_{\beta}\phi,
\endaligned
\end{equation}
\begin{equation}\label{eq conformal aug sec6 b}
\gv^{\alpha'\beta'}\del_{\alpha'}\del_{\beta'}\phi = 2\gv^{\alpha'\beta'}\del_{\alpha'}\phi \del_{\beta'}\varrho,
\end{equation}
\begin{equation}\label{eq conformal aug sec6 c}
\gv^{\alpha'\beta'}\del_{\alpha'}\del_{\beta'}\varrho
= \frac{W_2(\varrho)}{6e^{2\varrho}}
+ \frac{f(\rhoR(\varrho))}{6e^{4\varrho}}
- \frac{4\pi}{3e^{2\varrho}}\gv^{\alpha'\beta'}\del_{\alpha'}\phi\del_{\beta'}\phi.
\end{equation}
\end{subequations}
Clearly, this is quasi-linear system of wave equations in diagonalized form and, in order to formulate a well-posed problem, the initial data set should include  the functions
\be
\aligned
&\gv_{\alpha\beta}(0,x) = \gvzero_{\alpha\beta},\qquad \qquad 
&&\del_t \gv_{\alpha\beta}(0,x) = \gvone_{\alpha\beta},
\\
&\varrho(0,x) = \varrho_0,\quad 
&&\del_t\varrho(0,x) = \varrho_1,
\\
&\phi(0,x) = \phi_0 = \phizerodd,\quad 
&&\del_t\phi(0,x) = \phi_1 = \phionedd.
\endaligned
\ee
There are $24$ functions to be prescribed, but the geometrical initial data set 
$$
(\Mb, \gb^{\ddag}, \Kv, \varrhozero, \varrhoone, \phizerodd, \phionedd)
$$
contains $16$ functions only. In fact, in order to construct a solution of \eqref{eq conformal aug sec6} that also satisfies \eqref{main eq trans}, we see that, by Proposition \ref{prop aug 1}, the conditions \eqref{eq preserve aug 1-1}, \eqref{constraint trans-aug Hamiltonian} and \eqref{constraint trans-aug momentum} must hold on the initial hypersurface. These conditions form a nonlinear PDE's system of eight equations, and it is expected that the $8$ remaining initial data components could be determined from these $8$ constraint equations. This task, however, is not a trivial one and further investigation would be needed to fully clarify the set of initial data. 

From now on, we assume that this system of $8$ constraint equations together with the $16$ functions prescribed by the geometric initial data set uniquely determine our PDE's initial data set.
Throughout, we denote by $(\Mb, \gvzero_{\alpha\beta}, \gvone_{\alpha\beta},\varrho_0,\varrho_1, \phi_0,\phi_1)$ the PDE initial data determined by $(\Mb, \gb^{\ddag}, \Kv, \varrhozero, \varrhoone, \phizerodd, \phionedd)$ and the constraint equations \eqref{constraint trans-aug Hamiltonian}, \eqref{constraint trans-aug momentum} and \eqref{eq preserve aug 1-1}.


\subsection{Simplifying the field equations of $f(R)$ gravity}

For definiteness, we focus on the role of second-order terms in $f$ and assume that
\begin{equation}\label{eq 7-1-1}
f(r) := r + \frac{\coeff}{2}r^2, \qquad r \in \RR 
\end{equation}
for some $\coeff\geq 0$. We recall $e^{2\rho} = f'(R_g)$, so that 
$$
e^{2\rho} = 1 + \coeff R_g, \qquad R_g = \frac{e^{2\rho}-1}{\coeff},
$$
and
$$
\aligned
W_2(s) 
& = \frac{f(r) - rf'(r)}{f'(r)} 
\\
& = \frac{r + \frac{\coeff}{2}r^2 - r(1+\coeff r)}{1+\coeff r}
= -\frac{\coeff r^2}{2(1+\coeff r)},
\endaligned
$$
with 
$e^{2s} = f'(r) = 1+\coeff r$
so that 
\be
W_2(s) = - \frac{(e^{2s}-1)^2}{2\coeff e^{2s}}.
\ee

The spacetime $(M, \gv)$ under consideration is endowed with a foliation
$$
M = [0,T]\times M_t, \quad t\in [0,T]
$$
and we assume that, for each $t$, $M_t$ is diffeomorphic to $\RR^3$. The spacetime metric is supposed to be sufficiently close to Minkowski metric and, especially, is asymptotically flat, so that the following notation is convenient: 
\be
h_{\alpha\beta} := \gv_{\alpha\beta} - m_{\alpha\beta}
\ee
and we thus seek for unknowns triples $(h_{\alpha\beta},\varrho,\phi)$. Sometimes, we will write $(h_{\alpha\beta}, \varrho, \phi) = (h_{\alpha\beta}^{\coeff}, \varrho^{\coeff}, \phi^{\coeff})$ in orde to emphasize that the solutions of \eqref{eq conformal aug} depend on the coefficient $\coeff$.

With these notation the system \eqref{eq conformal aug} take in the alternative form:
\begin{subequations}\label{eq conformal aug sec7}
\begin{equation}\label{eq conformal aug sec7 a}
\aligned
&(m^{\alpha'\beta'} + H^{\alpha'\beta'}(h^{\coeff}))\del_{\alpha'}\del_{\beta'}h_{\alpha\beta}^\coeff
\\
&=  F_{\alpha\beta}(h^\coeff;\del h^\coeff,\del h^\coeff) - 16\pi\del_{\alpha}\phi^\coeff\del_{\beta}\phi^\coeff
 - 12\del_{\alpha}\varrho^\coeff\del_{\beta}\varrho^{\coeff}
\\
& \quad
-\frac{\big(e^{2\varrho^\coeff}-1\big)^2}{2\coeff e^{4\varrho^\coeff}}\big(m_{\alpha\beta}+h^{\coeff}_{\alpha\beta}\big),
\endaligned
\end{equation}
\begin{equation}\label{eq conformal aug sec7 b}
(m^{\alpha'\beta'} + H^{\alpha'\beta'}(h^{\coeff}))\del_{\alpha'}\del_{\beta'}\phi^\coeff
= 2(m^{\alpha'\beta'} + H^{\alpha'\beta'}(h^{\coeff}))\del_{\alpha'}\phi^{\coeff} \del_{\beta'}\varrho^\coeff,
\end{equation}
\begin{equation}\label{eq conformal aug sec7 c}
\aligned
& (m^{\alpha'\beta'} + H^{\alpha'\beta'}(h^{\coeff}))\del_{\alpha'}\del_{\beta'}\varrho^\coeff -\frac{e^{2\varrho^\coeff}-1}{6\coeff e^{4\varrho^\coeff}}
\\
& =
- \frac{4\pi}{3e^{2\varrho^\coeff}}(m^{\alpha'\beta'} + H^{\alpha'\beta'}(h^{\coeff}))\del_{\alpha'}\phi^\coeff\del_{\beta'}\phi^\coeff,
\endaligned
\end{equation}
\end{subequations}
where, from $h^{\coeff} = (h^{\coeff}_{\alpha\beta})$, we have determined 
\begin{equation}\label{eq def1 7-2}
\big( m_{\alpha\beta} + H^{\alpha\beta}(h^{\coeff}) \big)
\, \text{ as the inverse of }  \big(m_{\alpha\beta} + h^{\coeff}_{\alpha\beta} \big).
\end{equation}

With this notation, the PDE initial data denoted by 
$(\Mb, h_0, h_1, \varrho_0,\varrho_1,\phizerodd, \phionedd)$ 
is thus rewritten in terms of $h$, with
$$
h_{\alpha\beta}|_{t=0} = {h_0}_{\alpha\beta} := {\gv_0}_{\alpha\beta} - m_{\alpha\beta}, \quad
\qquad
\del_th_{\alpha\beta} = {h_1}_{\alpha\beta} := {\gv_1}_{\alpha\beta}.
$$
The system under consideration is composed of $12$ quasi-linear wave equations: $11$ of them (those on $h_{\alpha\beta}^{\coeff}$ and $\phi^\coeff$) are quasi-linear wave equations, while the equation on $\varrho^\coeff$ is a quasi-linear Klein-Gordon equation with defocusing potential. 

\begin{remark} The coefficients $H^{\alpha\beta}$ are clearly smooth functions of $h$, in a sufficiently small neighborhood of the origin, at least. Hence, we can find a positive constant $\eps_0$ such that if
$ 
|h| \leq \eps_0,
$
then for any integer $k$, the $k$-th order derivatives of $H^{\alpha\beta}$ with respect to $h$, say 
$D^kH^{\alpha\beta}$, are well-defined and  
$$
\sup_{|h| \leq \eps'}|D^k H^{\alpha\beta}(h)\big|\leq C(k,\eps_0)
$$
Standard linear algebra arguments show that when  
$|h| \leq \eps_0$ (with $\eps_0$ sufficiently small), then
\begin{equation}\label{eq 1 7-2}
H^{\alpha\beta}(h) = - h_{\alpha\beta} + Q^{\alpha\beta} \big( h,h)(1+R^{\alpha\beta}(h) \big), 
\end{equation}
where $Q^{\alpha\beta}$ is a quadratic form in its argument and
$
|R^{\alpha\beta}(h)|\leq C(\eps_0) \, |h|.
$
\end{remark}


\subsection{Vector fields and notation}

We introduce the three generators of the spatial rotations
\be
\Omega_{ij} := x_i\del_j - x_j\del_i = x^i\del_j - x^j\del_i,
\ee
which are known to commute with the wave operator, as well as with the Klein-Gordon operator. 
Here, the coordinate indices are raised and lowered with the Minkowski metric. We also write 
$$
\Omega_1 := \Omega_{12},\quad \Omega_2 := \Omega_{23}, \quad \Omega_3 := \Omega_{13}.
$$
Note that
\be
[\Omega_a,\Box] = 0,\quad [\Omega_a,\Box + 1] = 0. 
\ee

The following notation about multi-indices will be used.  Given a finite set 
$\mathscr{I} = \{\alpha_1,\alpha_2,\dots, \alpha_n\}$, 
we call $n$ the order of $\mathscr{I}$, denoted by $|\mathscr{I}| = n$. 
We introduce an ordering relation denoted by $\preceq$ on $\mathscr{I}$, defined by
$$
\alpha_i\preceq\alpha_j \quad \text{ if and only if}\quad i\leq j.
$$
The pair $(\mathscr{I},\preceq)$ is called a abstract multi-index. Obviously, a subset of $\mathscr{I}$ can also be regarded as a multi-index endowed with the same (restricted) order. The order $\preceq$ describes the location of each differential operator in a product.

A partition of an abstract index $\mathscr{I}$ is defined as follows. Let $\mathscr{J}_k$ be family of subsets of $\mathscr{I}$, with
$$
\bigcup_{k=1}^m \mathscr{J}_k = \mathscr{I}, 
\qquad
\quad \mathscr{J}_k \cap \mathscr{J}_{k'} = \varnothing.
$$
Then, we say that $\{\mathscr{J}_k\}$ is an $m$-partition of $\mathscr{I}$ and we write $\sum_{k=1}^m \mathscr{J}_k = \mathscr{I}$. We observe that each $\mathscr{J}_k$ can be regarded  as a multi-index and $\sum_{k=1}^m|\mathscr{J}_k| = |\mathscr{I}|$.

If for all $k=1,2,\dots m$, we have $\mathscr{J}_k\neq \varnothing$, then we say $\{\mathscr{J}_k\}$ is a proper $m$-partition of $\mathscr{I}$ and we write $\sum_{k=1}^m \mathscr{J}_k = \mathscr{I}, |\mathscr{J}_k|\geq1$.

Now, let us return to the case of multi-indices in the context of differential operators. Let $Z$ be a family of order one differential operator, say $Z = \{Z_1,Z_2,\dots Z_p\}$. A $n$-multi-index on the family $Z$ is a map
$$
\aligned
I : \mathscr{I} &\to  \{1,2,\dots n\},
\\
\alpha_i &\mapsto  I(\alpha_i)\in   \{1,2,\dots n\},
\endaligned
$$
and we write
$Z^I := Z_{I(\alpha_1)}\circ Z_{I(\alpha_2)}\circ\dots \circ Z_{I(\alpha_n)}$. With some abuse of notation, we often write $I = (\alpha_n,\alpha_{n-1},\dots,\alpha_1)$ with $\alpha_i\in \{1,2,\dots n\}$, where each $\alpha_i$ is replaced by $I(\alpha_i)$.

An $m$-partition of index $I$ is defined as follows. 
Let $\{\mathscr{J}_k\}$ be a (proper) $m$-partition of an abstract index $\mathscr{I}$. Then we restrict the map $I$ on each ordered set $\mathscr{J}_k$, and this yields us an $m$-multi-index, denoted by $\{J_k\} = I(\{\mathscr{J}_k\})$. Then, we call $J_k$ a (proper) $m$-partition of $I$ and we denote it by $I = \sum_{k=1}^m J_k, (|J_k|\geq 1)$.

We often consider the set $\mathscr{P}(\mathscr{I},m)$ composed by all possible $m$-partitions of $\mathscr{I}$. Then, each partition in $\mathscr{P}(\mathscr{I},m)$ can be associated with a partition of $I$. We observe that if $\prod_{k=1}^m u_k$ is a product of $m$ functions, then
$$
Z^I\bigg(\prod_{k=1}^mu_k\bigg) = \sum_{\{\mathscr{J}_k\}\in\mathscr{P}(\mathscr{I},m)}\prod_{i=1}^mZ^{J_k}u_k
$$
with $\{J_k\} = I(\{\mathscr{J}_k\})$. However, for the sake of simplicity in the notation and whenever there is no risk of confusion, we will often write 
$$
Z^I\bigg(\prod_{k=1}^mh_k\bigg) = \sum_{\sum_{k=1}^mJ_k = I}\prod_{k=1}^m Z^{J_k}u_k. 
$$ 


\subsection{Functional spaces of interest}
\label{subsec 7-fs}

We recall that in the classical case,  an initial data set of a Cauchy problem of the general relativity satisfies the constraint equations \eqref{GC-3}. This system leads to a nonlinear elliptic system and by the positive mass theorem, the non-trivial part of the metric $\bar{g}-\bar{m}$ decreases exactly like $r^{-1}$ at spatial infinity. If $\bar{g}-\bar{m}$ decreases faster than $r^{-1}$, then $\bar{g} = \bar{m}$.

In our case, the constraint equations \eqref{constraint trans-aug Hamiltonian} and \eqref{constraint trans-aug momentum} are much more complicated than the classical system, the parallel result to the positive mass theorem is not known. But as we focus our attention on the convergence result, where we take an ``Einsteinian initial data set" having $\varrho_0 = \varrho_1 \equiv 0$, which satisfies the classical constraint equations. So we have to handle the quasi-linear wave system with initial data decreasing as $r^{-1}$ at spatial infinity. These functions as (in general) not in $L^2(\RR^3)$. So we need to construct our local solution with the aid of the following functional space.

We need to introduce some norms about $C_c^{\infty}(\RR^3)$ functions, that is, smooth functions with compact support. A first norm to be introduced is
$$
\|\varrho\|_{X^d}
:= \sum_{|I_1|+|I_2|\leq d}\|\del_x^{I_1}\Omega^{I_2}\varrho\|_{L^2(\RR^3)}.
$$
The $L^2$ norm is with respect to the standard volume form, i.e. the Lebesgue measure.

The second norm defined for the $C_c^{\infty}(\RR^3)$ functions is
$$
\|\varrho\|_{X_P^{d+1}} := \sum_{|I_1|+|I_2|\leq d}\|\del_x\del_x^{I_1}\Omega^{I_2} \varrho\|_{L^2(\RR^3)},
$$
where $\del_x \phi$ refers to the spatial gradient of $\phi$.

The first functional space to be used in our analysis, the space $X_P^{d+1}$, is defined as the completion of the norm $\|\cdot\|_{X_P^{d+1}}$ on the $C_c^{\infty}$ functions.
We denote by 
\be
\|\cdot\|_{X_R^{d+1}} :=\|\cdot\|_{X^d} + \|\cdot\|_{X_P^{d+1}}.
\ee
 The second functional space, $X_R^{d+1}$, to be used in our analysis, is defined as the completion of the norm $\|\cdot\|_{X_R^{d+1}}$ on the $C_c^{\infty}$ functions.

W also define the weighted sup-norm
$$
\|f\|_{\mathcal{E}_{-1}} : = \sup_{r\geq 0}\{(1+r)|f|\}
$$
and 
$$
\|f\|_{X_H^{d+1}}
:= \|f\|_{\mathcal{E}_{-1}} + \|f\|_{X_P^{d+1}},
$$
so that the functional spaces $\mathcal{E}_{-1}$
and $X_H^{d+1}$ are obtained by completion from the $C_c(\RR^3)$ functions with respect of the norm under consideration.
 
The relations among these functional spaces are as follows: 
\be
X^{d+1} \subset X_R^{d+1}\subset X^d,\quad X_R^d \subset X_P^d,\quad X_H^d \subset X_P^d.
\ee
In the next section, when $d\geq 2$, by \eqref{eq 1 lem 7-6-1}, we will also see that 
\be
X^d\subset \mathcal{E}_{-1},\quad  X_R^d\subset X_H^d\subset X_P^d.
\ee

Finally, we define the norm of a triple $S_0 := ({h_0}_{\alpha\beta}, \varrho_0, \phi_0)$:
$$
\|S\|_{X_0^{d+1}} := \sum_{\alpha,\beta}\|{h_0}_{\alpha\beta}\|_{X_H^{d+1}} + \|\varrho_0\|_{X_R^{d+1}} + \|\phi_0\|_{X_P^{d+1}}
$$
and we set 
\be
X_0^{d+1} = X_H^{d+1}\times X_R^{d+1}\times X_P^{d+1}.
\ee
Similarly, for triples $S_1 = ({h_1}_{\alpha\beta},\phi_1,\varrho_1)$, we define
$$
\|S_1\|_{X_1^d} := \sum_{\alpha,\beta}\|{h_1}_{\alpha\beta}\|_{X^d} + \|\varrho_1\|_{X^d} + \|\phi_1\|_{X^d}
$$

We are now ready to discuss the notion of asymptotically flat PDE initial data.

\begin{definition}\label{def 7-4-1}
A {\bf PDE initial data set in wave coordinates}  for the initial value problem stated 
in Definition \ref{def36}, say $(\Mb, {h_0}_{\alpha\beta}, {h_1}_{\alpha\beta},\varrho_0,\varrho_1, \phi_0,\phi_1)$ , 
 is said to be {\bf asymptotically flat} if 
\begin{itemize}
\item the initial slice $M_0$ is diffeomorphic to $\mathbb{R}^3$ and in its canonical coordinate system, the initial data set satisfies the wave constraint equations.

\item in the canonical coordinate system $(x^1,x^2,x^3)$,
$$
\|{h_0}_{\alpha\beta}\|_{\mathcal{E}_{-1}}\leq \eps_0,
$$
where $\eps_0$ represents the ADM mass.

\item the $\mathcal{E}_{-1}(\RR^3)$ norm of $\varrho_0,\varrho_1,\phi_1$ are finite.
\end{itemize}
\end{definition}

Hence, the initial data behaves like $r^{-1}$ at spatial infinity. A geometrical initial data $(\Mb, \gb^{\ddag}, \Kv, \varrhozero, \varrhoone, \phizerodd, \phionedd)$ is called asymptotically flat, if it gives a asymptotically flat PDE initial data.

We recall that the components of the solution to the system \eqref{eq conformal aug sec6} are functions defined in $\RR^4$ with three spatial variables and one time variable. To study these functions, say $u=u(t, \cdot)$, we need to the following norms and corresponding spaces: 
$$
\aligned
\|u(t,\cdot)\|_{E^d} :&= \sum_{|I_1|+|I_2|\leq d}\|\del^{I_1}\Omega^{I_2}u(t,\cdot)\|_{L^2(\RR^3)}
\\
&= \sum_{k+|J_1|+|J_2|\leq d}\|\del_t^k\del_x^{J_1}\Omega^{J_2}u(t,\cdot)\| = \sum_{k+l\leq d}\|\del_t^k u(t,\cdot)\|_{X^l},
\endaligned
$$
$$
\aligned
\|u(t,\cdot)\|_{E_P^{d+1}}
:&= \sum_{\alpha\atop |I_1|+|I_2|\leq d}\|\del_{\alpha}\del^{I_1}\Omega_{I_2}u(t,\cdot)\|_{L^2(\RR^3)}
\\
&=\sum_{\alpha}\sum_{k+|J_1|+|J_2|\leq d}\|\del_{\alpha}\del_t^k\del_x^{J_1}\Omega^{J_2}u(t,\cdot)\|_{L^2(\RR^3)}
=\sum_{k+l\leq d}\|\del_t^ku(t,\cdot)\|_{X_P^l},
\endaligned
$$
and thn 
$$
\|u(t,\cdot)\|_{E_R^{d+1}} := \|u(t,\cdot)\|_{E^d} + \|u(t,\cdot)\|_{E_P^{d+1}},
$$
$$
\|u(t,\cdot)\|_{E_H^{d+1}} := \|u(t,\cdot)\|_{\mathcal{E}_{-1}} + \|u(t,\cdot)\|_{E_P^{d+1}}.
$$
We also define several norms on the time interval $[0,T]$ for $p\in[1,\infty]$:
$$
\aligned
\|u\|_{L^p([0,T];E^d)} :&=\big\| \|u(t,\cdot)\|_{E^d}\big\|_{L^p([0,T])},
\quad
\|u\|_{L^([0,T];E_P^d)} :=\big\| \|u(t,\cdot)\|_{E_P^d}\big\|_{L^p([0,T])},
\\
\|u\|_{L^p([0,T];E_R^d)} :&=\big\| \|u(t,\cdot)\|_{E_R^d}\big\|_{L^p([0,T])},
\quad
\|u\|_{L^p([0,T];E_H^d)} :=\big\| \|u(t,\cdot)\|_{E_H^d}\big\|_{L^p([0,T])}.
\endaligned
$$
Finally, for $p<\infty$, we have the functional spaces of interest 
$$
L^p([0,T];E^d), \qquad
L^p([0,T];E_P^d),
\qquad 
L^p([0,T];E_R^d), 
$$
and $L^p([0,T];E_H^d)$, defined by completion with respect to the corresponding norms on the $C_c^\infty(\RR^4)$ functions . This leads us to the definition of
 $C([0,T],;E^d)$, $C([0,T];E_P^d)$ and $C([0,T];E_H^d)$ by completion of $C_c^{\infty}$ functions with respect to the corresponding norms.


\subsection{Existence theorem for the nonlinear field equations}

We  introduce the following norm for the initial data $S_0:=(h_0,h_1,\phi_0,\phi_1,\varrho_0,\varrho_1)$:
$$
\aligned
\|S_0\|_{X_{\coeff}^{d+1}} := \max\{&\|{h_0}_{\alpha\beta}\|_{X_H^{d+1}},\,\|{h_1}_{\alpha\beta}\|_{X^d},\,\coeff^{-(1/2)[d/2]+1/4}\|\phi_0\|_{X_P^{d+1}},
\\
& \coeff^{-(1/2)[d/2]+1/4}\|\phi_1\|_{X^d},\,\coeff^{-[d/2]-1/2}\|\varrho_0\|_{X_P^{d+1}},\, \coeff^{-[d/2]-1/2}\|\varrho_1\|_{X^d}\}.
\endaligned
$$ 
We are ready to state one of our main results. 

\begin{theorem}[Local existence with uniform bounds]\label{prop 7-10-1} 
Given any integer $d\geq 4$, assume that
$$
({h_0}_{\alpha\beta},{h_1}_{\alpha\beta}) \in X_H^{d+1}\times X^d, \quad
(\phi_0,\phi_1)\in X_P^{d+1}\times X^d,\quad
(\varrho_0,\varrho_1)\in X_R^{d+1}\times X^d
$$
and suppose that 
$S_0:=(h_0,h_1,\phi_0,\phi_1,\varrho_0,\varrho_1)$ satisfies 
\be
\|S_0\|_{X_{\coeff}^{d+1}}\leq \eps\leq \eps_0
\ee
for some sufficiently small $\eps_0>0$.
Then, there exist constants $A, T^*>0$ which are {\bf independent} of $\coeff$ and such that, within the time interval $[0,T^*]$, the Cauchy problem \eqref{eq conformal aug sec7} (with $0< \coeff \leq 1$)  has a unique solution $(h_{\alpha\beta}^\coeff,\phi^\coeff,\varrho^\coeff)$ in the following functional space (with $0\leq k\leq d-1$): 
$$
\aligned
&\del_t^k h^{\coeff}_{\alpha\beta} \in C([0,T],X_H^{d-k})\cap C^1([0,T],X_H^{d-k-1}),\quad
\\
&
\del_t^k \phi^{\coeff} \in C([0,T],X_P^{d-k})\cap C^1([0,T],X_P^{d-k-1}),
\\
&\del_t^k \varrho^\coeff \in  C([0,T],X_R^{d-k})\cap C^1([0,T],X_R^{d-k-1}),
\endaligned
$$
Furthermore, the following estimates hold with constant {\bf independent} of $\coeff$:
\begin{subequations}\label{eq 1 prop 7-10-1}
\begin{equation}\label{eq 1 prop 7-10-1 a}
\|h_{\alpha\beta}^{\coeff}(t,\cdot)\|_{E_H^d} \leq A\eps, 
\end{equation}
\begin{equation}\label{eq 1 prop 7-10-1 b}
\|\phi^\coeff\|_{E_P^d} \leq A\eps, 
\end{equation}
\begin{equation}\label{eq 1 prop 7-10-1 c}
\|\varrho_n^\coeff\|_{E_P^d} + \coeff^{-1/2}\|\varrho_n^\coeff\|_{E^{d-1}}\leq A\eps. 
\end{equation} 
\end{subequations} 
\end{theorem}

Equipped with the above theorem, we are thus able to build the local solution of the original Cauchy problem stated in Definition \ref{Cauchy conformal aug}.

\begin{theorem}[Existence of modified gravity developments] Consider an initial data set $(\Mb, \gb^{\ddag}, \Kv, \varrhozero, \varrhoone, \phizerodd, \phionedd)$ for the Cauchy problem in Definition \ref{def36aug} and assume that its associated PDE initial data $S_0 = (h_0,h_1,\phi_0,\phi_1,\varrho_0,\varrho_1)$ is asymptotically flat and satisfies the conditions in Theorem \ref{prop 7-10-1}. Let $ (h_{\alpha\beta}^\coeff,\phi^\coeff,\varrho^\coeff)$ be the corresponding solution of \eqref{eq conformal aug} associated with $S_0$. Then the spacetime $([0,T]\times\RR^3,\gv)$ is a modified gravity development of the initial data  $(\Mb, \gb^{\ddag}, \Kv, \varrhozero, \varrhoone, \phizerodd, \phionedd)$.
\end{theorem}

\begin{proof} 
We simply note that the local solution $(h_{\alpha\beta}^\coeff,\phi^\coeff,\varrho^\coeff)$ is sufficiently regular and that $h = \gv - m$ sufficiently small, which
guarantees that the metric $\gv$ is globally hyperbolic on $[0,T]\times \RR^3$. 
We can apply the result about the preservation of constraints in Proposition~\ref{prop aug 1}. Once the constraints
$$
\aligned
&e^{2\varrho} = \frac{1}{2}\ln f'(R_g),
\qquad\quad 
\Gammav^\lambda = 0
\endaligned
$$
hold, we see that the pair $(\gv, \phi)$ satisfies the conformal field equations \eqref{main eq trans}.
\end{proof}


\newpage 

\section{Technical tools for the local existence theory}

\subsection{Estimates on commutators}

From this subsection we will make some preparations in order to prove Theorem \ref{prop 7-10-1}. In this subsection we derive commutator estimates. First we point out the following commutation relations:
\begin{equation}\label{eq com 7-6}
[\del_a, \Omega_{bc}] = \delta_{ab}\del_c - \delta_{ac}\del_b,\quad [\del_t, \Omega_{ab}] = 0.
\end{equation}

\begin{lemma}\label{lem 7-6'-1}
If $u$ is a smooth  function defined on $[0,T]\times \RR^3$, then for any multi-index $I_1, I_2$ with $|I_1| + |I_2| = n$, the following estimates hold:
\begin{equation}\label{eq 1 lem 7-6'-1}
\aligned
&[\Omega^I,\del_\alpha]u = \sum_{|J|<|I|}\Theta^{Ic}_{\alpha J}\del_c\Omega^Ju,
\endaligned
\end{equation}
where $\Theta_{\alpha J}^{Ic}$ are constants. When $\alpha = 0$, one has $\Theta_{\alpha J}^{Ic} = 0$.
\end{lemma}

\begin{proof} First, we observe that when $\alpha = 0$, $\del_0 = \del_t$ commutes with $\Omega^I$. When $\alpha>0$, this is proven by induction on the order of $|J|$. When $|J| = 1$, the result is proven by \eqref{eq com 7-6}. We denote by
$$
[\Omega_a,\del_\beta] = \theta_{a\beta}^c\del_c.
$$

Now assume that for $|J|\leq k$, \eqref{eq 1 lem 7-6'-1} is valid. For $|J| = k$, we find 
$$
\Omega^J = \Omega_{a_1}\Omega^{J_1},
$$
where $|J_1|=k-1$. Then, we have 
$$
\aligned
\,[\Omega^J, \del_a]u &= \Omega_{a_1}\Omega^{J_1}\del_\alpha u - \del_\alpha\Omega_{a_1}\Omega^{J_1}u
= \Omega_{a_1}\Omega^{J_1}\del_\alpha u - \Omega_{a_1}\del_\alpha\Omega^{J_1}u
+ \Omega_{a_1}\del_\alpha\Omega^{J_1}u - \del_\alpha\Omega_{a_1}\Omega^{J_1}u
\\
&= \Omega_{a_1}\big([\Omega^{J_1},\del_\alpha]u\big) + [\Omega_{a_1},\del_\alpha]\Omega^{J_1}u
\\
&= \Omega_{a_1}\bigg(\sum_{|K_1|<|J_1|}\Theta_{\alpha K_1}^{J_1 b}\del_b\Omega^{K_1} u\bigg) + \theta_{a_1 a}^b\del_b\Omega^{J_1}u
\\
&= \sum_{|K_1|<|J_1|}\Theta_{\alpha K_1}^{J_1 b} \Omega_{a_1} \del_b\Omega^{K_1} u + \theta_{a_1 a}^b\del_b\Omega^{J_1}u
\\
&=\sum_{|K_1|<|J_1|}\Theta_{\alpha K_1}^{J_1 b} \del_b\Omega_{a_1}\Omega^{K_1} u + \theta_{a_1 a}^b\del_b\Omega^{J_1}u
+\sum_{|K_1|<|J_1|}\Theta_{\alpha K_1}^{J_1 b} [\Omega_{a_1},\del_b]\Omega^{K_1} u
\\
&=\sum_{|K_1|<|J_1|}\Theta_{\alpha K_1}^{J_1 b} \del_b\Omega_{a_1}\Omega^{K_1} u + \theta_{a_1 a}^b\del_b\Omega^{J_1}u
+\sum_{|K_1|<|J_1|}\Theta_{\alpha K_1}^{J_1 b} \theta_{a_1 b}^c\del_c\Omega^{K_1} u.
\endaligned
$$ 
\end{proof}

\begin{lemma}\label{lem 7-6'-2}
Let u be a smooth function defined in $[0,T]\times \RR^3$. Then the following estimates hold for $|I_1|+|I_2|=d$:
\begin{subequations}\label{eq com lem 7-6'-1}
\begin{equation}\label{eq com lem 7-6'-1 d}
|\del_x^{I_1}\Omega^{I_2}\del_x u| \leq |\del_x\del_x^{I_1}\Omega^{I_2}u| + C(d)\sum_{|I_2'|< |I_2|}|\del_x \del_x^{I_1}\Omega^{I_2'} u|,
\end{equation}
\begin{equation}\label{eq com lem 7-6'-1 e}
|\del_x \del_x^{I_1}\Omega^{I_2} u|\leq |\del_x^{I_1}\Omega^{I_2}\del_x u| + C(d)\sum_{|I_1'|+|I_2'|< d}| \del_x^{I_1'}\Omega^{I_2'} \del_x u|,
\end{equation}
\begin{equation}\label{eq com lem 7-6'-1 a}
\big|[\del_x^{I_1}\Omega^{I_2},\del_{\alpha}\del_{\beta}]u\big|
\leq C(d)\sum_{a,\alpha\atop |J_2|<|I_2|}|\del_a\del_{\alpha}\del_x^{I_1}\Omega^{J_2}u|,
\end{equation}
\begin{equation}\label{eq com lem 7-6'-1 b}
|\del_r^k\Omega^{I_1} u (t,x)|\leq C(k)\sum_{|J_1|\leq k}|\del^{J_1}_x \Omega^{I_1} u(t,x)|, \quad \text{ for } |x|\geq 1,
\end{equation}
\begin{equation}\label{eq com lem 7-6'-1 c}
\aligned
|[\del_x^{I_1}\Omega^{I_2}, H^{\alpha\beta}\del_{\alpha}\del_{\beta}]u|
& \leq \sum_{J_1 + J_1' = I_1 \atop {J_2+J_2' = I_2 \atop |J_1'|+|J_2'|>0}}
|\del_x^{J_1'}\Omega^{J_2'}H^{\alpha\beta}|\,|\del_{\alpha}\del_{\beta}\del_x^{J_1}\Omega^{J_2}u|
\\
&+C(d) \sum_{J_1+J_1'=I_1\atop {|J_2|+|J_2'|<|I_2|\atop \alpha,\beta,\alpha',a}}
|\del_x^{J_1'}\Omega^{J_2'}H^{\alpha\beta}|\,|\del_{\alpha'}\del_a \del_x^{J_1}\Omega^{J_2}u|,
\endaligned
\end{equation}
where $H^{\alpha\beta}$ are smooth functions.
\end{subequations}
\end{lemma}

\begin{proof} In view of \eqref{eq 1 lem 7-6'-1}, the following identity is immediate: 
$$
[\del_x^{I_1}\Omega^{I_2},\del_{\alpha}] = \sum_{|K|<|I_2|}\Theta_{\alpha K}^{I_2 b}\del_b\del_x^{I_1}\Omega^Ku.
$$

   To derive \eqref{eq com lem 7-6'-1 d},  we observe that
$$
[\del_x^{I_1}\Omega^{I_2},\del_{\alpha}]u
= \del_x^{I_1}\big([\Omega^{I_1},\del_\alpha]u\big)
= \sum_{c,|J_2|\leq |I_2|}\del_x^{I_1}\big(\Theta_{\alpha J_2}^{I_1 c}\del_{c}\Omega^{J_2}u\big)
= \sum_{c,|J_2|\leq |I_2|}\Theta_{\alpha J_2}^{I_1 c}\del_{c}\del_x^{I_1}\Omega^{J_2}u
$$
and
\begin{equation}\label{eq pr0 lem 7-6'-2}
\sum_{c,|J_2|\leq |I_2|}\big|\Theta_{\alpha J_2}^{I_1 c}\del_{c}\del_x^{I_1}\Omega^{J_2}u\big|
\leq C(d)\sum_{c,|I_2'|<|I_2|}|\del_c\del^{I_1}\Omega^{I_2'}u|.
\end{equation}
This establishes \eqref{eq com lem 7-6'-1 d}.

   The equation \eqref{eq com lem 7-6'-1 e} is derived by induction on $d$. Clearly, \eqref{eq com lem 7-6'-1 e} holds for $d=0$. If it holds for $d\leq k$, let us prove the case $d=k+1$:
$$
\big|\del_\alpha\del^{I_1}\Omega^{I_2} u\big|\leq \big|\del^{I_1}\Omega^{I_2}\del_\alpha u\big| + \big|[\del_\alpha,\del^{I_1}\Omega^{I_2}]u\big|.
$$
Then by \eqref{eq pr0 lem 7-6'-2},
$$
\aligned
\big|\del_\alpha\del^{I_1}\Omega^{I_2} u\big|
& \leq  \big|\del^{I_1}\Omega^{I_2}\del_\alpha u\big| + C(d)\sum_{c,|I_2'|<|I_2|}|\del_c\del^{I_1}\Omega^{I_2'}u|. 
\endaligned
$$
Note that $|I_2'|\leq |I_2|-1\leq k$. We apply the induction assumption on the second term in the right-hand-side and obtain 
$$
\aligned
|\del_c\del^{I_1}\Omega^{I_2'}u|\leq |\del^{I_1}\Omega^{I_2'}\del_c u| + C(d-1)\sum_{c',|I_2''|>|I_2'|}\big|\del^{I_1}\Omega^{I_2''}\del_{c'}u\big|, 
\endaligned
$$
which proves \eqref{eq com lem 7-6'-1 e}.

   The proof of \eqref{eq com lem 7-6'-1 a} is a direct application of \eqref{eq 1 lem 7-6'-1}.
\begin{equation}\label{eq pr1 lem 7-6'-2}
\aligned
\,[\del_x^{I_1}\Omega^{I_2},\del_{\alpha}\del_{\beta}]u
&= [\del_x^{I_1}\Omega^{I_2},\del_{\alpha}]\del_{\beta}u + \del_{\alpha}\big([\del_x^{I_1}\Omega^{I_2},\del_{\beta}]u\big)
\\
&=\sum_{|J_2|<|I_2|}\Theta_{\alpha J_2}^{I_2 c}\del_c\del_x^{I_1}\Omega^{J_2}\del_{\beta}u
 + \sum_{|J_2|<|I_2|}\Theta_{\beta J_2}^{I_2 c}\del_{\alpha}\del_c\del_x^{I_1}\Omega^{J_2}u
\\
&=\sum_{|J_2|<|I_2|}\Theta_{\alpha J_2}^{I_2 c}\del_c\del_x^{I_1}\del_{\beta}\Omega^{J_2}u
 + \sum_{|J_2|<|I_2|}\Theta_{\alpha J_2}^{I_2 c}\del_c\del_x^{I_1}[\Omega^{J_2},\del_\beta]u
\\
 &+ \sum_{|J_2|<|I_2|}\Theta_{\beta J_2}^{I_2 c}\del_{\alpha}\del_c\del_x^{I_1}\Omega^{J_2}u
\\
&=\sum_{|J_2|<|I_2|\atop |K_2|<|I_2|}
\Theta_{\alpha J_2}^{I_2 c}\Theta_{\beta K_2 }^{J_2c'}\del_c\del_{c'}\del_x^{I_1}\Omega^{K_2}u
\\
 &+ \sum_{|J_2|<|I_2|}\Theta_{\beta J_2}^{I_2 c}\del_{\alpha}\del_c\del_x^{I_1}\Omega^{J_2}u
 +\sum_{|J_2|<|I_2|}\Theta_{\alpha J_2}^{I_2 c}\del_c\del_x^{I_1}\del_{\beta}\Omega^{J_2}u
\endaligned
\end{equation}
which leads to \eqref{eq com lem 7-6'-1 a}.

   The proof of \eqref{eq com lem 7-6'-1 b} needs the notion of homogeneous functions. A smooth function $f$ defined in the pointed region $\RR^3\setminus\{0\}$ is said to be {\bf homogeneous of degree $i$} if
$$
f(rx) = r^if(x), \text{ for any } r>0 \text{ and } x\in \RR^3\setminus\{0\}.
$$
It is well-known that the partial derivatives of a homogeneous function of degree $i$ are also homogeneous and of degree $i-1$.

We denote by $\omega^a:=\frac{x^a}{r}$ and we note that they are homogeneous functions of degree $0$. And recall the definition of radial derivative $\del_r = \omega^a\del_a$. We will prove that
\begin{equation}\label{eq pr 1 lem 7-6-2}
\del_r^k = \sum_{a,\atop |I|\leq k} A_I^k\del^I_x,
\end{equation}
where $A_I^k$ is a homogeneous function of degree $-k+|I|$.
For $k=1$, this is guaranteed by the expression of $\del_r$. Assume that this  holds for the integer less than or equal to $k$, we will prove the case of $k+1$:
$$
\aligned
\del^{k+1}_r u &= \del_r\del^k_r u = \del_r \bigg(\sum_{I,a} A_I^k\del^I_x\bigg)
\\
&=\sum_{I,a}\omega^b\del_b\big(A_I^k\del^I_x\big) = \sum_{I,a}\omega^b\del_bA_I^k \del^I_x
+ \sum_{I,a}\omega^bA_I^k\del_b\del_x^I
\endaligned
$$
We observe that $\omega^b\del_b A_I^k$ is homogeneous of degree $-k-1+|I|$ and $\omega^b A_I^k$ is homogeneous of degree $-(k+1)+|I|+1$. This concludes \eqref{eq pr 1 lem 7-6-2}. Next,  we see that \eqref{eq com lem 7-6'-1 b} follows immediately from \eqref{eq pr 1 lem 7-6-2}.

   To prove \eqref{eq com lem 7-6'-1 c}, we perform the following calculation:
$$
\aligned
&\,[\del_x^{I_1}\Omega^{I_2},H^{\alpha\beta}\del_{\alpha}\del_{\beta}u]
\\
&= \sum_{J_1+J_1'=I_1,J_2+J_2'=I_2\atop |J_1'|+|J_2'|>0}
\del_x^{J_1'}\Omega^{J_2'}H^{\alpha\beta}\del_x^{J_1}\Omega^{J_2}\del_{\alpha}\del_{\beta}u
+H^{\alpha\beta}[\del_x^{I_1}\Omega^{I_2},\del_{\alpha}\del_{\beta}]u
\\
&= \sum_{J_1+J_1'=I_1,J_2+J_2'=I_2\atop |J_1'|+|J_2'|>0}\del_x^{J_1'}
\Omega^{J_2'}H^{\alpha\beta}\del_{\alpha}\del_{\beta}\del_x^{J_1}\Omega^{J_2}u
+ \sum_{J_1+J_1'=I_1\atop J_2+J_2'=I_2}\del_x^{J_1'}
\Omega^{J_2'}H^{\alpha\beta}[\del_x^{J_1}\Omega^{J_2},\del_{\alpha}\del_{\beta}]u
\\
&= \sum_{J_1+J_1'=I_1,J_2+J_2'=I_2\atop |J_1'|+|J_2'|>0}\del_x^{J_1'}
\Omega^{J_2'}H^{\alpha\beta}\del_{\alpha}\del_{\beta}\del_x^{J_1}\Omega^{J_2}u
\\
&+\sum_{J_1+J_1'=I_1\atop J_2+J_2'=I_2}\del_x^{J_1'}
\Omega^{J_2'}H^{\alpha\beta}
\bigg(
\sum_{|K_2|<|J_2|\atop|K_2'|<|K_2|}\Theta_{\alpha K_2}^{J_2 c}\Theta_{\beta K_2'}^{c'K_2}\del_c\del_{c'}\del_x^{J_1}\Omega^{K_2'}u
\\
& \hskip5.cm 
+\sum_{|K_2|<|J_2|}\Theta_{\beta K_2}^{J_2 c}\del_{\alpha}\del_c\del_x^{J_1}\Omega^{K_2}u
\bigg),
\endaligned
$$
where $\eqref{eq pr1 lem 7-6'-2}$ is applied. Then \eqref{eq com lem 7-6'-1 c} follows from this identity.
\end{proof}

We also need the commutator estimates on the product in the form $\del^{I_1}\Omega^{I_2}$.

\begin{lemma}\label{lem 7-6'-3}
Let u be a smooth function defined in $[0,T]\times \RR^3$. Then the following estimates hold for $|I_1|+|I_2|=d$:
\begin{subequations}\label{eq com' lem 7-6'-1}
\begin{equation}\label{eq com' lem 7-6'-1 d}
|\del^{I_1}\Omega^{I_2}\del_{\alpha} u| \leq |\del_\alpha\del^{I_1}\Omega^{I_2}u| + C(d)\sum_{a,|I_2'|< |I_2|}|\del_a \del^{I_1}\Omega^{I_2'} u|,
\end{equation}
\begin{equation}\label{eq com' lem 7-6'-1 e}
|\del_{\alpha} \del^{I_1}\Omega^{I_2} u|\leq |\del^{I_1}\Omega^{I_2}\del_{\alpha} u| + C(d)\sum_{\beta\atop |I_1'|+|I_2'|< d}| \del^{I_1'}\Omega^{I_2'} \del_{\beta} u|,
\end{equation}
\begin{equation}\label{eq com' lem 7-6'-1 a}
\big|[\del^{I_1}\Omega^{I_2},\del_{\alpha}\del_{\beta}]u\big|
\leq C(d)\sum_{a,\alpha\atop |J_2|<|I_2|}|\del_a\del_{\alpha}\del^{I_1}\Omega^{J_2}u|,
\end{equation} 
\begin{equation}\label{eq com' lem 7-6'-1 c}
\aligned
|[\del^{I_1}\Omega^{I_2}, H^{\alpha\beta}\del_{\alpha}\del_{\beta}]u|
& \leq \sum_{J_1 + J_1' = I_1 \atop {J_2+J_2' = I_2 \atop |J_1'|+|J_2'|>0}}
|\del^{J_1'}\Omega^{J_2'}H^{\alpha\beta}|\,|\del_{\alpha}\del_{\beta}\del^{J_1}\Omega^{J_2}u|
\\
&+C(d) \sum_{J_1+J_1'=I_1\atop {|J_2|+|J_2'|<|I_2|\atop \alpha,\beta,\alpha',a}}
|\del^{J_1'}\Omega^{J_2'}H^{\alpha\beta}|\,|\del_{\alpha'}\del_a \del^{J_1}\Omega^{J_2}u|. 
\endaligned
\end{equation}
\end{subequations}
\end{lemma}

\begin{proof} From \eqref{eq 1 lem 7-6'-1}, the following identity is immediate:
$$
[\del^{I_1}\Omega^{I_2},\del_{\alpha}] = \sum_{|K|<|I_2|}\Theta_{\alpha K}^{I_2 b}\del_b\del_x^{I_1}\Omega^Ku.
$$
Then we perform exactly the same calculation as in the proof of Lemma~\ref{lem 7-6'-2} with $\del_x^{I_1}$ replaced by $\del^{I_1}$.
\end{proof}


\subsection{Global Sobolev inequalities and embedding properties}

For completeness, we re-derive a classical estimate (due to Klainerman).

\begin{lemma}\label{lem 7-6-1}
For all $u\in X^{d}$ with $d\geq 2$, one has 
\begin{equation}\label{eq 1 lem 7-6-1}
\|u\|_{\mathcal{E}_{-1}}\leq C\|u\|_{X^{d}}.
\end{equation}
\end{lemma}

\begin{proof} We only prove this inequality for smooth functions since, by regularization, it then extends to the whole $X^d$. We consider $\RR^3$ equipped with the polar coordinates, i.e.:
$$
x^1 = r\sin\theta\cos\varphi,\quad x^2 = r\sin\theta\sin\varphi,\quad x^3 = r\cos\theta.
$$
Note that \eqref{eq 1 lem 7-6-1} is equivalent to the following inequality for all $x_0\in \RR^3$:
$$
|u(x_0)|(1+r_0)\leq C\|u\|_{X^{d}}
$$
with $x_0 = (x_0^1,x_0^2,x_0^3) = (r_0\sin\theta_0\cos\varphi_0, r_0\sin\theta_0\sin\varphi_0,r_0\cos\theta_0)$.

The case $r_0\leq 1$ is direct by classical Sobolev inequality. We thus focus on the case $r_0>1$ and we consider the estimate on the the following open subset of $\RR^3$ defined by
$$
V = (r_0-1/2,r_0+1/2)\times(\theta_0-1/2,\theta_0+1/2)\times(\varphi_0-1/2,\varphi_0+1/2)
$$

Now let $u$ be a smooth function and denote also by $u$ its restriction on $V$, with 
$$
v(r,\theta,\varphi) := u(r\sin\theta\cos\varphi, r\sin\theta\sin\varphi,r\cos\theta)
$$
with $r_0>1$. Then by the classical Sobolev inequality, we have 
$$
|u(x_0)|^2 = |v(r_0,\theta_0,\phi_0)|^2
\leq C\sum_{k_0+k_1+k_2\geq 2}\int_V|\del_r^{k_0}\del_{\theta}^{k_1}\del_{\varphi}^{k_2}v|^2drd\theta d\varphi.
$$
Note that in $V$, $r_0-1/2<r<r_0+1/2$, which leads to $1 - 1/(2r_0)< r/r_0 < 1+1/(2r_0)$. Recall that $r_0>1$, then
$$
\frac{1}{2}<\frac{r}{r_0}<\frac{3}{2}.
$$
Thus, we have 
$$
\aligned
\frac{1}{2}\int_V|\del_r^{k_0}\del_{\theta}^{k_1}\del_{\varphi}^{k_2}v|^2drd\theta d\varphi
& \leq \int_V|\del_r^{k_0}\del_{\theta}^{k_1}\del_{\varphi}^{k_2}v|^2\frac{r^2}{r_0^2}drd\theta d\varphi
\\
&= r_0^{-2} \int_V|\del_r^{k_0}\del_{\theta}^{k_1}\del_{\varphi}^{k_2}v|^2 r^2dr d\theta d\varphi
\\
&= r_0^{-2} \int_V|\del_r^{k_0}\del_{\theta}^{k_1}\del_{\varphi}^{k_2}v|^2dx
\\
& \leq r_0^{-2} \int_{\RR^3}|\del_r^{k_0}\del_{\theta}^{k_1}\del_{\varphi}^{k_2}v|^2dx,
\endaligned
$$
where $dx$ is the standard volume form of $\RR^3$.

Here, we observe that
$$
\del_\theta  v
=\cos\varphi \Omega_{31}u + \sin\varphi \Omega_{32}u
= \frac{x^1}{\big((x^1)^2+(x^2)^2\big)^{1/2}}\Omega_{31}u + \frac{x^2}{\big((x^1)^2+(x^2)^2\big)^{1/2}}\Omega_{32}u
$$
and
$$
\del_{\varphi}v = \Omega_{12}u, \quad \del_r v = \del_r u.
$$

Note that $\Omega^{J}\cos\varphi$, $\Omega^{J}\sin\varphi$ are homogeneous of degree $0$. So, by homogeneity, that for $r\geq 1/2$, we have 
$$
|\del_r^{k_0}\del_\theta^{k_1}\del_{\varphi}^{k_2}v|
=\big|\del_r^{k_0}\big(\cos\varphi \Omega_{31}+ \sin\varphi\Omega_{33}\big)^{k_1}\Omega_{12}^{k_2}u\big|
\leq C\sum_{|I|\leq k_1+k_2}|\del_r^{k_0}\Omega^Iu|, 
$$
which leads to
\begin{equation}
|u(x_0)|^2\leq Cr_0^{-2}\sum_{|I|+k_0\leq 2}\int_{\RR^3}|\del_r^{k_0}\Omega^Iu|^2dx. 
\end{equation}
Then by \eqref{eq com lem 7-6'-1 b}, the desired result is proven. 
\end{proof}

We will also need the following embedding result.

\begin{lemma}\label{lem 7-6-2}
Let $u$ be a function in $X_H^{d+2}$ and $v\in L^{\infty}([0,T];E_H^{d+2})$. Then for all pair of multi-indices $(I_1,I_2)$ with $|I_1|+|I_2|\leq d$, the following estimate holds:
\begin{subequations}
\begin{equation}\label{eq 1 lem 7-6-2}
\big{\|}\del_x^{I_1}\Omega^{I_2} u \big{\|}_{L^{\infty}(\RR^3)}\leq C\|u\|_{X_H^{d+2}}.
\end{equation}
\begin{equation}\label{eq 1' lem 7-6-2}
\big{\|}\del^{I_1}\Omega^{I_2} v(t,\cdot) \big{\|}_{L^{\infty}(\RR^3)}\leq C\|v(t,\cdot)\|_{E_H^{d+2}}\quad \text{ for } t\in[0,T].
\end{equation}
\end{subequations}
\end{lemma}

\begin{proof}
We only prove this inequality when $u\in C_c^{\infty}(\RR^3)$ and $v\in C_c^{\infty}(\RR^4)$. Then by  regularization, it extends on $X_H^{d+2}$ and $L^{\infty}([0,T];E_H^{d+2})$.

   We begin with \eqref{eq 1 lem 7-6-2}, and the proof is decomposed as follows. 

\

\noindent {\bf Case 1: $I_1 = I_2 = 0$.} The left-hand-side of \eqref{eq 1 lem 7-6-2} is controlled by its $\mathcal{E}_{-1}$ norm so \eqref{eq 1 lem 7-6-2} holds.

\

\noindent {\bf Case 2: $|I_1|>0$.} In this case we suppose that $I_1 = (a_1,a_2,\cdots,a_n)$ and denote by $I_1' = (a_2,\cdots,a_n)$. Then $I_1'$ is of order $n-1\geq 0$ and, by classical Sobolev's inequality,
$$
\big{\|}\del_x^{I_1}\Omega^{I_2}u\big{\|}_{L^{\infty}(\RR^3)} = \big{\|}\del_{a_1}\del_x^{I_1'}\Omega^{I_2}u\big{\|}_{L^{\infty}(\RR^3)}
\leq C\sum_{|J_1|\leq 2}\big{\|}\del_x^{J_1}\del_{a_1}\del_x^{I_1'}\Omega^{I_2}u\big{\|}_{L^2(\RR^3)}\leq C\|u\|_{X_P^{d+2}}.
$$

\

\noindent {\bf Case 3: $|I_2| = 0, |I_2|>0$.} In this case we suppose that $I_2 = (b_1,b_2,\cdots,b_n)$ and denote by $I_b' = (b_2,\cdots,b_n)$. Then $I_2'$ is of order $n-1\geq 0$, and 
$$
\big{\|}\del_x^{I_1}\Omega^{I_2}u\big{\|}_{L^{\infty}(\RR^3)} = \big{\|}\Omega_{b_1}\Omega^{I'_2}u \big{\|}_{L^{\infty}(\RR^3)}
\leq C\sum_b \big{\|}(1+r)\del_b\Omega^{I'_2}u\big{\|}_{L^{\infty}(\RR^3)} = C\sum_b\big{\|}\del_b\Omega^{I'_2}u\big{\|}_{\mathcal{E}_{-1}}. 
$$
Then by \eqref{eq 1 lem 7-6-1}, we have 
$$
\aligned
\big{\|}\del_x^{I_1}\Omega^{I_2}u\big{\|}_{L^{\infty}(\RR^3)}
& \leq C\sum_b\big{\|}\del_b\del_x^{I_1}\Omega^{I'_2}u\big{\|}_{X^2}
= C\sum_{|J_1|+|J_2|\leq 2}\sum_b\big{\|}\del_x^{J_1}\Omega^{J_2}\del_b\del_x^{I_1}\Omega^{I'_2}u\big{\|}_{L^2(\RR^3)}
\\
& \leq C\sum_{|J_1|+|J'_2|\leq 2\atop b'}\big{\|}\del_{b'}\del_x^{I_1'}\del_x^{J_1}\Omega^{J_2'}\Omega^{I'_2}u\big{\|}_{L^2(\RR^3)} \leq C\|u\|_{X_P^{d+2}},
\endaligned
$$
where the commutator estimate \eqref{eq com lem 7-6'-1 d} is used. 

By combining these three cases together, \eqref{eq 1 lem 7-6-2} is proven.

   We then prove \eqref{eq 1' lem 7-6-2}. The proof is similar and we also discuss three different cases. 

\

\noindent {\bf Case 1: $I_1 = I_2 = 0$.} The left-hand-side of \eqref{eq 1' lem 7-6-2} is controlled by its $\mathcal{E}_{-1}$ norm so \eqref{eq 1' lem 7-6-2} holds.

\

\noindent {\bf Case 2: $I_1>0$.} In this case we also suppose that $I_1 = (\alpha_1,\alpha_2,\dots,\alpha_n)$ and denote by $I'_1 = (\alpha_2,\dots,\alpha_n)$. Then $|I'_1|\geq 0$, and also by the classical Sobolev's inequality, 
$$
\aligned
\|\del^{I_1}\Omega^{I_2}v (t,\cdot)\|_{L^{\infty}}
& \leq C\sum_{|J_1|\leq 2}\|\del_x^{J_1}\del_{\alpha_1}\del^{I'_1}\Omega^{I_2}v(t,\cdot)\|_{L^2(\RR^3)}
\\
=& C\sum_{|J_1|\leq 2}\|\del_{\alpha_1}\del_x^{J_1}\del^{I'_1}\Omega^{I_2}v(t,\cdot)\|_{L^2(\RR^3)}
\leq C\|v(t,\cdot)\|_{E_P^{d+2}}.
\endaligned
$$

\

\noindent {\bf Case 3: $|I_1|=0,|I_2|>0$.} We suppose that $I_2= (b_1,b_2,\dots, a_n)$ and denote by $I_2' = (b_2,\dots,b_n)$. Then, we have $|I_2'|\geq 0$, and 
$$
\aligned
\|\del^{I_1}\Omega^{I_2}v(t,\cdot)\|_{L^{\infty}(\RR^3)} &= \|\Omega_{b_1}\Omega^{I_2'}v(t,\cdot)\|_{L^{\infty}}
\\
&\leq C\sum_{b}\|(1+r)\del_b\Omega^{I_2'}v(t,\cdot)\|_{L^{\infty}(\RR^3)}
= C\sum_{b}\|\Omega^{I_2'}v(t,\cdot)\|_{\mathcal{E}_{-1}}
\\
& \leq C\sum_{b}\|\del_b\Omega^{I_2'}v(t,\cdot)\|_{X^2}
=  C\sum_{b\atop |J_1|+|J_2|\leq 2}\|\del_x^{J_1}\Omega^{J_2}\del_b\Omega^{I_2'}v(t,\cdot)\|_{X^2}
\\
& \leq  C\sum_{b'\atop |J_2|+|J_2'|\leq 2}\|\del_x^{J_1}\del_{b'}\Omega^{J_2'}\Omega^{I_2'}v(t,\cdot)\|_{X^2}
\leq  C\|v(t,\cdot)\|_{E_P^{d+2}}.
\endaligned
$$
By combining these three cases, \eqref{eq 1' lem 7-6-2} is established.
\end{proof}

\begin{lemma}\label{lem 7-6-3}
Let $u\in X_H^d$ and $v\in L^{\infty}([0,T];E_P^d)$, then the following estimate holds for all pair of multi-index $(I_1,I_2)$ with $1\leq |I_1|+|I_2|\leq d$:
\begin{subequations}
\begin{equation}\label{eq 1 lem 7-6-3}
\|(1+r)^{-1}\del_x^{I_1}\Omega^{I_2}u\|_{L^2(\RR^3)}\leq C\|u\|_{X_P^d}.
\end{equation}
\begin{equation}\label{eq 1' lem 7-6-3}
\|(1+r)^{-1}\del^{I_1}\Omega^{I_2}v(t,\cdot)\|_{L^2(\RR^3)}\leq C\|v(t,\cdot)\|_{E_P^d}, \quad \text{for } t\in[0,T]. 
\end{equation}
\end{subequations}
\end{lemma}

\begin{proof}
   The proof of \eqref{eq 1 lem 7-6-3} is decomposed into several cases, as follows. 

\

\noindent {\bf Case 1: $|I_1|>0$.}  In this case we suppose that $I_1 = (a_1,a_2,\cdots,a_n)$ and denote by $I_1' = (a_2,\cdots,a_n)$. Then $I_1'$ is of order $n-1\geq 0$, and we obtain 
$$
\big{\|}\del_x^{I_1}\Omega^{I_2}u\big{\|}_{L^2(\RR^3)} = \big{\|}\del_{a_1}\del_x^{I_1'}\Omega^{I_2}u\big{\|}_{L^2(\RR^3)}
\leq \big{\|}u\big{\|}_{X_P^d}.
$$

\

\noindent {\bf Case 2: $|I_1|=0, |I_2|> 0$.} In this case we suppose that $I_2 = (b_1,b_2,\cdots,b_n)$ and denote by $I_b' = (b_2,\cdots,b_n)$. Then $I_2'$ is of order $n-1\geq 0$ and we have 
$$
\aligned
\big{\|}(1+r)^{-1}\del_x^{I_1}\Omega^{I_2}u\big{\|}_{L^2(\RR^3)} 
& = \big{\|}(1+r)^{-1}\Omega_{b_1}\Omega^{I'_2}u \big{\|}_{L^2(\RR^3)}
\\
& \leq C\sum_b \big{\|}\del_b\Omega^{I'_2}u\big{\|}_{L^2(\RR^3)} \leq C\|u\|_{X_P^d}.
\endaligned
$$
By combining these two cases, \eqref{eq 1 lem 7-6-3} is proven.

   The proof of \eqref{eq 1' lem 7-6-3} is exactly the same if we replace $\del_x$ by $\del$ in the above proof.
\end{proof}

\begin{lemma}\label{lem 7-6-4}
For all function $u$ of class $X^2$, one has 
\begin{equation}
\|u\|_{L^{\infty}(\RR^3)}\leq C\sum_{|I_1|+|I_2|\leq 2}\|(1+|x|)^{-1}\del_x^{I_1}\Omega^{I_2}u\|_{L^2}.
\end{equation}
\end{lemma}

\begin{proof}
This inequality is equivalent to
\begin{equation}\label{eq pr1 lem 7-6-4}
|u(x)|\leq C \sum_{|I_1|+|I_2|\leq 2}\|(1+|x|)^{-1}\del_x^{I_1}\Omega^{I_2}u\|_{L^2}, \quad x\in \RR^3
\end{equation}
for all $u\in C_c^{\infty}(\RR^3)$. Then by regularization, this inequality is hold for all $u\in X^2$. This is proven by distinguish different $x$.

Let $\chi(\cdot)$ a $C^{\infty}$ function defined on $[0,\infty)$ with
$$
\chi(s) =
\left\{
\aligned
&1, \quad  s\leq 1/2,
\\
&0, \quad 1 \leq s < \infty.
\endaligned
\right.
$$
Then $\chi(|\cdot|)$ is a $C_c^{\infty}(\RR^3)$ function.

Now we consider  the case $|x|\leq 1/2$. We consider the function
$$
f(x) := \chi(|x|)u(x)
$$
which is in class $C_c^{\infty}(\{|x|<1\})$. Then we apply the classical Sobolev's inequality:
$$
\aligned
|u(x)| &= |f(x)|\leq C\sum_{|I|\leq 2}\|\del_x^If\|_{L^2(\{|x|<1\})} =  C\sum_{|I|\leq 2}\big\|\del_x^I\big(\chi(|x|)u(x)\big)\big\|_{L^2(\{|x|<1\})}
\\
&= C\sum_{|I_1|+|I_2|\leq |I|,\atop |I|\leq 2}\big\|\del_x^{I_1}\chi(|x|)\del_x^{I_2}u(x)\big\|_{L^2(\{|x|<1\})}\leq C\sum_{|I_2|\leq 2}\|\del_x^{I_2}u(x)\|_{L^2(\{|x|<1\})}
\\
& \leq C\sum_{|I_2|\leq 2}\big\|(1+|x|)^{-1}\del_x^{I_2}u(x)\big\|_{L^2(\{|x|<1\})}
\leq C\sum_{|I_2|\leq 2}\big\|(1+|x|)^{-1}\del_x^{I_2}u(x)\big\|_{L^2(\RR^3)},
\endaligned
$$
where we used that $\del_x^{I_1}\chi(|x|)$ are bounded.

When $|x|\geq 1/2$, we have 
$$
\aligned
|u(x)| &= \big|\big(1-\chi(2|x|)\big)u(x)\big| = \big\|(1+|x|)\cdot(1+|x|)^{-1}\big(1-\chi(2|x|)\big)u(x)\big\|_{L^\infty(\RR^3)}
\\
& \leq \big\|(1+|x|)^{-1}\big(1-\chi(2|x|)\big)u(x)\big\|_{\mathcal{E}_{-1}}
\endaligned
$$
and the by \eqref{eq 1 lem 7-6-1},
$$
\aligned
|u(x)|& \leq \big\|(1+|x|)^{-1}\big(1-\chi(2|x|)\big)u(x)\big\|_{\mathcal{E}_{-1}}
\\
& \leq C\sum_{|I_1|+|I_2|\leq 2}\big\|\del_x^{I_1}\Omega^{I_2}\big((1+|x|)^{-1}\big(1-\chi(2|x|)\big)u(x)\big)\big\|_{L^2(\RR^3)}
\\
&= C\sum_{|J_1|+|J_1'|=I_1\atop |I_1|+|I_2|\leq 2}
\big\|\del_x^{J_1'}\big((1+|x|)^{-1}\big(1-\chi(2|x|)\big)\big)\,
\\
& \hskip3.cm \del_x^{J_1}\Omega^{I_2}\big((1+|x|)^{-1}\big(1-\chi(2|x|)\big)u(x)\big)\big\|_{L^2(\RR^3)},
\endaligned
$$
where for the last equality we have used the fact that
$$
\Omega_a  \big((1+|x|)^{-1}\big(1-\chi(2|x|)\big)\big) = 0, 
$$
since both factors are radial symmetric. Then we will prove that
$$
\del^{J_1'}\big((1+|x|)^{-1}\big(1-\chi(2|x|)\big)\big)\leq C(1+|x|)^{-1}
$$
for $|J_1'|\leq 2$. This is check directly by calculating and the fact that $\big(1-\chi(2|x|)\big)$ and its derivatives are supported out of the ball $\{|x|<1/4\}$. Then the desired result is established.
\end{proof}
 

\subsection{Linear estimates}

We begin with the linear theory of wave equation with the initial data given in $X^{d+1}_H(\RR^3)\times X^d(\RR^3)$. For the simplicity of proof, we introduce the energy functional with respect to a metric $g^{\alpha\beta}$ defined in $\RR^4$ as follows:
\begin{equation}
E_g(t,u) := \bigg(\int_{\RR^3}\big(-g^{00}|\del_tu|^2 + g^{ab}\del_au\del_bu\big)(t,\cdot) dx\bigg)^{1/2}
\end{equation}
and
\begin{equation}
E_{g,c}(t,u) := \bigg(\int_{\RR^3}\big(-g^{00}|\del_tu|^2 + g^{ab}\del_au\del_bu + c^2v^2\big)(t,\cdot)  dx\bigg)^{1/2}. 
\end{equation}
A metric $g^{\alpha\beta}$ is said to be {\bf coercive with constant $C>0$} if
\begin{equation}\label{eq coe 7-7}
C^{-1}\|\nabla u\|_{L^2(\RR^3)}\leq E_g(t,u) \leq C\|\nabla u\|_{L^2(\RR^3)},
\end{equation}
where $\nabla u$ refers to the spacetime gradient of $u$. At this juncture, let us introduce a 
notation for the $C_c^{\infty}$ functions defined in the region $[0,T]\times \RR^3$:
$$
\|u(t,\cdot)\|_{X_E^{d+1}} := \sum_{|I_1|+|I_2|\leq d}\|\nabla \del_x^{I_1}\Omega^{I_2}u(t,\cdot)\|_{L^2(\RR^3)}.
$$
By \eqref{eq com lem 7-6'-1 d} and \eqref{eq com lem 7-6'-1 e}, the norm $\|\cdot\|_{X_E^{d+1}}$ and $\|\nabla(\cdot)\|_{X^d}$ are equivalent.

We also introduce 
$$
E_g^d(s,u): = \sum_{|I_1|+|I_2|\leq d}E_g(s,\del^{I_1}\Omega^{I_2}u),\quad
E_{g,c}^d(s,u): = \sum_{|I_1|+|I_2|\leq d}E_{g,c}(s,\del^{I_1}\Omega^{I_2}u).
$$
Then when the coercivity condition \eqref{eq coe 7-7} is assumed,
\begin{equation}\label{eq coek0 7-7}
C^{-1}\|u(t,\cdot)\|_{E_P^{d+1}}\leq E_g^k(s,u)\leq C\|u(t,\cdot)\|_{E_P^{d+1}}.
\end{equation}
 
The existence result in the next section is based on the following linear estimate.

\begin{lemma}[$L^\infty$ type estimate for wave equation]\label{lem 7-7-1}
Let $u$ be a smooth function defined in the region $[0,T]$ and let $F = -\Box u$, then for any $0\leq  t\leq T$
\begin{equation}\label{eq 1 lem 7-7-1}
\|u(t,\cdot)\|_{\mathcal{E}_{-1}(\RR^3)}
\leq Ct\int_0^t\|F(s,\cdot)\|_{\mathcal{E}_{-1}(\RR^3)}ds
+ C\big(\|{u(0,x)}\|_{\mathcal{E}_{-1}(\RR^3)} + t \|\nabla u(0,x)\|_{\mathcal{E}_{-1}(\RR^3)}\big),
\end{equation}
where $\nabla u$ refers to the spacetime gradient of $u$.
\end{lemma}

\begin{proof}
This estimate based on the explicit expression of the linear wave equation. We consider the Cauchy problem:
\begin{equation}
\aligned
&\Box u = -F, 
\\
&u(0,x) = -f(x),\quad \del_t u(0,x) = -g(x). 
\endaligned
\end{equation}
Then, $u$ can be expressed by
$$
\aligned
u(t,x) &= \frac{1}{4\pi}\int_0^t \frac{1}{t-s}\int_{|y| = t-s}F(s,x-y)d\sigma(y)ds
\\
&+ \frac{1}{4\pi t}\int_{|y| = t}g(x-y)d\sigma(y)
+ \frac{1}{4\pi t^2}\int_{|y| = t}\big(f(x-y) - <\del_xf (x-y),y>\big) d\sigma(y)
\\
:&= u_1(t,x) + u_2(t,x) + u_3(t,x).
\endaligned
$$
Here, $d\sigma(y)$ refers to the standard Lebesgue measure on the sphere $|y| = t-s$ or $|y| = t$, and $<\cdot,\cdot>$ refers to the standard scalar product on $\RR^3$. The notation $\del_xf$ stands for the (spacial) gradient of $f$. The remained work is to estimate $u_i$ with $i=1,2,3$.

When $|x|\leq 1$, we make use of the fact that $\|\cdot\|_{L^{\infty}(\RR^3)}\leq \|\cdot\|_{\mathcal{E}_{-1}}$. Then, we have 
$$
\aligned
|u_1(t,x)|
& \leq \frac{1}{4\pi}\int_0^t \frac{\|F(s,\cdot)\|_{\mathcal{E}_{-1}}}{t-s}\int_{|y| = t-s}d\sigma(y)ds
\\
& \leq \int_0^t \|F(s,\cdot)\|_{\mathcal{E}_{-1}}(t-s)ds \leq t\int_0^t \|F(s,\cdot)\|_{\mathcal{E}_{-1}}.
\endaligned
$$
The other terms $u_2$ and $u_3$ are estimated similarly. Then \eqref{eq 1 lem 7-7-1} is proven in this case.

When $|x|>1$, we need to establish the decay estimate of the solution at spatial infinity. We begin with $u_1$.
$$
\aligned
|u_1(t,x)| & \leq \frac{1}{4\pi}\int_0^t\frac{1}{t-s}\int_{|y| = t-s}\frac{\|F(s,\cdot)\|_{\mathcal{E}_{-1}}}{1+|x-y|}d\sigma(y)ds
\\
&= \frac{1}{4\pi}\int_0^t\frac{\|F(s,\cdot)\|_{\mathcal{E}_{-1}}}{t-s}\int_{|y| = t-s}(1+|x-y|)^{-1}d\sigma(y)ds. 
\endaligned
$$
We focus on the expression $\int_{|y| = t-s}(1+|x-y|)^{-1}d\sigma(y)$. We make the following parametrization of the sphere $\{|y| = t-s\}$. Let $\theta$ be the angle from the vector $-x$ to the vector $-y$, and $\varphi$ refers the angle from the plan determined by pair of vector $(x,y)$ to a fixed plan containing $x$ (for example the plan determined by $x$ and $(0,0,1)$). With this parametrization, the volume form $\sigma(y)$ has the following expression:
$$
d\sigma(y) = (t-s)^2\sin\theta d\theta d\varphi. 
$$
Also,  by the classical trigonometrical theorem ``Law of cosines'',
$$
|x-y|^2 = |x|^2 + |t-s|^2 - 2 |x|(t-s)\cos\theta.
$$
Then, we obtain 
$$
\aligned
\int_{|y| = t-s}(1+|x-y|)^{-1}d\sigma(y)
&=  \int_0^{2\pi} \int_0^\pi \frac{(t-s)^2\sin\theta d\theta d\varphi}{1+\big(|x|^2 + |t-s|^2 - 2|x|(t-s)\cos \theta\big)^{1/2}}
\\
&= 2\pi\int_{-1}^1\frac{(t-s)^2d\gamma}{1+\big(|x|^2 + |t-s|^2 - 2|x|(t-s)\gamma\big)^{1/2}},
\endaligned
$$
where $\gamma := \cos\theta$. Then, we have 
$$
\aligned
\int_{|y| = t-s}(1+|x-y|)^{-1}d\sigma(y)
&=\frac{2\pi(t-s)}{|x|}\int_{-1}^1\frac{|x|(t-s)d\gamma}{1+\big(|x|^2 + |t-s|^2 - 2|x|(t-s)\gamma\big)^{1/2}}
\\
&=\frac{2\pi(t-s)}{|x|}\int_{||x| - (t-s)|}^{|x|+(t-s)}\frac{\tau d\tau}{1+ \tau},
\endaligned
$$
where $\tau := \big(|x|^2 + |t-s|^2 - 2|x|(t-s)\gamma\big)^{1/2}$. Then, we obtain 
$$
\aligned
|u_1(t,x)|& \leq \frac{1}{4\pi}\int_0^t\frac{\|F(s,\cdot)\|_{\mathcal{B}^-1}}{t-s}\int_{|y| = t-s}(1+|x-y|)^{-1}d\sigma(y)ds
\\
& \leq \frac{1}{2|x|}\int_0^t\|F(s,\cdot)\|_{\mathcal{E}_{-1}}\int_{||x| - (t-s)|}^{|x|+(t-s)}\frac{\tau d\tau}{1+ \tau}ds.
\endaligned
$$
Now the discussion should distinguish between two cases. 

\

\noindent {\bf Case 1: $|x|>t$.} In this case $|x|>t-s$ always holds, and we find 
$$
\aligned
|u_1(t,x)|& \leq \frac{1}{2|x|}\int_0^t\|F(s,\cdot)\|_{\mathcal{E}_{-1}}\int_{||x| - (t-s)|}^{|x|+(t-s)}\frac{\tau d\tau}{1+ \tau}ds
\\
&= \frac{1}{2|x|}\int_0^t\|F(s,\cdot)\|_{\mathcal{E}_{-1}}\int_{|x| - (t-s)}^{|x|+(t-s)}\frac{\tau d\tau}{1+ \tau}ds
\\
& \leq \frac{1}{2|x|}\int_0^t\|F(s,\cdot)\|_{\mathcal{E}_{-1}}\int_{|x| - (t-s)}^{|x|+(t-s)}d\tau ds
\\
&=\frac{1}{|x|}\int_0^t(t-s)\|F(s,\cdot)\|_{\mathcal{E}_{-1}}ds \leq \frac{t}{|x|}\int_0^t\|F(s,\cdot)\|_{\mathcal{E}_{-1}}ds, 
\endaligned
$$
which leads to
$$
|x||u(t,x)|\leq t\int_0^t\|F(s,\cdot)\|_{\mathcal{E}_{-1}}ds.
$$

\

\noindent {\bf Case 2: $|x|\leq t$.} In this case we need a more precise calculation: 
$$
\aligned
|u_1(t,x)|& \leq \frac{1}{2|x|}\int_0^t\|F(s,\cdot)\|_{\mathcal{E}_{-1}}\int_{||x| - (t-s)|}^{|x|+(t-s)}\frac{\tau d\tau}{1+ \tau}ds
\\
&= \frac{1}{2|x|}\int_0^{t-|x|}\|F(s,\cdot)\|_{\mathcal{E}_{-1}}\int_{(t-s)-|x|}^{|x|+(t-s)}\frac{\tau d\tau}{1+ \tau}ds
\\
& \qquad 
+ \frac{1}{2|x|}\int_{t-|x|}^t\|F(s,\cdot)\|_{\mathcal{E}_{-1}}\int_{|x|-(t-s)}^{|x|+(t-s)}\frac{\tau d\tau}{1+ \tau}ds
\\
& \leq \frac{1}{2|x|}\int_0^{t-|x|}2|x|\|F(s,\cdot)\|_{\mathcal{E}_{-1}}ds
    + \frac{1}{2|x|}\int_{t-|x|}^t2(t-s)\|F(s,\cdot)\|_{\mathcal{E}_{-1}}ds
\\
& \leq \int_0^{t-|x|}\|F(s,\cdot)\|_{\mathcal{E}_{-1}}ds + \int_{t-|x|}^t\|F(s,\cdot)\|_{\mathcal{E}_{-1}}ds
= \int_0^t\|F(s,\cdot)\|_{\mathcal{E}_{-1}}ds. 
\endaligned
$$
Recalling that $|x|\leq t\leq T$, we find 
$$
|u_1(t,x)|\leq \int_0^t\|F(s,\cdot)\|_{\mathcal{E}_{-1}}ds \leq \frac{t}{|x|}\int_0^t\|F(s,\cdot)\|_{\mathcal{E}_{-1}}ds, 
$$
so that
$$
|x||u_1(t,x)|\leq t\int_0^t\|F(s,\cdot)\|_{\mathcal{E}_{-1}}ds.
$$

The estimate of $u_2$ and $u_3$ are similar: 
$$
\aligned
|u_2(t,x)|& \leq \frac{1}{4\pi t}\int_{|y| = t}|g(x-y)|d\sigma(y)
\\
& \leq  \frac{\|g\|_{\mathcal{E}_{-1}}}{4\pi t}\int_{|y| = t}(|x-y|+1)^{-1}d\sigma(y). 
\endaligned
$$
By the same parametrization made in the estimate of $u_1$ and similar calculation, we can conclude with
$$
|x||u_2(t,x)|\leq t\|g\|_{\mathcal{E}_{-1}}.
$$

In the same way, we have 
$$
\aligned
|u_3(t,x)| & \leq \frac{1}{4\pi t^2}\int_{|y|=t}\frac{\|f\|_{\mathcal{E}_{-1}}d\sigma(y)}{|x-y|+1}
 + \frac{1}{4\pi t}\int_{|y|=t}\frac{\|\del_x f\|_{\mathcal{E}_{-1}}d\sigma(y)}{|x-y|+1}
\\
& \leq  \frac{1}{|x|}\|f\|_{\mathcal{E}_{-1}} + \frac{t}{|x|}\|\del_x f\|_{\mathcal{E}_{-1}}.
\endaligned
$$
By combining the estimates made in $|x|\leq 1$ and $|x|>1$, the desired result is established.
\end{proof}

\begin{lemma}[$L^2$ type estimate for wave equation]\label{lem 7-7-2}
Let $u$ be a smooth function defined in the region $[0, T]\times \RR^3$ and $g^{\alpha\beta} = m^{\alpha\beta} + H^{\alpha\beta}$ be a smooth metric, where $m^{\alpha\beta}$ is the Minkowski metric with signature $(-,+,+,+)$. Assume that $g^{\alpha\beta}$ is coercive with constant $C$. Let $f = - g^{\alpha\beta}\del_{\alpha}\del_{\beta} u$, then for any $0\leq t\leq T$,
\begin{equation}\label{eq 1 lem 7-7-2}
\frac{d}{dt}E_g(t,u) \leq C\|f(t,\cdot)\|_{L^2(\RR^3)} + C\sum_{\alpha,\beta}\|\nabla g^{\alpha\beta}(t,\cdot)\|_{L^{\infty}}E_g(t,u),
\end{equation}
where $\nabla u$ refers to the spacetime gradient of $u$.
\end{lemma}

\begin{proof}
This is a standard calculation and we write 
$$
\aligned
\del_t u g^{\alpha\beta}\del_{\alpha}\del_{\beta} u
&=\frac{1}{2}\del_0\big(g^{00}(\del_0 u)^2 - g^{ab}\del_a u\del_b u\big) + \del_a\big(g^{a\beta}\del_t u \del_{\beta} u\big)
\\
&\quad 
+ \frac{1}{2}\del_tg^{\alpha\beta}\del_{\alpha}u\del_{\beta}u - \del_{\alpha}g^{\alpha\beta}\del_t u \del_{\beta}u. 
\endaligned
$$
Integrating on the slice $\{t=\tau\}$ and applying Stokes' formula, we find 
$$
\aligned
&\frac{1}{2}\int_{\RR^3}\del_0\big(-g^{00}(\del_t u)^2 + g^{ab}\del_au\del_bu \big)dx
\\
& = \int_{\RR^3}\del_tufdx
+ \frac{1}{2}\int_{\RR^3}\big(\del_tg^{\alpha\beta}\del_{\alpha}u\del_{\beta}u - 2\del_{\alpha}g^{\alpha\beta}\del_t u \del_{\beta}u\big)dx, 
\endaligned
$$
which leads to (by the coercivity condition)
$$
\aligned
\frac{1}{2}\frac{d}{dt}\big(E_g(t,u)\big)^2
= \int_{\RR^3}\del_tu fdx + \frac{1}{2}\int_{\RR^3}\big(\del_tg^{\alpha\beta}\del_{\alpha}u\del_{\beta}u - 2\del_{\alpha}g^{\alpha\beta}\del_t u \del_{\beta}u\big)dx, 
\endaligned
$$
which is
$$
\aligned
E_g(t,u)\frac{d}{dt}E_g(t,u)
&= \int_{\RR^3}\del_tufdx
+ \frac{1}{2}\int_{\RR^3}\big(\del_tg^{\alpha\beta}\del_{\alpha}u\del_{\beta}u - 2\del_{\alpha}g^{\alpha\beta}\del_t u \del_{\beta}u\big)dx.
\endaligned
$$
So we have 
$$
\aligned
E_g(t,u)\frac{d}{dt}E_g(t,u)
\leq \|\del_t u\|_{L^2(\RR^3)}\|f\|_{L^2(\RR^3)} + C\|\nabla g^{\alpha\beta}\|_{L^{\infty(\RR^3)}}\|\nabla u\|^2_{L^2(\RR^3)}
\endaligned
$$
and by recalling \eqref{eq coe 7-7}
$$
\frac{d}{dt}E_g(t,u)
\leq C\|f\|_{L^2(\RR^3)} + C\|\nabla g^{\alpha\beta}\|_{L^{\infty(\RR^3)}}E_g(t,u).
$$
\end{proof}

Now we combine Lemmas~\ref{lem 7-7-1} and \ref{lem 7-7-2} with the Sobolev estimate \eqref{eq 1 lem 7-6-1} in order to get a estimate on the $E^{d+1}_H$ norm.   

\begin{lemma}[$E_H^{d}$ norm estimate on wave equation]\label{lem 7-7-3}
Let $u$ be a smooth function defined in the region $[0,T]\times \RR^3$. Let $g^{\alpha\beta}$ be a smooth metric, $g^{\alpha\beta} = m^{\alpha\beta} + H^{\alpha\beta}$ with $m^{\alpha\beta}$ the standard Minkowski metric. Assume that $g$ satisfies the following coercivity condition with a constant $C>0$ and $|H^{00}|\leq 1/2$. 
Let $g^{\alpha\beta}\del_{\alpha\beta} u = -F$. Then the following estimates hold for $d\geq 3$:
\begin{subequations}\label{eq 1 lem 7-7-3}
\begin{equation}\label{eq 1 lem 7-7-3 a}
\aligned
\|u(t,\cdot)\|_{E_P^{d+1}} & \leq C\|u(0,\cdot)\|_{E_P^{d+1}} e^{C\int_0^tD_d(\tau)d\tau} + C\int_0^t\|F(\tau,\cdot)\|_{E^d}\,e^{C\int_\tau^tD_d(s)ds} d\tau,
\endaligned
\end{equation}
\end{subequations}
\begin{equation}\label{eq 2 lem 7-7-3}
\aligned
\|u(t,\cdot)\|_{\mathcal{E}_{-1}}
& \leq Ct\int_0^t\|F(s,\cdot)\|_{\mathcal{E}_{-1}}ds
+ CtE^3_g(0,u)\sum_{\alpha,\beta}\int_0^t \|H^{\alpha\beta}(s,\cdot)\|_{L^{\infty}}e^{C\int_0^sD_3(\tau)d\tau}ds
\\
&+Ct\sum_{\alpha,\beta}\int_0^t\|H^{\alpha\beta}(s,\cdot)\|_{L^{\infty}}\int_0^s\|F(\tau,\cdot)\|_{E^3}
e^{C\int_\tau^sD_3(\lambda)d\lambda}ds + CtE_g^3(0,u),
\endaligned
\end{equation}
where
$$
D_k(t):= \max_{\alpha,\beta}\|H^{\alpha\beta}(t,\cdot)\|_{E_H^{k+1}}.
$$
\end{lemma}

\begin{proof}
We begin with the estimate on $E^{d}$ norm of $\nabla u$. We derive the equation with respect to $\del^{I_1}\Omega^{I_2}$, where $|I_1|+|I_2|=l $. Remember that this product of operator commute with the linear wave operator, Then
$$
-\del^{I_1}\Omega^{I_2} F = \Box\del^{I_1}\Omega^{I_2} u + H^{\alpha\beta}\del_{\alpha}\del_{\beta}\del^{I_1}\Omega^{I_2}u + [\del^{I_1}\Omega^{I_2},H^{\alpha\beta}\del_{\alpha}\del_{\beta}]u
$$
which is
$$
g^{\alpha\beta}\del_{\alpha}\del_{\beta}\del^{I_1}\Omega^{I_2}u = -F - [\del^{I_1}\Omega^{I_2},H^{\alpha\beta}\del_{\alpha}\del_{\beta}]u.
$$
Apply \eqref{eq 1 lem 7-7-2},
\begin{equation}\label{eq pr2 lem 7-7-3}
\aligned
\frac{d}{dt}E_g(t,\del^{I_1}\Omega^{I_2}u)
& \leq C\|\del^{I_1}\Omega^{I_2}F(t,\cdot)\|_{L^2(\RR^3)}
+ C\|[\del^{I_1}\Omega^{I_2},H^{\alpha\beta}\del_{\alpha}\del_{\beta}]u(t,\cdot)\|_{L^2(\RR^3)}
\\
&+ C\sum_{\alpha,\beta}\|\nabla H^{\alpha\beta}(t,\cdot)\|_{L^{\infty}(\RR^3)}E_g(t,\del^{I_1}\Omega^{I_2}u).
\endaligned
\end{equation}
We should estimate the second term in right-hand-side. By \eqref{eq com lem 7-6'-1 c},
\begin{equation}\label{eq pr3 lem 7-7-3}
\aligned
\big{\|}[\del^{I_1}\Omega^{I_2},H^{\alpha\beta}\del_{\alpha}\del_{\beta}]u\big{\|}_{L^2(\RR^3)}
& \leq \sum_{{J_1+J_1'=I_1\atop J_2+J_2'=I_2}\atop |J_1'|+|J_2'|>0}
\|\del^{J_1'}\Omega^{J_2'}H^{\alpha\beta}\, \del_{\alpha}\del_{\beta}\del^{J_1}\Omega^{J_2}u\|_{L^2(\RR^3)}
\\
&+\sum_{J_1+J_1'=I_1\atop {|J_2|+|J_2'|<|I_2|\atop \alpha,\beta,\alpha',a}}
\|\del^{J_1'}\Omega^{J_2'}H^{\alpha\beta}\,\del_{\alpha'}\del_a\del^{J_1}\Omega^{J_2}u\|_{L^2(\RR^3)}
\\
=:& T_1 + T_2.
\endaligned
\end{equation}
Here, we make the convention that when $l\leq 0$, $[\del^{I_1}\Omega^{I_1},\del_{\alpha}\del_{\beta}] = 0$. We see that both terms can be bounded by $CD_d\,E^d(s,u)$:
For $T_1$, when $|J_1'|+|J_2'|\leq d-1$, Then by \eqref{eq 1' lem 7-6-2},
$$
\aligned
\|\del^{J_1'}\Omega^{J_2'}H^{\alpha\beta}\, \del_{\alpha}\del_{\beta}\del^{J_1}\Omega^{J_2}u\|_{L^2(\RR^3)}
& \leq \|\del^{J_1'}\Omega^{J_2'}H^{\alpha\beta}\|_{L^\infty(\RR^3)}
\|\del_{\alpha}\del_{\beta}\del^{J_1}\Omega^{J_2}u\|_{L^2(\RR^3)}
\\
& \leq  C\|H^{\alpha\beta}\|_{E_H^{d+1}}\|u\|_{E_P^{d+1}}.
\endaligned
$$
\\
When $|J_1'|+|J_2'|\geq d$, then $J_1 = J_2 = 0$. Then by \eqref{eq 1' lem 7-6-3} and \eqref{eq 1 lem 7-6-1},
$$
\aligned
\|\del^{J_1'}\Omega^{J_2'}H^{\alpha\beta}\, \del_{\alpha}\del_{\beta}\del^{J_1}\Omega^{J_2}u\|_{L^2(\RR^3)}
& \leq \|(1+r)^{-1}\del^{J_1'}\Omega^{J_2'}H^{\alpha\beta}\|_{L^2(\RR^3)}
\|(1+r)\del_{\alpha}\del_{\beta}u\|_{L^\infty(\RR^3)}
\\
& \leq C\|H^{\alpha\beta}\|_{E_H^{d+1}} \|\del_{\alpha}\del_{\beta}u\|_{E^2}\leq C\|H^{\alpha\beta}\|_{E_H^{d+1}}\|u\|_{E_P^{4}}.
\endaligned
$$
The term $T_2$ is bounded in the same manner and we omit the details.

Combining this with \eqref{eq pr2 lem 7-7-3}, we find 
\begin{equation}\label{eq pr4 lem 7-7-3}
\aligned
\frac{d}{dt}E_g(t,\del^{I_1}\Omega^{I_2}u)
& \leq C\|\del^{I_1}\Omega^{I_2}F(t,\cdot)\|_{L^2(\RR^3)}
+ CD_d E^d_g(t,u)
\\
&+ C\sum_{\alpha,\beta}\|\nabla H^{\alpha\beta}(t,\cdot)\|_{L^{\infty}(\RR^3)}E_g(t,\del^{I_1}\Omega^{I_2}u)
\endaligned
\end{equation}
Taking the sum with respect to $(I_1,I_2)$ for $|I_1|+|I_2|\leq d$, then
\begin{equation}\label{eq pr5 lem 7-7-3}
\frac{d}{dt}E_g^d(t,u)\leq C\|F(t,\cdot)\|_{E^d} + CD_d E^d_g(t,u).
\end{equation}
Integrating \eqref{eq pr5 lem 7-7-3}, we obtain 
\begin{equation}\label{eq pr6 lem 7-7-3}
E_g^d(t,u) \leq E_g^d(0,u)e^{C\int_0^tD_d(\tau)d\tau} + C\int_0^t \|F(\tau,\cdot)\|_{E^d}\, e^{C\int_\tau^tD_d(s)ds}d\tau,
\end{equation}
which leads to (by \eqref{eq coek0 7-7}) \eqref{eq 1 lem 7-7-3 a}.

We now turns to estimate the $\mathcal{E}_{-1}$ norm. We can easily deduce that
\begin{equation}\label{eq pr7 lem 7-7-3}
\Box u = -(1-H^{00})^{-1}\big(H^{ab}+H^{00}m^{ab}\big)\del_a\del_b u - 2(1-H^{00})^{-1}H^{a0}\del_t\del_a u - (1-H^{00})^{-1}F. 
\end{equation}
Here, we use the assumption that $H^{00}\leq 1/2$ to make sure that $(1-H^{00})^{-1}$ is well defined. By Lemma~\ref{lem 7-7-1}, we have 
$$
\aligned
\|u\|_{\mathcal{E}_{-1}}
& \leq Ct\int_0^t\|(1-H^{00})^{-1}F(s,\cdot)\|_{\mathcal{E}_{-1}}ds
\\
&+ Ct\int_0^t\|(1-H^{00})^{-1}H^{ab}\del_{a}\del_{b}u\|_{\mathcal{E}_{-1}}ds
+ Ct\int_0^t\|(1-H^{00})^{-1}H^{a0}\del_{a}\del_t u\|_{\mathcal{E}_{-1}}ds
\\
&+C\big(\|{u(0,x)}\|_{\mathcal{E}_{-1}(\RR^3)} + t \|\nabla u(0,x)\|_{\mathcal{E}_{-1}(\RR^3)}\big)
\\
& \leq Ct\int_0^t\|F(s,\cdot)\|_{\mathcal{E}_{-1}}ds +  Ct\int_0^t\|H^{ab}\del_{a}\del_{b}u\|_{\mathcal{E}_{-1}}ds
+ Ct\int_0^t\|H^{a0}\del_{a}\del_t u\|_{\mathcal{E}_{-1}}ds
\\
&+C\big(\|{u(0,x)}\|_{\mathcal{E}_{-1}(\RR^3)} + t \|\nabla u(0,x)\|_{\mathcal{E}_{-1}(\RR^3)}\big). 
\endaligned
$$
We observe that the following estimates are guaranteed by \eqref{eq 1 lem 7-6-1}: 
$$
\aligned
\|H^{a\beta}\del_a\del_{\beta}u\|_{\mathcal{E}_{-1}}
&\leq \|H^{a\beta}\|_{L^{\infty}}\, \|\del_a\del_{\beta}u\|_{\mathcal{E}_{-1}}
\\
&\leq \|H^{\alpha\beta}\|_{L^{\infty}}\,\|\del_a\del_{\beta}u\|_{E^2}
\leq C\sum_{a,\beta}\|H^{a\beta}\|_{L^{\infty}}E_g^3(s,u)
\endaligned
$$
and
$$
\|\nabla u(0,\cdot)\|_{\mathcal{E}_{-1}}\leq C\|\nabla u(0,\cdot)\|_{E^2}. 
$$
By combining these two estimate, we get 
\begin{equation}\label{eq pr1 lem 7-7-3}
\aligned
\|u(t,\cdot)\|_{\mathcal{E}_{-1}}
& \leq  Ct\int_0^t\|F(s,\cdot)\|_{\mathcal{E}_{-1}}ds
+ Ct\sum_{\alpha,\beta}\int_0^t\|H^{\alpha\beta}(s,\cdot)\|_{L^{\infty}}E_g^3(s,u)ds
+CtE_g^2(0,u).
\endaligned
\end{equation}

By combining \eqref{eq pr6 lem 7-7-3} and \eqref{eq pr1 lem 7-7-3}, we obtain 
\begin{equation}\label{eq pr7 lem 7-7-3-deux}
\aligned
\|u(t,\cdot)\|_{\mathcal{E}_{-1}}
& \leq Ct\int_0^t\|F(s,\cdot)\|_{\mathcal{E}_{-1}}ds
+ CtE^3_g(0,u)\sum_{a,\beta}\int_0^t \|H^{a\beta}(s,\cdot)\|_{L^{\infty}}e^{C\int_0^sD_3(\tau)d\tau}ds
\\
&\quad +Ct\sum_{a,\beta}\int_0^t\|H^{a\beta}(s,\cdot)\|_{L^{\infty}}\int_0^s\|F(\tau,\cdot)\|_{E_R^3}
e^{C\int_\tau^sD_3(\lambda)d\lambda}ds + CtE_g^3(0,u).
\endaligned
\end{equation}
\end{proof}

Furthermore, the following $L^2$ estimate for Klein-Gordon equations is essential in our analysis. 

\begin{lemma}[$L^2$-type estimate for KG equations]\label{lem 7-7-4}
Let $v$ be in a smooth function defined $[0,T]\times \RR^3$ and let $F = c^2v - g^{\alpha\beta}\del_{\alpha}\del_{\beta} v$, $c>0$. Suppose that 
$g^{\alpha\beta} = m^{\alpha\beta} + H^{\alpha\beta}$ satisfies the coercivity condition with a constant $C$, i.e. \eqref{eq coe 7-7}. Then, the following estimate holds for $0\leq t<T$:
\begin{subequations}\label{eq 1 lem 7-7-4}
\begin{equation}\label{eq 1 lem 7-7-4 a}
\aligned
c\|v(t,\cdot)\|_{E^d} + \|v(t,\cdot)\|_{E_P^{d+1}}
& \leq C\big(\|v(0,\cdot)\|_{E_P^{d+1}} + c\|v(0,\cdot)\|_{E^d}\big)e^{C\int_0^tD_d(\tau)d\tau}
\\
&+ C\int_0^t\|F(s,\cdot)\|_{E^d}\, e^{C\int_s^t D_d(\tau)}d\tau ds.
\endaligned
\end{equation} 
\end{subequations}
\end{lemma}

\begin{proof}
The proof is essentially the same to that of Lemma~\ref{lem 7-7-1}. The only difference comes from the potential term:
$$
\aligned
&\del_t v \big(g^{\alpha\beta}\del_{\alpha}\del_{\beta}v - c^2v\big)
\\
&= \frac{1}{2}\del_0\big(g^{00}(\del_0 v)^2 - g^{ab}\del_a v\del_b v\big) - \frac{1}{2}(c\del_0v)^2 + \del_a\big(g^{a\beta}\del_t v\del_{\beta} v\big)
+ \frac{1}{2}\del_tg^{\alpha\beta}\del_{\alpha}v\del_{\beta}v - \del_{\alpha}g^{\alpha\beta}\del_t v \del_{\beta}v.
\endaligned
$$
Then the same calculation of the proof in \eqref{lem 7-7-1} leads to
\begin{equation}\label{eq pr1 lem 7-7-4}
\frac{d}{dt}E_{g,c}(t,v)
 \leq C\|F(t,\cdot)\|_{L^2(\RR^3)}
+ C\sum_{\alpha,\beta}\|\nabla H^{\alpha\beta}(t,\cdot)\|_{L^{\infty}(\RR^3)}\,E_{g,c}(t,v).
\end{equation}

Now we derive the equation with respect to $\del^{I_1}\Omega^{I_2}$, and perform the same calculation as we done in the proof of Lemma~\ref{lem 7-7-3}, the we arrive at:
\begin{equation}
E_{g,c}^d(t,u) \leq E_{g,c}^d(0,u)e^{C\int_0^tD_d(\tau)d\tau} + C\int_0^t \|F(\tau,\cdot)\|_{E^d}\, e^{C\int_\tau^tD_d(s)ds}d\tau.
\end{equation}
Then combined with the expression of $E_{g,c}$ and the coercivity condition \eqref{eq coe 7-7}, the desired result is proven.
\end{proof}

At the end of this subsection, we establish the following estimate on second order time derivative of the solution.

\begin{lemma}\label{lem 7-7-5}
Let $u$ be a smooth function defined in $\RR^4$ and suppose that $u$ satisfies the following wave/Klein-Gordon equation:
$$
g^{\alpha\beta}\del_{\alpha}\del_{\beta}u - c^2 u = - F,
$$
where $c\geq 0$. Suppose that $g^{\alpha\beta} = m^{\alpha\beta} + H^{\alpha\beta}$ with $\big|H^{00}\big|\leq 1/2$. Then the following estimate hold for all pair of multi-index $(I_1,I_2)$
\begin{equation}\label{eq 1 lem 7-7-5}
\aligned
&\big\|\del_x^{I_1}\Omega^{I_2}\del_t\del_t u\big\|_{L^2(\RR^3)}
\\
& \leq
\big\|\del_x^{I_1}\Omega^{I_2}\big((1-H^{00})^{-1}(m^{ab}+H^{ab})\del_a\del_b u\big)\big\|_{L^2(\RR^3)}
+c^2\big\|\del_x^{I_1}\Omega^{I_2}\big((1-H^{00})^{-1}u\big)\big\|_{L^2(\RR^3)}
\\
&+ 2\big\|\del_x^{I_1}\Omega^{I_2}\big((1-H^{00})^{-1}H^{0a}\del_t\del_a u\big)\big\|_{L^2(\RR^3)}
+\big\|\del_x^{I_1}\Omega^{I_2}\big((1-H^{00})^{-1}F\big)\big\|_{L^2(\RR^3)}.
\endaligned
\end{equation}
\end{lemma}

\begin{proof} By decomposing the wave operator, we have 
$$
g^{\alpha\beta}\del_{\alpha}\del_{\beta} u = \big(-1+H^{00}\big)\del_t\del_t u + 2H^{0a}\del_t\del_a u + \big(m^{ab}+H^{ab}\big)\del_a\del_b u
$$
and thanks to the equation
$$
\big(-1+H^{00}\big)\del_t\del_t u + 2H^{0a}\del_t\del_a u + \big(m^{ab}+H^{ab}\big)\del_a\del_b u - c^2u = -F,
$$
we have
$$
\del_t\del_t u
= -\frac{c^2u}{1-H^{00}} + \frac{2H^{0a}\del_t\del_a u}{1-H^{00}} + \frac{\big(m^{ab}+H^{ab}\big)\del_a\del_b u}{1-H^{00}}
+ \frac{F}{1-H^{00}}. 
$$ 
\end{proof}


\subsection{Existence results for linear equations}

We now establish the existence theory for linear wave and Klein-Gordon equations with initial data in the corresponding functional spaces defined in subsection \ref{subsec 7-fs}. We begin with the wave equation.

\begin{proposition}[Existence of linear wave equation in $E_H^{d+1}$]\label{prop 7-7-1}
Let $d\geq 3$ be an integer. Assume that $F\in L^1([0,T];E^d)$, $(u_0,u_1)\in E_H^{d+1}\times E^d$. Assume that $g^{\alpha\beta}$ is a smooth metric defined on $[0,T]\times \RR^3$ and coercive with constant $C>0$,  and $H^{\alpha\beta} = g^{\alpha\beta} - m^{\alpha\beta}$ is in the class $C([0,T];E_H^{d+1})$ with $|H^{00}|<1/2$. Then the following Cauchy problem
\begin{equation}\label{eq 7-7-cauchy wave}
\aligned
&g^{\alpha\beta}\del_{\alpha}\del_{\beta} u = -F,
\\
& u(0,x) = u_0(x),\quad \del_t u(0,x) = u_1(x)
\endaligned
\end{equation}
has a unique solution in class $C([0,T];E^{d+1}_H)$ with $\del_t^k u\in C([0,T];E^{d+1-k})$, $1\leq k\leq d$. Furthermore, this solution satisfies the following estimate:
\begin{subequations}\label{eq 1 prop 7-7-1}
\begin{equation}\label{eq 1 prop 7-7-1 b}
\|u(t,\cdot)\|_{E_P^{d+1}}  \leq C\|\nabla u(0,\cdot)\|_{E^d} e^{C\int_0^tD_d(\tau)d\tau}
 + C\int_0^t\|F(\tau,\cdot)\|_{E^d}\,e^{C\int_\tau^tD_d(s)ds} d\tau,
\end{equation}
\end{subequations}
\begin{equation}\label{eq 2 prop 7-7-1}
\aligned
\|u(t,\cdot)\|_{\mathcal{E}_{-1}}
& \leq Ct\int_0^t\|F(s,\cdot)\|_{\mathcal{E}_{-1}}ds
+ CtE^3_g(0,u)\sum_{\alpha,\beta}\int_0^t \|H^{\alpha\beta}(s,\cdot)\|_{L^{\infty}}e^{\int_0^sD_3(\tau)d\tau}ds
\\
&+Ct\sum_{\alpha,\beta}\int_0^t\|H^{\alpha\beta}(s,\cdot)\|_{L^{\infty}}\int_0^s\|F(\tau,\cdot)\|_{E^3}
e^{\int_\tau^sD_3(\lambda)d\lambda}ds + CtE_g^3(0,u),
\endaligned
\end{equation}
where
$$
D_k(t):= \max_{\alpha,\beta}\|H^{\alpha\beta}(t,\cdot)\|_{E_H^{k+1}}.
$$
\end{proposition}

\begin{proof}[Proof of Proposition \ref{prop 7-7-1}]
The uniqueness is direct by applying Lemma~\ref{lem 7-7-2}.

The existence is based on the regularization and the estimate proved in Lemma~\ref{lem 7-7-3}.
We proceed by make a series of triple $(u_0^n,u_1^n,F^n)$ which converges to $(u_0,u_1,F)$ in the following sense:
$$
\aligned
&\lim_{n\rightarrow \infty}\|u_0^n -u_0\|_{E^{d+1}_H} = 0,\quad \lim_{n\rightarrow \infty}\|u_1^n-u_0\|_{E^d} = 0
\\
&\lim_{n \rightarrow\infty}\|F^n - F\|_{L^1([0,T];E^{d})} =0,\quad \lim_{n \rightarrow\infty}\|H^{\alpha\beta}_n - H^{\alpha\beta}\|_{L^\infty([0,T];E^{d})} =0,
\endaligned
$$
where $u^n_0,u^n_1$ are $C^\infty_c(\RR^3)$ functions and for all $t\in [0,T]$, $F^n(t,\cdot)\in C^\infty_c(\RR^3)$.

By classical existence theorem of linear wave equation (see for example \cite{Taylor11}), fix the time interval $[0,T]$, each triple $(u_0^n,u_1^n,F^n)$ determines a unique smooth solution by \eqref{eq 7-7-cauchy wave}. These solution, denoted by $u^n$, formes a series.

Now we take the difference of the equation satisfied by $u^n$ and $u^{n-1}$:
$$
g_n^{\alpha\beta}\del_{\alpha}\del_{\beta}(u^n - u^{n-1}) = (g^{\alpha\beta}_{n-1} - g^{\alpha\beta}_n)\del_{\alpha}\del_{\beta}u^{n-1} + (F^n - F^{n-1}).
$$
The apply to this equation the estimate \eqref{eq 2 lem 7-7-3}, we see that the sequence $\{u^n\}$ converges with respect the norm $L^\infty([0,T];\mathcal{E}_{-1})$.

By estimate \eqref{eq 1 lem 7-7-3 a}, $\{u^n\}$ is bounded in $L^{\infty}([0,T],E_P^{d+1})$. We recall the estimate \eqref{eq pr6 lem 7-7-3} and apply this on the time interval $[t',t'']\subset [0,T]$, we get
$$
E^d_g(t'',u^n) - E_g^d(t',u^n)\leq E_g^d\big(e^{C\int_{t'}^{t''}D^n_d(\tau)d\tau}-1\big)
+ C\int_{t'}^{t''}\|F^n(\tau,\cdot)\|_{E^d}e^{C\int_\tau^{t''}D^n_d(s)}d\tau,
$$
where
$$
D_k(t):= \max_{\alpha,\beta}\|H^{\alpha\beta}(t,\cdot)\|_{E_H^{k+1}}.
$$
Recall that $D_d^n(\tau)$ and $\|F(\tau,\cdot)\|_{E^d}$ are uniformly (with respect to $n$) bound. This implies that $\{u^n\}$ is equicontinuous with respect to the norm $L^{\infty}([0,T], E_P^{d+1})$. Then there is a sub-sequence of $\{u^n\}$ converges in the sense of $L^{\infty}([0,T], E_P^{d+1})$. We denote it again by $u^n$. Then we see that $\{u^n\}$ converges in $L^{\infty}([0,T], E_H^{d+1})$. We denote by $u$ its limit.

When $d\geq 3$, \eqref{eq 2 lem 7-7-3} shows that $\{u^n\}$ is a Cauchy sequence in $L^{\infty}([0,T],\mathcal{E}_{-1})$. So that $\{u^n\}$ converges in $C([0,T],E_H^{d+1})$ ($u^n$ are $C_c^{\infty}$ functions so $u^n\in C([0,T];\mathcal{E}_{-1})$). Furthermore, since $u^n$ are $C_c^{\infty}$ functions so they are in $C([0,T],E_H^{d+1})$ which is a closed subspace of $L^{\infty}([0,T],E_H^{d+1})$. Then $\{u^n\}$ converges in $C([0,T],E_H^{d+1})$. We denote by $u$ the $C([0,T],E_H^{d+1})$-limit of $\{u^n\}$. Then we see that $u\in C([0,T];E_H^d)$

We apply the same argument on $\{\del^k_t u\}$ and get the desired regularity. The estimate on $u$ is gained by taking limit of the estimate on $u^n$. 
\end{proof}

If we analyse carefully the proof of Proposition \ref{prop 7-7-1}, we can conclude that if the triple $(F,u_0,u_1)$ is only supposed to be in $L^1([0,T],E^d)\times E_P^{d+1}\times E^d$, the Cauchy problem \eqref{eq 7-7-cauchy wave} determines also a unique solution in $C([0,T],E_P^{d+1})$. We prefer to state this result separately in the following proposition:
\begin{proposition}\label{prop 7-7-2}
Let $d\geq 3$ and assume that the triple $(F,u_0,u_1)$ is only supposed to be in $L^1([0,T],E^d)\times E_P^{d+1}\times E^d$. And assume that $g^{\alpha\beta}$ is a $C^{\infty}$ metric defined on $[0,T]\times \RR^3$ and coercive with constant $C>0$,  and $H^{\alpha\beta} = g^{\alpha\beta} - m^{\alpha\beta}$ is in the class $C([0,T];E_H^{d+1})$. Then the Cauchy problem \eqref{eq 7-7-cauchy wave} has a unique solution $u$ in $C([0,T],E_P^{d+1})$ and $\del_t^k u\in C([0,T];E^{d+1-k})$ for $0\leq k\leq d$. Furthermore, it satisfies the following estimate: 
\begin{subequations}\label{eq 1 prop 7-7-2} 
\begin{equation}\label{eq 1 prop 7-7-2 b}
\| u(t,\cdot)\|_{E_P^{d+1}} \leq C\|\nabla u(0,\cdot)\|_{E^d} e^{\int_0^tCD_d(\tau)d\tau}
 + C\int_0^t\|F(\tau,\cdot)\|_{E^d}\,e^{\int_\tau^tD_d(s)ds} d\tau,
\end{equation}
\end{subequations}
where
$$
D_k(t):= \max_{\alpha,\beta}\|H^{\alpha\beta}(t,\cdot)\|_{E_H^{k+1}}.
$$
\end{proposition} 

Apply Lemma~\ref{lem 7-7-4} and taking the same regularization argument as in the proof of Proposition \ref{prop 7-7-1}, the following existence result for linear Klein-Gordon equation holds.

\begin{proposition}[Existence for KG equation]\label{prop 7-7-3}
Let $d\geq 3$ and the triple $(v_0,v_1,F)$ be in class $E_R^{d+1}\times E_R^d\times L^1([0,T],E_R^d)$. Assume that $g^{\alpha\beta}$ is a $C^{\infty}$ metric defined on $[0,T]\times \RR^3$ and coercive with constant $C>0$,  and $H^{\alpha\beta} = g^{\alpha\beta} - m^{\alpha\beta}$ is in the class $C[0,T];E_H^{d+1})$. Then the following Cauchy problem
\begin{equation}\label{eq 7-7-cauchy KG}
\aligned
&g^{\alpha\beta}\del_{\alpha}\del_{\beta} v - c^2v = F, \quad  c>0,
\\
& v(0,x) = v_0(x),\quad \del_t v(0,x) = v_1(x)
\endaligned
\end{equation}
has a unique solution in class $C([0,T];E^{d+1}_R)\cap C^1([0,T];E^d_R)$. Furthermore, it satisfies the following estimate
\begin{subequations}
\begin{equation}\label{eq 1 prop 7-7-3 a}
\aligned
&\|v(t,\cdot)\|_{E_P^{d+1}} + \|\del_t v(t,\cdot)\|_{E^d} + c\|v(t, \cdot)\|_{E^d} 
\\
& \leq
 C\big(\|\nabla u(0,\cdot)\|_{E^d} + c\|v(0,\cdot)\|_{E^d}\big)e^{C\int_0^tD_d(\tau,\cdot)d\tau}
\\
&\quad + C\int_0^t\|F(s,\cdot)\|_{E^d}\, e^{C\int_s^t D_d(\tau,\cdot)}d\tau ds,
\endaligned
\end{equation}
\begin{equation}\label{eq 1 prop 7-7-3 b}
\aligned
\|v(t,\cdot)\|_{E_P^{d+1}} + c\|v(t, \cdot)\|_{E^d} & \leq
 C\big(\|\nabla u(0,\cdot)\|_{E^d} + c\|v(0,\cdot)\|_{E^d}\big)e^{C\int_0^tD_d(\tau,\cdot)d\tau}
\\
&+ C\int_0^t\|F(s,\cdot)\|_{E^d}\, e^{C\int_s^t D_d(\tau,\cdot)}d\tau ds,
\endaligned
\end{equation}
\end{subequations}
where
$$
D_k(t):= \max_{|I_1|+|I_2|\leq k\atop \alpha,\beta}\|\del^{I_1}\Omega^{I_2}H^{\alpha\beta}(t,\cdot)\|_{L^{\infty}(\RR^3)}.
$$
\end{proposition} 


\subsection{Nonlinear estimates}

To estimate the solution of quasi-linear system, we will need the following estimates on nonlinear terms.

\begin{lemma}\label{lem 7-10-0}
Let $F$ be a $C^{\infty}$ function from $\RR^m$ to $\RR$ and $u$ a $C^{\infty}$ application form $\RR^4$ to $\RR^m$ with components denoted by $u = (u_1,u_2,\cdots, u_m)$. Let $Z$ be a family of one order linear differential operate $Z = \{Z_{\alpha}\}$ with $\alpha\in\Lambda$, where $\Lambda$ is a subset of $\mathbb{N}^*$. Then the following identity holds for all multi-index $I = (\alpha_1,\alpha_2,\cdots, \alpha_{|I|})$ with $|I|\geq 1$:
\begin{equation}\label{eq lem 7-10-0}
Z^I\big (F(u)\big) = \sum_{1\leq|L|\leq |I|}P^LF(u)\sum_{\sum_{ji}K_{ji}=I}\prod_{j=1}^m\prod_{i=1}^{l_j}Z^{K_{ji}}u_j.
\end{equation}
Here
$$
P^L = \prod_{i=1}^m\del_i^{l_i}
$$
is a product of partial differential operator with $L=(l_1,\cdots l_m)$ and the convention:
$$
Z^Iu = 1, \quad \text{ if }\quad |I| = 0
$$
is applied. Furthermore, in a product if the set of index is empty, this product is regarded as $1$. For example,
$$
\prod_{i=1}^{l_j} Z^{K_{ji}}u_j = 1, \quad \text{ if } l_j=0.
$$
\end{lemma}

\begin{proof}  We observe that in the right-hand-side, for a fixed $L$, the sum is taken over all the proper $|L|$-partition of index $I$. That is, over all the proper $L$-partition of abstract index $\mathscr{I}$ with $|\mathscr{I}| = I$. We denote by
$$
\mathscr{K}_{ji} = \mathscr{I} \quad\text{ with }\quad \{K_{ji}\} = I(\{\mathscr{K}_{ji}\})
$$
and we denote by $\mathscr{P}_p(\mathscr{I},|L|)$ the set of all proper $|L|$-partition of $\mathscr{I}$.  Then \eqref{eq lem 7-10-0} can be written as
$$
Z^I\big (F(u)\big) = \sum_{1\leq|L|\leq |I|}P^LF(u)\sum_{\{\mathscr{K}_{ji}\}\in \mathscr{P}_p(\mathscr{I},|L|)}\prod_{j=1}^m\prod_{i=1}^{l_j}Z^{K_{ji}}u_j, \quad \{K_{ji}\} = I(\{\mathscr{K}_{ji}\}).
$$
Now we associate each term in the right-hand-side to a pair $(L,\{\mathscr{K}_{ij}\})$: in the sum, each term
$P^LF(u)\prod_{j=1}^m\prod_{i=1}^{l_j}Z^{K_{ji}}u_j$ corresponds to an operator $P^L$.
The quantity $\{K_{ji}\}$ is a partition of $I$ which is a restriction of $I$ on a abstract partition $\{\mathscr{K}_{ji}\}$. Note that for fixed $L$ the sum is taken on $\mathscr{P}_p(\mathscr{I},|L|)$, so we have constructed a bijection from the terms in right hand side to the following set
$$
\mathscr{Q}(\mathscr{I}) = \{(L,\{\mathscr{K}_{ji}\})|\mathscr{K}_{ji}\in \mathscr{P}_P(\mathscr{I},|L|),1\leq |L|\leq |\mathscr{I}|\}
$$

We will prove \eqref{lem 7-10-0} by induction on the order of $|I|$ with associated abstract multi-index $\mathscr{I} = \{\alpha_1,\alpha_2,\dots, \alpha_n\}$. We check by direct calculation that this identity is valid for $|I| =1$.  Suppose that it holds for $|I|\leq n$, we consider $|I'| = n+1$. Let $\mathscr{I}'$ be the associated $n+1$ order abstract multi-index composed by $\{\alpha_1,\alpha_2\dots,\alpha_n,\alpha_{n+1}\}$ and the restriction of $I'$ on $\mathscr{I}$ is coincide with $I$.

Then
$$
\aligned
&Z^{I'}\big(F(u)\big)
\\
&=Z_{I'(\alpha_{n+1})}\bigg(\sum_{1\leq|L|\leq |I|}P^LF(u)\sum_{\{\mathscr{K}_{ji}\}\in \mathscr{P}_p(\mathscr{I},|L|)}\prod_{j=1}^m\prod_{i=1}^{l_j}Z^{K_{ji}}u_j\bigg)
\\
&=\sum_{1\leq|L|\leq |I|}Z_{I'(\alpha_{n+1})}\big(P^LF(u)\big)
\sum_{\{\mathscr{K}_{ji}\}\in \mathscr{P}_p(\mathscr{I},|L|)}\prod_{j=1}^m\prod_{i=1}^{l_j}Z^{K_{ji}}u_j
\\
&+\sum_{1\leq|L|\leq |I|}P^LF(u)\sum_{\{\mathscr{K}_{ji}\}\in \mathscr{P}_p(\mathscr{I},|L|)}Z_{I'(\alpha_{n+1})}\prod_{j=1}^m\prod_{i=1}^{l_j}Z^{K_{ji}}u_j
\\
&=: T_1 + T_2.
\endaligned
$$
For $T_1$, we observe that
$$
\aligned
T_1 &= \sum_{1\leq|L|\leq |I|}\sum_{k=1}^m P^{L_k'}F(u)Z_{I'(\alpha_{n+1})}u_k
\sum_{\{\mathscr{K}_{ji}\}\in \mathscr{P}_p(\mathscr{I},|L|)}\prod_{j=1}^m\prod_{i=1}^{l_j}Z^{K_{ji}}u_j
\\
&=\sum_{1\leq|L|\leq |I|}\sum_{k=1}^m P^{L_k'}F(u)
\sum_{\{\mathscr{K}_{ji}\}\in \mathscr{P}_p(\mathscr{I},|L|)}\prod_{j=1}^m\prod_{i=1}^{l'_j}Z^{K'_{ji}}u_j
\endaligned
$$
with $L_k' = (l'_1,l'_2,\dots l'_k,\dots,l'_m)$ with $l'_j = l_j$ for $j\neq k$ and $l'_k = l_k+1$ and $\mathscr{K}'_{ji} = \mathscr{K}_{ji}$ with $(j,i)\neq (k,l'_k)$ and $\mathscr{K}'_{kl'_k} = \alpha_{n+1}$. Here, $K_{ij}$ is the restriction of $I$ on $\mathscr{K}_{ji}$ while $K'_{ji}$ is the restriction of $I'$ on $\mathscr{K}'_{ji}$.

For $T_2$,
$$
\aligned
T_2 &= \sum_{1\leq|L|\leq |I|}P^LF(u)\sum_{\{\mathscr{K}_{ji}\}\in \mathscr{P}_p(\mathscr{I},|L|)}Z_{I'(\alpha_{n+1})}\prod_{j=1}^m\prod_{i=1}^{l_j}Z^{K_{ji}}u_j
\\
& =\sum_{1\leq|L|\leq |I|}P^LF(u)\sum_{\{\mathscr{K}_{ji}\}\in \mathscr{P}_p(\mathscr{I},|L|)}
\sum_{1\leq j_0\leq m\atop 1\leq i_0\leq l_j}\prod_{j=1}^m\prod_{i=1}^{l_j}Z^{K'_{ji}}u_j,
\endaligned
$$
where  $\mathscr{K'}_{ji} = \mathscr{K}_{ji}$ when $(j,i)\neq (j_0,i_0)$ and $\mathscr{K'}_{j_0i_0} = \mathscr{K}_{j_0i_0}\cup\{\alpha_{n+1}\}$.

Now, we associate to each term in $T_1$ and $T_2$ a pair o$(L',\{\mathscr{K}_{ji}\})$ in the same manner. This defines an injection from the terms contained in $T_1$ and $T_2$ to the set
$$
\mathscr{Q}(\mathscr{I}') = \{(L',\{\mathscr{K}_{ji}\})|1\leq |L'|\leq |\mathscr{I}'|,\{\mathscr{K}_{ji}\}\in \mathscr{P}_P(|L|,\mathscr{I}')\}.
$$
The injectivity porperty is by checked from the fact that for two terms if $L'= \tilde{L}'$, the different terms correspond to a different partition (by our definition of sum over partitions).

Denote by $\mathscr{A}$ the image of the terms in $T_1$ and $T_2$ under this injection. This is a subset of $\mathscr{Q}(\mathscr{I}')$. We will prove that $\mathscr{A} = \mathscr{Q}(\mathscr{I}')$ which leads to the equality in the case $|I'| = n+1$. To do so, let
$$
(L',\{\mathscr{K}'_{ji}\})\in \mathscr{Q}(\mathscr{I}'), \quad L' = (l_1',l_2'\dots, l_m')
$$
Then we see that as $\alpha_{n+1}\in \mathscr{I}'$ there is one and only one $(j_0,i_0)$ such that $\alpha_{n+1}\in\mathscr{K}'_{j_0i_0}$ since $\bigcup_{(j,i)}\mathscr{K}'_{ji} = \mathscr{I}'$ and they are disjoint to each other. We will prove that $(L',\{\mathscr{K}'_{ji}\})\in \mathscr{A}$.

We define $L = (l_1,l_2,\dots,l_m)$ with $l_j = l'_j$ for $j\neq j_0$ and $l_{j_0} = l'_{j_0}-1$. We construct $\{\mathscr{K}_{ji}\}$ as follows:
$\mathscr{K}_{ji} = \mathscr{K}'_{ji}$ if $(j,i)\neq (j_0,i_0)$ and $\mathscr{K}_{j_0i_0} = \mathscr{K}'_{j_0i_0} \cap \{\alpha_{n+1}\}^c$.
Such constructed pair $(L,\mathscr{K}_{ji})$ to be in $\mathscr{Q}(\mathscr{I})$.

When $\mathscr{K}_{j_0i_0} = \varnothing$, we see that $(L',\{\mathscr{K}'_{ji}\})$ corresponds to a term in $T_1$. More precisely in the following term:
$$
\aligned
&Z_{I'(\alpha_{n+1})}\big(P^LF(u)\big)\sum_{\{\mathscr{K}_{ji}\}\in \mathscr{P}_p(\mathscr{I},|L|)}\prod_{j=1}^m\prod_{i=1}^{l_j}Z^{K_{ji}}u_j
\\
&=P^{L'_k}F(u)\sum_{k=1}^mZ_{\alpha_{n+1}}u_k\sum_{\{\mathscr{K}_{ji}\}\in \mathscr{P}_p(\mathscr{I},|L|)}\prod_{j=1}^m\prod_{i=1}^{l_j}Z^{K_{ji}}u_j,
\endaligned
$$
where we can see it by fixing $k = j_0$ and $\{\mathscr{K}_{ji}\}$ in the sum.

When $\mathscr{K}_{j_0i_0} \neq \varnothing$, $(L',\{\mathscr{K}'_{ji}\})$ corresponds to a term in $T_2$:
$$
\aligned
&P^LF(u)\sum_{\{\mathscr{K}_{ji}\}\in \mathscr{P}_p(\mathscr{I},|L|)}Z_{I'(\alpha_{n+1})}\prod_{j=1}^m\prod_{i=1}^{l_j}Z^{K_{ji}}u_j
\\
&=P^LF(u)\sum_{\{\mathscr{K}_{ji}\}\in \mathscr{P}_p(\mathscr{I},|L|)}
\sum_{1\leq j_0\leq m\atop 1\leq i_0\leq l_j}\prod_{j=1}^m\prod_{i=1}^{l_j}Z^{K'_{ji}}u_j
\endaligned
$$
we see it by fixing  $(i_i,j_0)$ and $\{\mathscr{K}_{ji}\}$ in the sum.
\end{proof}

\begin{lemma}\label{lem 7-10-1}
Let $F$ be a $C^{\infty}$ function from $\RR^m$ to $\RR$ and $u$ a $C^{\infty}$ application form $\RR^4$ to $\RR^m$. Then the following identity holds for any multi-index $I_1, I_2$ with $|I_1|+|I_2|\geq 1$:
\begin{equation}\label{eq lem 7-10-1}
\del^{I_1}\Omega^{I_2}\big(F(u)\big) = \sum_{l=1}^{|I_1|+|I_2|}\sum_{\sum_{j=1}^m l_j = l}P^LF(u)\sum_{\sum_{j,i}K_{1ji} =I_1\atop \sum_{j,i}K_{2ji}=I_2}
\prod_{j=1}^m\prod_{i=1}^{l_j}\del^{K_{1ji}}\Omega^{K_{2ji}}u_j,
\end{equation}
where $L = (l_1,l_2,\cdots, l_m)$ is a $m$-dimensional vector with its components taking value in $\mathbb{N}$ and $P^L$ the partial differential operator
$$
P^L = \prod_{i=1}^m\del_i^{l_i}
$$
and the convention:
$$
\del^{I_1}\Omega^{I_2}u = 1, \quad \text{ if } |I_1| = |I_2| =0
$$
is applied.
\end{lemma}

\begin{proof} 
The proof is an application of \eqref{eq lem 7-10-0}. Let $D^I = \del^{I_1}\Omega^{I_2}$ with $D = \{\del_\alpha, \Omega_a\}$. We denote by
$$
D_\alpha = \del_\alpha \text{ for } \alpha = 0, 1,2,3,\quad D_\alpha = \Omega_{\alpha-3}, \text{ for } \alpha=4,5,6.
$$
We denote the components of $I_1$ and $I_2$ by
$$
I_1 = (\beta_1,\beta_2,\cdots \beta_{n_1}),\quad I_2 = (\gamma_1,\gamma_2,\cdots,\gamma_{n_2})
$$
Then $I$ is determined by
$$
I = (\alpha_1,\alpha_2,\cdots \alpha_{n_1},\alpha_{n_1+1}\cdots,\alpha_{n_1+n_2})
$$
with $\alpha_i = \beta_i$ for $i=1,2,n_1$ and $\alpha_{i+n_1} = \gamma_i+3$ for $i=1,2,n_2$.

Remake that $D$ is a family of first-order linear differential operator. Then by \eqref{eq lem 7-10-0}
\begin{equation}
\del^{I_1}\Omega^{I_2}\big(F(u)\big)
 = \sum_{1\leq|L|\leq |I|}\sum_{\sum_{ji}K_{ji} = I}P^L F(u)\prod_{j=1}^{m}\prod_{i=1}^{l_j}Z^{K_{ji}}u_j.
\end{equation}
Then, since $\sum_{ji}K_{ji}=I$ is a partition of $I$. Then
$$
D^{K_{ji}} = \del^{K_{1ji}}\Omega^{K_{2ji}'}
$$
with $\sum_{ji}K_{1ji}=I_1$ a partition of $I_1$ and $\sum_{ji}K_{2ji}=I_2$ a partition of $I_2$. This gives
$$
\del^{I_1}\Omega^{I_2}\big(F(u)\big) = \sum_{1\leq|L|\leq |I|}\sum_{\sum_{ji}K_{1ji} = I_1\atop \sum_{ji}K_{2ji} = I_2}P^L F(u)\prod_{j=1}^{m}\prod_{i=1}^{l_j}\del^{K_{1ji}}\Omega^{K_{2ji}}u_j.
$$ 
Then the desired result is proven.
\end{proof}

The result of Lemma~\ref{lem 7-10-1} will be applied in the following case where $F(\cdot)$ is supposed to vanish at $0$ in second order.

\begin{lemma}\label{lem 7-10-2}
Let $F$ be a $C^{\infty}$ function defined in a compact neighborhood $V$ of $0$ in $\RR^{m}$ and $F(0) = \nabla F(0)= 0$. Let $d\geq 3$ and suppose that $u$ map from $\RR^4$ to $V$ with its components $u_j$ in $L^{\infty}([0,T];E^d)$. Then the following estimates hold for any couple of index $(I_1,I_2)$ with $|I_1|+|I_2|\leq d$:
\begin{equation}\label{eq 1 lem 7-10-2}
\big\|\del^{I_1}\Omega^{I_2}\big(F(u)\big)(t,\cdot)\big\|_{L^2(\RR^3)}\leq C(F,V,d) \sum_{k=2}^{|I_1|+|I_2|} \|u(t,\cdot)\|^k_{E^d},\quad \text{ for } t\in[0,T],
\end{equation}
where $C(F,V,d)$ is a constant determined by $F$, $V$ and $d$, and $\|u(t,\cdot)\|_{E^d}: = \max_{j}\|u_j(t,\cdot)\|_{E^d}$.
\end{lemma}

\begin{proof}
When $|I_1|+|I_2|=0$, by the condition $F(0) = \nabla F(0) = 0$,
$$
|F(u)|\leq C(F,V)\sum_{j=1}^m|u_j|^2
$$
which leads to the desired result.

For $|I_1|+|I_2|\geq 1$, proof is based on Lemma~\ref{lem 7-10-1}. We take the expression:
$$
\del^{I_1}\Omega^{I_2}\big(F(u)\big) = \sum_{l=1}^{|I_1|+|I_2|}\sum_{\sum_{j=1}^m l_j = l}P^LF(u)\sum_{\sum_{j,i}K_{1ji} =I_1\atop \sum_{j,i}K_{2ji}=I_2}
\prod_{j=1}^m\prod_{i=1}^{l_j}\del^{K_{1ji}}\Omega^{K_{2ji}}u_j
$$
and observe that for $|L| = \sum_{j=1}^m l_j = 1$, $P^L F(0)=0$. Then we have, in the compact neighborhood V of $0$,
$$
|\del_jF(u)|\leq C(F,V)|u|,
$$
where $C(F,V)$ is determined by $V$ and $F$ and $|u|:=\max_j|u_j|$. Then,
\begin{equation}\label{eq pr1 lem 7-10-2}
\aligned
\big|\del^{I_1}\Omega^{I_2}\big(F(u)\big)\big|
& \leq \sum_{j=1}^m|\del_jF(u)|\big|\del^{I_1}\Omega^{I_2}u_j\big|
\\
+&\sum_{l=2}^{|I_1|+|I_2|}\sum_{\sum_{j=1}^m l_j = l}|P^LF(u)|\sum_{\sum_{j,i}K_{1ji} =I_1\atop \sum_{j,i}K_{2ji}=I_2}
\prod_{j=1}^m\prod_{i=1}^{l_j}\big|\del^{K_{1ji}}\Omega^{K_{2ji}}u_j\big|
\\
& \leq C(F,V)|u|\,\big|\del^{I_1}\Omega^{I_2}u\big|
\\
&  + \sum_{l=2}^{|I_1|+|I_2|}\sum_{\sum_{j=1}^m l_j = l}|P^LF(u)|\sum_{\sum_{j,i}K_{1ji} =I_1\atop \sum_{j,i}K_{2ji}=I_2}
\prod_{j=1}^m\prod_{i=1}^{l_j}\big|\del^{K_{1ji}}\Omega^{K_{2ji}}u_j\big|.
\endaligned
\end{equation}

The first term in right-hand-side is estimated as follows:
$$
\aligned
\big{\|}u(t,\cdot)\,\del^{I_1}\Omega^{I_2}u(t,\cdot)\big{\|}_{L^2(\RR^3)}
& \leq \|u(t,\cdot)\|_{L^{\infty}(\RR^3)}\|\del^{I_1}\Omega^{I_2}u(t,\cdot)\|_{L^2(\RR^3)}
\\
& \leq C\|u(t,\cdot)\|_{E^2} \|u(t,\cdot)\|_{X^d}\leq C\|u(t,\cdot)\|^2_{E^d},
\endaligned
$$
where the Sobolev's inequality is applied.

The second term in right-hand-side of \eqref{eq pr1 lem 7-10-2} is estimated as follows: we observe the term:
$$
\sum_{\sum_{j,i}K_{1ji} =I_1\atop \sum_{j,i}K_{2ji}=I_2}
\prod_{j=1}^m\prod_{i=1}^{l_j}\big|\del^{K_{1ji}}\Omega^{K_{2ji}}u_j\big|
$$
Recall that $|L| = \sum_jl_j\geq 2$, then we see that in the product there are at least two factors and
$$
\sum_{ji}K_{1ji} = I_1,\quad \sum_{ji}K_{2ji}  = I_2
$$
is a partition of $(I_1,I_2)$ in $|L| = \sum_jl_j$ pieces with $|L|\geq 2$. Note that
$$
\sum_{ji}|K_{1ji}|  = |I_1|,\quad \sum_{ji}|K_{2ji}| = |I_2|.
$$
We observe that there among the index $K_{1ji}$ there is at most one, denoted by $K_{1i_0j_0}$, is of order higher than $|I_1|/2$. In an other world,
$$
|K_{1ji}|\leq \big[|I_2|/2\big], \quad \text{if } i\neq i_0,j\neq j_0,
$$
where $[x]$ denotes the biggest integer less than or equal to $x$. The same result holds for the index $K_{2ji}$. Then we conclude that in the decomposition of $(I_1,I_2)$ there is at most one pair of index, denoted by $(K_{1i_0j_0}, K_{2i_0j_0})$, is of order higher than $d/2$. In an other world,
$$
|K_{1ji}| + |K_{2ji}| \leq [d/2],\quad \text{if } i\neq i_0,j\neq j_0.
$$
Then if we take the $L^2$ norm, we will find that
$$
\aligned
&\bigg{\|}\prod_{j=1}^m\prod_{i=1}^{l_j}\del^{K_{1ji}}\Omega^{K_{2ji}}u_j(t,\cdot)\bigg{\|}_{L^2(\RR^3)}
\\
& \leq \prod_{j=1\atop j\neq j_0}^m\prod_{i=1\atop i\neq i_0}^{l_j}\big{\|}\del^{K_{1ji}}\Omega^{K_{2ji}}u_j(t,\cdot)\big{\|}_{L^{\infty}(\RR^3)}\cdot \big{\|}\del^{K_{1i_0j_0}}\Omega^{K_{2i_0j_0}}u_{j_0}(t,\cdot)\big{\|}_{L^2(\RR^3)}
\\
& \leq \prod_{j=1\atop j\neq j_0}^m\prod_{i=1\atop i\neq i_0}^{l_j}\big{\|}\del^{K_{1ji}}\Omega^{K_{2ji}}u_j(t,\cdot)\big{\|}_{E^2} \cdot
\|u_{j_0}(t,\cdot)\|_{E^d}
\\
& \leq \prod_{j=1\atop j\neq j_0}^m\prod_{i=1\atop i\neq i_0}^{l_j}\|u(t,\cdot)\|_{E^{2+[d/2]}}\|u(t,\cdot)\|_{E^d}.
\endaligned
$$
Here, we have applied \eqref{eq com lem 7-6'-1 d}. Now recall $d\geq 3$, then $2+[d/2]\leq d$. Then, we get 
$$
\bigg{\|}\prod_{j=1}^m\prod_{i=1}^{l_j}\del^{K_{1ji}}\Omega^{K_{2ji}}u(t,\cdot)\bigg{\|}_{L^2(\RR^3)} \leq \|u\|^{|L|}_{E^d}.
$$

Also, we observe that in the compact neighborhood $V$, $\sup_{x\in V}|P^LF(u)|\leq C(F,V)$ with $C(F,V,|L|)$ a constant determined by $F$ and $V$ and $|L|$. Then the desired result is proven.
\end{proof}

Now, we  combine Lemma~\ref{lem 7-10-2} with the global Sobolev inequality \eqref{eq 1 lem 7-6-1}.

\begin{lemma}\label{lem 7-10-2.1}
Based on the same assumptions on $F$ and $u$ as in Lemma~\ref{lem 7-10-2}, for $|I_1|+|I_2|\leq d-2$, the following estimate holds:
\begin{equation}\label{eq 1 lem 7-10-2.1}
\|\del^{I_1}\Omega^{I_2}\big(F(u)\big)(t,\cdot)\|_{\mathcal{E}_{-1}}\leq C(F,V,d) \sum_{k=2}^{|I_1|+|I_2|+2} \|u(t,\cdot)\|^k_{E^d}.
\end{equation}
\end{lemma}

We also need the following estimate in the following discussion.

\begin{lemma}\label{lem 7-10-2.2}
Let $d\geq 3$ and assume that $H(\cdot)$ be a $C^{\infty}$ function defined in a compact neighborhood $V$ of $0$ in $\RR^m$ and assume that $u$ is a map from $\RR^3$ to $V$ with its components $u_j$ in class $L^\infty([0,T];E_H^d)$. Then the following estimate holds for $d\geq |I_1|+|I_2|\geq 1$:
\begin{equation}\label{eq 1 lem 7-10-2.2}
\big\|(1+r)^{-1}\del^{I_1}\Omega^{I_2}\big(H(u)\big)(t,\cdot)\big\|_{L^2(\RR^3)}
\leq C(H,V,d)\sum_{k=1}^{|I_1|+|I_2|}\|u(t,\cdot)\|^k_{E_H^d},\quad t\in [0,T].
\end{equation}
\end{lemma}

\begin{proof}
We apply the expression \eqref{eq lem 7-10-1}:
$$
\del^{I_1}\Omega^{I_2}\big(H(u)\big) = \sum_{l=1}^{|I_1|+|I_2|}\sum_{\sum_{j=1}^ml_j=l}P^LH(u)\sum_{\sum_{ji}K_{1ji}=I_1\atop \sum_{ji}K_{2ji}=I_2}\prod_{j=1}^m\prod_{i=1}^{l_j}\del^{K_{1ji}}\Omega^{K_{2ji}}u_j
$$
Note that $P^LH(u)$ is bounded by a constant $C(H,V,|L|)$ determined by $V$. The estimate of $\big\|\del^{I_1}\Omega^{I_2}\big(H(u)\big)(t,\cdot)\big\|_{L^2(\RR^3)}$ reduced into the estimate of
$$
\bigg\|(1+r)^{-1}\prod_{j=1}^m\prod_{i=1}^{l_j}\del^{K_{1ji}}\Omega^{K_{2ji}}u_j(t,\cdot)\bigg\|_{L^2(\RR^3)},
$$
where $K_{1ji}$ and $K_{2ji}$ is a possible partition of $(I_1,I_2)$. We take the same argument to that the proof of Lemma~\ref{lem 7-10-2}. Suppose that $|K_{1j_0i_0}| + |K_{2j_0i_0}|\geq |K_{1ji}| + |K_{2ji}|$ for all pares $(j,i)$, i.e. $(K_{1j_0i_0},K_{2j_0i_0})$ is the pair of index with highest order. Then, we find 
$$
|K_{1ji}|+|K_{2ji}|\leq [d/2], \quad \text{ if } i\neq i_0, j\neq j_0
$$
and
$$
|K_{1j_0i_0}| + |K_{2j_0i_0}| \geq 1, 
$$
and we also have 
$$
\aligned
&\bigg\|(1+r)^{-1}\prod_{j=1}^m\prod_{i=1}^{l_j}\del^{K_{1ji}}\Omega^{K_{2ji}}u_j(t,\cdot)\bigg\|_{L^2(\RR^3)}
\\
& \leq  \prod_{j=1,\atop j\neq j_0}^m \prod_{i=1,\atop i\neq i_0}^{l_j}
\big\|\del^{K_{1ji}}\Omega^{K_{2ji}}u_j(t,\cdot)\big\|_{L^{\infty}(\RR^3)} \big\|(1+r)^{-1}\del^{K_{1j_0i_0}}\Omega^{K_{2i_0j_0}}u_j(t,\cdot)\big\|_{L^2(\RR^3)}. 
\endaligned
$$
In view of \eqref{eq 1' lem 7-6-2} and \eqref{eq 1' lem 7-6-3}, the desired result is proven.
\end{proof}

\begin{lemma}\label{lem 7-10-2.3}
By taking the same assumption on $H$ and $u$ as in Lemma~\ref{lem 7-10-2.2}, and assuming furthermore that $H(0) = 0$. Then for all $|I_1|+|I_2|\leq d-2$, the following estimate holds:
\begin{equation}\label{eq 1 lem 7-10-2.3}
\big\|\del^{I_1}\Omega^{I_2}\big(H(u)\big)(t,\cdot)\big\|_{L^\infty(\RR^3)}\leq C(H,V,d)\sum_{k=1}^{|I_1|+|I_2|+2} \|u(t,\cdot)\|_{E^d_H}^k.
\end{equation}
\end{lemma}

\begin{proof}
When $I_1=I_2=0$, we recall the that the condition $H(0)=0$ implies
$$
\|H(u)\|_{L^\infty(\RR^3)}  \leq C(H,V)\|u\|_{L^\infty(\RR^3)} \leq C(H,V)\|u\|_{\mathcal{E}_{-1}}.
$$
When $|I_1|+I_2|\geq 1$, we apply \eqref{eq 1 lem 7-10-2.2} combined with Lemma~\ref{lem 7-6-4}:
$$
\aligned
\|\big\|\del^{I_1}\Omega^{I_2}\big(H(u)\big)(t,\cdot)\big\|_{L^\infty(\RR^3)}
& \leq C\sum_{|J_1|+|J_2|\leq 2}\|\del_x^{J_1}\Omega^{J_2}\del^{I_1}\Omega^{I_2}\big(H(u)\big)(t,\cdot)\|_{L^2(\RR^3)}
\\
& \leq C\sum_{|J_1|+|J_2'|\leq 2\atop |I_1'|\leq|I_1|}\|\del_x^{J_1}\del^{I_1'}\Omega^{J_2'}\Omega^{I_2}\big(H(u)\big)(t,\cdot)\|_{L^2(\RR^3)}. 
\endaligned
$$

Here, we observe that when $|I_1|\geq 1$,  by applying \eqref{eq com' lem 7-6'-1 d} successively ($|I_1|$ times), we see that $|I'_1|\geq 1$, this leads to the fact that $|J_1|+|I_1'|+|J_2|+|I_2'|\geq 1$. Then we apply \eqref{eq 1 lem 7-10-2.2}.

When $|I_1|=0$, then $|I_2|\geq 1$ then we can also apply \eqref{eq 1 lem 7-10-2.2}.
\end{proof}

\begin{lemma}\label{lem 7-10-2.4}
Let $F$ be a $C^\infty$ function defined in a compact neighborhood $V_1$ of $0$ in $\RR^m$ and $H$ be a $C^{\infty}$ function defined in a compact neighborhood $V_2$ of $0$ in $\RR^n$. Assume that $F(0) = 0$ and $\nabla(F) = 0$. Let $u$ be a map from $\RR^4$ to $V_1$ with its components $u_j$ in class $L^\infty([0,T];E_H^d)$ and $v$ be a map form $\RR^4$ to $V_2$ with its components $v_l$ in class $L^\infty([0,T];E^d)$. Then the following estimate holds for $|I_1|+|I_2|\leq d$ with $d\geq 3$:
\begin{equation}\label{eq 1 lem 7-10-2.4}
\aligned
&
\big\|\del^{I_1}\Omega^{I_2}\big(H(u)F(v)\big)(t,\cdot)\big\|_{L^2(\RR^3)}
\\&
\leq C(H, F,V,d)\bigg(\sum_{k=0}^d\|u(t,\cdot)\|^k_{E_H^d}\bigg) \bigg(\sum_{k=2}^d\|v(t,\cdot)\|^k_{E^d}\bigg),\quad t\in[0,T].
\endaligned
\end{equation}
\end{lemma}

\begin{proof}
The proof is based on Lemmas~\ref{lem 7-10-2}, \ref{lem 7-10-2.1}, \ref{lem 7-10-2.2}, and \ref{lem 7-10-2.3}: 
$$
\del^{I_1}\Omega^{I_2}\big(H(u)F(v)\big)
 =  \sum_{J_1+J_1' = I_1\atop J_2+J_2'=I_2}\del^{J_1}\Omega^{J_2}\big(H(u)-H(0)\big) \del^{J_1'}\Omega^{J_2'}\big(F(v)\big)
 + H(0)\del^{I_1}\Omega^{I_2}\big(F(v)\big).
$$
When $d\geq |J_1'|+|J_2'|\geq d/2$, $|J_1|+|J_2|\leq [d/2]\leq d-2$:
$$
\aligned
&\big\|\del^{J_1}\Omega^{J_2}\big(H(u)-H(0)\big)(t,\cdot) \del^{J_1'}\Omega^{J_2'}\big(F(v)\big)(t,\cdot)\big\|_{L^2(\RR^3)}
\\
& \leq \big\|\del^{J_1}\Omega^{J_2}\big(H(u)-H(0)\big)(t,\cdot)\big\|_{L^\infty(\RR^3)}
      \big\|\del^{J_1'}\Omega^{J_2'}\big(F(v)\big)(t,\cdot)\big\|_{L^2(\RR^3)}
\\
& \leq C(H,F,V,d)\sum_{k=1}^{|J_1|+|J_2|+2}\|u(t,\cdot)\|^k_{E_H^d}\sum_{k=2}^{|J_1'|+|J_2'|}\|v(t,\cdot)\|^k_{E^d},
\endaligned
$$
where Lemma~\ref{lem 7-10-2.3} and Lemma~\ref{lem 7-10-2.1} are applied (on function $H(u)-H(0)$ and $F(v)$).
The term $H(0)\del^{I_1}\Omega^{I_2}\big(F(v)\big)$ is estimated by \eqref{eq 1 lem 7-10-2}.

When $|J_1|+|J_2|\geq d/2>1$, $|J_1'+|J_2'|\leq [d/2]\leq d-2$:
$$
\aligned
&\big\|\del^{J_1}\Omega^{J_2}\big(H(u)\big)(t,\cdot) \del^{J_1'}\Omega^{J_2'}\big(F(u)\big)(t,\cdot) \big\|_{L^2(\RR^3)}
\\
& \leq \big\|(1+r)^{-1}\del^{J_1}\Omega^{J_2}\big(H(u)\big)(t,\cdot) \big\|_{L^2(\RR^3)}
      \big\|(1+r)\del^{J_1'}\Omega^{J_2'}\big(F(u)\big)(t,\cdot) \big\|_{L^\infty(\RR^3)}
\\
&=\big\|(1+r)^{-1}\del^{J_1}\Omega^{J_2}\big(H(u)\big)(t,\cdot) \big\|_{L^2(\RR^3)}
      \big\|\del^{J_1'}\Omega^{J_2'}\big(F(u)\big)(t,\cdot) \big\|_{\mathcal{E}_{-1}}
\\
& \leq  C(H,F,V,d)\bigg(\sum_{k=1}^{|J_1|+|J_2|}\|u(t,\cdot)\|^k_{E_H^d}\bigg) \bigg(\sum_{k=2}^{|J_1'|+|J_2'|+2}\|v(t,\cdot)\|^k_{E^d}\bigg),
\endaligned
$$
where Lemma~\ref{lem 7-10-2.2} and Lemma~\ref{lem 7-10-2.1} are applied.
\end{proof}
 
We will need the following estimate on multi-linear functions.

\begin{lemma}\label{lem 7-10-2.45}
Let $d\geq 3$ be an integer and $u_i, v_i$, $i=1,2,\cdots m$ be functions in class $L^{\infty}([0,T];E^d)$. Then the following estimates hold for $m\geq 2$ and $|I_1|+|I_2|\leq d$:
\begin{equation}\label{eq 1 lem 7-10-2.45}
\aligned
&
\bigg\|\del^{I_1}\Omega^{I_2}\bigg(\prod_{i=1}^{m}u_i - \prod_{i=1}^mv_i\bigg)(t,\cdot)\bigg\|_{L^2(\RR^3)}
\\
&\leq C(V,d)\|u(t,\cdot)-v(t,\cdot)\|_{E^d}\sum_{k=1}^{m-1}\sum_{k_1+k_2=k}\|u(t,\cdot)\|_{E^d}^{k_1}\|v(t,\cdot)\|_{E^d}^{k_2}, 
\endaligned
\end{equation}
where $\|u\|_{X^d} := \max_i\|u_i\|_{X^d}$ and $\|v\|_{X^d}:=\max_i\|v_i\|_{X^d}$.
\end{lemma}

\begin{proof}
We observe the following indentity
$$
\prod_{i=1}^{m}u_i - \prod_{i=1}^mv_i = \sum_{k=1}^m(u_k-v_k)\prod_{j<k}u^j\prod_{j>k}v^j.
$$
Then, for each term, we have 
$$
\aligned
&\del^{I_1}\Omega^{I_2}\bigg((u_k-v_k)\prod_{j<k}u_j\prod_{j>k}v_j\bigg)
\\
&= \sum_{\sum_iK_{1i}=I_1\atop \sum_iK_{2i}=I_2}
\del^{K_{1k}}\Omega^{K_{2k}}(u_k-v_k)\prod_{j>k}\del^{K_{1j}}\Omega^{K_{2j}}u_j\prod_{j<k}\del^{K_{1j}}\Omega^{K_{2j}}v_j.
\endaligned
$$
To estimate the product
$$
\del^{K_{1k}}\Omega^{K_{2k}}(u_k-v_k)\prod_{j>k}\del^{K_{1j}}\Omega^{K_{2j}}u_j\prod_{j<k}\del^{K_{1j}}\Omega^{K_{2j}}v_j,
$$
we apply the same reasoning as in the proof of Lemma~\ref{lem 7-10-2}, there is at most one pair of multi index of order bigger that $d/2$. Then \eqref{eq 1 lem 7-10-2.45} is proven by applying the classical Sobolev's inequality.
\end{proof}

\begin{lemma}\label{lem 7-10-2.45'}
Let $d\geq 3$ be an integer and $u_i, v_i$, $i=1,2,\cdots m$ be functions in class $L^{\infty}([0,T];E_H^d)$. Then the following estimates hold for $m\geq 2$ for $1\leq |I_1|+|I_2|\leq d$:
\begin{equation}\label{eq 1 lem 7-10-2.45'}
\aligned
&\bigg\|(1+r)^{-1}\del^{I_1}\Omega^{I_2}\bigg(\prod_{i=1}^{m}u_i - \prod_{i=1}^mv_i\bigg)(t,\cdot)\bigg\|_{L^2(\RR^3)}
\\
& \leq C(V,d)\|u(t,\cdot)-v(t,\cdot)\|_{E_H^d}
\sum_{k=1}^{m-1}\sum_{k_1+k_2=k}\|u(t,\cdot)\|_{E_H^d}^{k_1}\|v(t,\cdot)\|_{E^d_H}^{k_2}, \quad t\in [0,T],
\endaligned
\end{equation}
where $\|u\|_{E_H^d} := \max_i\|u_i\|_{E_H^d}$ and $\|v\|_{X^d}:=\max_i\|v_i\|_{E_H^d}$.
\end{lemma}

\begin{proof}
We observe the following indentity
$$
\prod_{i=1}^{m}u_i - \prod_{i=1}^mv_i = \sum_{k=1}^m(u_k-v_k)\prod_{j<k}u_j\prod_{j>k}v_j.
$$
Then, for each term, we have 
$$
\aligned
&\del^{I_1}\Omega^{I_2}\bigg((u_k-v_k)\prod_{j<k}u_j\prod_{j>k}v_j\bigg)
\\
&= \sum_{\sum_iK_{1i}=I_1\atop \sum_iK_{2i}=I_2}
\del^{K_{1k}}\Omega^{K_{2k}}(u_k-v_k)\prod_{j>k}\del^{K_{1j}}\Omega^{K_{2j}}u_j\prod_{j<k}\del^{K_{1j}}\Omega^{K_{2j}}v_j.
\endaligned
$$
To estimate the product
$$
\del^{K_{1k}}\Omega^{K_{2k}}(u_k-v_k)\prod_{j>k}\del^{K_{1j}}\Omega^{K_{2j}}u_j\prod_{j<k}\del^{K_{1j}}\Omega^{K_{2j}}v_j,
$$
we also apply the same reasoning as in the proof of Lemma~\ref{lem 7-10-2}, there is at most one pair of multi index of order bigger that $d/2$. We discuss three cases.

\

\noindent {\bf Case 1: $|K_{1k}| + |K_{2k}|\geq |K_{1j}| + |K_{2j}|$ for $j=1,2,\cdots m$}. Hence,   $(K_{1k},K_{2k})$ takes the highest order. For $j\neq k$, $|K_{1j}| + |K_{2j}|\leq [d/2]\leq d-2$ and $|K_{1k}| + |K_{2k}|\geq 1$, and we have 
$$
\aligned
&\bigg\|(1+r)^{-1}\del^{K_{1k}}\Omega^{K_{2k}}(u_k-v_k)\prod_{j>k}\del^{K_{1j}}\Omega^{K_{2j}}u_j\prod_{j<k}\del^{K_{1j}}\Omega^{K_{2j}}v_j
\bigg\|_{L^2(\RR^3)}
\\
& \leq 
\big\|(1+r)^{-1}\del^{K_{1k}}\Omega^{K_{2k}}(u_k-v_k)\big\|_{L^2(\RR^3)}
\bigg\|\prod_{j>k}\del^{K_{1j}}\Omega^{K_{2j}}u_j
\prod_{j<k}\del^{K_{1j}}\Omega^{K_{2j}}v_j\bigg\|_{L^{\infty}(\RR^3)}
\\
& \leq  C\|u_k-v_k\|_{E_P^d} \bigg\|\prod_{j>k}\del^{K_{1j}}\Omega^{K_{2j}}u_j
\prod_{j<k}\del^{K_{1j}}\Omega^{K_{2j}}v_j\bigg\|_{L^\infty(\RR^3)}. 
\endaligned
$$
Here, we used \eqref{eq 1 lem 7-6-2} on the first factor. Note that $m\geq 2$ then there exists a $j_0\neq k$. So the second factor  is bounded by applying \eqref{eq 1 lem 7-10-2.3} and the fact that $|K_{1j}| + |K_{2j}|\leq d-2$.

\

\noindent {\bf Case 2: $|K_{1k}| + |K_{2k}|\leq [d/2]\leq d-2$ and there exists $j_0<k$ such that $|K_{1j_0}| + |K_{2j_0}|\geq |K_{1j}| + |K_{2j}|$.}
Then, we have 
$$
\aligned
&\bigg\|(1+r)^{-1}\del^{K_{1k}}\Omega^{K_{2k}}(u_k-v_k)
\prod_{j>k}\del^{K_{1j}}\Omega^{K_{2j}}u_j
\prod_{j<k}\del^{K_{1j}}\Omega^{K_{2j}}v_j
\bigg\|_{L^2(\RR^3)}
\\
& \leq  \big\|\del^{K_{1k}}\Omega^{K_{2k}}\big(u_k-v_k\big)\big\|_{L^{\infty}(\RR^3)}
\prod_{j>k}\big\|\del^{K_{1j}}\Omega^{K_{2j}}u_j\big\|_{L^\infty(\RR^3)}
\prod_{j<k\atop j\neq j_0}\big\|\del^{K_{1j}}\Omega^{K_{2j}}v_j\big\|_{L^\infty(\RR^3)}
\\
&\times \big\|(1+r)^{-1}\del^{K_{1j_0}}\Omega^{K^{2j_0}} u^{j_0}\big\|_{L^2(\RR^3)}
\\
& \leq  C\|u-v\|_{E_H^d} \prod_{j>k}\big\|u_j\big\|_{E_H^d} \prod_{j<k\atop j\neq j_0}\big\|v_j\big\|_{E_H^d} \|u^{j_0}\|_{E_H^d},
\endaligned
$$
where \eqref{eq 1' lem 7-6-2} and \eqref{eq 1' lem 7-6-3} are applied.

\

\noindent {\bf Case 3: $|K_{1k}| + |K_{2k}|\leq [d/2]\leq d-2$ and there exists $j_0>k$ such that $|K_{1j_0}| + |K_{2j_0}|\geq |K_{1j}| + |K_{2j}|$.} The proof is exactly the same as in the last case provided we exchange the roles of $u_j$ and $v_j$.
\end{proof}

Finally, we are able to estimate the difference of two quadratic functions. 

\begin{lemma}\label{lem 7-10-2.46}
Let $d\geq 3$ be an integer. Let $F$ be a $C^{\infty}$ function defined in a compact neighborhood $V$ of $0$ in $\RR^{m}$ with $F(0) = \nabla F(0) = 0$. Assume that $u$ and $v$ are maps from $\RR^4$ to $V$ with their components in $L^\infty([0,T];E^d)$. Then the following estimate hold for all pair of index $(I_1,I_2)$ with $|I_1|+|I_2|\leq d$:
\begin{equation}\label{eq 1 lem 7-10-2.46}
\big\|\del_x^{I_1}\Omega^{I_2}\big(F(u) - F(v)\big)\big\|_{L^2(\RR^3)}\leq C(F,V,d)\|u-v\|_{X^d}\sum_{k=1}^{|I_1|+|I_2|-1}\sum_{k_1+k_2=k}\|u\|_{X^d}^{k_1}\|v\|_{X^d}^{k_2}.
\end{equation}
\end{lemma}

\begin{proof}
When $I_1 = I_2 = 0$, we apply the mean value theorem: there is a $\theta\in[0,1]$ such that
$$
|F(u) - F(v)| \leq \big|\nabla F\big(\theta u+(1-\theta) v\big)\big| u-v\big|.
$$
Then by the condition $\nabla F(0) = 0$,
$$
|F(u) - F(v)| \leq \big|\nabla F\big(\theta u+(1-\theta) v\big)|u-v|\leq C(F,V)|\theta u+ (1-\theta)v||u-v|.
$$
Then
$$
\aligned
\|F(u)(t,\cdot) - F(v)(t,\cdot)\|_{L^2(\RR^3)} & \leq C(F,V)\|u-v\|_{L^2(\RR^3)}\|\theta u+ (1-\theta)v\|_{L^{\infty}(\RR^3)}
\\
& \leq C(V)\|u(t,\cdot)-v(t,\cdot)\|_{E^d}(\|u(t,\cdot)\|_{E^d} + \|v(t,\cdot)\|_{E^d}).
\endaligned
$$

For the case for $|I_1|+|I_2|\geq 1$, we recall the expression \eqref{eq lem 7-10-1} :
$$
\aligned
&\del^{I_1}\Omega^{I_2}\big(F(u)-F(v)\big)
\\
&= \sum_{l=1}^{|I_1|+|I_2|}\sum_{\sum_jl_j = l}
\bigg(P^LF(u)\sum_{\sum{j,i}K_{1ji} = I_1\atop \sum_{j,i}K_{2ji}=I_2}
\prod_{j=1}^m\prod_{i=1}^{l_j} \del^{K_{1ji}}\Omega^{K_{2ji}}u_j 
\\
&\hskip3.cm - P^LF(v)\sum_{\sum{j,i}K_{1ji} = I_1\atop \sum_{j,i}K_{2ji}=I_2}
\prod_{j=1}^m\prod_{i=1}^{l_j} \del^{K_{1ji}}\Omega^{K_{2ji}}v_j\bigg)
\\
&= \sum_j \del_jF(u)\big(\del^{I_1}\Omega^{I_2}u_j - \del^{I_1}\Omega^{I_2}v_j\big)
\\
 &+ \sum_{l=2}^{|I_1|+|I_2|}\sum_{\sum_jl_j = l}P^LF(u)\sum_{\sum{j,i}K_{1ji} = I_1\atop \sum_{j,i}K_{2ji}=I_2}
\bigg(\prod_{j=1}^m\prod_{i=1}^{l_j}\del^{K_{1ji}}\Omega^{K_{2ji}}u_j - \prod_{j=1}^m\prod_{i=1}^{l_j}\del^{K_{1ji}}\Omega^{K_{2ji}}v_j\bigg)
\\
&+ \sum_{l=1}^{|I_1|+|I_2|}\sum_{\sum_jl_j = l}\big(P^LF(u) - P^LF(v)\big)\sum_{\sum{j,i}K_{1ji} = I_1\atop \sum_{j,i}K_{2ji}=I_2}
\prod_{j=1}^m\prod_{i=1}^{l_j} \del^{K_{1ji}}\Omega^{K_{2ji}}v_j
\\
\\
&=: T_1 + T_2 + T_3.
\endaligned
$$

To estimate $T_1$, we take into consideration of the condition $\nabla F(0)=0$ with leads to
$$
\big|\del_jF(u)(t,\cdot)\big|\leq C(V)|u|.
$$
So $\|\del_j F(u)(t,\cdot)\|_{L^{\infty}(\RR^3)}\leq C(V)\|u(t,\cdot)\|_{L^{\infty}(\RR^3)}\leq C(F,V)\|u(t,\cdot)\|_{X^d}$.
Then
$$
\|T_1(t,\cdot)\|_{L^2(\RR^3)}\leq C(F,V)\|u(t,\cdot)-v(t,\cdot)\|_{X^d}\|u(t,\cdot)\|_{X^d}.
$$

To estimate $T_2$, we need to apply Lemma~\ref{lem 7-10-2.45}. To do so, we observe the following relation:
\begin{equation}
\del^{I_1}\Omega^{I_2}\bigg(\prod_{i=1}^m u_j^{l_j} - \prod_{i=1}^mv_j^{l_j}\bigg) =  \sum_{\sum{j,i}K_{1ji} = I_1\atop \sum_{j,i}K_{2ji}=I_2}
\bigg(\prod_{j=1}^m\prod_{i=1}^{l_j}\del^{K_{1ji}}\Omega^{K_{2ji}}u_j - \prod_{j=1}^m\prod_{i=1}^{l_j}\del^{K_{1ji}}\Omega^{K_{2ji}}v_j\bigg)
\end{equation}
Then
$$
T_2 = \sum_{l=2}^{|I_1|+|I_2|}\sum_{\sum_jl_j = l}P^LF(u)\del^{I_1}\Omega^{I_2}\bigg(\prod_{j=1}^m u_j^{l_j} - \prod_{j=1}^mv_j^{l_j}\bigg)
$$
Recall that $|P^LF(u)|$ is bounded by a constant $C(F,V,|L|)$ determined by the neighborhood $V$, the function $F$ and the order $|L|$. Then we apply \eqref{eq 1 lem 7-10-2.45}.

The estimate of $T_3$ is as follows:
\begin{equation}
\|T_3\|_{L^2(\RR^3)} \leq \sum_{l=1}^{|I_1|+|I_2|}\sum_{\sum_jl_j = l}\big\|P^LF(u) - P^LF(v)\big\|_{L^{\infty}(\RR^3)}\sum_{\sum{j,i}K_{1ji} = I_1\atop \sum_{j,i}K_{2ji}=I_2}
\big\|\prod_{j=1}^m\prod_{i=1}^{l_j} \del_x^{K_{1ji}}\Omega^{K_{2ji}}v_j\big\|_{L^2(\RR^3)}
\end{equation}
As in the estimate on $T_2$, we see that
$$
\sum_{\sum{j,i}K_{1ji} = I_1\atop \sum_{j,i}K_{2ji}=I_2}
\big\|\prod_{j=1}^m\prod_{i=1}^{l_j} \del_x^{K_{1ji}}\Omega^{K_{2ji}}v(t,\cdot)\big\|_{L^2(\RR^3)}\leq \sum_{l=1}^{|I_1|+|I_2|}\|v(t,\cdot)\|^l_{X^d}.
$$
and
$$
\aligned
\big\|P^LF(u)(t,\cdot) - P^LF(v)(t,\cdot)\big\|_{L^{\infty}(\RR^3)} 
& \leq C(F,V,|L|)\|u(t,\cdot)-v(t,\cdot)\|_{L^{\infty}}
\\
& \leq C(V,|L|)\|u(t,\cdot)-v(t,\cdot)\|_{E^d} 
\endaligned
$$
and then
$$
\|T_3(t,\cdot)\|_{L^2(\RR^3)}\leq C(V)\|u(t,\cdot)-v(t,\cdot)\|_{E^d}\sum_{k=1}^{d-1}\|v\|_{E^d}^k
$$
which leads us to the conclusion \eqref{eq 1 lem 7-10-2.46}.
\end{proof}

The following $L^{\infty}$ type estimate is a direct application of Lemma~\ref{lem 7-10-2.46} and the global Sobolev inequality \eqref{eq 1 lem 7-6-1}:
\begin{lemma}\label{lem 7-10-2.460}
Let $F$, $G$, $u$ and $v$ take the assumption as in Lemma~\ref{lem 7-10-2.46}, then for any $|I_1|+|I_2|\leq d-2$, the following estimate holds:
\begin{equation}
\aligned
&\big\|\del^{I_1}\Omega^{I_2}\big(F(u)(t,\cdot)-F(v)(t,\cdot)\big)(t,\cdot)\big\|_{\mathcal{E}_{-1}(\RR^3)}
\\
& \leq C(F,V,d)\|u(t,\cdot)-v(t,\cdot)\|_{E^d}\sum_{k=1}^{d-1}\sum_{k_1+k_2=k}\|u(t,\cdot)\|_{E^d}^{k_1}\|v(t,\cdot)\|_{E^d}^{k_2}.
\endaligned
\end{equation}
\end{lemma}

\begin{lemma}\label{lem 7-10-2.461}
Let $d\geq 3$ be an integer. Let $H$ be a $C^{\infty}$ function which is defined in a compact neighborhood $V$ of $0$ in $\RR^m$ and let $u,v$ be maps from $\RR^4$ to $V$ with their components in class $L^\infty([0,T],E_H^d)$. Then the following estimate holds for $1\leq |I_1|+|I_2|\leq d$:
\begin{equation}\label{eq 1 lem 7-10-2.461}
\aligned
&\big\|(1+r)^{-1}\del^{I_1}\Omega^{I_2}\big(H(u) - H(v)\big)(t,\cdot)\big\|_{L^2(\RR^3)}
\\
& \leq  C(V)\|u(t,\cdot)-v(t,\cdot)\|_{E_H^d}\sum_{k=0}^{|I_1|+|I_2|}\sum_{k_1+k_2=k}
\|u(t,\cdot)\|_{E_H^d}^{k_1}\|v(t,\cdot)\|_{E_H^d}^{k_2}.
\endaligned
\end{equation}
\end{lemma}
\begin{proof}
The proof is quite similar to that of Lemma~\ref{lem 7-10-2.46}.

When $|I_1|+|I_2|\geq 1$, we apply the same calculation:
$$
\aligned
&\del^{I_1}\Omega^{I_2}\big(H(u)-H(v)\big)
\\
&= \sum_{l=1}^{|I_1|+|I_2|}\sum_{\sum_jl_j = l}
\bigg(P^LH(u)\sum_{\sum{j,i}K_{1ji} = I_1\atop \sum_{j,i}K_{2ji}=I_2}
\prod_{j=1}^m\prod_{i=1}^{l_j} \del^{K_{1ji}}\Omega^{K_{2ji}}u_j 
\\
& \hskip3.cm -
P^LH(v)\sum_{\sum{j,i}K_{1ji} = I_1\atop \sum_{j,i}K_{2ji}=I_2}
\prod_{j=1}^m\prod_{i=1}^{l_j} \del^{K_{1ji}}\Omega^{K_{2ji}}v_j\bigg)
\\
&= \sum_{l=1}^{|I_1|+|I_2|}\sum_{\sum_jl_j = l}P^LH(u)\sum_{\sum{j,i}K_{1ji} = I_1\atop \sum_{j,i}K_{2ji}=I_2}
\bigg(\prod_{j=1}^m\prod_{i=1}^{l_j}\del^{K_{1ji}}\Omega^{K_{2ji}}u_j - \prod_{j=1}^m\prod_{i=1}^{l_j}\del^{K_{1ji}}\Omega^{K_{2ji}}v_j\bigg)
\\
&+ \sum_{l=1}^{|I_1|+|I_2|}\sum_{\sum_jl_j = l}\big(P^LH(u) - P^LH(v)\big)\sum_{\sum{j,i}K_{1ji} = I_1\atop \sum_{j,i}K_{2ji}=I_2}
\prod_{j=1}^m\prod_{i=1}^{l_j} \del^{K_{1ji}}\Omega^{K_{2ji}}v_j
\\
&=: T_1 + T_2.
\endaligned
$$
The estimate of $T_1$ and $T_2$ are similar to the estimate made in the proof of Lemma~\ref{lem 7-10-2.46}. $T_1$ is estimated by apply Lemma~\ref{lem 7-10-2.45'} and $T_2$ is by applying the following estimate
$$
\big\|P^LH(u) - P^LH(v)\big\|_{L^{\infty}(\RR^3)} \leq C(H,V,|L|)\|u-v\|_{L^{\infty}}\leq C(V,|L|)\|u-v\|_{X^d}
$$
and the fact that (by applying \eqref{eq 1 lem 7-10-2.2})
$$
\bigg\|(1+r)^{-1}\prod_{j=1}^m\prod_{i=1}^{l_j} \del_x^{K_{1ji}}\Omega^{K_{2ji}}v_j\bigg\|_{L^2(\RR^3)}\leq C(V)\|v\|^{|L|}_{X_H^d}.
$$
\end{proof}

We also need a $L^{\infty}$ estimate on $H(u) - H(v)$:
\begin{lemma}\label{lem 7-10-2.470}
Let $d\geq 3$.
Assume that $H$ be a $C^{\infty}$ function defined in a compact neighborhood $V$ of $0$ in $\RR^m$ and assume that $u,v$ are maps from $\RR^4$ to $V$ with their components in the class $L^\infty([0,T];E_H^d)$. Then the following estimate holds for $|I_1|+|I_2|\leq d-2$:
\begin{equation}\label{eq 1 lem 7-10-2.470}
\aligned
&\|\del^{I_1}\Omega^{I_2}\big(H(u) - H(v)\big)(t,\cdot)\|_{L^{\infty}(\RR^3)}
\\
& \leq C(H,V,d)\|u(t,\cdot)-v(t,\cdot)\|_{E_H^d}\sum_{k=0}^{|I_1|+|I_2|+1}\sum_{k_1+k_2=k}
\|u(t,\cdot)\|_{E_H^d}^{k_1}\|v(t,\cdot)\|_{E_H^d}^{k_2}
\endaligned
\end{equation}
\end{lemma}
\begin{proof}
When $I_1=I_2=0$, then there exists $\theta \in [0,1]$ such that
$$
|H(u)-H(v)|\leq |\nabla H(\theta u + (1-\theta)v)||u-v|\leq C(H,V)|u-v|.
$$
which proves \eqref{eq 1 lem 7-10-2.470}.

For $|I_1|+|I_2|\geq 1$, we apply Lemma~\ref{lem 7-6-4} combined with \eqref{eq 1 lem 7-10-2.461}.
\end{proof}

\begin{lemma}\label{lem 7-10-2.47}
Let $F$ be a $C^{\infty}$ function defined in a compact neighborhood $V_1$ of $0$ in $\RR^m$ and $H$ be a $C^{\infty}$ function defined in a neighborhood $V_2$ of $0$ in $\RR^n$. Let $d$ be an integer and $d\geq 3$. Assume that $F(0) = 0, \nabla F(0)=0$. Let $u_1$ and $u_2$ be maps from $\RR^4$ to $V_1$ with their components in $L^\infty([0,T];E^d_H)$ and $v_1$, $v_2$ be maps from $\RR^4$ to $V_2$ with their components in class $L^\infty([0,T];E^d)$. Then the following estimate holds for $|I_1|+|I_2|\leq d$ and $d\geq 3$. Then, we have 
\begin{equation}\label{eq 1 lem 7-10-2.47}
\aligned
&\big\|\del^{I_1}\Omega^{I_2}\big(H(u_1)F(v_1) - H(u_2)F(v_2)\big)(t,\cdot)\big\|_{L^2(\RR^3)}
\\
& \leq C(V)\big(\|u_1(t,\cdot)-u_2(t,\cdot)\|_{E^d_H} + \|v_1(t,\cdot)-v_2(t,\cdot)\|_{E^d}\big)\sum_{k=1}^{d-1}M^k(t),
\endaligned
\end{equation}
where
$$
M(t):= \max\{\|{u_1}_j(t,\cdot)\|_{E_H^d},\|{u_2}_j(t,\cdot)\|_{E_H^d},\|{v_1}_j(t,\cdot)\|_{E^d},\|{v_2}_j(t,\cdot)\|_{E^d}\}.
$$
\end{lemma}

\begin{proof} We have 
$$
\aligned
&\del^{I_1}\Omega^{I_2}\big(H(u_1)F(v_1) - H(u_2)F(v_2)\big)
\\
&=\del^{I_1}\Omega^{I_2}\big(H(u_1)F(v_1) - H(u_1)F(v_2) + H(u_1)F(v_2) - H(u_2)F(v_2)\big)
\\
&= \del^{I_1}\Omega^{I_2}\big(H(u_1)(F(v_1)-F(v_2))\big) +  \del^{I_1}\Omega^{I_2}\big(F(v_2)(H(u_1) - H(u_2))\big)
\\
&=: T_1 + T_2. 
\endaligned
$$
The term $T_1$ is estimated as follows:
$$
\aligned
T_1 =  \del^{I_1}\Omega^{I_2}\big(H(u_1)(F(v_1)-F(v_2))\big)
= \sum_{J_1+J_2=I_1\atop J_1'+J_2'=I_2}\del^{J_1}\Omega^{J_2}\big(H(u_1)\big)\del^{J_1'}\Omega^{J_2'}(F(v_1)-F(v_2)).
\endaligned
$$
When $|J_1|+|J_2|\leq [d/2]\leq d-2$, by applying \eqref{eq 1 lem 7-10-2.3} and \eqref{eq 1 lem 7-10-2.46}
$$
\aligned
&\big\|\del^{J_1}\Omega^{J_2}\big(H(u_1)\big)\del^{J_1'}\Omega^{J_2'}(F(v_1)-F(v_2))(t,\cdot)\big\|_{L^2(\RR^3)}
\\
&=\big\|\del^{J_1}\Omega^{J_2}\big(H(u_1)\big)(t,\cdot)\big\|_{L^{\infty}(\RR^3)}
\big\|\del^{J_1'}\Omega^{J_2'}(F(v_1)-F(v_2))(t,\cdot)\big\|_{L^2(\RR^3)}
\\
& \leq  C(H,F,V,d)\bigg(\sum_{k=0}^{|J_1|+|J_2|+2}\|u_1(t,\cdot)\|_{E_H^d}\bigg)
\\
& \hskip2.cm
\bigg(\|v_1(t,\cdot)-v_2(t,\cdot)\|_{E^d}\sum_{k=1}^{|J_1'|+|J_2'|-1}\sum_{k_1+k_2=k}\|v_1(t,\cdot)\|_{E^d}^{k_1}\|v_2(t,\cdot)\|_{E^d}^{k_2}\bigg)
\endaligned
$$
When $|J_1|+|J_2|\geq d/2>1$, $|J_1'|+|J_2'|\leq [d/2]$, we apply \eqref{eq 1 lem 7-10-2.2}:
$$
\aligned
&\big\|\del^{J_1}\Omega^{J_2}\big(H(u_1)\big)\del^{J_1'}\Omega^{J_2'}(F(v_1)-F(v_2))(t,\cdot)\big\|_{L^2(\RR^3)}
\\
& \leq \|(1+r)^{-1}\del^{J_1}\Omega^{J_2}\big(H(u_1)\big)(t,\cdot)\|_{L^2(\RR^3)}
\|\del^{J_1'}\Omega^{J_2'}(F(v_1)-F(v_2))(t,\cdot)\|_{\mathcal{E}_{-1}}
\\
& \leq  C(V)\sum_{k=1}^{|J_1|+|J_2|}\|u\|_{E_H^d}^k \cdot C(H,F,V,d)\|v_1(t,\cdot)-v_2(t,\cdot)\|_{E^d}\sum_{k=1}^{|J_1'|+|J_2'|+1}\sum_{k_1+k_2=k}\|v_1(t,\cdot)\|_{E^d}^{k_1}\|v_2(t,\cdot)\|_{E^d}^{k_2}
\\
& \leq  C(H,F,V,d)\|v_1(t,\cdot)-v_2(t,\cdot)\|_{E^d}\sum_{k=1}^{|I_1|+|I_2|+2}
\sum_{k_1+k_2+k_3=k}\|u_1(t,\cdot)\|_{E_H^d}^{k_1}\|v_1(t,\cdot)\|_{E^d}^{k_2}\|v_2(t,\cdot)\|_{E^d}^{k_3}
\endaligned
$$

The term $T_2$ is estimated similarly by applying \eqref{eq 1 lem 7-10-2}, \eqref{eq 1 lem 7-10-2.461}, \eqref{eq 1 lem 7-10-2.1} and \eqref{eq 1 lem 7-10-2.470}. We omit the details, but we write out the estimate
$$
\|T_2(t,\cdot)\|_{L^2(\RR^3)} \leq \|u_1(t,\cdot)-u_2(t,\cdot)\|_{E_H^d}\sum_{k=2}^{|I_1|+|I_2|+1}\sum_{k_1+k_2+k_3=k}\|u_1\|_{E_H^d}^{k_1}\|u_2\|_{E_H^d}^{k_2}\|v_2\|_{E^d}^{k_3}.
$$
\end{proof}


\newpage 

\section{Proof of the local existence}

\subsection{Iteration and uniform bound}

In this section we will begin the proof of Theorem \ref{prop 7-10-1}. The proof of this theorem will occupy the following two subsections and follows a classical iteration procedure:
\begin{subequations}\label{eq conformal aug sec7-9}
\begin{equation}\label{eq conformal aug sec7-9 a}
\aligned
(m^{\alpha'\beta'} + H^{\alpha'\beta'}(h^{\coeff}_n))\del_{\alpha'}\del_{\beta'}h_{\alpha\beta}^{\coeff,n+1}
&=  F_{\alpha\beta}(h^\coeff_n;\del h^\coeff_n,\del h^\coeff_n) - 16\pi\del_{\alpha}\phi^\coeff_n\del_{\beta}\phi^\coeff_n
\\
& \quad - 12\del_{\alpha}\varrho^\coeff_n\del_{\beta}\varrho^{\coeff}_n
-\coeff^{-1}V_h(\varrho_n^{\coeff})\big(m_{\alpha\beta}+h^{\coeff,n}_{\alpha\beta}\big),
\endaligned
\end{equation}
\begin{equation}\label{eq conformal aug sec7-9 b}
(m^{\alpha'\beta'} + H^{\alpha'\beta'}(h^{\coeff}_n))\del_{\alpha'}\del_{\beta'}\phi^\coeff_{n+1}
= 2(m^{\alpha'\beta'} + H^{\alpha'\beta'}(h^{\coeff}_n))\del_{\alpha'}\phi^{\coeff}_n \del_{\beta'}\varrho^\coeff_n,
\end{equation}
\begin{equation}\label{eq conformal aug sec7-9 c}
\aligned
& (m^{\alpha'\beta'} + H^{\alpha'\beta'}(h^{\coeff}_n))\del_{\alpha'}\del_{\beta'}\varrho^{\coeff}_{n+1} -\frac{\varrho^\coeff_{n+1}}{3\coeff}
\\
&=
\coeff^{-1}V_\rho(\varrho^{\coeff}_n)
- \frac{4\pi}{3e^{2\varrho^{\coeff}_n}}(m^{\alpha'\beta'} + H^{\alpha'\beta'}(h^{\coeff}_n))\del_{\alpha'}\phi^\coeff_n\del_{\beta'}\phi^\coeff_n,
\endaligned
\end{equation}
where
$$
\coeff^{-1}V_h(s) := \frac{\big(e^{2s}-1\big)^2}{2\coeff e^{4s}},\quad
\coeff^{-1}V_\rho(s) :=\frac{e^{2s}-1}{6\coeff e^{4s}} - \frac{s}{3\coeff}
$$
\end{subequations}
and with initial data
$$
\aligned
&h^{\coeff,n}_{\alpha\beta}(0,x) = {h_0}_{\alpha\beta},\quad
\phi^{\coeff}_n(0,x) = \phi_0(x),\quad \varrho^{\coeff}_n(0,x) =\varrho_0(x),
\\
&\del_t h^{\coeff,n}_{\alpha\beta}(0,x) = {h_1}_{\alpha\beta},\quad
\del_t\phi^{\coeff}_n(0,x) = \phi_1(x)\quad
\del_t\varrho^\coeff_n(0,x) = \varrho_1(x).
\endaligned
$$

Recall the function $H^{\alpha\beta}(h)$ are defined in \eqref{eq def1 7-2} and the associated estimates are in \eqref{eq 1 7-2}.
We take $(h_{\alpha\beta}^{\coeff,0})$ and set 
$$
S_n^{\coeff} : = (h_{\alpha\beta}^{\coeff, n}, \phi^{\coeff}_n, \varrho^{\coeff}_n).
$$
We also denote by $F_H(S_n^{\coeff}), F_P(S_n^\coeff)$ and $F_R(S_n^\coeff)$ the terms in right-hand-side of \eqref{eq conformal aug sec7-9 a}, \eqref{eq conformal aug sec7-9 b} and \eqref{eq conformal aug sec7-9 c}.

We take $S_0^{\coeff} = (h_{\alpha\beta}^{\coeff, 0}, \phi^{\coeff}_0, \varrho^{\coeff}_0)$ as the solution of the following homogeneous Cauchy problem:
$$
\aligned
&\Box h_{\alpha\beta} = 0, \quad \Box \phi = 0,\quad \Box \varrho - \frac{\varrho}{3\coeff} = 0,
\\
&h^{\coeff,0}_{\alpha\beta}(0,x) = {h_0}_{\alpha\beta},\quad
\phi^{\coeff}_0(0,x) = \phi_0(x),\quad \varrho^{\coeff}_0(0,x) = \varrho_0(x),
\\
&\del_t h^{\coeff,0}_{\alpha\beta}(0,x) = {h_1}_{\alpha\beta},\quad
\del_t\phi^{\coeff}_0(0,x) = \phi_1(x)\quad
\del_t\varrho^\coeff_0(0,x) = \varrho_1(x).
\endaligned
$$

We see that the source terms and metric coefficients in \eqref{eq conformal aug sec7-9} are sufficiently regular and the initial data are in the corresponding class required in Propositions \ref{prop 7-7-1}, \ref{prop 7-7-2} and \ref{prop 7-7-3}. Then, by the theory of local existence for linear equations, this iteration procedure is well defined in a fixed time interval $[0,T]$, where the metric coefficients and source terms are in the corresponding class and $|H(h^\coeff_n)^{00}|\leq 1/2$. We see that is iteration defines a sequence of triple $S_n^{\coeff} := (h_{\alpha\beta}^{\coeff, n}, \phi^{\coeff}_n, \varrho^{\coeff}_n)$. In order to get the local existence, we will prove that $S_n^{\coeff}$ converges in the following norm: 
$$
\aligned
&\|S(t,\cdot)^{\coeff}_n\|_{E_\coeff^{d+1}}
\\
: &= \max \{ \|h_{\alpha\beta}^{\coeff,n}(t,\cdot)\|_{E_H^{d+1}},\,
\|\phi^\coeff_n(t,\cdot)\|_{E_P^{d+1}},\,
\|\varrho^\coeff_n(t,\cdot)\|_{E_P^{d+1}},\,  \coeff^{-1/2}\|\varrho_n^{\coeff}(t,\cdot)\|_{E^d}\},
\endaligned
$$
for all $t \in [0, T^*]$, in which $T^*>0$ will be defined. 

We suppose that for certain $d\geq 3$, for all $k\leq n$ and all $t\in [0,T]$,
\begin{equation}\label{eq 1 7-10}
\|S_k^\coeff (t,\cdot)\|_{E_\coeff^{d+1}}\leq A\eps,
\end{equation}
we will show that $\|S_{n+1}^\coeff(t,\cdot)\|_{X_\coeff^{d+1}}\leq A\eps$ with $T$ and $A$ well chosen.

First, we observe that when \eqref{eq 1 7-10} holds with $A\eps\leq \eps_0$ with $\eps_0$ small enough, the metric $g^{\coeff,k}_{\alpha\beta} := m^{\alpha\beta} + h_{\alpha\beta}^{\coeff,k}$ are coercive with constant $C(\eps_0)$, where $C(\eps_0)$ is determined by $\eps_0$.

Now, by combining \eqref{eq 1 7-10} with Lemma~\ref{lem 7-10-2}, the following estimates on the source terms follow.

\begin{lemma}\label{lem 7-10-2.5}
Assume that \eqref{eq 1 7-10} holds with $d\geq 3$. Suppose that $0\leq A\eps\leq \eps_0$, where $\eps_0$ is a constant sufficiently small. Then the following estimates holds for $k\leq n$:
\begin{subequations}\label{eq 1 lem 7-10-2.5}
\begin{equation}\label{eq 1 lem 7-10-2.5 a}
\|F_{\alpha\beta}(h_k^\coeff,\del h_k^{\coeff},\del h_k^{\coeff})(t,\cdot)\|_{E^d}
+\|\del_{\alpha}\phi_k^{\coeff}\del_{\beta}\phi^\coeff_k(t,\cdot)\|_{E^d}
+\|\del_{\alpha}\varrho_k^{\coeff}\del_{\beta}\varrho_k^{\coeff}(t,\cdot)\|_{E^d}
\leq C(\eps_0,d)(A\eps)^2,
\end{equation}
\begin{equation}\label{eq 1 lem 7-10-2.5 b}
\|(m^{\alpha'\beta'} + H^{\alpha'\beta'}(h_k^{\coeff}))\del_{\alpha'}\phi^{\coeff}_k\del_{\beta'}\phi^{\coeff}_k(t,\cdot)\|_{E^d}\leq C(\eps_0,d)(A\eps)^2,
\end{equation}
\begin{equation}\label{eq 1 lem 7-10-2.5 c}
\|e^{-2\varrho_k^{\coeff}}(m^{\alpha'\beta'}+H^{\alpha'\beta'}(h_n^\coeff))\del_{\alpha'}\phi_k^{\coeff}\del_{\beta'}\phi_k^{\coeff}(t,\cdot)\|_{E^d}\leq C(\eps_0,d)(A\eps)^2,
\end{equation}
\begin{equation}\label{eq 1 lem 7-10-2.5 d}
\|\coeff^{-1}V_\rho(\varrho_k^\coeff)(t,\cdot)\|_{E^d}\leq  C(\eps_0,d) (A\eps)^2,\quad \|\coeff^{-1}V_h(\varrho_k^\coeff)(m_{\alpha\beta}+h_{\alpha\beta}^{\coeff,k})(t,\cdot)\|_{E^d}\leq  C(\eps_0,d)(A\eps)^2.
\end{equation}
\end{subequations}
\end{lemma}
These are classical estimate when we establish the local theory of existence for quasi-linear wave equations with quadratic nonlinearity. The only thing important is \eqref{eq 1 lem 7-10-2.5 d}, where the $\coeff$ appears in the left-hand-side in order to get estimates independent of $\coeff$.
\begin{proof}
These estimates are applications of \eqref{eq 1 lem 7-10-2} and \eqref{eq 1 lem 7-10-2.4} combined with \eqref{eq 1 7-10}. The only thing we need to be pay attention is that to guarantee the $C^{\infty}$ regularity of the function $H^{\alpha\beta}(\cdot)$, we need to restrict its defined in a compact neighborhood $V$ of $0$ in $\RR^{10}$ with
$$
V = \{\max_{1\leq1\leq10}|x^i|\leq \eps_0\}
$$
with $\eps_0$ sufficiently small. This can be guaranteed by taking $\eps\leq \eps_0$ in \eqref{eq 1 7-10}.

We observe that $F_{\alpha\beta}(h_n^\coeff, \del h_n^\coeff,\del h_n^\coeff)$ is quadratic with respect to $\del h_n^\coeff$ and 
$C^{\infty}$ with respect to $h_n^\coeff$. Then we apply \eqref{eq 1 lem 7-10-2.4}. The estimate on the term
$$
(m^{\alpha'\beta'} + H^{\alpha'\beta'}(h_k^{\coeff}))\del_{\alpha'}\phi^{\coeff}_k\del_{\beta'}\phi^{\coeff}_k
$$
is established in the same manner.

The estimate of \eqref{eq 1 lem 7-10-2.5 d} is checked by using the estimate
$$
\coeff^{-1/2}\|\varrho_k^\coeff\|_{E^d}\leq A\eps,
$$
which follows from \eqref{eq 1 7-10}.
\end{proof}

Now we begin the discussion of the commutators such as $[\del_x^{I_1}\Omega^{I_2},H^{\alpha'\beta'}(h^{\coeff}_n)\del_{\alpha'}\del_{\beta'}]h^{\coeff,n}_{\alpha\beta}$ which appears in the estimates of $\|h_{\alpha\beta}^{\coeff,n+1}\|_{X_E^d}$.

\begin{lemma}\label{lem 7-10-4}
There exists a positive constant $\eps_0$ such that if \eqref{eq 1 7-10} holds with $d\geq 3$, $A\eps\leq \eps_0\leq 1$, then the following estimates hold for all couple of index $(I_1,I_2)$ with $1\leq |I_1|+|I_2|\leq d$:
\begin{subequations}\label{eq 1 lem 7-10-4}
\begin{equation}\label{eq 1 lem 7-10-4a}
\big{\|}[\del^{I_1}\Omega^{I_2},H^{\alpha'\beta'}(h^{\coeff}_n)\del_{\alpha'}\del_{\beta'}]h^{\coeff,n+1}_{\alpha\beta}(t,\cdot)\big{\|}_{E^d}
\leq C(\eps_0,d)A\eps \|h_{\alpha\beta}^{\coeff,n+1}(t,\cdot)\|_{E_P^{d+1}},
\end{equation}
\begin{equation}\label{eq 1 lem 7-10-4b}
\big{\|}[\del^{I_1}\Omega^{I_2},H^{\alpha'\beta'}(h^{\coeff}_n)\del_{\alpha'}\del_{\beta'}]\phi^\coeff_{n+1}(t,\cdot)\big{\|}_{E^d}
\leq C(\eps_0,d)A\eps \|\phi_{n+1}^{\coeff}(t,\cdot)\|_{E_P^{d+1}},
\end{equation}
\begin{equation}\label{eq 1 lem 7-10-4c}
\big{\|}[\del^{I_1}\Omega^{I_2},H^{\alpha'\beta'}(h^{\coeff}_n)\del_{\alpha'}\del_{\beta'}]\varrho^{\coeff}_{n+1}(t,\cdot)\big{\|}_{E^d}
\leq C(\eps_0,d)A\eps \|\varrho_{n+1}^{\coeff}(t,\cdot)\|_{E_P^{d+1}}.
\end{equation}
\end{subequations}
\end{lemma}

\begin{proof}
The estimate of these three commutators are similar, and we only prove the first statement.
Let $(I_1,I_2)$ be a pair of multi-indices, $|I_1|+|I_2|\leq d$. Recall the estimate of commutator \eqref{eq com lem 7-6'-1 c}:
$$
\aligned
\big|[\del^{I_1}\Omega^{I_2}, H^{\alpha'\beta'}\del_{\alpha'}\del_{\beta'}]h_{\alpha'\beta'}^{\coeff,n+1}\big|
& \leq \sum_{J_1+J_1'=I_1\atop {|J_2|+|J_2'|<|I_2|\atop \alpha',\beta',\alpha'',a}}
|\del^{J_1'}\Omega^{J_2'}H^{\alpha'\beta'}|\,|\del_{\alpha''}\del_a \del^{J_1}\Omega^{J_2}h_{\alpha\beta}^{\coeff,n+1}|
\\
&+\sum_{J_1 + J_1' = I_1 \atop {J_2+J_2' = I_2 \atop |J_1'|+|J_2'|>0}}
|\del^{J_1'}\Omega^{J_2'}H^{\alpha'\beta'}|\,|\del_{\alpha'}\del_{\beta'}\del^{J_1}\Omega^{J_2}h_{\alpha\beta}^{\coeff,n+1}|
\\
=:& T_1 + T_2.
\endaligned
$$

   We begin with $T_2$ and distinguish between two cases.

\

\noindent {\bf Case 1: $1\leq |J_1'|+|J_2'|\leq d-1$, $1\leq |J_1|+|J_2|\leq d-1$.} In this case, we have 
$$
\aligned
&\big{\|}\del^{J_1'}\Omega^{J_2'}H^{\alpha'\beta'}(h_n^\coeff)
\del_{\alpha'}\del_{\beta'}\del^{J_1}\Omega^{J_2}h_{\alpha\beta}^{\coeff,n+1}\big{\|}_{L^2(\RR^3)}
\\
& \leq \|\del^{J_1'}\Omega^{J_2'}H^{\alpha'\beta'}(h_n^\coeff)\|_{L^{\infty}(\RR^3)}
\|\del_{\alpha'}\del_{\beta'}\del^{J_1}\Omega^{J_2}h_{\alpha\beta}^{\coeff,n+1}\big{\|}_{L^2(\RR^3)}
\\
& \leq C(\eps_0,d)A\eps \|h^\coeff_n(t,\cdot)\|_{E_P^{d+1}},
\endaligned
$$
where we used $A\eps\leq 1$ and \eqref{eq 1 lem 7-10-2.3}.

\

\noindent {\bf Case 2: $J'_1 = I_1,J'_2=I_2$, $J_1=J_2=0$.} Recall that $d\geq 3$ then $|J_1|+|J_2|\leq d-3$:
$$
\aligned
&\big{\|}\del_x^{J_1'}\Omega^{J_2'}H^{\alpha'\beta'}(h_n^\coeff)
\del_{\alpha'}\del_{\beta'}\del_x^{J_1}\Omega^{J_2}h_{\alpha\beta}^{\coeff,n+1}\big{\|}_{L^2(\RR^3)}
\\
& \leq 
\big{\|}(1+r)^{-1}\del_x^{I_1}\Omega^{I_2}H^{\alpha'\beta'}(h_n^\coeff)\big{\|}_{L^2(\RR^3)}
\big{\|}\del_{\alpha'}\del_{\beta'}h_{\alpha\beta}^{\coeff,n+1}\big{\|}_{\mathcal{E}_{-1}}
\\
& \leq  C(\eps_0,d)A\eps \|h_{\alpha\beta}^{\coeff,n+1}\|_{E_P^{d+1}},
\endaligned
$$
where we used $A\eps \leq 1$ and \eqref{eq 1 lem 7-10-2.2}. 

The estimate of term $T_1$ is quite simpler. Recall that in the expression of $T_1$, the sum is taken over the index satisfying the following conditions:
$$
J_1+J_1'=I_1,\quad |J_2| + |J_2'|\leq |I_2|-1
$$
So
$$
(|J_1| + |J_1'|) + (|J_2|+|J_2'|)\leq d-1
$$
which leads to
$$
|J_1'| + |J_2'|\leq d-1.
$$
So
$$
\big{\|}\del_x^{J_1'}\Omega^{J_2'}H^{\alpha'\beta'}
\del_{\alpha''}\del_a \del_x^{I_1}\Omega^{J_2}h_{\alpha\beta}^{\coeff,n+1}\big{\|}_{L^2(\RR^3)}
\leq \big{\|}\del_x^{J_1'}\Omega^{J_2'}H^{\alpha'\beta'}\big{\|}_{L^{\infty}}
\big{\|}\del_{\alpha''}\del_a \del_x^{I_1}\Omega^{J_2}h_{\alpha\beta}^{\coeff,n+1}\big{\|}_{L^2(\RR^3)}
$$

As in the estimate of $T_2$, we see that with $|J_1'|+|J_2'|\leq d-1$,
$$
 \big{\|}\del_x^{J_1'}\Omega^{J_2'}H^{\alpha'\beta'}\big{\|}_{L^{\infty}(\RR^3)}\leq C(\eps_0,d)A\eps.
$$
The second factor on $h_{\alpha\beta}^{\coeff,n+1}$ is bounded directly by $\|h_{\alpha\beta}^{\coeff,n+1}\|_{X_E^{d+1}}$. So we conclude with \eqref{eq 1 lem 7-10-4a}.
\end{proof}

Now we need to discuss the bound of the initial data $E^d_{g_n}(0,h_{\alpha\beta}^{\coeff,n+1}), E^d_{g_n}(0,\phi^\coeff_{n+1})$ and $E^d_{g_n,c}(0,\varrho_{n+1}^\coeff)$. We will see that these norms are controlled by $\|S_0\|_{X_\coeff^{d+1}}$:
\begin{lemma}\label{lem 7-10-4.5}
When $\|S_0\|_{X_\coeff^{d+1}}$ is supposed to be bounded by $\eps\leq \eps_0$ for $\eps_0$ sufficiently small, there exists a positive constant determined by $\eps_0$ and $d$ such that
$$
E_{g_n}^d(0,h^\coeff_{n+1}) + E_{g_{n+1}}^d(0,\phi^\coeff) + E_{g_{n+1},\coeff^{-1/2}}^d(0,\varrho^\coeff)\leq C(\eps_0,d)\eps.
$$
\end{lemma}
\begin{proof}
We recall that
$$
\aligned
\|S_0\|_{X_{\coeff}^{d+1}} := \max\{&\|{h_0}_{\alpha\beta}\|_{X_H^{d+1}},\,\|{h_1}_{\alpha\beta}\|_{X^d},\,\coeff^{-(1/2)[d/2]+1/4}\|\phi_0\|_{X_P^{d+1}},
\\
& \coeff^{-(1/2)[d/2]+1/4}\|\phi_1\|_{X^d},\,\coeff^{-[d/2]-1/2}\|\varrho_0\|_{X_P^{d+1}},\, \coeff^{-[d/2]-1/2}\|\varrho_1\|_{X^d}\}.
\endaligned
$$

We observe that when $0\leq k\leq 1$ the norm $\|\del_t^k\del_x^{I_1}\Omega^{I_2}\varrho^\coeff(0,\cdot)\|_{L^2(\RR^3)}$ is determined directly by $\varrho_0$ and $\varrho_1$ thus, bounded by $\coeff^{-[d/2]-1/2}\leq \coeff^{-1/2}$.

When $2\leq k\leq d$, we need to use the equation.

We will prove that for $0\leq |k|\leq d-2$ and $|I_1|+|I_2|\leq d-1-k$
\begin{equation}\label{eq pr1 lem 7-10-4.5}
\aligned
&\|\del_t\del_t\del_t^k\del_x^{I_1}\Omega^{I_2}h_{\alpha\beta}^\coeff(0,\cdot)\|_{L^2(\RR^3)}\leq C(\eps_0,d)\eps,
\\
&\|\del_t\del_t\del_t^k\del_x^{I_1}\Omega^{I_2} \varrho^{\coeff}(0,\cdot)\|_{L^2(\RR^3)}
\leq C(\eps_0,d)\coeff^{[d/2]+1/2-[k/2]}\eps,
\\
&\|\del_t\del_t\del_t^k \del_x^{I_1}\Omega^{I_2}\phi^{\coeff}(0,\cdot)\|_{L^2(\RR^3)}\leq C(\eps_0,d)\coeff^{1/2[d/2]-1/4}\eps
\endaligned
\end{equation}

This is proven by induction on $k$. We see that for $k=0,1$, the estimates hold by direct verification. Suppose that \eqref{eq pr1 lem 7-10-4.5} holds for $(k-1,k)$ we will prove the case $k+1$.

The estimate of $\|\del_t\del_t\del_t^{k+1}\del_x^{I_1}\Omega^{I_2}\phi^{\coeff}\|_{L^2(\RR^3)}$, is a bit complicated. We see that by Lemma~\ref{lem 7-7-5},
$$
\aligned
\|\del_t\del_t\del_t^{k+1}\del_x^{I_1}\Omega^{I_2} \varrho^{\coeff}\|_{L^2(\RR^3)}
& \leq \|\del_t^{k+1}\del_x^{I_1}\Omega^{I_2}\big((1-H^{00})^{-1}\big(m^{ab}+H^{ab}\big)\del_a\del_b\varrho^\coeff\big)\|_{L^2(\RR^3)}
\\
&+ 2\|\del_t^{k+1}\del_x^{I_1}\Omega^{I_2}\big((1-H^{00})^{-1}H^{0a}\del_t\del_a\varrho^\coeff\big)\|_{L^2(\RR^3)}
\\
&+ \frac{3}{\coeff}\|\del_t^{k+1}\del_x^{I_1}\Omega^{I_2}\varrho^\coeff\|_{L^2(\RR^3)}
\\
&+ \|\del_t^{k+1}\del_x^{I_1}\Omega^{I_2}\big((1-H^{00})F_R\big)\|_{L^2(\RR^3)}
\endaligned
$$
Then we see that
$$
\|\del_t\del_t\del_t^k\del_x^{I_1}\Omega^{I_2} \varrho^{\coeff}\|_{L^2(\RR^3)}\leq C(\eps_0,d)\coeff^{[d/2]+1/2-[d/2]}\eps.
$$

We observe that by Lemma~\ref{lem 7-7-5},
$$
\aligned
\|\del_t\del_t\del_t^k \del_x^{I_1}\Omega^{I_2} h_{\alpha\beta}^{\coeff}\|_{L^2(\RR^3)}
& \leq \|\del_t^k \del_x^{I_1}\Omega^{I_2}\big((1-H^{00})^{-1}\big(m^{ab}+H^{ab}\big)\del_a\del_b h_{\alpha\beta}^\coeff\big)\|_{L^2(\RR^3)}
\\
&+ 2\|\del_t^k \del_x^{I_1}\Omega^{I_2}\big((1-H^{00})^{-1}H^{0a}\del_t\del_ah_{\alpha\beta}^\coeff\big)\|_{L^2(\RR^3)}
\\
&+ \|\del_t^k \del_x^{I_1}\Omega^{I_2}\big((1-H^{00})F_H\big)\|_{L^2(\RR^3)}
\endaligned
$$
Then by the bounds prescribed by $\|S_0\|_{\coeff}^{d+1}$ and \eqref{eq pr1 lem 7-10-4.5}, we see that $\|\del_t\del_t\del_x^{I_1}\Omega^{I_2} h_{\alpha\beta}^{\coeff}\|_{L^2(\RR^3)}$ is bounded by $C(\eps_0,d)\eps$.

In the same manner, we see that for $0\leq k\leq d-2$
$$
\|\del_t\del_t\del_t^k \del_x^{I_1}\Omega^{I_2}\phi^{\coeff}\|_{L^2(\RR^3)}\leq C(\eps_0,d)\coeff^{-(1/2)[d/2]+1/4}\eps.
$$
\end{proof}

Now we are ready to estimate the $L^2$ type norm of $S_n^{\coeff}$.

\begin{lemma}\label{lem 7-10-5}
There exists a positive constant $\eps_0$ such that if \eqref{eq 1 7-10} holds for $A\eps\leq \eps_0\leq 1$ and $d\geq 3$, then
\begin{subequations}\label{eq lem 7-10-5}
\begin{equation}\label{eq lem 7-10-5 a}
\|h^{\coeff,n+1}_{\alpha\beta}(t,\cdot)\|_{E_P^{d+1}} + \|\del_t h_{\alpha\beta}^{\coeff,n+1}(t,\cdot)\|_{E^d}
\leq C(\eps_0,d)\big(\eps e^{C(\eps_0,d)A\eps t}
+ A\eps\big(e^{C(\eps_0,d)A\eps t} - 1\big)\big)
\end{equation}
\begin{equation}\label{eq lem 7-10-5 b}
\|\phi_{n+1}^\coeff(t,\cdot)\|_{E_P^{d+1}} + \|\del_t \phi_{n+1}^{\coeff}(t,\cdot)\|_{E^d}
\leq C(\eps_0,d)\big(\eps e^{C(\eps_0,d)A\eps t}
+ A\eps\big(e^{C(\eps_0,d)A\eps t} - 1\big)\big)
\end{equation}
\begin{equation}\label{eq lem 7-10-5 c}
\aligned
\|\varrho_{n+1}^\coeff(t,\cdot)\|_{E_P^{d+1}}
+ \|\del_t \varrho_{n+1}^{\coeff}(t,\cdot)\|_{E^d} &+ \coeff^{-1/2}\|\varrho_{n+1}^\coeff(t,\cdot)\|_{E^d}
\\
& \leq C(\eps_0,d)\big(\eps e^{C(\eps_0,d)A\eps t}
+ A\eps\big(e^{C(\eps_0,d)A\eps t} - 1\big)\big),
\endaligned
\end{equation}
\end{subequations}
where $C(\eps_0,d)$ is a positive constant determined by $\eps_0$ and $d$.
\end{lemma}

\begin{proof} This is an application of the $L^2$ estimate \eqref{eq 1 lem 7-7-2}. We consider the estimate for $h_{\alpha\beta}^{\coeff,n+1}$. To do so, we derive the equation \eqref{eq conformal aug sec7-9 a} with respect to a product $\del^{I_1}\Omega^{I_2}$ with $|I_1|+|I_2|\leq d$:
$$
\big(m^{\alpha'\beta'} + H^{\alpha'\beta'}(h_n^{\coeff})\big)\del_{\alpha'}\del_{\beta'}\del^{I_1}\Omega^{I_2}h_{\alpha\beta}^{\coeff,n+1}
=\del^{I_1}\Omega^{I_2}F_H(h_n^{\coeff},\phi_n^{\coeff},\varrho_n^{\coeff}) - [\del^{I_1}\Omega^{I_2},H^{\alpha'\beta'}\del_{\alpha'}\del_{\beta'}]h_{\alpha\beta}^{\coeff,n+1},
$$
where $F_H$ denotes the terms in right-hand-side of \eqref{eq conformal aug sec7-9 a}, which is
$$
\aligned
F_H(h_n^{\coeff},\phi_n^{\coeff},\varrho_n^{\coeff}) &= F_{\alpha\beta}(h^\coeff_n;\del h^\coeff_n,\del h^\coeff_n) - 16\pi\del_{\alpha}\phi^\coeff_n\del_{\beta}\phi^\coeff_n
 - 12\del_{\alpha}\varrho^\coeff_n\del_{\beta}\varrho^{\coeff}_n
-\coeff^{-1}V_h(\varrho_n^{\coeff})\big(m_{\alpha\beta}+h^{\coeff,n}_{\alpha\beta}\big).
\endaligned
$$
Then by \eqref{eq 1 lem 7-7-2},
$$
\aligned
\frac{d}{dt}E_{g_n}(t,\del^{I_1}\Omega^{I_2}h_{\alpha\beta}^{\coeff,n+1}) & \leq C\|\del^{I_1}\Omega^{I_2}F_H\|_{L^3(\RR^3)} + C\big{\|}[\del^{I_1}\Omega^{I_2},H^{\alpha'\beta'}\del_{\alpha'}\del_{\beta'}]h_{\alpha\beta}^{\coeff,n+1}\big{\|}_{L^2(\RR^3)}
\\
&+ C\sum_{\alpha\beta}\big{\|}\nabla H^{\alpha'\beta'}(h_n^{\coeff})\big{\|}_{L^{\infty}(\RR^3)}E_{g_n}(t,\del^{I_1}\Omega^{I_2}h_{\alpha\beta}^{\coeff,n+1}).
\endaligned
$$
By Lemma~\ref{lem 7-10-2.5} and \eqref{lem 7-10-4}, and the (equi-)coercivity of $g_n$ guaranteed by $A\eps\leq \eps_0$,
$$
\aligned
\frac{d}{dt}E_{g_n}(t,\del^{I_1}\Omega^{I_2}h_{\alpha\beta}^{\coeff,n+1})
 & \leq C(\eps_0,d)(A\eps)^2 + C(\eps_0,d)A\eps \sum_{\alpha\beta}E^d_g(t,h_{\alpha\beta}^{\coeff,n+1})
\\
&+ C(\eps_0,d)A\eps E_{g_n}(t,\del^{I_1}\Omega^{I_2}h_{\alpha\beta}^{\coeff,n+1}).
\endaligned
$$
Taking the sum over the index $(I_1,I_2)$ with $|I_1|+|I_2|\leq d$ and $\alpha,\beta$:
$$
\aligned
\frac{d}{dt}\sum_{\alpha\beta}E_{g_n}^d(t,h_{\alpha\beta}^{\coeff,n+1})
 & \leq C(\eps_0,d)(A\eps)^2 + C(\eps_0,d)A\eps \sum_{\alpha\beta}E^d_{g_n}(t,h_{\alpha\beta}^{\coeff,n+1}),
\endaligned
$$
which leads to
\begin{equation}
\sum_{\alpha\beta}E_{g_n}^d(t,h_{\alpha\beta}^{\coeff,n+1})
\leq \sum_{\alpha\beta}E_{g_n}^d(0,h_{\alpha\beta}^{n+1}) e^{C(\eps_0,d)A\eps t}
+ A\eps\big(e^{C(\eps_0,d)A\eps t} - 1\big)
\end{equation}
Note that by Lemma~\ref{lem 7-10-4.5}, $E_{g_n}^d(0,h_{\alpha,\beta})$ is controlled by $\|h_0\|_{X_P^{d+1}}$  and $\|h_1\|_{X^d}$, so it can be controlled by $C(\eps_0,d)\eps$, where $C(\eps_0,d)$ is a constant depending only on $d$ and $\eps_0$.

The estimate of $E_{g_n}^d(t,\phi_{n+1}^{\coeff})$ is exactly the same, and we omit the details.

The estimate of on $\varrho_{n+1}^{\coeff}$ is similar. By the same energy estimate, we arrive at the following estimate:
\begin{equation}
\aligned
\frac{d}{dt}E_{g_n,\coeff^{-1/2}}(t,\del^{I_1}\Omega^{I_2}\varrho_{n+1}^\coeff)
& \leq C\|\del^{I_1}\Omega^{I_2}F_R\|_{L^3(\RR^3)} + C\big{\|}[\del^{I_1}\Omega^{I_2},H^{\alpha'\beta'}\del_{\alpha'}\del_{\beta'}]\varrho_{n+1}^{\coeff}\big{\|}_{L^2(\RR^3)}
\\
&+ C\sum_{\alpha\beta}\big{\|}\nabla H^{\alpha'\beta'}(h_n^{\coeff})\big{\|}_{L^{\infty}(\RR^3)}E_{g_n,\coeff^{-1/2}}(t,\del^{I_1}\Omega^{I_2}\varrho_{n+1}^\coeff),
\endaligned
\end{equation}
where $F_R$ denotes the right-hand-side of \eqref{eq conformal aug sec7-9 c}. Then, also by Lemma~\ref{lem 7-10-2.5} and \eqref{eq 1 lem 7-10-4c} and the same calculation in the estimate of  $h_{\alpha\beta}^{\coeff,n+1}$,
$$
E_{g_n,\coeff^{-1/2}}^d(t,\varrho_{n+1}^\coeff)\leq E_{g_n,\coeff^{-1/2}}^d(0,\varrho_{n+1}^{\coeff}) e^{C(\eps_0,d)A\eps t}
+ A\eps\big(e^{C(\eps_0,d)A\eps t} - 1\big).
$$ 
\end{proof}

Now we begin to make the choice of the couple $(A_0,T_0)$ such that when $A\leq A_0$, $T\leq T_0$, \eqref{eq lem 7-10-5} implies
\begin{equation}\label{eq 2 7-10}
\aligned
&\|h^{\coeff,n+1}_{\alpha\beta}(t,\cdot)\|_{E_P^{d+1}} + \|\del_t h_{\alpha\beta}^{\coeff,n+1}(t,\cdot)\|_{E^d}\leq A\eps,
\\
&\|\phi_{n+1}^\coeff(t,\cdot)\|_{E_P^{d+1}} + \|\del_t \phi_{n+1}^{\coeff}(t,\cdot)\|_{E^d}\leq A\eps,
\\
&\|\varrho_{n+1}^\coeff(t,\cdot)\|_{E_P^{d+1}} + \|\del_t \varrho_{n+1}^{\coeff}(t,\cdot)\|_{E^d} + \coeff^{-1/2}\|\varrho_{n+1}^\coeff(t,\cdot)\|_{E^d}\leq A\eps
\endaligned
\end{equation}
on the time interval $[0,T]$.

\begin{lemma}\label{lem 7-10-5.5}
There exists a couple of positive constants $(\eps_0,A_0(\eps_0,d))$, where $\eps_0$ is a universal constant and $A_0(\eps_0,d)$ is determined by $d$ and $\eps_0$ such that when $\eqref{eq 1 7-10}$ is valid with $A\geq A_0$ and $A\eps\leq \eps_0\leq 1$ on the time interval $[0,T]$ with
$$
T \leq T_0:= \frac{\ln\big(1+(2C(\eps_0,d))^{-1}\big)}{C(\eps_0,d)A\eps},
$$
where $C(\eps_0,d)$ is a constant determined by $\eps_0$ and $d$. Therefore, \eqref{eq 2 7-10} hold.
\end{lemma}

\begin{proof}
By Lemma~\ref{lem 7-10-5}, we chose $A_0(\eps_0,d)$ and $T$ such that when $A\geq A_0(\eps_0,d)$ and $t\leq T$
$$
 C(\eps_0,d)\big(\eps e^{C(\eps_0,d)A\eps t}
+ A\eps\big(e^{C(\eps_0,d)A\eps t} - 1\big)\big)\leq A\eps.
$$
This can be guaranteed by
$$
\left\{
\aligned
&e^{C(\eps_0,d)A\eps T_0}-1\leq \frac{1}{2C(\eps_0,d)},
\\
&e^{C(\eps_0,d)A\eps T_0} \leq \frac{A}{2C(\eps_0,d)}
\endaligned
\right.
$$
which is equivalent to
$$
T\leq \frac{\ln\big(1+(2C(\eps_0,d))^{-1}\big)}{C(\eps_0,d)A\eps}\, \quad A\geq 2C(\eps_0,d) +1.
$$
Then we can take $A_0 = 2C(\eps_0,d) +1$ and $T_0 = \frac{\ln\big(1+(2C(\eps_0,d))^{-1}\big)}{C(\eps_0,d)A\eps}$.
\end{proof} 

Then we are about to estimate the $\mathcal{E}_{-1}$ norm of $h_{\alpha\beta}^{\coeff,n+1}$.

\begin{lemma}\label{lem 7-10-6}
There exists a positive constant $\eps_0$ such that if \eqref{eq 1 7-10} holds with $d\geq 3$, $A\geq A_0$, $A\eps\leq \eps_0$ on the time interval $[0,T]$, $T\leq T_0$, where $(A_0,T_0)$ are constants determined in Lemma~\ref{lem 7-10-5.5}. Then the following estimate holds: 
\begin{equation}\label{eq 1 lem 7-10-6}
\|h_{\alpha\beta}^{\coeff,n+1}(t,\cdot)\|_{\mathcal{E}_{-1}}\leq C(\eps_0)T^2(A\eps)^2 + C(T+1)\eps.
\end{equation}
\end{lemma}

\begin{proof}
We will apply \eqref{eq 1 lem 7-7-1}, and we note that 
$$
\Box h_{\alpha\beta}^{\coeff,n+1} = F_H  - H^{\alpha'\beta'}\del_{\alpha'}\del_{\beta'}h_{\alpha\beta}^{\coeff,n+1}.
$$ 
Then by \eqref{eq 1 lem 7-7-1}, we have 
$$
\aligned
\|h_{\alpha\beta}^{\coeff,n+1}(t,\cdot)\|_{\mathcal{E}_{-1}}
& \leq Ct\int_0^t\|F_H(\tau,\cdot)\|_{\mathcal{E}_{-1}}d\tau
+ Ct\sum_{\alpha,\beta}\int_0^t\|H^{\alpha'\beta'}\del_{\alpha'}\del_{\beta'}h_{\alpha\beta}^{\coeff,n+1}(\tau,\cdot)\|_{\mathcal{E}_{-1}}d\tau
\\
&+ C\big(\|h_{\alpha\beta}^{\coeff,n+1}(0,\cdot)\|_{\mathcal{E}_{-1}} + t\|\nabla h_{\alpha\beta}^{\coeff,n+1}(0,\cdot)\|_{\mathcal{E}_{-1}}\big)
\endaligned
$$
Then we can apply on each term the global Sobolev inequality \eqref{eq 1 lem 7-6-1} to get estimates on the $\mathcal{E}_{-1}$ norms:
$$
\aligned
\|F_H\|_{\mathcal{E}_{-1}}\leq C\|F_H\|_{X^2}\leq C\|F_H\|_{E^2}\leq C(\eps_0)(A\eps)^2,
\endaligned
$$
where we used \eqref{eq 1 lem 7-10-2.5 a} and \eqref{eq 1 lem 7-6-1}.
We have 
$$
\aligned
\|H^{\alpha'\beta'}\del_{\alpha'}\del_{\beta'} h_{\alpha\beta}^{\coeff,n+1}(t,\cdot)\|_{\mathcal{E}_{-1}}
& \leq \|H^{\alpha'\beta'}\|_{L^{\infty}(\RR^3)}\| \del_{\alpha'}\del_{\beta'}h_{\alpha\beta}^{\coeff,n+1}\|_{X^2} 
\\
& \leq C\|H^{\alpha\beta}\|_{E_H^2}\|h_{\alpha\beta}^{\coeff,n+1}\|_{E_P^4}
\\
& \leq C(\eps_0)(A\eps)^2.
\endaligned
$$
Here, we used Lemma~\ref{lem 7-10-5.5}.

The initial terms $\|h_{\alpha\beta}^{\coeff,n+1}(0,x)\|_{\mathcal{E}_{-1}}$ and $\|\nabla h_{\alpha\beta}^{\coeff,n+1}\|_{\mathcal{B}{-1}}$ are determined by the initial data $h_0$ and $h_1$, hence, can be controlled by $C\eps$, where $C$ is a universal constant. So we conclude with the desired result.
\end{proof}

Now we can conclude that, with suitable choice of $A$ and $T$ and sufficient small $\eps$, the sequence $\{S_n\}$ is bounded with respect to the norm $\|\cdot\|_{X_S^{d+1}}$. More rigorously, the following proposition.
\begin{proposition}\label{prop 7-10-2}
There exists a couple of positive constant $(A,T)$ depends only on $\eps_0$, $\eps$ and $d$ such that if \eqref{eq 1 7-10} holds on $[0,T]$, then
\begin{equation}\label{eq 1 prop 7-10-2}
\|S^\coeff_n(t,\cdot)\|_{E_\coeff^{d+1}}\leq A\eps\leq \eps_0\leq 1,
\end{equation}
which means that the sequence of triple $\{S^{\coeff}_n\}$ is bounded in the Banach space $E_{\coeff}^{d+1}$. Furthermore, if $\eps\rightarrow 0^+$,
$$
T\rightarrow \infty.
$$
\end{proposition}
Note that the choice of $(A,T)$ are {\bf independent} of $\coeff$.
\begin{proof}
By Lemma~\ref{lem 7-10-5.5}, we take already $A\geq A_0(\eps_0,d)$ and $T\leq T_0$ such that \eqref{eq 2 7-10} holds. In order to prove \eqref{eq 1 prop 7-10-2}, we need only to guarantee, by \eqref{eq 1 lem 7-10-6}, the following inequality:
$$
 C(\eps_0)T^2(A\eps)^2 + C(T+1)\eps \leq A\eps.
$$
This can be guaranteed by
$$
\left\{
\aligned
&T\leq \frac{A}{2C} - 1,
\\
&T^2\leq \frac{1}{2C(\eps_0)A\eps}.
\endaligned
\right.
$$
So we require that $\frac{A}{2C} - 1 > 0$. Taking into consideration of the conditions in Lemma~\ref{lem 7-10-5.5}:
$$
\left\{
\aligned
&A\geq A_0(\eps_0,d) = 2C(\eps_0,d)+1
\\
&T\leq T_0 = \frac{\ln\big(1+(2C(\eps_0,d))^{-1}\big)}{C(\eps_0,d)A\eps}
\endaligned
\right.
$$
together withe the condition $A\eps\leq \eps_0$. So we see that when $\eps$ sufficiently small such that
$$
A_0(\eps_0,d)\leq \eps_0\eps^{-1/3},
$$
we can take $A = \eps_0\eps^{-1/3}$ and $T =\min\{A(2C)^{-1}-1,(2C(\eps_0)A\eps)^{-1/2}, T_0\}$. Then  there exist a constant $C'(\eps_0,d)$ such that $T \geq C'(\eps_0,d) \eps^{-1/3}$. This proves the desired result.
\end{proof}


\subsection{Contraction property and local existence}

To establish theorem \ref{prop 7-10-1}, we need to prove that the sequence $\{S_n^{\coeff}\}$ is contracting. 

\begin{proposition}\label{prop 7-12-1}
Let \eqref{eq 1 prop 7-10-2} holds with $(A,T)$ determined in by Proposition \ref{prop 7-10-2}. Assume that $d\geq 4$. Then there exist a time interval $[0, T^*]$ determined by $\eps_0, \eps$ and $d$ such that the sequence $\{S_n\}$ is contracting in the following sense:
\begin{equation}\label{eq 1 prop 7-12-1}
\|S^\coeff_{n+1} - S^\coeff_n\|_{L^{\infty}([0,T^*];X_\coeff^{d})} \leq \lambda \|S^\coeff_n(t,\cdot) - S^\coeff_{n-1}(t,\cdot)\|_{L^{\infty}([0,T^*];X_\coeff^{d})}.
\end{equation}
with a fixed $0<\lambda<1$. Furthermore, when $\eps\rightarrow 0^+$, we can take $T^*$ such that $T^*(\eps)\rightarrow +\infty$.
\end{proposition}

We emphasize that here the lower bound of the life-span-time $T^*$ given here {\bf does not} depend on the coefficient $\coeff$.

The rest of this section is mainly devoted to the proof of this proposition. To do so, we start by taking the difference of between the iteration relation for the pair $(S_{n+1}^{\coeff},S_n^{\coeff})$ and that of $(S_n^\coeff,S_{n-1}^\coeff)$. This leads to the following differential system
\begin{subequations}\label{eq difference sec7-10}
\begin{equation}\label{eq difference sec7-10 a}
\aligned
\big(m^{\alpha'\beta'} + H^{\alpha'\beta'}(h_n^\coeff)\big)\del_{\alpha'}\del_{\beta'}\big(h_{\alpha\beta}^{\coeff,n+1} - h_{\alpha\beta}^{\coeff,n}\big)
&= \big(H^{\alpha'\beta'}(h_{n-1}^\coeff) - H^{\alpha'\beta'}(h^\coeff_n)\big)\del_{\alpha'}\del_{\beta'}h_{\alpha\beta}^{\coeff,n}
\\
&+ F_H(S_n^{\coeff}) - F_H(S_{n-1}^\coeff),
\endaligned
\end{equation}
\begin{equation}\label{eq difference sec7-10 b}
\aligned
\big(m^{\alpha'\beta'} + H^{\alpha'\beta'}(h_n^\coeff)\big)\del_{\alpha'}\del_{\beta'}\big(\phi_{n+1}^\coeff - \phi_n^\coeff \big)
&= \big(H^{\alpha'\beta'}(h_{n-1}^\coeff) - H^{\alpha'\beta'}(h^\coeff_n)\big)\del_{\alpha'}\del_{\beta'}\phi_n^\coeff
\\
&+ F_P(S_n^{\coeff}) - F_P(S_{n-1}^\coeff),
\endaligned
\end{equation}
\begin{equation}\label{eq difference sec7-10 c}
\aligned
\big(m^{\alpha'\beta'} + H^{\alpha'\beta'}(h_n^\coeff)\big)
\del_{\alpha'}\del_{\beta'}\big(\varrho_{n+1}^\coeff - \varrho_n^\coeff \big)
&= \big(H^{\alpha'\beta'}(h_{n-1}^\coeff) - H^{\alpha'\beta'}(h^\coeff_n)\big)\del_{\alpha'}\del_{\beta'}\varrho_n^\coeff
\\
&+ F_R(S_n^{\coeff}) - F_R(S_{n-1}^\coeff)
\endaligned
\end{equation}
with zero initial data
$$
\aligned
&h_{\alpha\beta}^{\coeff,n+1}(0,x) - h_{\alpha\beta}^{\coeff,n}(0,x) = 0,\quad
\del_t\big(h_{\alpha\beta}^{\coeff,n+1} - h_{\alpha\beta}^{\coeff,n}\big)(0,x) = 0,
\\
&\phi_{n+1}^\coeff(0,x) - \phi_n^\coeff (0,x) = 0,\quad \del_t\big(\phi_{n+1}^\coeff-\phi_n^\coeff\big) = 0,
\\
&\varrho_{n+1}^\coeff(0,x) - \varrho_n^\coeff(0,x) = 0,\quad \del_t\big(\varrho_{n+1} - \varrho_n\big)(0,x) = 0
\endaligned
$$
\end{subequations}
For simplicity of expression, we denote by $D_H(S_n,S_{n-1})$ the right-hand-side of \eqref{eq difference sec7-10 a}, by $D_P(S_n,S_{n-1})$ the right-hand-side of \eqref{eq difference sec7-10 b} and by $D_R(S_n,S_{n-1})$ the right-hand-side of \eqref{eq difference sec7-10 c}. We need to estimate
\begin{equation}\label{eq 1 7-12}
\aligned
&\big\|h_{\alpha\beta}^{\coeff,n+1} - h_{\alpha\beta}^{\coeff,n}\big\|_{X_E^{d}},\quad
\big\|\phi_{n+1}^\coeff - \phi_n^\coeff\big\|_{X_E^{d}},\quad
\big\|\varrho_{n+1}^\coeff - \varrho_n^\coeff \big\|_{X_E^{d}},
\\
&\big\|h_{\alpha\beta}^{\coeff,n+1} - h_{\alpha\beta}^{\coeff,n}\big\|_{\mathcal{E}_{-1}}.
\endaligned
\end{equation}

First we recall the uniform bound of the sequence constructed in the last subsection:
\begin{equation}\label{eq pr1 lem 7-10-7}
\|S_n^\coeff(t,\cdot)\|_{E^{d+1}_{\coeff}}\leq A\eps\leq \eps_0\leq 1
\end{equation}
with $d\geq 3$ for $0\leq t\leq T$. We observe that this condition is equivalent to \eqref{eq 2 7-10} for all $n\in \mathbb{N}^*$.

Now we will make a series of estimates to bound the norm listed in \eqref{eq 1 7-12}.
\begin{lemma}\label{lem 7-12-1}
Let $\{S_n^\coeff\}$ be the sequence constructed by \eqref{eq conformal aug sec7-9} which satisfies the uniform bound condition \eqref{eq pr1 lem 7-10-7} with $d\geq 4$. Then the following estimate holds for $|I_1|+|I_2|\leq d-1$:
\begin{equation}\label{eq 1 lem 7-12-1}
\big\|\del^{I_1}\Omega^{I_2}D_H(S^\coeff_{n+1},S_n^\coeff)(t,\cdot)\big\|_{L^2(\RR^3)}\leq C(\eps_0,d)A\eps\|S_n^\coeff(t,\cdot) - S_{n-1}^\coeff(t,\cdot)\|_{E_\coeff^d}. 
\end{equation}
\end{lemma}

\begin{proof}
This is guaranteed by \eqref{eq 1 lem 7-10-2.47} and \eqref{eq 1 7-10}. Recall that
$$
\aligned
D_H(S_n, S_{n-1})&= -\big(H^{\alpha'\beta'}(h_{n}^\coeff) - H^{\alpha'\beta'}(h_{n-1}^\coeff)\big)\del_{\alpha'}\del_{\beta'}h_{\alpha\beta}^{\coeff,n}
\\
&+\big(F_{\alpha\beta}(h^\coeff_n,\del h_n^\coeff,\del h_n^\coeff)
 - F_{\alpha\beta}(h_{n-1},\del h_{n-1}^\coeff,\del h_{n-1}^\coeff)\big)
\\
&-16\pi \big(\del_{\alpha}\phi_n^\coeff\del_{\beta}\phi_n^\coeff - \del_{\alpha}\phi_{n-1}^\coeff\del_{\beta}\phi_{n-1}^\coeff\big)
 - 12\big(\del_{\alpha}\varrho_n^\coeff\del_{\beta}\varrho_n^\coeff - \del_{\alpha}\varrho_{n-1}^\coeff\del_{\beta}\varrho_{n-1}^\coeff\big)
\\
&- \coeff^{-1}\Big(V_h(\varrho^\coeff_n)\big(m_{\alpha\beta} + h_{\alpha\beta}^{\coeff,n}\big)
 - V_h(\varrho_{n-1}^\coeff)\big(m_{\alpha\beta} + h_{\alpha\beta}^{\coeff,n-1}\big)\Big)
\\
&=:T_1 + T_2 +T_3 + T_4.
\endaligned
$$

   We observe that $\|T_1\|_{E^d}$ is bounded by $C(\eps_0,d)A\eps \|h_n^\coeff - h_{n-1}^\coeff\|_{X_H^d}$.
$$
\aligned
\|\del^{I_1}\Omega^{I_2}T_1\|_{L^2(\RR^3)}
& \leq \sum_{J_1+J_1'=I_1\atop J_2+J_2'=I_2}\big\|\del^{J_1}\Omega^{J_2}\big(H^{\alpha'\beta'}(h_{n}^\coeff) - H^{\alpha'\beta'}(h_{n-1}^\coeff)\big)\del^{J_1'}\Omega^{J_2'}\del_{\alpha'}\del_{\beta'}h_{\alpha\beta}^{\coeff,n}\big\|_{L^2(\RR^3)}.
\endaligned
$$
When $|J_1|+|J_2|\leq d-3$, $d-1\geq |J_1'|+|J_2'|\geq 2$,
$$
\aligned
&\big\|\del^{J_1}\Omega^{J_2}\big(H^{\alpha'\beta'}(h_{n}^\coeff) - H^{\alpha'\beta'}(h_{n-1}^\coeff)\big)\del^{J_1'}\Omega^{J_2'}\del_{\alpha'}\del_{\beta'}h_{\alpha\beta}^{\coeff,n}\big\|_{L^2(\RR^3)}
\\
& \leq \big\|\del^{J_1}\Omega^{J_2}\big(H^{\alpha'\beta'}(h_{n}^\coeff) -H^{\alpha'\beta'}(h_{n-1}^\coeff)\big)\big\|_{L^{\infty}(\RR^3)}
\big\|\del^{J_1'}\Omega^{J_2'}\del_{\alpha'}\del_{\beta'}h_{\alpha\beta}^{\coeff,n}\big\|_{L^2(\RR^3)}
\\
& \leq C(\eps_0,d)\|h_n^{\coeff} - h_{n-1}^{\coeff}\|_{E_H^{d}}\|h_{\alpha\beta}^{\coeff,n}\|_{E_P^{d+1}}
\\
& \leq C(\eps_0,d)A\eps \|h_n^{\coeff} - h_{n-1}^{\coeff}\|_{E_H^{d}}
\\
& \leq C(\eps_0,d)A\eps \|S_n^\coeff - S_{n-1}^{\coeff}\|_{E_\coeff^{d}},
\endaligned
$$
Where \eqref{eq 1 lem 7-10-2.470} is applied.
\\
When $d-1\geq |J_1|+|J_2|\geq d-2$, we have $0\leq|J_1'| + |J_2'|\leq 1$. Then recall that $d\geq 4$, $|J_1'|+|J_2'|+1\leq d-2$.
Then
$$
\aligned
&\big\|\del^{J_1}\Omega^{J_2}\big(H^{\alpha'\beta'}(h_{n}^\coeff) - H^{\alpha'\beta'}(h_{n-1}^\coeff)\big)\del^{J_1'}\Omega^{J_2'}\del_{\alpha'}\del_{\beta'}h_{\alpha\beta}^{\coeff,n}\big\|_{L^2(\RR^3)}
\\
& \leq \big\|(1+r)^{-1}\del^{J_1}\Omega^{J_2}\big(H^{\alpha'\beta'}(h_{n}^\coeff) -H^{\alpha'\beta'}(h_{n-1}^\coeff)\big)\big\|_{L^2(\RR^3)}
\big\|\del^{J_1'}\Omega^{J_2'}\del_{\alpha'}\del_{\beta'}h_{\alpha\beta}^{\coeff,n}\big\|_{\mathcal{E}_{-1}}
\\
& \leq C(\eps_0,d) \|h_n^{\coeff} - h_{n-1}^{\coeff}\|_{E_H^d}
\big\|\del^{J_1'}\Omega^{J_2'}\del_{\alpha'}\del_{\beta'}h_{\alpha\beta}^{\coeff,n}\big\|_{X^2}
\\
& \leq C(\eps_0,d)\|S_n^{\coeff} - S_{n-1}^{\coeff}\|_{E_\coeff^d}\|h_{\alpha\beta}^{\coeff,n}\|_{E_P^d}
\\
& \leq C(\eps_0,d)A\eps \|S_n^{\coeff} - S_{n-1}^{\coeff}\|_{E_\coeff^d},
\endaligned
$$
where \eqref{eq 1 lem 7-6-1}, \eqref{eq 1 lem 7-10-2.461} and \eqref{eq com' lem 7-6'-1 d} are applied. Note that since of the term with second order derivative $\del_{\alpha'}\del_{\beta'}h_{\alpha\beta}^{\coeff,n}$, we can only bound the $E^d$ norm, i.e. one order of regularity is lost.

   The $E^d$ term $T_2$ and $T_3$ are bounded by \eqref{eq 1 lem 7-10-2.46} and \eqref{eq pr1 lem 7-10-7}.

   We should pay additional attention to the term $T_4$:
$$
\coeff^{-1}\Big(V_h(\varrho^\coeff_n)\big(m_{\alpha\beta} + h_{\alpha\beta}^{\coeff,n}\big)
 - V_h(\varrho_{n-1}^\coeff)\big(m_{\alpha\beta} + h_{\alpha\beta}^{\coeff,n-1}\big)\Big)
$$
The $E^d$ norm of this term can be bounded by $C(\eps_0,d)(A\eps)^2$. This is garanteed by \eqref{eq 1 lem 7-10-2} and the assumption
$$
\coeff^{-1/2}\|\varrho_k^\coeff\|_{X^d}\leq C(\eps_0,d)A\eps
$$
deduced from \eqref{eq pr1 lem 7-10-7}.
\end{proof}

\begin{lemma}\label{lem 7-12-3}
Let $\{S_n^\coeff\}$ be the sequence constructed by \eqref{eq conformal aug sec7-9} which satisfies the uniform bound condition \eqref{eq pr1 lem 7-10-7}. Then the following estimate holds for $|I_1|+|I_2|\leq d-1$:
\begin{equation}\label{eq 1 lem 7-12-3}
\big\|[\del^{I_1}\Omega^{I_2},H^{\alpha'\beta'}(h_n^\coeff)\del_{\alpha'}\del_{\beta'}]\big(h_{\alpha\beta}^{\coeff,n+1} - h_{\alpha\beta}^{\coeff,n}\big)\big\|_{L^2(\RR^3)}\leq C(\eps_0)A\eps\|h_{\alpha\beta}^{\coeff,n+1}-h_{\alpha\beta}^{\coeff,n}\|_{E_P^d}. 
\end{equation}
\end{lemma}

\begin{proof}
We perform the same calculation as in the proof of Lemma~\ref{lem 7-10-4}:
$$
\aligned
&\big|[\del^{I_1}\Omega^{I_2}, H^{\alpha'\beta'}\del_{\alpha'}\del_{\beta'}]\big(h_{\alpha\beta}^{\coeff,n+1} - h_{\alpha\beta}^{\coeff,n}\big)\big|
\\
& \leq \sum_{J_1+J_1'=I_1\atop {|J_2|+|J_2'|<|I_2|\atop \alpha',\beta',\alpha'',a}}
|\del^{J_1'}\Omega^{J_2'}H^{\alpha'\beta'}|\,|\del_{\alpha''}\del_a \del^{J_1}\Omega^{J_2}\big(h_{\alpha\beta}^{\coeff,n+1} - h_{\alpha\beta}^{\coeff,n}\big)|
\\
&+\sum_{J_1 + J_1' = I_1 \atop {J_2+J_2' = I_2 \atop |J_1'|+|J_2'|>0}}
|\del^{J_1'}\Omega^{J_2'}H^{\alpha'\beta'}|\,|\del_{\alpha'}\del_{\beta'}\del^{J_1}\Omega^{J_2}\big(h_{\alpha\beta}^{\coeff,n+1} - h_{\alpha\beta}^{\coeff,n}\big)|
\\
=:& T_1 + T_2.
\endaligned
$$

To estimate $T_1$, we observe that since $|J_1'|+|J_2'|+|J_2|+|J_2'|\leq d-2$, by \eqref{eq 1 lem 7-10-2.3} and \eqref{eq pr1 lem 7-10-7},
$$
\|\del^{J_1'}\Omega^{J_2'}H^{\alpha'\beta'}(h_{\alpha\beta}^{\coeff,n})\|_{L^{\infty}(\RR^3)}\leq C(\eps_0,d)A\eps
$$
and also since $|J_1|+|J_2|\leq d-2$:
$$
\big\|\del_{\alpha'}\del_a \del^{I_1}\Omega^{J_2}\big(h_{\alpha\beta}^{\coeff,n+1} - h_{\alpha\beta}^{\coeff,n}\big)\big\|_{L^2(\RR^3)}
\leq C\big\|h_{\alpha\beta}^{\coeff,n+1} - h_{\alpha\beta}^{\coeff,n}\big\|_{E_P^{d}}
$$
Then we see that
$$
\|T_1\|_{L^2(\RR^3)}\leq C(\eps_0)A\eps \big\|h_{\alpha\beta}^{\coeff,n+1} - h_{\alpha\beta}^{\coeff,n}\big\|_{E_P^d}.
$$

The estimate on $T_2$ is established in a bit complicated. We see that in $T_2$, $|J_1|+|J_2|\leq d-2$ so
$$
\big\|\del_{\alpha'}\del_a \del^{I_1}\Omega^{J_2}\big(h_{\alpha\beta}^{\coeff,n+1} - h_{\alpha\beta}^{\coeff,n}\big)\big\|_{L^2(\RR^3)}
\leq C\big\|h_{\alpha\beta}^{\coeff,n+1} - h_{\alpha\beta}^{\coeff,n}\big\|_{E_P^{d}}
$$
When $1\leq |J_1'|+|J_2'|\leq d-1$
$$
\|\del^{J_1'}\Omega^{J_2'}H^{\alpha'\beta'}(h_{\alpha\beta}^{\coeff,n})\|_{L^{\infty}(\RR^3)}\leq C(\eps_0,d)A\eps, 
$$
where we used \eqref{eq 1 lem 7-10-2.3} combined with \eqref{eq pr1 lem 7-10-7}.
\end{proof}

Now we are ready to estimate the term $\big\|h_{\alpha\beta}^{\coeff,n+1} - h_{\alpha\beta}^{\coeff,n}\big\|_{E_P^{d}}$.

\begin{lemma}\label{lem 7-12-4}
Let $\{S_n^\coeff\}$ be the sequence constructed by \eqref{eq conformal aug sec7-9} which satisfies the uniform bound condition \eqref{eq pr1 lem 7-10-7} with $d\geq 4$. Then the following estimate holds:
\begin{equation}\label{eq 1 lem 7-12-4}
\aligned
\frac{d}{dt}E_{g_n}^{d-1}(t,(h_{\alpha,\beta}^{\coeff,n+1} - h_{\alpha\beta}^{\coeff,n}))
& \leq C(\eps_0,d)A\eps\big\|S_n(t,\cdot) - S_{n-1}(t,\cdot)\big\|_{E_\coeff^d}
\\
&+ C(\eps_0,d)A\eps E_{g_n}^{d-1}(t,(h_{\alpha,\beta}^{\coeff,n+1} - h_{\alpha\beta}^{\coeff,n})).
\endaligned
\end{equation}
\end{lemma}

\begin{proof}
We derive the equation \eqref{eq difference sec7-10 a} with respect to a product $\del^{I_1}\Omega^{I_2}$ with $|I_1|+|I_2|\leq d-1$. Recall the relation of commutation, we get
$$
\aligned
\big(m^{\alpha'\beta'} + H^{\alpha'\beta'}(h_n^\coeff)\big)\del_{\alpha'}\del_{\beta'}\del^{I_1}\Omega^{I_2}\big(h_{\alpha\beta}^{\coeff,n+1} - h_{\alpha\beta}^{\coeff,n}\big)
&= \del^{I_2}\Omega^{I_2}\Big(\big(H^{\alpha'\beta'}(h_{n-1}^\coeff) - H^{\alpha'\beta'}(h^\coeff_n)\big)\del_{\alpha'}\del_{\beta'}h_{\alpha\beta}^{\coeff,n}\Big)
\\
&+\del^{I_1}\Omega^{I_2}\big( F_H(S_n^{\coeff}) - F_H(S_{n-1}^\coeff)\big)
\\
&- [\del^{I_1}\Omega^{I_2},H^{\alpha'\beta'}\del_{\alpha'}\del_{\beta'}]\big(h_{\alpha\beta}^{\coeff,n+1} - h_{\alpha\beta}^{\coeff,n}\big).
\endaligned
$$
Then we apply \eqref{eq 1 lem 7-7-2},
$$
\aligned
&\frac{d}{dt}E_{g_n}(t,\del^{I_1}\Omega^{I_2}\big(h_{\alpha\beta}^{\coeff,n+1} - h_{\alpha\beta}^{\coeff,n}\big))
\\
& \leq C\big\|\del^{I_1}\Omega^{I_2}D_H(S_n^\coeff,S_{n-1}^\coeff)\big\|_{L^2(\RR^3)}
 + \big\|[\del^{I_1}\Omega^{I_2},H^{\alpha'\beta'}(h_n^\coeff)\del_{\alpha'}\del_{\beta'}]\big(h_{\alpha\beta}^{\coeff,n+1} - h_{\alpha\beta}^{\coeff,n}\big)\big\|_{L^2(\RR^3)}
\\
&+ \sum_{\alpha',\beta'}\big\|\nabla H^{\alpha'\beta'}(h_n^\coeff)\big\|_{L^\infty(\RR^3)}E_{g_n}(t,\del^{I_1}\Omega^{I_2}\big(h_{\alpha\beta}^{\coeff,n+1} - h_{\alpha\beta}^{\coeff,n}\big)).
\endaligned
$$
Note that by \eqref{eq 1 lem 7-10-2.3} and \eqref{eq pr1 lem 7-10-7},
$$
\big\|\nabla H^{\alpha'\beta'}(h_n^\coeff)\big\|_{L^\infty(\RR^3)}\leq C(\eps_0,d)A\eps.
$$
Then by Lemma~\ref{lem 7-12-1} and \ref{lem 7-12-3},
$$
\aligned
\frac{d}{dt}E_{g_n}(t,\del^{I_1}\Omega^{I_2}\big(h_{\alpha\beta}^{\coeff,n+1} - h_{\alpha\beta}^{\coeff,n}\big))
& \leq C(\eps_0,d)A\eps \big\|S_n(t,\cdot) - S_{n-1}(t,\cdot)\big\|_{E_\coeff^d}
\\
&+ C(\eps_0,d)A\eps\big\|h_{\alpha,\beta}^{\coeff,n+1} - h_{\alpha\beta}^{\coeff,n}\big\|_{E_P^d}
\\
&+  C(\eps_0)A\eps E_{g_n}(t,\del_x^{I_1}\Omega^{I_2}(h_{\alpha,\beta}^{\coeff,n+1} - h_{\alpha\beta}^{\coeff,n})).
\endaligned
$$
Then by taking the sum over all the pair of multi-index $(I_1,I_2)$ with $|I_1|+|I_2|\leq d-1$, and observe that (by \eqref{eq pr1 lem 7-10-7}):
$$
\|h_{\alpha,\beta}^{\coeff,n+1} - h_{\alpha\beta}^{\coeff,n}\big\|_{E_P^d}
\leq C(\eps_0,d)E_{g_n}^{d-1}(t,(h_{\alpha,\beta}^{\coeff,n+1} - h_{\alpha\beta}^{\coeff,n})),
$$
and the desired result is proven.
\end{proof}

The estimates on $\big(\phi_{n+1}^\coeff - \phi_n^\coeff\big)$ and $\big(\varrho_{n+1}^\coeff - \varrho_n^\coeff\big)$ are established in the same manner.

\begin{lemma}\label{lem 7-12-5}
Let $\{S_n^\coeff\}$ be the sequence constructed by \eqref{eq conformal aug sec7-9} which satisfies the uniform bound condition \eqref{eq pr1 lem 7-10-7} with $d\geq 4$. Then the following estimates hold:
\begin{subequations}\label{eq 1 lem 7-12-5}
\begin{equation}\label{eq 1 lem 7-12-5 a}
\frac{d}{dt}E_{g_n}^{d-1}(t,(\phi_{n+1}^\coeff - \phi_n^\coeff))
\leq C(\eps_0,d)A\eps\big\|S_n(t,\cdot) - S_{n-1}(t,\cdot)\big\|_{E_\coeff^d}
+ C(\eps_0,d)A\eps E_{g_n}^{d-1}(t,(\phi_{n+1}^\coeff - \phi_n^\coeff)),
\end{equation}
\begin{equation}\label{eq 1 lem 7-12-5 b}
\frac{d}{dt}E_{g_n,\coeff^{-1/2}}^{d-1}(t,(\varrho_{n+1}^\coeff - \varrho_n^\coeff))
\leq C(\eps_0,d)A\eps\big\|S_n(t,\cdot) - S_{n-1}(t,\cdot)\big\|_{E_\coeff^d}
+ C(\eps_0,d)A\eps E_{g_n,\coeff^{-1/2}}^{d-1}(t,(\varrho_{n+1}^\coeff - \varrho_n^\coeff)).
\end{equation}
\end{subequations}
\end{lemma}

At this juncture, we can finally estimate the $\mathcal{E}_{-1}$ norm of $\big(h_{\alpha\beta}^{\coeff,n+1} - h_{\alpha\beta}^{\coeff,n}\big)$.

\begin{lemma}\label{lem 7-12-6}
Let $\{S_n^\coeff\}$ be the sequence constructed by \eqref{eq conformal aug sec7-9} which satisfies the uniform bound condition \eqref{eq pr1 lem 7-10-7} with $d\geq 4$. Then the following estimate holds:
\begin{equation}\label{eq 1 lem 7-12-6}
\big\|h_{\alpha\beta}^{\coeff,n+1} - h_{\alpha\beta}^{\coeff,n}\big\|_{\mathcal{E}_{-1}}
\leq C(\eps_0)t^2A\eps\big\|S_n^\coeff - S_{n-1}^\coeff\big\|_{E_\coeff^d}.
\end{equation}
\end{lemma}

\begin{proof} We are going to apply \eqref{eq 1 lem 7-7-1}. To do so we need to establish the following estimates:
\begin{equation}\label{eq pr1 lem 7-12-6}
\aligned
\big\|H^{\alpha'\beta'}(h_n^\coeff)\del_{\alpha'}\del_{\beta'}\big(h_{\alpha\beta}^{\coeff,n+1} - h_{\alpha\beta}^{\coeff,n}\big)\big\|_{\mathcal{E}_{-1}}
\leq C(\eps_0)A\eps \|h_{\alpha\beta}^{\coeff,n+1} - h_{\alpha\beta}^{\coeff,n}\|_{E_H^d},
\endaligned
\end{equation}
\begin{equation}\label{eq pr1 lem 7-12-6}
\big\|D_H(S_n^\coeff,S_{n-1}^\coeff)\big\|_{\mathcal{E}_{-1}}\leq C(\eps_0)A\eps \big\|S_{n}^\coeff - S_{n-1}^\coeff\big\|_{E_\coeff^d}.
\end{equation}
We see that \eqref{eq pr1 lem 7-12-6} follows from \eqref{eq 1 lem 7-12-1} and \eqref{eq 1 lem 7-6-1}.
To establish \eqref{eq pr1 lem 7-12-6}, we see that
$$
\aligned
\big\|H^{\alpha'\beta'}(h_n^\coeff)\del_{\alpha'}\del_{\beta'}\big(h_{\alpha\beta}^{\coeff,n+1} - h_{\alpha\beta}^{\coeff,n}\big)\big\|_{\mathcal{E}_{-1}}
& \leq \big\|H^{\alpha'\beta'}(h_n^\coeff)\big\|_{\mathcal{E}_{-1}} \big\|\del_{\alpha'}\del_{\beta'}\big(h_{\alpha\beta}^{\coeff,n+1} - h_{\alpha\beta}^{\coeff,n}\big)\big\|_{L^{\infty}(\RR^3)}
\\
& \leq C(\eps_0)A\eps \|h_{\alpha\beta}^{\coeff,n+1} - h_{\alpha\beta}^{\coeff,n}\|_{E_P^4},
\endaligned
$$
where \eqref{eq com' lem 7-6'-1} is applied.
\end{proof}

\begin{proof}[Proof of Proposition \ref{prop 7-12-1}]
We integrate \eqref{eq 1 lem 7-12-4}, \eqref{eq 1 lem 7-12-5 a} and \eqref{eq 1 lem 7-12-5 b} and get the following estimates:
\begin{subequations}\label{eq pr1 prop 7-12-1}
\begin{equation}\label{eq pr1 prop 7-12-1 a}
E_{g_n}^{d-1}(t,(h_{\alpha,\beta}^{\coeff,n+1} - h_{\alpha\beta}^{\coeff,n}))
\leq \big(e^{C(\eps_0,d)A\eps t} - 1\big) \big\|S_n^\coeff -S_{n-1}^\coeff\big\|_{L^{\infty}([0,T^*];E_\coeff^{d})}
\end{equation}
\begin{equation}\label{eq pr1 prop 7-12-1 b}
E_{g_n}^{d-1}(t,(\phi_{n+1}^\coeff - \phi_n^\coeff))
\leq \big(e^{C(\eps_0,d)A\eps t} - 1\big) \big\|S_n^\coeff -S_{n-1}^\coeff\big\|_{L^{\infty}([0,T^*];E_\coeff^{d})}
\end{equation}
\begin{equation}\label{eq pr1 prop 7-12-1 c}
E_{g_n,\coeff^{-1/2}}^{d-1}(t,(\varrho_{n+1}^\coeff -\varrho_n^\coeff))
\leq \big(e^{C(\eps_0,d)A\eps t} - 1\big) \big\|S_n^\coeff -S_{n-1}^\coeff\big\|_{L^{\infty}([0,T^*];E_\coeff^{d})}
\end{equation}
\end{subequations}

Recall that the metric $g_n$ is coercive with constant $C(\eps_0)$. We have 
$$
\aligned
&\|S_{n+1}^\coeff(t,\cdot) -S_n^\coeff(t,\cdot)\|_{E_\coeff^d}
\\
& \leq C(\eps_0,d)
\max\{E_{g_n}^{d-1}(t,(h_{\alpha,\beta}^{\coeff,n+1} - h_{\alpha\beta}^{\coeff,n})), E_{g_n}^{d-1}(t,(\phi_{n+1}^\coeff - \phi_n^\coeff)), E_{g_n,\coeff^{-1/2}}^{d-1}(t,(\varrho_{n+1}^\coeff -\varrho_n^\coeff))\}.
\endaligned
$$
Then we conclude with 
\begin{equation}
\|S_{n+1}^\coeff -S_n^\coeff\|_{L^{\infty}([0,T^*];E_\coeff^{d})}
\leq C(\eps_0,d)
\big(e^{C(\eps_0,d)A\eps T^*} - 1\big) \big\|S_n^\coeff - S_{n-1}^\coeff\big\|_{L^{\infty}([0,T^*];E_\coeff^{d})}.
\end{equation}
Then if we choose
$$
T^*= \frac{\ln \big(1+(2C(\eps_0))^{-1}\big)}{C(\eps_0)A\eps},
$$
then
$$
\lambda := e^{C(\eps_0)A\eps t} - 1 = 1/2<1, 
$$
which satisfies the contraction condition. Furthermore, recall that in Proposition \ref{prop 7-10-2} we can take $A = \eps_0\eps^{-1/3}$ and $T \geq C'(\eps_0,d)\eps^{-1/3}$ when $\eps$ sufficiently small. So here we can also take $T^* = C''(\eps_0,d)\eps^{-1/3}$ for $\eps$ sufficiently small. This leads to the limit of $T^*(\eps)$ when $\eps\rightarrow 0^+$.
\end{proof}
 
Now we apply the fixed point theorem of Banach and see that $\{S_n^\coeff\}$ converges to a triple $S^{\coeff} := (h^{\coeff},\phi^{\coeff},\varrho^{\coeff})$ in the sense of $L^{\infty}([0,T^*],E_\coeff^d)$. Then we will prove that $S^{\coeff}$ is a solution of \eqref{eq conformal aug sec7-9}.

\begin{proposition}\label{prop 7-10-3}
When $d\geq 4$,
the $S^{\coeff}$ constructed above is a solution of \eqref{eq conformal aug sec7-9} in the sense of distribution. Furthermore,
$$
\aligned
&h^{\coeff}_{\alpha\beta} \in C([0,T_0],E_H^d),\quad \del_t^k h^{\coeff}_{\alpha\beta} \in C([0,T_0],E_P^{d-k}), k\leq d
\\
&\phi^{\coeff} \in C([0,T_0],E_P^d),\quad \del_t^k h^{\coeff}_{\alpha\beta} \in C([0,T_0],E_P^{d-k}), k\leq d
\\
&\varrho^{\coeff} \in C([0,T_0],E_R^d),\quad \del_t^k h^{\coeff}_{\alpha\beta} \in C([0,T_0],E_R^{d-k}), k\leq d.
\endaligned
$$
Furthermore, $\|S^{\coeff}(t,\cdot)\|_{E_\coeff^d}\leq A\eps$ with $0\leq t\leq T^*$.
\end{proposition}

\begin{proof}
The proof is based on taking the limit in both side of \eqref{eq conformal aug sec7-9}. The convergence of $\{S_n^{\coeff}\}$ in sense of $L^{\infty}([0,T^*];E_\coeff^d)$ can guarantee the convergence of both sides of \eqref{eq conformal aug sec7-9}.
Recall that the sequence $\{\del_t\del_t h_{\alpha\beta}^{\coeff,n}\}$ also converges in the sense $L^{\infty}([0,T^*];E^{d-2})$ and so does $\{\del_t\del_t\phi_n^\coeff\}$ and $\{\del_t\del_t\varrho_n^\coeff\}$.

The convergence of $\{S_n^\coeff\}$ in $E_\coeff^d$ guarantees the following convergence (remark that $\coeff\leq 1$:
\begin{equation}
\aligned
&h_{\alpha\beta}^{\coeff,n}\rightarrow h_{\alpha\beta}^\coeff \text{ in } L^{\infty}([0,T^*]; E_H^d),\quad
\nabla h_{\alpha\beta}^{\coeff,n}\rightarrow \nabla h_{\alpha\beta}^\coeff \text{ in } L^{\infty}([0,T^*]; E^{d-1}),
\\
&\phi_n^\coeff\rightarrow \phi^\coeff \text{ in } L^{\infty}([0,T^*]; E_P^d),\quad
\nabla \phi_n^\coeff\rightarrow \nabla \phi^\coeff \text{ in } L^{\infty}([0,T^*]; E^{d-1}),
\\
&\varrho_n^\coeff\rightarrow \varrho^\coeff \text{ in } L^{\infty}([0,T^*]; E_P^d)\cap L^{\infty}([0,T^*]; E^{d-1}),\quad
\nabla \varrho_n^\coeff\rightarrow \nabla \varrho^\coeff \text{ in } L^{\infty}([0,T^*]; E^{d-1}).
\endaligned
\end{equation}
Here, $\nabla$ denotes the spacetime divergence. By Sobolev embedding ($d-1\geq 2$), $\{h_n^\coeff\}$, $\{\nabla\phi_n^\coeff\}$ and $\{\varrho^\coeff_n\}$ converges in $L^{\infty}([0,T^*]\times\RR^3)$.
Furthermore, we have
\begin{equation}
\aligned
&\del_t\del_x h_{\alpha\beta}^{\coeff,n}\rightarrow \del_t\del_x h_{\alpha\beta}^\coeff \text{ in } L^{\infty}([0,T^*]; E^{d-2}),
\quad \del_x\del_x h_{\alpha\beta}^{\coeff,n}\rightarrow \del_x\del_x h_{\alpha\beta}^\coeff \text{ in } L^{\infty}([0,T^*]; E^{d-2})
\\
&\del_t\del_x \phi_n^\coeff \rightarrow \del_t\del_x \phi^\coeff \text{ in } L^{\infty}([0,T^*]; E^{d-2}),
\quad \del_x\del_x \phi_n^\coeff \rightarrow \del_x\del_x \phi^\coeff \text{ in } L^{\infty}([0,T^*]; E^{d-2})
\\
&\del_t\del_x h_{\alpha\beta}^{\coeff,n}\rightarrow \del_t\del_x h_{\alpha\beta}^\coeff \text{ in } L^{\infty}([0,T^*]; E^{d-2}),
\quad \del_x\del_x h_{\alpha\beta}^{\coeff,n}\rightarrow \del_x\del_x h_{\alpha\beta}^\coeff \text{ in } L^{\infty}([0,T^*]; E^{d-2}).
\endaligned
\end{equation}
These convergence properties are sufficient to guarantee the convergence of both side of \eqref{eq conformal aug sec7-9} since both side depend linearly the terms with second order derivatives. And the lower order terms converge in $L^\infty$ sense.
\end{proof}

\begin{proof}[Proof of Theorem \ref{prop 7-10-1}] We have checked that the triple $S^\coeff$ is a local solution of \eqref{eq conformal aug sec7}. Furthermore, we notice that the lower bound of life-span-time $T^*$ constructed in Proposition \ref{prop 7-10-3} does not depend on $\coeff$. The estimates are established by taking the limit of the \eqref{eq pr1 lem 7-10-7}.
\end{proof}


\newpage 

\section{Comparing the $f(R)$ theory to the classical theory}
\label{sec convergence}

\subsection{Statement of the main estimate}

In this section, we compare the solutions given by the $f(R)$ theory with the solutions of the classical Einstein theory. We denote by $S^0 := (h_{\alpha\beta},\phi)$ the triple determined by the following Cauchy problem: 
\begin{subequations}\label{eq Einstein-wave}
\begin{equation}\label{eq Einstein-wave a}
\big(m^{\alpha'\beta'}+H^{\alpha'\beta'}(h)\big)\del_{\alpha'}\del_{\beta'}h_{\alpha\beta} = F_{\alpha\beta}(h,\del h,\del h) -16\pi\del_{\alpha}\phi\del_\beta\phi,
\end{equation}
\begin{equation}\label{eq Einstein-wave b}
\big(m^{\alpha'\beta'}+H^{\alpha'\beta'}(h)\big)\del_{\alpha'}\del_{\beta'}\phi = 0
\end{equation}
\end{subequations}
with initial data
$$
\aligned
&h_{\alpha\beta}(0,x) = {h_0}_{\alpha\beta},\quad &\del_th_{\alpha\beta}(0,x) = {h_1}_{\alpha\beta},
\\
&\phi(0,x) = \phi_0, \quad &\del_t\phi(0,x) = \phi_1.
\endaligned
$$
This imiting problem is defined by replacing $\rho^\coeff$ by $0$ in our formulation (6.3). 
As before, if the initial data satisfies the corresponding constraint conditions, 
then $g_{\alpha\beta} = m_{\alpha\beta} + h_{\alpha\beta}$ and $\phi$ satisfy the classical Einstein's field equation coupled with the real masse less scalar field $\phi$. For the convenience of discussion we introduce the norm
$$
\big\|S_0\big\|_{X_0^{d+1}} := \max\{\|{h_0}_{\alpha\beta}\|_{X_H^{d+1}},\|{h_1}_{\alpha\beta}\|_{X^d},\,\|\phi_0\|_{X_P^{d+1}},\|\phi_1\|_{X^d}\}.
$$

\begin{proposition}[Local existence theory for the classical gravity system]
\label{thm 7-13-0}
Suppose that $({h_0}_{\alpha\beta},{h_1}_{\alpha\beta})\in X_H^{d+1}\times X^{d}$ and $(\phi_0,\phi_1)\in X_P^{d+1}\times X^d$ and with $d\geq 4$. Denote by $S_0 = (h_0,h_1,\phi_0,\phi_1)$ and assume that
$$
\big\|S_0\big\|_{X_0^{d+1}}\leq \eps\leq \eps_0\leq 1
$$
with a sufficiently small $\eps_0$. Then there exist positive constants $A, T^*$ determined from $\eps_0$, $\eps$ and $d$ such that the Cauchy problem \eqref{eq Einstein-wave} with initial data $S_0$ has a unique solution (in sense of distribution) $(h_{\alpha\beta},\phi)$ in the time interval $[0,T^*]$. Here
$$
\aligned
h_{\alpha\beta}\in C([0,T^*];E_H^d),\quad \phi\in C([0,T^*];E_P^d).
\endaligned
$$
When $\eps\rightarrow 0^+$, we can take
$$
\lim_{\eps\rightarrow 0^+}T^* = +\infty.
$$
Furthermore, the local solution satisfies the following estimates in the time interval $[0, T^*]$:
\begin{equation}\label{eq 1 thm 7-13-0}
\aligned
\big\|h_{\alpha\beta} \big\|_{E_H^d}+\big\|\phi\big\|_{E_P^d}  \leq A\eps.
\endaligned
\end{equation}
\end{proposition}

The proof is similar to that of Theorem \ref{prop 7-10-1}: we make an iteration and estimate the sequence constructed by this iteration and we prove that with suitable choice of $(A, T^*)$, this sequence is contracting. The details of the argument are omitted. 

Let $S_0 = ({h^0_0}_{\alpha\beta},{h^0_1}_{\alpha\beta},\phi^0_0,\phi^0_1)$ be an initial data which satisfies the Einstein's constraint equation \eqref{GC-3} and $S_1=({h_0}_{\alpha\beta},{h_1}_{\alpha\beta},\phi_0,\phi_1,\varrho_0,\varrho_1)$ be an initial data which satisfies the nonlinear constraint equations \eqref{constraint trans-aug Hamiltonian} and \eqref{constraint trans-aug momentum}. Define the following function $\mathcal{D}_{\coeff}(S_0,S_1)$:
$$
\aligned
\mathcal{D}_{\coeff}^d(S_0,S_1):= \max\{& \big\|{h^0_0}_{\alpha\beta} - {h_0}_{\alpha\beta}\big\|_{X_H^{d+1}},\,
\big\|{h^0_1}_{\alpha\beta} - {h_1}_{\alpha\beta}\big\|_{X^d},\,\big\|\phi^0_0 - \phi_0\big\|_{X_P^{d+1}},\,
\\
&\big\|\phi^0_1 - \phi_1\big\|_{X_P^{d+1}},\big\|\varrho_0\big\|_{X_P^{d+1}},\big\|\varrho_1\big\|_{X^d}, \coeff^{-1/2}\big\|\varrho_0\big\|_{X^d}\}.
\endaligned
$$
Denote by
$$
S^0(t) = (h^0_{\alpha\beta}(t),\phi^0(t))\in  C([0,T^*];X_H^d)\cap C^1([0,T];X_H^{d-1})\times C([0,T^*];X_P^d)\cap C^1([0,T];X_P^{d-1}).
$$
 the local solution of Cauchy problem \eqref{eq Einstein-wave} with initial data $S^0(0) = S_0$, and
$$
\aligned
S^\coeff(t)
&= (h^\coeff_{\alpha\beta},\phi^\coeff,\varrho^\coeff)\in C([0,T^*];X_H^d)\cap C^1([0,T];X_H^{d-1})\times C([0,T^*];X_P^d)\cap C^1([0,T];X_P^{d-1})
\\
&\times C([0,T^*];X_P^d)\cap C^1([0,T];X_P^{d-1})\cap C([0,T^*];X_R^{d-1}).
\endaligned
$$
We introduce the ``distance'' form $S^0$ to $S^\coeff$:
$$
\mathcal {D}^{d}(S^0,S^\coeff)(t)
:= \sum_{\alpha\beta}\|h^0_{\alpha\beta} - h^\coeff_{\alpha\beta}\|_{E_H^d} + \|\phi^0-\phi^\coeff\|_{E_P^d}
$$
and we are ready to state the key estimate derived in the present work.

\begin{theorem}[Comparison estimate]\label{thm 7-13-1}
There exists a positive constant $\eps_0$ such that if
$$
\max\{\|S_1\|_{X_\coeff^{d+1}},\,\|S_0\|_{X_0^{d+1}}\} \leq \eps\leq \eps_0\leq 1
$$
(with $d\geq 4$), then in the common interval of existence $[0,T^*]$ (which depends only on $\eps, \eps_0$ and $d$), the following estimates hold:
\begin{equation}\label{eq 1 thm 7-13-1}
\mathcal{D}^{d-1}(S^0,S^\coeff)(t)
\leq C(\eps_0,d)\big(\mathcal{D}^{d-1}(S^0,S^\coeff)(0)
+\big( E_{g,\coeff^{-1/2}}^{d-2}(0,\varrho^\coeff)\big)^2 + \coeff^{1/2}(A\eps)^3\big).
\end{equation}
\end{theorem}


\subsection{Proof of the comparison estimate}

The proof of Theorem \ref{thm 7-13-1} requires better estimates on $\|\varrho_{n+1}^{\coeff}\|_{X^{d-1}}$.  
First we establish an improved bound on the $L^2$ norm of $\del_t\del^{I_1}\Omega^{I_2}\varrho^\coeff$. The following lemma is immediate from \eqref{eq 1 prop 7-10-1 c}.

\begin{lemma}\label{lem 7-13-1}
Let $S^{\coeff} = (h_{\alpha\beta}^\coeff, \phi^\coeff, \varrho^\coeff)$ be the solution of Cauchy problem \eqref{eq conformal aug sec7} with $d\geq 4$. Then the following estimate holds for all $|I_1|+|I_2|\leq d-2$:
\begin{equation}\label{eq 1 lem 7-13-1}
\coeff^{-1/2}\|\del_t\del^{I_1}\Omega^{I_2}\varrho^\coeff\|_{L^2(\RR^3)} \leq C(\eps_0,d)A\eps.
\end{equation}
\end{lemma}

\begin{lemma}\label{lem 7-13-2}
There exists a positive constant $\eps_0$ such that if \eqref{eq 1 7-10} holds for $d\geq 4$ and $A\eps\leq \eps_0$, then
\begin{equation}\label{eq 1 lem 7-13-2}
\|\varrho^\coeff(t,\cdot)\|^2_{E^{d-2}}
\leq \coeff^{3/2}C(\eps_0, d)t(A\eps)^3 + C(\eps_0,d)\coeff \big(E_{g,\coeff^{-1/2}}^{d-2}(0, \varrho^\coeff)\big)^2.
\end{equation}
\end{lemma}
\begin{proof}
This is a modified energy estimate. We derive the equation \eqref{eq conformal aug sec7 c} with respect to $\del^{I_1}\Omega^{I_2}$ with $|I_1|+|I_2|\leq d-2$. 
As in the proof of Lemma~\ref{lem 7-7-4}:
$$
\aligned
&\del_t\del^{I_1}\Omega^{I_2}\varrho^\coeff \big(g^{\alpha\beta}\del_{\alpha}\del_{\beta}\del^{I_1}\Omega^{I_2}\varrho^\coeff - 3\coeff^{-1}\del^{I_1}\Omega^{I_2}\varrho^\coeff\big)
\\
&= \frac{1}{2}\del_0\big(g^{00}(\del_0 \del^{I_1}\Omega^{I_2}\varrho^\coeff)^2 - g^{ab}\del_a \del^{I_1}\Omega^{I_2}\varrho^\coeff\del_b \del^{I_1}\Omega^{I_2}\varrho^\coeff\big) - \frac{1}{2}\del_0\big((3\coeff)^{-1/2}\del^{I_1}\Omega^{I_2}\varrho^\coeff\big)^2
\\
& \quad + \del_a\big(g^{a\beta}\del^{I_1}\Omega^{I_2}\varrho^\coeff \del_{\beta} \del^{I_1}\Omega^{I_2}\varrho^\coeff\big)
+ \frac{1}{2}\del_tg^{\alpha\beta}\del_{\alpha}\del^{I_1}\Omega^{I_2}\varrho^\coeff\del_{\beta}\del^{I_1}\Omega^{I_2}\varrho^\coeff
\\
&\quad - \del_{\alpha}g^{\alpha\beta}\del_0\del^{I_1}\Omega^{I_2}\varrho^\coeff \del_{\beta}\del^{I_1}\Omega^{I_2}\varrho^\coeff.
\endaligned
$$
For simplicity,  we set $v = \del^{I_1}\Omega^{I_2}\varrho^\coeff$ and obtain 
$$
\aligned
&\del_tv \big(\del^{I_1}\Omega^{I_2}F_R(h^\coeff, \phi^\coeff, \varrho^\coeff)\big)
\\
&= \frac{1}{2}\del_t\big(g^{00}(\del_t v)^2 - g^{ab}\del_a v\del_b v\big) - \frac{1}{2}\del_t\big((3\coeff)^{-1/2}v\big)^2
+ \del_a\big(g^{a\beta}\del_t v \del_{\beta}v\big)
\\
&\quad
+ \frac{1}{2}\del_tg^{\alpha\beta}\del_{\alpha}v\del_{\beta}v - \del_{\alpha}g^{\alpha\beta}\del_tv \del_{\beta}v. 
\endaligned
$$
Integrating this equation in the region $[0,t]\times \RR^3$ and using Stokes' formula, we obtain 
\begin{equation}\label{eq pr1 lem 7-13-1}
\aligned
&\int_0^t\int_{\RR^3}\del_t v \big(\del^{I_1}\Omega^{I_2}F_R(h^\coeff, \phi^\coeff, \varrho^\coeff)\big)dx dt
\\
&=\frac{1}{2}\big(E_{g,\coeff^{-1/2}}(t,\del^{I_1}\Omega^{I_2} \varrho^\coeff)\big)^2 - \frac{1}{2}\big(E_{g,\coeff^{-1/2}}(0,\del^{I_1}\Omega^{I_2} \varrho^\coeff)\big)^2
\\
 &\quad + \int_0^t\int_{\RR^3}\big(\frac{1}{2}\del_tg^{\alpha\beta}\del_{\alpha}v\del_{\beta}v  - \del_{\alpha}g^{\alpha\beta}\del_tv \del_{\beta}v\big)dxdt
\endaligned
\end{equation}
Then by \eqref{eq 1 lem 7-10-2}, \eqref{eq 1 lem 7-10-2.4}, \eqref{eq 1 lem 7-13-1} and \eqref{eq 1 prop 7-10-1}, we see that
$$
\bigg|\int_0^t\int_{\RR^3}\del_t v \big(\del^{I_1}\Omega^{I_2}F_R(h^\coeff, \phi^\coeff, \varrho^\coeff)\big)dx dt\bigg|
 \leq C(\eps_0, d)t(A\eps)^3\coeff^{1/2}.
$$
Also by \eqref{eq 1 prop 7-10-1} and the fact that $d\geq 4$
$$
\big|\del_\alpha g^{\alpha\beta}\big|_{L^\infty(\RR^3)}\leq C(\eps_0)\|h_{\alpha\beta}\|_{E_P^{d-1}}\leq C(\eps_0)A\eps,
$$
then
$$
\bigg|\int_0^t\int_{\RR^3}\big(\frac{1}{2}\del_tg^{\alpha\beta}\del_{\alpha}v\del_{\beta}v  - \del_{\alpha}g^{\alpha\beta}\del_tv \del_{\beta}v\big)dxdt\bigg| \leq C(\eps_0,d)t(A\eps)^3\coeff.
$$
Combining these two estimate with \eqref{eq pr1 lem 7-13-1}, we see that (with $0<\coeff\leq 1$)
$$
\big(E_{g,\coeff^{-1/2}}(t,\del^{I_1}\Omega^{I_2} \varrho^\coeff)\big)^2
\leq \big(E_{g,\coeff^{-1/2}}(0,\del^{I_1}\Omega^{I_2} \varrho^\coeff)\big)^2 + C(\eps_0,d)t(A\eps)^3\coeff^{1/2}.
$$
In view of the definition of $(E_{g,\coeff^{-1/2}}(t,u)$, this leads us to 
$$
\coeff^{-1}\|\del^{I_1}\Omega^{I_2}\varrho^{\coeff}(t,\cdot)\|^2_{L^2(\RR^3)}\leq C(\eps_0,d)\big(E_{g,\coeff^{-1/2}}(0,\del^{I_1}\Omega^{I_2} \varrho^\coeff)\big)^2 + C(\eps_0,d)t(A\eps)^3\coeff^{1/2}.
$$
\end{proof}

\begin{proof}[Proof of Theorem \ref{thm 7-13-1}]
By taking the difference of \eqref{eq conformal aug sec7 a} and \eqref{eq Einstein-wave a}, and the difference of \eqref{eq conformal aug sec7 b} and \eqref{eq Einstein-wave b}, we obtain the following two equations:
\begin{subequations}\label{eq pr1 thm 7-13-1}
\begin{equation}\label{eq pr1 thm 7-13-1 a}
\aligned
&\big(m^{\alpha'\beta'} + H^{\alpha'\beta'}(h^0)\big)\del_{\alpha'}\del_{\beta'}\big((h^0_{\alpha\beta} - h^\coeff_{\alpha\beta}\big)
\\
&= -\big(H^{\alpha'\beta'}(h^0) - H^{\alpha'\beta'}(h^\coeff)\big)\del_{\alpha'}\del_{\beta'}h^\coeff_{\alpha\beta}
 + \big(F_{\alpha\beta}(h^0,\del h^0,\del h^0) - F_{\alpha\beta}(h^\coeff,\del h^\coeff,\del h^\coeff)\big)
\\
 & \quad -16\pi\big(\del_{\alpha}\phi^0\del_{\beta}\phi^0 - \del_{\alpha}\phi^\coeff\del_{\beta}\phi^\coeff\big)
 +12 \del_{\alpha}\varrho^{\coeff}\del_\beta\varrho^\coeff + \coeff^{-1}V_h(\varrho^\coeff)\big(m_{\alpha\beta}+h_{\alpha\beta}^\coeff\big),
\endaligned
\end{equation}
\begin{equation}\label{eq pr1 thm 7-13-1 b}
\aligned
\big(m^{\alpha'\beta'} + H^{\alpha'\beta'}(h^0)\big)\del_{\alpha'}\del_{\beta'}\big(\phi^0 - \phi^\coeff\big)
&= -\big(H^{\alpha'\beta'}(h^0) - H^{\alpha'\beta'}(h^\coeff)\big)\del_{\alpha'}\del_{\beta'}\phi^\coeff
\\
 &-2\big(m^{\alpha'\beta'} + H^{\alpha'\beta'}(h^\coeff)\big)\del_{\alpha'}\phi^\coeff\del_\beta\varrho^\coeff.
\endaligned
\end{equation}
\end{subequations}
Then we will establish the estimates \eqref{eq 1 thm 7-13-1} from these two equations. The proof is quite similar to that of Lemma~\ref{lem 7-12-4} and Lemma~\ref{lem 7-12-5}.

\

\noindent{\bf Step I. The estimates on $L^2$ type norms.}
Let us begin with the $E^{d-2}_P$ norm of $\big(h_{\alpha\beta}^0 - h_{\alpha\beta}^\coeff\big)$. Let $(I_1,I_2)$ be a pair of multi-indices with $|I_1| + |I_2|\leq d-2$. We derive \eqref{eq pr1 thm 7-13-1} with respect to $\del^{I_1}\Omega^{I_2}$:
$$
\aligned
&\big(m^{\alpha'\beta'} + H^{\alpha'\beta'}(h^0)\big)\del_{\alpha'}\del_{\beta'}\del^{I_1}\Omega^{I_2}\big(h^0_{\alpha\beta} - h^\coeff_{\alpha\beta}\big)
\\
&= -\del^{I_1}\Omega^{I_2}\Big(\big(H^{\alpha'\beta'}(h^0) - H^{\alpha'\beta'}(h^\coeff)\big)\del_{\alpha'}\del_{\beta'}h^\coeff_{\alpha\beta}\Big)
+ \del^{I_1}\Omega^{I_2}\big(F_{\alpha\beta}(h^0,\del h^0,\del h^0) - F_{\alpha\beta}(h^\coeff,\del h^\coeff,\del h^\coeff)\big)
\\
 & \quad -16\pi\del^{I_1}\Omega^{I_2}\big(\del_{\alpha}\phi^0\del_{\beta}\phi^0 - \del_{\alpha}\phi^\coeff\del_{\beta}\phi^\coeff\big)
 -[\del^{I_1}\Omega^{I_2},H^{\alpha'\beta'}(h^0)\del_{\alpha'}\del_{\beta'}]\big(h^0_{\alpha\beta} - h^\coeff_{\alpha\beta}\big)
\\
 & \quad +12 \del^{I_2}\Omega^{I_2}\big(\del_{\alpha}\varrho^{\coeff}\del_\beta\varrho^\coeff\big) + \coeff^{-1}\del_x^{I_1}\Omega^{I_2}\Big(V_h(\varrho^\coeff)\big(m_{\alpha\beta}+h_{\alpha\beta}^\coeff\big)\Big)
\\
 &=: T_1 + T_2 + T_3 + T_4 + T_5 + T_6.
\endaligned
$$
Then combined with \eqref{eq 1 lem 7-7-2},
\begin{equation}\label{eq pr1.5 thm 7-13-1}
\frac{d}{dt}E_g(t,\del^{I_1}\Omega^{I_2}\big(h^0_{\alpha\beta} - h^\coeff_{\alpha\beta}\big))
\leq C\sum_{i=1}^6\|T_i\|_{L^2} + C\sum_{\alpha,\beta}\big\|H^{\alpha'\beta'}(h^0)\big\|_{L^\infty(\RR^3)}E_g(t,\del^{I_1}\Omega^{I_2}\big(h^0_{\alpha\beta} - h^\coeff_{\alpha\beta}\big)).
\end{equation}
We will need to control the $L^2$ norm of these $T_i$ for $i=1,\cdots,6$. The term $T_i$ for $i=1,2,3,4$ can be bounded as follows:
\begin{equation}\label{eq pr2 thm 7-13-1}
\aligned
\big\|T_i\big\|_{L^2(\RR^3)}& \leq C(\eps_0,d)A\eps\big(\|h^0 - h^\coeff\|_{E_H^{d-1}} + \|h^0 - h^\coeff\|_{E_P^{d-1}} \|\phi^0-\phi^\coeff\|_{E_P^{d-1}}\big)
\\
& \leq  C(\eps_0,d)A\eps \mathcal{D}^{d-1}(S^0,S^\coeff)(t).
\endaligned
\end{equation}
The proof is exactly the same to the one of \eqref{eq 1 lem 7-12-1} and \eqref{eq 1 lem 7-12-3} and we omit the details.

The {\sl key terms} $T_5$ and $T_6$ are bounded as follows:
\begin{equation}\label{eq pr3 thm 7-13-1}
\aligned
&\|T_5\|_{L^2(\RR^3)}\leq C(\eps_0,d) \coeff(A\eps)^2,
\\
&\|T_6\|_{L^2(\RR^3)}\leq C(\eps_0,d) \coeff^{1/2}(A\eps)^3
 + C(\eps_0,d)\big(E_{g,\coeff^{-1/2}}^{d-2}(0,\varrho^\coeff)\big)^2
\endaligned
\end{equation} 

The estimates on $T_5$ and $T_6$ is related to the refined estimates \eqref{eq 1 lem 7-13-1} and \eqref{eq 1 lem 7-13-2}. More precisely, $T_5$ is estimated by \eqref{eq 1 lem 7-10-2}, \eqref{eq 1 prop 7-10-1 b} and \eqref{eq 1 lem 7-13-1}.
The term $T_6$ is estimated by \eqref{eq 1 lem 7-10-2} and \eqref{eq 1 lem 7-13-2}. 

Next, we combine together the above estimates on $T_i$ and observe that
$$
\big\|H^{\alpha'\beta'}(h^0)\big\|_{L^{\infty}}\leq C(\eps_0)A\eps
$$
and
$$
E_g(t,\del^{I_1}\Omega^{I_2}\big(h^0_{\alpha\beta} - h^\coeff_{\alpha\beta}\big))\leq C(\eps_0,d)\mathcal{D}^{d-1}(S^0,S^\coeff)(t). 
$$
We can thus deduce from \eqref{eq pr1.5 thm 7-13-1} that
\begin{equation}\label{eq pr4 thm 7-13-1}
\aligned
\frac{d}{dt}E^{d-2}_{g}(t,\big(h^0_{\alpha\beta} - h_{\alpha\beta}^\coeff\big))
& \leq C(\eps_0,d)A\eps \mathcal {D}^{d-1}(S^0,S^\coeff)(t)
+ C(\eps_0,d)\coeff^{1/2}(A\eps)^3
\\
&+ C(\eps_0,d)\big( E_{g,\coeff^{-1/2}}^{d-2}(0,\varrho^\coeff)\big)^2,
\endaligned
\end{equation}
where we recall the definition 
$$
\mathcal {D}^{d}(S^0,S^\coeff)(t)
:= \sum_{\alpha\beta}\|h^0_{\alpha\beta} - h^\coeff_{\alpha\beta}\|_{E_H^d}
 + \|\phi^0-\phi^\coeff\|_{E_P^d}.
$$

The estimate on the norm $\|\phi^0-\phi^1\|_{X_E^d}$ is similar to that of $h^0-h^\coeff$ (even simpler). We claim that the following estimate on the right-hand-side of \eqref{eq pr1 thm 7-13-1 b}:
\begin{subequations}\label{eq pr5 thm 7-13-1}
\begin{equation}\label{eq pr5 thm 7-13-1 a}
\big\|\big(H^{\alpha'\beta'}(h^0) - H^{\alpha'\beta'}(h^\coeff)\big)\del_{\alpha'}\del_{\beta'}\phi^\coeff\big\|_{E^{d-2}}
\leq C(\eps_0)A\eps\|h^0-h^\coeff\|_{E_H^{d-1}},
\end{equation}
\begin{equation}\label{eq pr5 thm 7-13-1 b}
\big\|\big(m^{\alpha\beta}+H^{\alpha\beta}(h^\coeff)\big)\del_{\alpha}\phi^\coeff \del_\beta \varrho^\coeff\big\|_{E^{d-2}}\leq C(\eps_0)\coeff^{1/2}(A\eps)^2.
\end{equation}
\end{subequations}
The first can be proved exactly as in the proof of \eqref{eq pr4 thm 7-13-1}. The second one is proven as follows: for any $(I_1,I_2)$ with $|I_1|+|I_2|\leq d-2$,
$$
\aligned
&\big\|\del^{I_2}\Omega^{I_2}\big((m^{\alpha\beta}+H^{\alpha\beta}(h^\coeff))\del_{\alpha}\phi^\coeff \del_\beta \varrho^\coeff\big)\big\|_{L^2(\RR^3)}
\\
& \leq \sum_{J_1+J_1'=I_1\atop J_2+J_2'=I_2}\big\|\del^{J_1}\Omega^{J_2}\big(m^{\alpha\beta} + H^{\alpha\beta}(h^\coeff)\big)\big\|_{L^{\infty}(\RR^3)}
\big\|\del^{J_1'}\Omega^{J_2'}\big(\del_{\alpha}\phi^\coeff\del_{\beta}\varrho^\coeff\big)\big\|_{L^2(\RR^3)}
\\
& \leq C(\eps_0,d)\sum_{\alpha,\beta\atop |J_1'|+|J_2'|\leq d-2}\big\|\del^{J_1'}\Omega^{J_2'}\big(\del_{\alpha}\phi^\coeff\del_{\beta}\varrho^\coeff\big)\big\|_{L^2(\RR^3)}
\\
& \leq C(\eps_0,d)\sum_{\alpha,\beta\atop |J_1'|+|J_2'|\leq d-2}\sum_{K_1+K_1'=J_1'\atop K_2+K_2'=J_2'}
\big\|\del^{K_1}\Omega^{K_2}\del_{\alpha}\phi^\coeff\del_x^{K_1'}\Omega^{K_2'}\del_{\beta}\varrho^\coeff\big\|_{L^2(\RR^3)}
\endaligned
$$
Then, when $|K_1|+|K_2|\leq d-3$,
$$
\aligned
\big\|\del^{K_1}\Omega^{K_2}\del_{\alpha}\phi^\coeff\del^{K_1'}\Omega^{K_2'}\del_{\beta}\varrho^\coeff\big\|_{L^2(\RR^3)}
& \leq \big\|\del^{K_1}\Omega^{K_2}\del_{\alpha}\phi^\coeff\big\|_{L^{\infty}(\RR^3)}
\big\|\del^{K_1'}\Omega^{K_2'}\del_{\beta}\varrho^\coeff\big\|_{L^2(\RR^3)}
\\
& \leq C(\eps_0,d)A\eps \|\varrho^\coeff\|_{E^{d-1}}
\\
& \leq C(\eps_0,d)\coeff^{1/2}(A\eps)^2.
\endaligned
$$
When $|K_1|+|K_2|=d-2$ and $K_1'=K_2'=0$, recall that $d\geq 4$:
$$
\aligned
\big\|\del^{K_1}\Omega^{K_2}\del_{\alpha}\phi^\coeff\del^{K_1'}\Omega^{K_2'}\del_{\beta}\varrho^\coeff\big\|_{L^2(\RR^3)}
& \leq \big\|\del^{K_1}\Omega^{K_2}\del_{\alpha}\phi^\coeff\big\|_{L^2(\RR^3)}
\big\|\del_{\beta}\varrho^\coeff\big\|_{L^\infty(\RR^3)}
\\
& \leq C(\eps_0,d)A\eps \big\|\varrho^\coeff\big\|_{E^{d-1}}
\leq C(\eps_0,d)\coeff^{1/2}(A\eps)^2
\endaligned
$$
So we conclude with \eqref{eq pr5 thm 7-13-1 b}, and combined with \eqref{eq 1 lem 7-7-2},
\begin{equation}\label{eq pr6 thm 7-13-1}
\aligned
\frac{d}{dt}E^{d-2}_{g}(t,\big(\phi^0 - \phi^\coeff\big))
& \leq C(\eps_0,d)A\eps \mathcal {D}^{d-1}(S^0,S^\coeff)(t) + C(\eps_0,d)\coeff^{1/2}(A\eps)^2.
\endaligned
\end{equation}

\

\noindent
{\bf Step II. Estimate of the $\mathcal{E}_{-1}$ norm.} To do so we rewrite the equation \eqref{eq pr1 thm 7-13-1 a} into the following form:
\begin{equation}\label{eq pr7 thm 7-13-1}
\aligned
\Box \big((h^0_{\alpha\beta} - h^\coeff_{\alpha\beta}\big)
&= - H^{\alpha'\beta'}(h^0)\del_{\alpha'}\del_{\beta'}\big((h^0_{\alpha\beta} - h^\coeff_{\alpha\beta}\big)
 -\big(H^{\alpha'\beta'}(h^0) - H^{\alpha'\beta'}(h^\coeff)\big)\del_{\alpha'}\del_{\beta'}h^\coeff_{\alpha\beta}
\\
 &\quad+ \big(F_{\alpha\beta}(h^0,\del h^0,\del h^0) - F_{\alpha\beta}(h^\coeff,\del h^\coeff,\del h^\coeff)\big)
 -16\pi\big(\del_{\alpha}\phi^0\del_{\beta}\phi^0 - \del_{\alpha}\phi^\coeff\del_{\beta}\phi^\coeff\big)
\\
 & \quad +12 \del_{\alpha}\varrho^{\coeff}\del_\beta\varrho^\coeff + \coeff^{-1}V_h(\varrho^\coeff)\big(m_{\alpha\beta}+h_{\alpha\beta}^\coeff\big)
\\&=: T_0 + T_1 + T_2 + T_3 + T_4 + T_5 + T_6.
\endaligned
\end{equation}
By \eqref{eq 1 lem 7-7-1} we need to control the $\mathcal{E}_{-1}$ norm of the terms $T_i$ for $i=1,\cdots 6$. By \eqref{eq 1 lem 7-6-1}, we need only to control the $X^2$ norm of these terms. Recall the condition $d\geq 4$, then $d-2\geq 2$. So we only need to control the $X^{d-2}$ norm of these terms. Note that in the {\bf Step I} we have already controlled this norm for the terms $T_i$ with $i\geq 1$. Now we only need to control the $X^2$ norm of $T_0$.
Let $(I_1,I_2)$ be a pair of multi-indices with $|I_1|+|I_2|\leq 2$. Then, we have 
$$
\aligned
&\big\|\del_x^{I_1}\Omega^{I_2}\big(H^{\alpha'\beta'}(h^0)\del_{\alpha'}\del_{\beta'}(h_{\alpha\beta}^0-h_{\alpha\beta}^\coeff)\big)\big\|_{L^2(\RR^3)}
\\
& \leq  \sum_{J_1+J_1'=I_1\atop J_2+J_2'=I_2}\big\|\del_x^{J_1}\Omega^{J_2}\big(H^{\alpha'\beta'}(h^0)\big)
\del_x^{J_1'}\Omega^{J_2'}\del_{\alpha'}\del_{\beta'}(h_{\alpha\beta}^0-h_{\alpha\beta}^\coeff)\big)\big\|_{L^2(\RR^3)}
\\
& \leq  \sum_{J_1+J_1'=I_1\atop J_2+J_2'=I_2}\big\|\del_x^{J_1}\Omega^{J_2}\big(H^{\alpha'\beta'}(h^0)\big)\big\|_{L^\infty(\RR^3)}
\big\|\del_x^{J_1'}\Omega^{J_2'}\del_{\alpha'}\del_{\beta'}(h_{\alpha\beta}^0-h_{\alpha\beta}^\coeff)\big)\big\|_{L^2(\RR^3)}
\\
& \leq  C(\eps_0)\|h^0\|_{X^4}\|h^0-h^\coeff\|_{E_P^{3}}
\\
& \leq  C(\eps_0)A\eps\|h^0-h^\coeff\|_{E_P^{d-1}}.
\endaligned
$$
Then by \eqref{eq 1 lem 7-7-1}, the following estimate on $\mathcal{E}_{-1}$ norm holds:
\begin{equation}\label{eq pr7.5 thm 7-13-1}
\|h_{\alpha\beta}^0(t,\cdot) - h_{\alpha\beta}^\coeff(t,\cdot)\|_{\mathcal{E}_{-1}}
\leq C(\eps_0)t A\eps \int_0^t \mathcal{D}^{d-1}(S^0,S^\coeff)(\tau) d\tau
+ C(1+t)\big(\|{h^0_0}_{\alpha\beta} - {h_0}_{\alpha\beta}\|_{X_H^{d-1}}\big)
\end{equation}

\

\noindent
{\bf Step III: the conclusion.}
Now by integrating \eqref{eq pr4 thm 7-13-1} and \eqref{eq pr7 thm 7-13-1}, we get the following estimate:
\begin{subequations}\label{eq pr8 thm 7-13-1}
\begin{equation}\label{eq pr8 thm 7-13-1 a}
\aligned
E_{g}^{d-2}(t,h_{\alpha\beta}^0 - h_{\alpha\beta}^{\coeff})& \leq 
E_{g}^{d-2}(0,h_{\alpha\beta}^0 - h_{\alpha\beta}^{\coeff})
+C(\eps_0,d)A\eps \int_0^t \mathcal{D}^{d-1}(S^0,S^\coeff)(\tau)d\tau
\\
&\quad + C(\eps_0,d)\coeff^{1/2}(A\eps)^3t + C(\eps_0,d)t\big( E_{g,\coeff^{-1/2}}^{d-2}(0,\varrho^\coeff)\big)^2, 
\endaligned
\end{equation}
\begin{equation}\label{eq pr8 thm 7-13-1 b}
\aligned
E^{d-2}_{g}(t,\phi^0 - \phi^\coeff)
& \leq E^{d-2}_{g}(0,\big(\phi^0 - \phi^\coeff\big))
+C(\eps_0,d)A\eps \int_0^t \mathcal {D}^{d-1}(S^0,S^\coeff)(\tau)d\tau
\\
&+ C(\eps_0,d)\coeff^{1/2}(A\eps)^2t.
\endaligned
\end{equation}
\end{subequations}
Recall that $g$ is coercive with constant $C(\eps_0)$ when $\eps_0$ is sufficiently small. Then, we have 
\begin{equation}\label{eq pr9 thm 7-13-1}
\aligned
& \mathcal{D}^{d-1}(S^0 - S^\coeff)(t)
\\
& \leq C(\eps_0,d)\sum_{\alpha,\beta}E_{g}^{d-2}(t,h_{\alpha\beta}^0 - h_{\alpha\beta}^{\coeff})
 + C(\eps_0,d)\sum_{\alpha,\beta}\|h_{\alpha\beta}^0(t,\cdot) - h_{\alpha\beta}^\coeff(t,\cdot)\|_{\mathcal{E}_{-1}}
\\
&\quad
+C(\eps_0,d)E^{d-2}_{g}(t,\phi^0 - \phi^\coeff)\leq C(\eps_0,d)^2\mathcal{D}^{d-1}(S^0 - S^\coeff)(t).
\endaligned
\end{equation}
Then by combining \eqref{eq pr7.5 thm 7-13-1}, \eqref{eq pr8 thm 7-13-1 a}, \eqref{eq pr8 thm 7-13-1 b} and \eqref{eq pr9 thm 7-13-1}, the following estimate holds:
\begin{equation}\label{eq pr10 thm 7-13-1}
\aligned
\mathcal{D}^{d-1}(S^0,S^\coeff)(t)
& \leq C(\eps_0,d)(1+T^*)\mathcal{D}^{d-1}(S^0,S^\coeff)(0)
+ C(\eps_0,d)T^*\big( E_{g,\coeff^{-1/2}}^{d-2}(0,\varrho^\coeff)\big)^2
\\
&\quad+ C(\eps_0,d)T^*\coeff^{1/2}(A\eps)^3
\\
  &\quad+C(\eps_0,d)(1+T^*)A\eps\int_0^t\mathcal{D}^{d-1}(S^0,S^\coeff)(\tau)d\tau, 
\endaligned
\end{equation}
which yields 
$$
\aligned
&\mathcal{D}^{d-1}(S^0,S^\coeff)(t)
\\
& \leq C(\eps_0,d)(1+T^*)\Big(\mathcal{D}^{d-1}(S^0,S^\coeff)(0)
+\big( E_{g,\coeff^{-1/2}}^{d-2}(0,\varrho^\coeff)\big)^2 + \coeff^{1/2}(A\eps)^3\Big)e^{C(\eps_0,d)A\eps(1+T^*)t}. 
\endaligned
$$ 
\end{proof}
 

\section*{Acknowledgments}

The authors were partially supported by the Agence Nationale de la Recherche (ANR) through the grant 06-2-134423  and ANR SIMI-1-003-01. Part of this research was done in the Fall Semester 2013 when the first author was a visiting professor at the Mathematical Sciences Research Institute (Berkeley) and was supported by the National Science Foundation under Grant No. 0932078 000.


\end{document}